\documentclass{SCAE}
\numberwithin{equation}{section}

\allowdisplaybreaks
\def\rr{{\mathbb R}}
\def\rn{{{\rr}^n}}
\def\zz{{\mathbb Z}}
\def\cc{{\mathbb C}}

\def\nn{{\mathbb N}}

\def\cb{{\mathcal B}}

\def\ce{{\mathcal E}}

\def\cg{{\mathcal G}}

\def\cl{{\mathcal L}}

\def\mr{{\mathcal R}}

\def\cx{{\mathcal X}}

\def\fz{\infty}
\def\az{\alpha}
\def\bz{\beta}
\def\dz{\delta}

\def\ez{\epsilon}
\def\gz{{\gamma}}

\def\kz{\kappa}
\def\lz{\lambda}

\def\tz{\theta}
\def\sz{\sigma}
\def\vz{\varphi}

\def\lf{\left}
\def\r{\right}

\def\hs{\hspace{0.25cm}}
\def\ls{\lesssim}

\def\noz{\nonumber}
\def\wz{\widetilde}

\def\st{\subset}

\def\bh{\backslash}
\def\bp{\bigcup}
\def\gfz{\genfrac{}{}{0pt}{}}

\def\supp{\mathop\mathrm{\,supp\,}}

\def\diam{\mathrm{\,diam}\,}
\def\dint{\displaystyle\int}

\def\lp{{L^p(\mu)}}
\def\lq{{L^q(\mu)}}

\def\lpt{{L^{p_2}(\mu)}}

\def\lon{{L^1(\mu)}}
\def\ltw{{L^2(\mu)}}

\def\cer{{{\mathcal E}^{1/p-1}_{\rho}(\mu)}}

\def\cera{{{\mathcal E}^{\alpha}_{\rho}(\mu)}}
\def\ceag{{{\mathcal E}^{\alpha,\,q}_{\rho,\,\gz}(\mu)}}
\def\cear{{{\mathcal E}^{\alpha,\,q}_{\rho}(\mu)}}

\def\ceaeg{{{\mathcal E}^{\alpha,\,q}_{\rho,\,\eta,\,\gz}(\mu)}}
\def\ceaog{{{\mathcal E}^{\alpha,\,q}_{\rho,\,\eta_1,\,\gz}(\mu)}}
\def\ceatg{{{\mathcal E}^{\alpha,\,q}_{\rho,\,\eta_2,\,\gz}(\mu)}}
\def\ceago{{{\mathcal E}^{\alpha,\,q}_{\rho_1,\,\eta,\,\gz}(\mu)}}
\def\ceagt{{{\mathcal E}^{\alpha,\,q}_{\rho_2,\,\eta,\,\gz}(\mu)}}
\def\li{{L^{\infty}(\mu)}}
\def\lip{{\mathop\mathrm{\,Lip}}}

\def\rbmo{\mathop\mathrm{\,{\rm RBMO}(\mu)}}

\def\kbsp{{\wz K^{(\rho),\,p}_{B,\,S}}}
\def\kbjp{{\wz K^{(\rho),\,p}_{B_j,\,B}}}
\def\kbsa{{\wz K^{(\rho),\,1/(\az+1)}_{B,\,S}}}

\def\hoq{{H_{\rm atb}^{1,\,q}(\mu)}}

\def\nhoq{{\wz H_{\rm atb}^{1,\,q}(\mu)}}

\def\hp{H_{\rm{atb},\,\rho}^{1,\,q,\,\gz}(\mu)}

\def\nhp{\wz H_{\rm{atb},\,\rho}^{p,\,q,\,\gz}(\mu)}

\def\nhop{\wz H_{\rm{atb},\,\rho}^{p_1,\,q,\,\gz}(\mu)}

\def\pnhp{\wz {\mathbb{H}}_{\rm{atb},\,\rho}^{p,\,q,\,\gz}(\mu)}

\def\pnhpo{\wz {\mathbb{H}}_{\rm{atb},\,\rho}^{p,\,q_1,\,\gz}(\mu)}
\def\pnhpt{\wz {\mathbb{H}}_{\rm{atb},\,\rho}^{p,\,q_2,\,\gz}(\mu)}

\def\mhp{\wz H_{\rm{mb},\,\rho}^{p,\,q,\,\gz,\,\ez}(\mu)}

\def\mhop{\wz H_{\rm{mb},\,\rho}^{p_1,\,q,\,\gz,\,\tz}(\mu)}
\def\mhpd{\wz H_{\rm{mb},\,2}^{p,\,q,\,1,\,\dz}(\mu)}
\def\mhpe{\wz H_{\rm{mb},\,2}^{p,\,q,\,1,\,\ez}(\mu)}
\def\mhpod{\wz H_{\rm{mb},\,2}^{p_1,\,q,\,1,\,\tz}(\mu)}

\def\pmhp{\wz {\mathbb{H}}_{\rm{mb},\,\rho}^{p,\,q,\,\gz,\,\ez}(\mu)}
\def\pmhpo{\wz {\mathbb{H}}_{\rm{mb},\,\rho}^{p,\,q_1,\,\gz,\,\ez}(\mu)}
\def\pmhpt{\wz {\mathbb{H}}_{\rm{mb},\,\rho}^{p,\,q_2,\,\gz,\,\ez}(\mu)}

\def\hoq{{H_{\rm atb}^{1,\,q}(\mu)}}
\def\nhoq{{\wz H_{\rm atb}^{1,\,q}(\mu)}}
\def\nhp{\wz H_{\rm{atb},\,\rho}^{p,\,q,\,\gz}(\mu)}
\def\mhp{\wz H_{\rm{mb},\,\rho}^{p,\,q,\,\gz,\,\ez}(\mu)}

\def\hhp{\widehat H_{\rm{atb},\,\rho}^{p,\,q,\,\gz}(\mu)}
\def\hhpq{\widehat H_{\rm{atb}}^{p,\,q}(\mu)}
\def\hmp{\widehat H_{\rm{mb},\,\rho}^{p,\,q,\,\gz,\,\ez}(\mu)}

\def\loc{{\mathop\mathrm{loc\,}}}
\def\dsum{\displaystyle\sum}

\begin{document}

\Year{2015} %
\Month{January}
\Vol{58} %
\No{1} %
\BeginPage{1} %
\EndPage{XX} %
\AuthorMark{Fu X {\it et al.}}
\ReceivedDay{August 7, 2014}
\AcceptedDay{November 19 , 2014}
\DOI{10.1007/s11425-014-4956-2} 

\title{Hardy spaces $H^p$ over
non-homogeneous metric measure spaces and their applications}{}


\author{FU Xing$^1$}{}
\author{LIN Haibo$^2$}{}
\author{YANG Dachun$^{1,}$}{Corresponding author}
\author{YANG Dongyong$^3$}{}

\address[{\rm1}]{School of Mathematical Sciences, Beijing Normal University,
Laboratory of Mathematics\\
and Complex Systems, Ministry of
Education, Beijing {\rm 100875}, China;}
\address[{\rm2}]{College of Science, China Agricultural University, Beijing {\rm 100083}, China;}
\address[{\rm3}]{School of Mathematical Sciences,
Xiamen University, Xiamen 361005, China}
\Emails{xingfu@mail.bnu.edu.cn, haibolincau@126.com, dcyang@bnu.edu.cn,
dyyang@xmu.edu.cn}\maketitle


 {\begin{center}
\parbox{14.5cm}{\begin{abstract}
Let $({\mathcal X},d,\mu)$
be a metric measure space
satisfying both the geometrically doubling and the upper doubling
conditions. Let $\rho\in (1,\infty)$, $0<p\le1\le q\le\infty$,
$p\neq q$, $\gamma\in[1,\infty)$ and $\epsilon\in(0,\infty)$.
In this article, the authors introduce
the atomic Hardy space
${\widetilde H_{\mathrm{atb},\,\rho}^{p,\,q,\,\gamma}(\mu)}$
and the molecular Hardy space
${\widetilde H_{\rm{mb},\,\rho}^{p,\,q,\,\gamma,\,\epsilon}(\mu)}$
via the discrete coefficient $\widetilde{K}^{(\rho),\,p}_{B,\,S}$,
and prove that the Calder\'on-Zygmund operator is bounded
from ${\widetilde H_{\rm{mb},\,\rho}^{p,\,q,\,\gamma,\,\delta}(\mu)}$
(or ${\widetilde H_{\rm{atb},\,\rho}^{p,\,q,\,\gamma}(\mu)}$)
into $L^p(\mu)$, and from
${\widetilde H_{\rm{atb},\,\rho(\rho+1)}^{p,\,q,\,\gamma+1}(\mu)}$ into
${\widetilde H_{\rm{mb},\,\rho}^{p,\,q,\,\gamma,\,\frac12(\delta
-\frac{\nu}{p}+\nu)}(\mu)}$.
The boundedness of the generalized fractional integral
$T_{\beta}$ ($\beta\in(0,1)$) from
$\widetilde H_{\rm{mb},\,\rho}^{p_1,\,q,\,\gamma,\,\theta}(\mu)$
(or ${\widetilde H}^{p_1,\,q,\,\gamma}_{{\rm atb},\,\rho}(\mu)$)
into $L^{p_2}(\mu)$ with $1/p_2=1/p_1-\beta$ is
also established.
The authors also introduce the $\rho$-weakly doubling condition,
with $\rho\in (1,\infty)$, of the measure $\mu$ and
construct a non-doubling measure $\mu$
satisfying this condition.
If $\mu$ is $\rho$-weakly doubling, the authors further
introduce the Campanato
space ${\mathcal E}^{\alpha,\,q}_{\rho,\,\eta,\,\gamma}(\mu)$
and show that ${\mathcal E}^{\alpha,\,q}_{\rho,\,\eta,\,\gamma}(\mu)$
is independent of the choices of $\rho$, $\eta$, $\gamma$ and $q$;
the authors then introduce the atomic Hardy space
$\widehat H_{\rm{atb},\,\rho}^{p,\,q,\,\gamma}(\mu)$
and the molecular Hardy space
$\widehat H_{\rm{mb},\,\rho}^{p,\,q,\,\gamma,\,\epsilon}(\mu)$,
which coincide with each other; the authors finally prove that
$\widehat{H}_{\rm{atb},\,\rho}^{p,\,q,\,\gamma}(\mu)$
is the predual of ${\mathcal E}^{1/p-1,\,1}_{\rho,\,\rho,\,1}(\mu)$.
Moreover, if $\mu$ is doubling, the authors show that
${\mathcal E}^{\alpha,\,q}_{\rho,\,\eta,\,\gamma}(\mu)$
and the Lipschitz space ${\rm Lip}_{\alpha,\,q}(\mu)$ ($q\in[1,\infty)$),
or $\widehat{H}_{\rm{atb},\,\rho}^{p,\,q,\,\gamma}(\mu)$
and the atomic Hardy space $H^{p,\,q}_{\rm at}(\mu)$
($q\in(1,\infty]$) of R. R. Coifman and G. Weiss coincide.
Finally, if $(\mathcal{X},d,\mu)$ is an RD-space
with $\mu(\mathcal{X})=\infty$, the authors prove that
$\widetilde H_{\rm{atb},\,\rho}^{p,\,q,\,\gamma}(\mu)$,
$\widetilde H_{\rm{mb},\,\rho}^{p,\,q,\,\gamma,\,\epsilon}(\mu)$
and $H^{p,\,q}_{\rm at}(\mu)$ coincide
for any $q\in(1,2]$. In particular, when
$(\mathcal{X},d,\mu):=(\mathbb{R}^D,|\cdot|,dx)$ with $dx$ being
the $D$-dimensional Lebesgue measure, the authors show that spaces
$\widetilde H_{\rm{atb},\,\rho}^{p,\,q,\,\gamma}(\mu)$,
$\widetilde H_{\rm{mb},\,\rho}^{p,\,q,\,\gamma,\,\epsilon}(\mu)$,
$\widehat{H}_{\rm{atb},\,\rho}^{p,\,q,\,\gamma}(\mu)$
and $\widehat{H}_{\rm{mb},\,\rho}^{p,\,q,\,\gamma,\,\epsilon}(\mu)$
all coincide with $H^p(\mathbb{R}^D)$ for any $q\in(1,\infty)$.
\vspace{-3mm}
\end{abstract}}\end{center}}

 \keywords{geometrically doubling measure, upper doubling measure,
$\rho$-weakly doubling measure, non-homogeneous metric measure space,
RD-space, Hardy space, Campanato space, Lipschitz space,
Calder\'on-Zygmund operator, fractional integral, predual,
atomic block, molecular block}

 \MSC{42B30, 42B20, 42B35, 30L99}

\renewcommand{\baselinestretch}{1.2}
\begin{center} \renewcommand{\arraystretch}{1.5}
{\begin{tabular}{lp{0.8\textwidth}} \hline \scriptsize
{\bf Citation:}\!\!\!\!&\scriptsize Fu X, Lin H, Yang Da, Yang Do.
Hardy spaces $H^p$ over non-homogeneous metric measure spaces and their applications.
Sci China Math, 2015, 58, doi: 10.1007/s11425-000-0000-0\vspace{1mm}
\\
\hline
\end{tabular}}\end{center}


\baselineskip 11pt\parindent=10.8pt  \wuhao

\section{Introduction}\label{s1}

\hskip\parindent It is well known that the real
variable theory of Hardy spaces
$H^p(\rr^D)$ on the $D$-dimensional Euclidean space
$\rr^D$ has many important applications
in various fields of analysis such as harmonic analysis and
partial differential equations; see, for example, \cite{sw,s70,fs,s93}.
When $p\in(1,\fz)$, $L^p(\rr^D)$ and $H^p(\rr^D)$ are essentially
the same; however, when $p\in(0,1]$, the space $H^p(\rr^D)$ is much better
adapted to problems arising in the theory of the
boundedness of operators, since some of singular
integrals (for example, Riesz transforms) are bounded on $H^p(\rr^D)$,
but not on $L^p(\rr^D)$. In 1972, Fefferman
and Stein \cite{fs} showed that the Hardy space
$H^1(\rr^D)$ is the predual of the space BMO($\rr^D$).
Later, Walsh \cite{w73} proved that the dual space of the Hardy space
$H^p(\rr^D)$ is the Campanato space introduced by Campanato \cite{c64}.
From then on, various characterizations of $H^p(\rr^D)$,
including the atomic and the molecular characterizations, and
their applications were studied extensively in harmonic analysis;
see, for example, \cite{co1,co2,la78,tw,ns,gly13,lhd,yy12,cy}.
Moreover, the atomic and the molecular
characterizations enabled the extension of the
real variable theory of Hardy spaces on $\rr^D$ to spaces of
homogeneous type in the sense of
Coifman and Weiss \cite{cw71, cw77}, which is a far more general setting for
function spaces and singular integrals than Euclidean spaces.

Recall that a metric space $(\cx, d)$ equipped with a non-negative measure
$\mu$ is called  a {\it space of homogeneous type}, if $(\cx, d,
\mu)$ satisfies the {\it measure doubling condition}:
there exists a positive constant $C_{(\mu)}$ such that, for all balls
$B(x,r):= \{y\in\cx:\ \, d(x, y)< r\}$ with $x\in\cx$ and $r\in(0, \fz)$,
\begin{equation}\label{1.1}
\mu(B(x, 2r))\le C_{(\mu)} \mu(B(x,r)).
\end{equation}
This measure doubling condition is
one of the most crucial assumptions in the classical harmonic analysis.
We point out that a space of homogeneous type in \cite{cw71,cw77} is
endowed with a quasi-metric. However, for simplicity, throughout
this article, we \emph{always assume} that a space of homogeneous type
is endowed with a metric.

Nevertheless, in recent years, it has been proved that
many results in the classical theory of
Hardy spaces and singular integrals on $\rr^D$
remain valid with the $D$-dimensional
Lebesgue measure replaced by a non-doubling measure (see, for example,
\cite{t01a,t01b,t03a,t03b,t04,t05,t14,ntv,cmy,hmy}).
Recall that a Radon measure $\mu$ on $\rr^D$
is called  a \emph{non-doubling measure}, if there exist positive
constants $C_0$ and $\kappa\in(0,D]$
such that, for all $x\in\rr^D$ and $r\in (0,\fz)$,
\begin{equation}\label{1.2}
\mu(B(x,r))\le C_0r^{\kz},
\end{equation}
where $B(x,r):=\{y\in\rr^D:\ |y-x|<r\}$.
Tolsa \cite{t01a,t03a} introduced the atomic Hardy space $ \hoq$,
for $q\in(1, \fz]$, and its dual space, $\rbmo$,
the \emph{space of functions with regularized bounded mean oscillation},
with respect to $\mu$ as in \eqref{1.2},
and proved that Calder\'on-Zygmund operators are bounded
from $ \hoq$ into $L^1(\mu)$. Later, Chen, Meng and Yang \cite{cmy}
showed that Calder\'on-Zygmund operators are bounded
on $ \hoq$. In \cite{hmy}, Hu, Meng and Yang established an
equivalent characterization of $ \hoq$ to
obtain the $L^q(\mu)$-boundedness of
commutators and their endpoint estimates. More
research on function spaces, mainly on Morrey spaces, and
their applications related to non-doubling measures
can be found in \cite{gs13, ss13, st05, st07, st07-2, st09, st09-2}.
We point out that the analysis on such non-doubling context
plays a striking role in solving
several long-standing problems related to the analytic capacity,
like Vitushkin's conjecture
or Painlev\'e's problem; see \cite{t03b, t04, t05, t14}.

However, as was pointed out by Hyt\"onen in \cite{h10},
the measure satisfying
\eqref{1.2} is different from, but not more general than,
the doubling measure.
In \cite{h10}, Hyt\"onen  introduced
a new class of metric measure spaces satisfying the so-called
geometrically doubling and the upper doubling conditions (see, respectively,
Definitions \ref{d2.1} and \ref{d2.3} below), which are also simply called
\emph{non-homogeneous metric measure spaces}. This new class of
non-homogeneous metric measure spaces
includes both spaces of homogeneous type and
metric spaces with non-doubling measures
as special cases. It is already known that singular integrals
on non-homogeneous metric measure spaces
arise naturally in the study of
complex and harmonic analysis questions in several complex variables
(see \cite{vw,hm} for the details).

In this new setting, Hyt\"onen \cite{h10}
introduced the space $\rbmo$ and established the corresponding
John-Nirenberg inequality. Later,
Hyt\"onen et al. \cite{hyy}, and  Bui and Duong \cite{bd},
independently, introduced the atomic Hardy
space $ \hoq$ and proved that the dual space of
$ \hoq$ is $\rbmo$.
Hyt\"onen et al. \cite{hlyy} and Liu et al. \cite{lyy2}
established some equivalent characterizations for the boundedness
of Calder\'on-Zygmund operators on $\lq$ with $q\in (1,\fz)$ and
their endpoint boundedness. Fu et al. \cite{fyy3}
introduced a version of the atomic Hardy space $\nhoq\ (\st\hoq)$
via the discrete coefficients
${\wz K}^{(\rho)}_{B,\,S}$, and showed that the Calder\'on-Zygmund
operator is bounded on $\nhoq$ via establishing
a molecular characterization of $\wz H_{\rm atb}^{1,\,q}(\mu)$
in this context. Recently, Fu et al. \cite{fyy2} introduced
generalized fractional integrals and
established the boundedness of generalized fractional integrals
and their commutators in this setting. More
research on the boundedness of various operators
on non-homogeneous metric measure spaces can be found in \cite{b13, hmy12, ly12,
lmy, ly14}. We refer the reader to the survey
\cite{yyf} and the monograph \cite{yyh} for more progress
on the theory of Hardy spaces and singular integrals
over non-homogeneous metric measure spaces.

We point out that the space $\nhoq$ seems to be more
useful in the study on the boundedness of operators,
since it was shown in \cite[Theorem 1.4]{fyy3} that
Calder\'on-Zygmund operators
are bounded on $\nhoq$, but the method does not work
for the boundedness of Calder\'on-Zygmund operators
on $\hoq$ over general non-homogeneous metric measure
spaces defined via the continuous coefficients
(see \cite[Remark 2.4]{fyy3} or Remark \ref{r4.2}(iv) below).

To the best of our knowledge, the theory of the Hardy space $H^p$ on
non-homogeneous metric measure spaces is still
unknown, even on Euclidean spaces endowed with non-doubling measures.
Let $(\cx, d, \mu)$ be a non-homogeneous metric measure space
in the sense of Hyt\"onen \cite{h10}. The main purposes of this article
are two-fold. First, via the discrete coefficients $\kbsp$,
we introduce the atomic Hardy space $\nhp$
and the molecular Hardy space $\mhp$, and give their
applications to the boundedness of Calder\'on-Zygmund operators
and generalized fractional integrals.
However, it is still unknown whether
$\nhp$ is independent of the choices of $\rho$, $\gz$
and $q$ or not even under some additional condition, called the
$\rho$-weakly doubling condition (see Definition \ref{d6.1} below).
Moreover, the dual space of $\nhp$ and the equivalence
between $\nhp$ and $\mhp$ are also unclear. Thus, we are forced to
turn to the other goal: introduce another
atomic Hardy space $\hhp$ and another molecular Hardy space $\hmp$,
and then show that $\hhp$ is independent
of the choices of $\rho$ and $\gz$
under the $\rho$-weakly doubling condition.
Then we study the Campanato space
${\mathcal E}^{\alpha,\,q}_{\rho,\,\eta,\,\gamma}(\mu)$,
the dual space of $\hhp$, and the equivalence
between $\hhp$ and $\hmp$ if $\mu$ is $\rho$-weakly doubling.
Moreover, if $\mu$ is doubling, we show that
${\mathcal E}^{\alpha,\,q}_{\rho,\,\eta,\,\gamma}(\mu)$
and the Lipschitz space ${\rm Lip}_{\alpha,\,q}(\mu)$ ($q\in[1,\infty)$),
or $\widehat{H}_{\rm{atb},\,\rho}^{p,\,q,\,\gamma}(\mu)$
and the atomic Hardy space $H^{p,\,q}_{\rm at}(\mu)$
($q\in(1,\infty]$) introduced by Coifman and Weiss in \cite{cw77} coincide
with equivalent quasi-norms. Finally, if $(\mathcal{X},d,\mu)$ is an RD-space
with $\mu(\mathcal{X})=\infty$, we prove that
$\widetilde H_{\rm{atb},\,\rho}^{p,\,q,\,\gamma}(\mu)$,
$\widetilde H_{\rm{mb},\,\rho}^{p,\,q,\,\gamma,\,\epsilon}(\mu)$
and $H^{p,\,q}_{\rm at}(\mu)$ coincide
for any $q\in(1,2]$, which is still unknown if $q\in(2,\infty]$.
In particular, when
$(\mathcal{X},d,\mu):=(\mathbb{R}^D,|\cdot|,dx)$ with $dx$ being the
$D$-dimensional Lebesgue measure, we show that the spaces
$\widetilde H_{\rm{atb},\,\rho}^{p,\,q,\,\gamma}(\mu)$,
$\widetilde H_{\rm{mb},\,\rho}^{p,\,q,\,\gamma,\,\epsilon}(\mu)$,
$\widehat{H}_{\rm{atb},\,\rho}^{p,\,q,\,\gamma}(\mu)$
and $\widehat{H}_{\rm{mb},\,\rho}^{p,\,q,\,\gamma,\,\epsilon}(\mu)$
all coincide with $H^p(\mathbb{R}^D)$ for any $q\in(1,\infty)$.

The organization of this article is as follows.

In Section \ref{s2}, we first recall some necessary notation and notions,
including the discrete coefficient $\kbsp$,
and give out some fundamental properties on $\kbsp$
which are crucial to the succeeding content.

In Section \ref{s3}, we introduce
the atomic Hardy space
${\widetilde H_{\mathrm{atb},\,\rho}^{p,\,q,\,\gz}(\mu)}$
via the discrete coefficient $\widetilde{K}^{(\rho),\,p}_{B,\,S}$
($\widetilde{K}^{(\rho),\,1}_{B,\,S}=\widetilde{K}^{(\rho)}_{B,\,S}$),
where the dominating function of the considered measure
appears in the size condition of the atomic
block, which seems to be well adapted to the study
of the boundedness of Calde\'on-Zygmund operators
and generalized fractional integrals,
and establish a useful property.
The key innovation in this section is
the definition of ${\widetilde H_{\mathrm{atb},\,\rho}^{p,\,q,\,\gz}(\mu)}$
as the completeness of a subspace of $\ltw$, $\pnhp$, which
is a suitable substitute of the classical fact that the set of all Schwartz
functions having infinite order vanishing moments
is dense in the Hardy space $H^p(\rr^D)$.

In Section \ref{s4},  we introduce the notion of the molecular Hardy space
${\widetilde H_{\rm{mb},\,\rho}^{p,\,q,\,\gz,\,\ez}(\mu)}$,
and prove that the Calder\'on-Zygmund operator is bounded
from ${\widetilde H_{\rm{mb},\,\rho}^{p,\,q,\,\gz,\,\dz}(\mu)}$
(or ${\widetilde H_{\rm{atb},\,\rho}^{p,\,q,\,\gz}(\mu)}$)
into $L^p(\mu)$ by borrowing some ideas from \cite[Theorem 4.2]{hyy}
with much more complicated arguments, and from
${\widetilde H_{\rm{atb},\,\rho(\rho+1)}^{p,\,q,\,\gz+1}(\mu)}$ into
${\widetilde H_{\rm{mb},\,\rho}^{p,\,q,\,\gz,\,(\delta
-{\nu}/{p}+\nu)/2}(\mu)}$ by using a method similar
to that used in the proof of \cite[Theorem 1.14]{fyy3} with
some technical modifications.

In Section \ref{s5}, we establish the boundedness
of the generalized fractional integral $T_{\bz}$ ($\bz\in(0,1)$)
from $\mhop$ (or $\nhop$)
into $\lpt$ with $1/p_2=1/p_1-\bz$.
The proof of the above result is parallel to that of the
conclusion for Calder\'on-Zygmund operators in Section \ref{s4}
with slight modifications.
For the sake of the clearness, we present the full details there.

Section \ref{s6} is mainly devoted to the theory of Campanato spaces.
We first introduce an additional assumption,
 called  the \emph{$\rho$-weakly doubling condition}
 (see \eqref{6.1} below), which is
satisfied by spaces of homogeneous type.
We also construct a non-trivial example to show that there
exist some non-homogeneous metric measure spaces satisfying
the $\rho$-weakly doubling condition \eqref{6.1};
see Example \ref{e6.3} below.
However, it turns out that there exist many
non-homogeneous metric measure spaces which
do not satisfy the $\rho$-weakly doubling condition;
see Example \ref{e6.4} below. Then we introduce the
Campanato space $\ceaeg$ and show that $\ceaeg$ is independent
of the choices of $\rho$, $\eta$, $\gz$ and $q$
under the assumption of $\rho$-weakly doubling conditions.
Precisely, via a useful property of $\ceaeg$
(see Proposition \ref{p6.7}(a) below)
and the geometrically doubling condition, we prove that $\ceaeg$
is independent of the choices of $\rho$ and $\eta$, where
the $\rho$-weakly doubling condition plays a decisive role.
Then, by establishing an equivalent characterization of
$\ceag:={\mathcal E}^{\alpha,\,q}_{\rho,\,\rho,\,\gz}(\mu)$
and a useful lemma (see Lemma \ref{l6.12} below),
which is analogous to \cite[Lemma 2.7]{hyy}, and by
borrowing some ideas from the proof of \cite[Proposition 2.5]{hyy},
we show that $\ceag$ is independent of the choice of $\gz$.
Next, by the above equivalent characterization of $\ceag$
and the $\rho$-weakly doubling condition, we establish the
John-Nirenberg inequality for
$\cear:={\mathcal E}^{\alpha,\,q}_{\rho,\,1}(\mu)$,
which further implies that $\cear$ is independent of the choice
of $q$. We point out that, on spaces of homogeneous type, the independence
of $q$ of $\cear$ is due to the coincidence between $\cear$ and
the Lipschitz space $\lip_{\az}(\mu)$; see \cite{ms1}.
However, this coincidence is unknown on non-homogeneous
metric measure spaces, even under the $\rho$-weakly doubling condition.
Alternatively, we adopt the method developed by Hyt\"onen
for the proof of the John-Nirenberg inequality for the BMO
type space in \cite{h10}; see also \cite{t01a}.
At the end of this section, we establish another useful
characterization of $\cera:={\mathcal E}^{\az,\,1}_{\rho}(\mu)$,
which plays important roles in the later context.

In Section \ref{s7}, we introduce the atomic Hardy space $\hhp$
and the molecular Hardy space
$\widehat H_{\rm{mb},\rho}^{p,q,\gz,\ez}(\mu)$ and investigate their relation
under the $\rho$-weakly doubling condition.
By using a method similar to that used in the proof of
\cite[Theorem 1.11]{fyy3}, together with some technical modifications,
we prove that $\hhp$ and $\hmp$ coincide with equivalent quasi-norms.
It is still unclear whether the above result holds true
or not on general non-homogeneous metric measure spaces,
even on Euclidean spaces with non-doubling
measures.

Section \ref{s8} is mainly devoted to investigating
the dual space of $\hhp$ under the $\rho$-weakly doubling condition.
To this end, we first show that $\hhp$ is independent
of the choices of $\rho$ and $\gz$.
Precisely, by the $\rho$-weakly doubling condition
and borrowing some ideas from the proof of \cite[Proposition 3.3(ii)]{hyy},
we first prove that $\hhp$ is independent of the choice of $\rho$.
By establishing the corresponding
result (see Lemma \ref{l6.11} below) to \cite[Lemma 9.2]{t01a}
and constructing a sequence of $(\rho, \bz_{\rho})$-doubling balls
which is a refinement of that appears in the proof of \cite[Lemma 9.3]{t01a},
we further show that $\hhp$ is independent of $\gz$.
Finally, via the independence of $\rho$ for $\hhp$
and the equivalent characterization of
$\cera:={\mathcal E}^{\az,\,1}_{\rho}(\mu)$ established in Section \ref{s6}
(see Proposition \ref{p6.18} below),
we show that $\hhp$ is the predual of $\cer$.
It is still unknown whether the above results hold true
or not on general non-homogeneous metric measure spaces,
even on Euclidean spaces with non-doubling measures.

In Section \ref{s9}, let $(\cx,d,\mu)$ be a space of
homogeneous type in the sense of Coifman and Weiss.
We investigate the relations between
the Campanato space $\ceaeg$ and the Lipschitz space $\lip_{\az,\,q}(\mu)$,
or between $\hhp$ and the atomic Hardy space $H^{p,\,q}_{\rm at}(\mu)$
introduced by Coifman and Weiss \cite{cw77}.
By carefully dividing the situation into several parts,
constructing a sequence of balls via using a method similar
to that used in the proof of the independence of $\gz$
for $\hhp$ in Section \ref{s6} and adopting some ideas from
\cite[Proposition 4.7]{h10}, we show that, if $q\in[1,\infty)$,
then $\ceaeg$ and $\lip_{\az,\,q}(\mu)$ coincide with equivalent norms.
By a method similar to that used in the proof of this result,
we also establish the coincidence of $\hhp$ and
$H^{p,\,q}_{\rm at}(\mu)$ for any $q\in(1,\infty]$ directly.

In Section \ref{s10}, suppose that $(\cx,d,\mu)$ is an RD-space
with $\mu(\cx)=\fz$ and $q\in(1,2]$. We show that
$\wz H_{\rm{atb},\,\rho}^{p,\,q,\,\gz}(\mu)$,
$\widetilde H_{\rm{mb},\,\rho}^{p,\,q,\,\gamma,\,\epsilon}(\mu)$ and
$H^{p,\,q}_{\rm at}(\mu)$ coincide.
Let $\pnhp$ and $\pmhp$ be dense subspaces of $\nhp$ and $\mhp$,
respectively (see Definitions \ref{d3.2} and \ref{d4.1} below).
We prove that
$$(H^{p,\,q}_{\rm at}(\mu)\cap\ltw)\st\pnhp\st
\wz{\mathbb{H}}^{p,\,q,\,\gz,\,\ez}_{\rm mb,\,\rho}(\mu)
\st (H^{p,\,q}_{\rm at}(\mu)\cap\ltw)$$
by two steps. In \textbf{Step 1}, to show that
$(H^{p,\,q}_{\rm at}(\mu)\cap\ltw)\st
\wz{\mathbb{H}}^{p,\,q,\,\gz}_{\rm atb,\,\rho}(\mu)$
for any $q\in(1,2]$, we first establish a
technical lemma (see Lemma \ref{l10.2} below).
Then we establish a useful criteria for the boundedness of some
integral operators (see Lemma \ref{l10.8} below).
Via this, a standard duality argument,
the Calder\'on reproducing formula and the
boundedness of the Littlewood-Paley
$g$-function on $\ltw$ obtained in \cite{hmy06},
we give out a key atomic decomposition
for all functions from $H_{\rm at}^{p,\,q}(\mu)\cap\ltw$
in $\ltw$ (see \eqref{10.5} below), which
plays an essential role in the proof of \textbf{Step 1}.
In \textbf{Step 2}, via the fact that
$\pnhp\st\pmhp$ (see Proposition \ref{p4.3} below)
and establishing the boundedness
of the Littlewood-Paley $S$-function from
$\pmhp$ into $\lp$, we conclude that, for any $q\in(1,\fz)$,
$$\pnhp\st\pmhp\st (H^p(\mu)\cap\ltw)=(H^{p,\,q}_{\rm at}(\mu)\cap\ltw),$$
where $H^p(\mu)$ is defined by the Littlewood-Paley $S$-function
as in \cite{gly,hmy06}.
By these two steps and a standard density argument, we
obtain the desired result.
Due to the defects of the above boundedness of the Littlewood-Paley
$g$-function and the criteria for the boundedness
of some integral operators, it is still unclear whether
$\widetilde H_{\rm{mb},\,\rho}^{p,\,q,\,\gamma,\,\epsilon}(\mu)\
({\rm or}\ \wz{H}^{p,\,q,\,\gz}_{\rm atb,\,\rho}(\mu))
=H^{p,\,q}_{\rm at}(\mu)$
over RD-spaces $(\cx,d,\mu)$ with $\mu(\cx)=\fz$ for $q\in(2,\infty]$.
Finally, if $(\mathcal{X},d,\mu):=(\mathbb{R}^D,|\cdot|,dx)$
with $dx$ being the $D$-dimensional Lebesgue measure,
we prove that the spaces
$\widetilde H_{\rm{atb},\,\rho}^{p,\,q,\,\gamma}(\mu)$,
$\widetilde H_{\rm{mb},\,\rho}^{p,\,q,\,\gamma,\,\ez}(\mu)$,
$\widehat{H}_{\rm{atb},\,\rho}^{p,\,q,\,\gamma}(\mu)$
and $\widehat{H}_{\rm{mb},\,\rho}^{p,\,q,\,\gamma,\,\epsilon}(\mu)$
all coincide with $H^p(\mathbb{R}^D)$ for any $q\in(1,\infty)$.

Finally, we make some conventions on notation.
Throughout this article, $C$ stands for a {\it positive constant} which
is independent of the choices of the main parameters,
but it may vary from line to
line. \emph{Constants with subscripts}, such as $C_0$, do
not change in different occurrences.
Furthermore, we use $C_{(\rho,\,\az,\,\ldots)}$
to denote a positive constant depending
on parameters $\rho,\,\az,\,\ldots$.
Let $\nn:=\{1,2,\ldots\}$ and $\zz_+:=\{0\}\cup\nn$.
For any ball $B$, the center and the
radius of $B$ are denoted, respectively, by $c_B$ and $r_B$.
For any subset $E$ of $\cx$, we use
$\chi_E$ to denote its {\it characteristic function}.

\section{Preliminaries}\label{s2}

\hskip\parindent In this section, we recall some necessary
notation and notions, including the discrete coefficient $\kbsp$,
and give out some fundamental properties on $\kbsp$ in the
non-homogeneous context.

The following notion of the geometrically doubling
is well known in analysis on metric spaces, which was originally introduced
by Coifman and Weiss in \cite[pp.\,66-67]{cw71} and is also
known as \emph{metrically doubling} (see, for example, \cite[p.\,81]{he}).

\begin{definition}\label{d2.1}
A metric space $(\cx,d)$ is said to be \emph{geometrically doubling} if there
exists some $N_0\in \nn$ such that, for any ball
$B(x,r)\st \cx$ with $x\in\cx$ and $r\in(0,\fz)$,
there exists a finite ball covering $\{B(x_i,r/2)\}_i$ of
$B(x,r)$ such that the cardinality of this covering is at most $N_0$.
\end{definition}

\begin{remark}\label{r2.2}
Let $(\cx,d)$ be a metric space. In \cite{h10}, Hyt\"onen showed that
the following statements are mutually equivalent:
\vspace{-0.25cm}
\begin{itemize}
  \item[\rm(i)] $(\cx,d)$ is geometrically doubling.
\vspace{-0.25cm}
  \item[\rm(ii)] For any $\ez\in (0,1)$ and any ball $B(x,r)\st \cx$
with $x\in\cx$ and $r\in(0,\fz)$,
there exists a finite ball covering $\{B(x_i,\ez r)\}_i$ of
$B(x,r)$ such that the cardinality of this covering is at most $N_0\ez^{-n_0}$,
here and hereafter, $N_0$ is as in Definition \ref{d2.1} and
$n_0:=\log_2N_0$.
\vspace{-0.25cm}
  \item[\rm(iii)] For every $\ez\in (0,1)$, any ball $B(x,r)\st \cx$
with $x\in\cx$ and $r\in(0,\fz)$ contains
at most $N_0\ez^{-n_0}$ centers of disjoint balls $\{B(x_i,\ez r)\}_i$.
\vspace{-0.25cm}
  \item[\rm(iv)] There exists $M\in \nn$ such that any ball $B(x,r)\st \cx$
with $x\in\cx$ and $r\in(0,\fz)$ contains at most $M$ centers $\{x_i\}_i$ of
  disjoint balls $\{B(x_i, r/4)\}_{i=1}^M$.
  \end{itemize}
\end{remark}

Recall that spaces of homogeneous type are geometrically doubling,
which was proved
by Coifman and Weiss in \cite[pp.\,66-68]{cw71}.

The following notion of upper doubling metric
measure spaces was originally introduced
by Hyt\"onen \cite{h10} (see also \cite{hlyy,lyy2}).

\begin{definition}\label{d2.3}
A metric measure space $(\cx,d,\mu)$ is said to be \emph{upper doubling} if
$\mu$ is a Borel measure on $\cx$ and there exist a \emph{dominating function}
$\lz:\cx \times (0,\fz)\to (0,\fz)$ and a positive constant $C_{(\lz)}$,
depending on $\lz$, such that, for each $x\in \cx$, $r\to \lz(x,r)$ is
non-decreasing and, for all $x\in \cx$ and $r\in (0,\fz)$,
\begin{equation}\label{2.1}
\mu(B(x,r))\le\lz(x,r)\le C_{(\lz)}\lz(x,r/2).
\end{equation}
A metric measure space $(\cx,d,\mu)$ is called  a
\emph{non-homogeneous metric measure space}
if $(\cx,d)$ is geometrically doubling and $(\cx,d,\mu)$ is upper doubling.
\end{definition}

\begin{remark}\label{r2.4}
(i) Obviously, a space of homogeneous type is a
special case of upper doubling spaces, where we take the dominating function
$\lz(x,r):=\mu(B(x,r))$ for all $x\in\cx$ and $r\in(0,\fz)$.
On the other hand, the $D$-dimensional Euclidean space
$\rr^D$ with any Radon measure $\mu$ as
in \eqref{1.2} is also an upper doubling
space by taking $\lz(x,r):=C_0r^{\kz}$ for all $x\in\rr^D$ and
$r\in(0,\fz)$.

(ii) Let $(\cx,d,\mu)$ be upper doubling with $\lz$ being the dominating
function on $\cx \times (0,\fz)$ as in Definition \ref{d2.3}. It was proved
in \cite{hyy} that there exists another
dominating function $\wz{\lz}$ such that $\wz{\lz}\le \lz$, $C_{(\wz{\lz})}
\le C_{(\lz)}$
and, for all $x,\,y\in \cx$ with $d(x,y)\le r$,
\begin{equation}\label{2.2}
\wz{\lz}(x,r)\le C_{(\wz{\lz})}\wz{\lz}(y,r).
\end{equation}

(iii) It was shown in \cite{tl} that the upper doubling condition
is equivalent
to the \emph{weak growth condition}:
there exist a dominating function $\lz:\cx\times(0,\fz)\to(0,\fz)$,
with $r\to\lz(x,r)$ non-decreasing, positive constants $C_{(\lz)}$,
depending on $\lz$, and $\ez$ such that

$\rm (iii)_1$ for all $r\in(0,\fz)$, $t\in[0,r]$, $x,\,y\in\cx$
and $d(x,y)\in[0,r]$,
$$|\lz(y,r+t)-\lz(x,r)|\le C_{(\lz)}\lf[\frac{d(x,y)+t}r\r]^{\ez}
\lz(x,r);$$

$\rm (iii)_2$ for all $x\in\cx$ and $r\in(0,\fz)$,
$\mu(B(x,r))\le\lz(x,r).$
\end{remark}

Based on Remark \ref{r2.4}(ii), from now on, we \emph{always assume} that
$(\cx,d,\mu)$ is a non-homogeneous metric measure space
with the dominating function $\lz$ satisfying \eqref{2.2}.

Though the measure doubling condition is not assumed uniformly for all balls
in the non-homogeneous metric measure space $(\cx,d,\mu)$,
it was shown in \cite{h10}
that there still exist many
balls which have the following $(\az,\bz)$-doubling property.

\begin{definition}\label{d2.5}
Let $\az,\,\bz\in (1,\fz)$. A ball $B\st \cx$ is said to be
\emph{$(\az,\bz)$-doubling}
if $\mu(\az B)\le \bz\mu(B)$, where, for any ball
$B:=B(c_B,r_B)$ and $\rho\in(0,\fz)$,
$\rho B:=B(c_B,\rho r_B)$.
\end{definition}

To be precise, it was proved in
\cite[Lemma 3.2]{h10} that, if a metric measure
space $(\cx,d,\mu)$ is upper doubling and $\az,\,\bz\in(1,\fz)$ with
$\bz>[C_{(\lz)}]^{\log_2\az}=:\az^\nu$,
then, for any ball $B\st \cx$, there exists some
$j\in \zz_+$ such that $\az^jB$ is $(\az,\bz)$-doubling.
Moreover, let $(\cx,d)$ be geometrically
doubling, $\bz>\az^{n_0}$ with $n_0:=\log_2N_0$
and $\mu$ a Borel measure on $\cx$
which is finite on bounded sets. Hyt\"onen \cite[Lemma 3.3]{h10} also
showed that, for $\mu$-almost every $x\in \cx$, there exist
arbitrary small $(\az,\bz)$-doubling
balls centered at $x$. Furthermore, the radii of these balls may be
chosen to be of the form $\az^{-j}r$ for $j\in\nn$ and any
preassigned number $r\in(0, \fz)$.
Throughout this article, for any $\az\in (1,\fz)$ and ball $B$,
the \emph{smallest
$(\az,\bz_\az)$-doubling ball of the form $\az^j B$ with $j\in \zz_+$}
is denoted by
$\wz B^\az$, where
\begin{equation*}
\bz_\az:=\az^{3(\max\{n_0,\,\nu\})}+[\max\{5\az,\,30\}]^{n_0}
+[\max\{3\az,30\}]^\nu.
\end{equation*}

Before we introduce the discrete coefficient $\kbsp$, we first give
an assumption on the relation between two balls $B$ and $S$,
which is \emph{supposed to hold true through the whole article}:

\textbf{(A)} If $B=S$, then $c_B=c_S$ and $r_B=r_S$.

Then we claim that, if $B\st S$, then $r_B\le2r_S$.
Indeed, assume that $r_B>2r_S$. By this
and $B\st S$, together with the triangle inequality
satisfied by $d$, we see that $S\st B$. Thus, $B=S$, which,
together with the assumption \textbf{(A)},
implies that $r_B=r_S$. This contradicts to $r_B>2r_S$,
which completes the proof of the above claim.

On the other hand, we give a simple example to
illustrate that, if $B\subsetneqq S$, then it may happen
that $r_B>r_S$. Let $(\cx,d):=(\{-1,1,3\},|\cdot|)$,
$B:=\{x\in\{-1,1,3\}:\ |x+1|<3\}$ and
$$
S:=\lf\{x\in\{-1,1,3\}:\ |x-1|<\frac52\r\}.
$$
Obviously, $r_B>r_S$ and $B=\{-1,1\}\subsetneqq\{-1,1,3\}=S$.

\begin{definition}\label{d2.6}
For any $\rho\in (1,\fz)$, $p\in(0,1]$ and any two balls
$B\st S\st\cx$, let
\begin{equation}\label{2.3}
\wz K^{(\rho),\,p}_{B,\,S}
:=\lf\{1+\sum_{k=-\lfloor\log_{\rho}2\rfloor}^{N^{(\rho)}_{B,\,S}}
\lf[\frac{\mu(\rho^{k}B)}{\lz(c_{B},\rho^{k}r_{B})}\r]^p\r\}^{1/p},
\end{equation}
here and hereafter, for any $a\in\rr$,
$\lfloor a\rfloor$ represents the \emph{biggest integer which is not bigger
than} $a$, and $N^{(\rho)}_{B,\,S}$ is the
\emph{smallest integer} satisfying
$\rho^{N^{(\rho)}_{B,\,S}}r_{B}\ge r_{S}$.
\end{definition}

\begin{remark}\label{r2.7}
(i) By a change of variables and \eqref{2.1}, we easily conclude that
$$
\wz K^{(\rho),\,p}_{B,\,S}
\sim\lf\{1+\sum_{k=1}^{N^{(\rho)}_{B,\,S}+\lfloor\log_{\rho}2\rfloor+1}
\lf[\frac{\mu(\rho^{k}B)}{\lz(c_{B},\rho^{k}r_{B})}\r]^p\r\}^{1/p},
$$
where the implicit equivalent positive constants are
independent of balls $B\st S\st\cx$, but depend on $\rho$ and $p$.

(ii) A continuous version, $K_{B,\,S}$, of the coefficient in
Definition \ref{d2.6} when $p=1$ was introduced in \cite{h10,hyy}
as follows: for any two balls $B\st S\st\cx$,
\begin{equation}\label{2.3x}
K_{B,\,S}:=1+\int_{(2S)\bh B}\frac1{\lz(c_B,d(x,c_B))}\,d\mu(x).
\end{equation}
It was proved in \cite[Lemma 2.2]{hyy} that $K_{B,\,S}$ has
all properties similar to those for $\kbsp$ as
in Lemma \ref{l2.8} below. Unfortunately, $K_{B,\,S}$ and
$\wz{K}_{B,\,S}^{(\rho),\,1}$ are usually not equivalent,
but this is true for $(\rr^D,|\cdot|,\mu)$ with $\mu$
as in \eqref{1.2}; see \cite{fyy3} for more details on this.
\end{remark}

Now we give some simple properties of ${\wz K}^{(\rho),\,p}_{B,\,S}$
defined by \eqref{2.3}
adapted from \cite[Lemma 3.5]{fyy1},
in which $\rho=6$ and $p=1$. The arguments therein are
still valid for the present case.
For the sake of reader's convenience, we present some details.
In what follows, for any $a\in\rr$, $\lceil a\rceil$ represents
the \emph{smallest integer which is not smaller than} $a$.

\begin{lemma}\label{l2.8}
Let $(\cx,d,\mu)$ be a non-homogeneous metric measure space and $p\in(0,1]$.

{\rm(i)} For any $\rho\in(1,\fz)$, there exists a
positive constant $C_{(\rho)}$,
depending on $\rho$, such that, for all balls $B\st R\st S$,
$[{\wz K}^{(\rho),\,p}_{B,\,R}]^p\le C_{(\rho)}
[{\wz K}^{(\rho),\,p}_{B,\,S}]^p$.

{\rm(ii)} For any $\az\in [1,\fz)$ and $\rho\in(1,\fz)$, there exists
a positive constant $C_{(\az,\,\rho)}$, depending on $\az$
and $\rho$, such that, for
all balls $B\st S$ with
$r_S\le \az r_B$, $[{\wz K}^{(\rho),\,p}_{B,\,S}]^p\le C_{(\az,\,\rho)}$.

{\rm(iii)} For any $\rho\in(1,\fz)$, there exists
a positive constant $C_{(\rho,\,p,\,\nu)}$, depending
on $\rho$, $p$ and $\nu$,
such that, for all balls $B$,
$[{\wz K}^{(\rho),\,p}_{B,\,\wz B^{\rho}}]^p\le C_{(\rho,\,p,\,\nu)}$.
Moreover, letting $\az,\,\bz\in(1,\fz)$, $B\st S$ be any
two concentric balls
such that there exists no $(\az,\bz)$-doubling ball in the form of
$\az^k B$, with $k\in\nn$, satisfying $B\st \az^k B\st S$,
then there exists a positive constant
$C_{(\az,\,\bz,\,p,\,\nu)}$, depending on $\az$,
$\bz$, $p$ and $\nu$, such that
$[{\wz K}^{(\az),\,p}_{B,\,S}]^p\le C_{(\az,\,\bz,\,p,\,\nu)}$.

{\rm(iv)} For any $\rho\in(1,\fz)$,
there exists a positive constant $c_{(\rho,\,p,\,\nu)}$,
depending on $\rho$, $p$ and $\nu$, such that, for all balls $B\st R\st S$,
$$\lf[{\wz K}^{(\rho),\,p}_{B,\,S}\r]^p\le \lf[{\wz K}^{(\rho),\,p}_{B,\,R}\r]^p
+c_{(\rho,\,p,\,\nu)}\lf[{\wz K}^{(\rho),\,p}_{R,\,S}\r]^p.$$

{\rm(v)} For any $\rho\in(1,\fz)$, there exists a positive constant
$\wz c_{(\rho,\,p,\,\nu)}$, depending on $\rho$, $p$ and $\nu$, such that,
for all balls $B\st R\st S$,
$[{\wz K}^{(\rho),\,p}_{R,\,S}]^p\le \wz c_{(\rho,\,p,\,\nu)}
[{\wz K}^{(\rho),\,p}_{B,\,S}]^p$.
\end{lemma}

\begin{proof}
Fix $p\in(0,1]$, $\rho\in(1,\fz)$ and $\az\in[1,\fz)$.
We first show (i). By $R\st S$, we have $r_R\le 2r_S$. Hence,
$r_R\le2r_S\le2\rho^{N^{(\rho)}_{B,\,S}}r_B
\le\rho^{\lceil\log_{\rho}2\rceil
+N^{(\rho)}_{B,\,S}}r_B$. Thus, $N^{(\rho)}_{B,\,R}
\le\lceil\log_{\rho}2\rceil+N^{(\rho)}_{B,\,S}$.
By this and \eqref{2.1}, we see that
\begin{align*}
\lf[{\wz K}^{(\rho),\,p}_{B,\,R}\r]^p&\le1
+\sum_{k=-\lfloor\log_{\rho}2\rfloor}^{\lceil
\log_{\rho}2\rceil+N^{(\rho)}_{B,\,S}}
\lf[\frac{\mu(\rho^{k}B)}{\lz(c_{B},\rho^{k}r_{B})}\r]^p\\
&\le1+\sum_{k=-\lfloor\log_{\rho}2\rfloor}^{N^{(\rho)}_{B,\,S}}
\lf[\frac{\mu(\rho^{k}B)}{\lz(c_{B},\rho^{k}r_{B})}\r]^p
+\lceil\log_{\rho}2\rceil
\le(1+\lceil\log_{\rho}2\rceil)\lf[{\wz K}^{(\rho),\,p}_{B,\,S}\r]^p.
\end{align*}
This shows (i).

Now we prove (ii). By the fact that $\rho^{N^{(\rho)}_{B,\,S}-1}r_B<r_S
\le\az r_B$, we have $N^{(\rho)}_{B,\,S}-1<\log_{\rho}\az$.
Thus, $N^{(\rho)}_{B,\,S}-1\le\lfloor\log_{\rho}\az\rfloor$.
From this and \eqref{2.1}, we deduce that
$$
\lf[{\wz K}^{(\rho),\,p}_{B,\,S}\r]^p\le1+N^{(\rho)}_{B,\,S}
+\lfloor\log_{\rho}2\rfloor
\le2+\lfloor\log_{\rho}\az\rfloor+\lfloor\log_{\rho}2\rfloor.
$$
Thus, (ii) holds true.

Let us now prove {\rm(iii)}.
To this end, let $N$ be the first integer such that $\rho^NB$ is
$(\rho,\bz_{\rho})$-doubling. If $B$ is $(\rho,\bz_{\rho})$-doubling,
the conclusion of (iii) holds true trivially.
Thus, without loss of generality, we may assume that $B$ is
non-$(\rho,\bz_{\rho})$-doubling, which implies that $N\ge1$.
For any $k\in \{-\lfloor\log_{\rho}2\rfloor,\ldots,N-1\}$, we have
$\mu(\rho^{k+1}B)>\bz_{\rho}\mu(\rho^kB)$.
Thus, for any $k\in \{-\lfloor\log_{\rho}2\rfloor,\ldots,N-1\}$,
$\mu(\rho^kB)<\frac{\mu(\rho^NB)}{\bz_{\rho}^{N-k}}$.
By this, together with \eqref{2.1} and the fact that
$\bz_{\rho}>[C_{(\lz)}]^{\log_22\rho}=(2\rho)^{\nu}$, we conclude that
\begin{align*}
\lf[{\wz K}^{(\rho),\,p}_{B,\,\wz B^{\rho}}\r]^p&=
\lf[{\wz K}^{(\rho),\,p}_{B,\,\rho^NB}\r]^p
\le 2+\sum^{N-1}_{k=-\lfloor\log_{\rho}2\rfloor}\lf[\frac{\mu(\rho^kB)}
{\lz(c_B,\rho^kr_B)}\r]^p\\
&\le 2+\sum^{N-1}_{k=-\lfloor\log_{\rho}2\rfloor}\lf[\frac{(2\rho)^{\nu}}
{\bz_{\rho}}\r]^{p(N-k)}
\lf[\frac{\mu(\rho^NB)}{\lz(c_B,\rho^Nr_B)}\r]^p\\
&\le 2+\sum^{\fz}_{k=1}\lf[\frac{(2\rho)^{\nu}}
{\bz_{\rho}}\r]^{pk}
\ls 1,
\end{align*}
where the implicit positive constant only depends on $\rho$, $p$ and $\nu$.
Similarly, the other part of (iii) holds true, the details being omitted.
This proves (iii).

Next we show (iv). By (i),
$ N^{(\rho)}_{B,R}\le N^{(\rho)}_{B,\,S}+\lceil\log_{\rho}2\rceil$.
If $N^{(\rho)}_{B,\,S}\le N^{(\rho)}_{B,R}\le N^{(\rho)}_{B,\,S}
+\lceil\log_{\rho}2\rceil$, then there exists nothing to prove.
If $N^{(\rho)}_{B,R}<N^{(\rho)}_{B,\,S}$,
from the facts that $N^{(\rho)}_{B,\,S}\le N^{(\rho)}_{B,R}+N^{(\rho)}_{R,S}$
(since $\rho^{N^{(\rho)}_{B,R}+N^{(\rho)}_{R,S}}r_B
\ge\rho^{N^{(\rho)}_{R,S}}r_R\ge r_S$),
$\rho^{N^{(\rho)}_{B,R}}r_B\ge r_R$,
$\rho^{k+N^{(\rho)}_{B,R}+1+\lfloor\log_{\rho}2\rfloor}
B\st\rho^{k+2+\lceil\log_{\rho}2\rceil+\lfloor\log_{\rho}2\rfloor}R$
for all $k\in\zz\cap[-\lfloor\log_{\rho}2\rfloor,\fz)$,
and \eqref{2.1}, it follows that
\begin{align*}
\lf[{\wz K}^{(\rho),\,p}_{B,\,S}\r]^p&
\le\lf[{\wz K}^{(\rho),\,p}_{B,\,R}\r]^p
+\sum_{k=N^{(\rho)}_{B,R}+1}^{N^{(\rho)}_{B,R}+N^{(\rho)}_{R,S}
+1+\lfloor\log_{\rho}2\rfloor}
\lf[\frac{\mu(\rho^kB)}{\lz(c_B,\rho^kr_B)}\r]^p\\
&=\lf[{\wz K}^{(\rho),\,p}_{B,\,R}\r]^p
+\sum_{k=-\lfloor\log_{\rho}2\rfloor}^{N^{(\rho)}_{R,S}}
\lf[\frac{\mu(\rho^{k+N^{(\rho)}_{B,R}+1+\lfloor\log_{\rho}2\rfloor}B)}
{\lz(c_B,\rho^{k+N^{(\rho)}_{B,R}+1
+\lfloor\log_{\rho}2\rfloor}r_B)}\r]^p\\
&\le\lf[{\wz K}^{(\rho),\,p}_{B,\,R}\r]^p
+\sum_{k=-\lfloor\log_{\rho}2\rfloor}^{N^{(\rho)}_{R,S}}
\lf[\frac{\mu(\rho^{k+2+\lceil\log_{\rho}2\rceil
+\lfloor\log_{\rho}2\rfloor}R)}
{\lz(c_B,\rho^{k+1+\lfloor\log_{\rho}2\rfloor}r_R)}\r]^p\\
&\le\lf[{\wz K}^{(\rho),\,p}_{B,\,R}\r]^p
+c_{(\rho,\,p,\,\nu)}\sum_{k=-\lfloor\log_{\rho}2\rfloor}^{N^{(\rho)}_{R,S}}
\lf[\frac{\mu(\rho^{k+2+\lceil\log_{\rho}2\rceil
+\lfloor\log_{\rho}2\rfloor}R)}
{\lz(c_B,\rho^{k+2+\lceil\log_{\rho}2\rceil
+\lfloor\log_{\rho}2\rfloor}r_R)}\r]^p\\
&\le\lf[{\wz K}^{(\rho),\,p}_{B,\,R}\r]^p
+c_{(\rho,\,p,\,\nu)}\sum_{k=-\lfloor\log_{\rho}2\rfloor}^{N^{(\rho)}_{R,S}
+2+\lceil\log_{\rho}2\rceil+\lfloor\log_{\rho}2\rfloor}
\lf[\frac{\mu(\rho^{k}R)}{\lz(c_R,\rho^{k}r_R)}\r]^p\\
&\le\lf[{\wz K}^{(\rho),\,p}_{B,\,R}\r]^p
+c_{(\rho,\,p,\,\nu)}\lf[{\wz K}^{(\rho),\,p}_{R,\,S}\r]^p,
\end{align*}
which shows (iv).

For (v), we first prove that
$N^{(\rho)}_{B,R}+N^{(\rho)}_{R,S}\le N^{(\rho)}_{B,\,S}+1$.
Since
$$r_R=\rho^{-N^{(\rho)}_{R,S}+1}
\rho^{N^{(\rho)}_{R,S}-1}r_R\le \rho^{-N^{(\rho)}_{R,S}+1}r_S
\le \rho^{-N^{(\rho)}_{R,S}+1}\rho^{N^{(\rho)}_{B,\,S}}
r_B= \rho^{N^{(\rho)}_{B,\,S}-N^{(\rho)}_{R,S}+1}r_B,$$
we obtain $N^{(\rho)}_{B,R}\le N^{(\rho)}_{B,\,S}-N^{(\rho)}_{R,S}+1$.
From this,
$r_R>\rho^{N^{(\rho)}_{B,R}-1}r_B$,
$\rho^kR\st\rho^{k+\lceil\log_{\rho}2\rceil+N^{(\rho)}_{B,R}}B$
for all $k\in\zz\cap[-\lfloor\log_{\rho}2\rfloor,\fz)$,
\eqref{2.1} and \eqref{2.2}, it follows that
\begin{align*}
\lf[{\wz K}^{(\rho),\,p}_{R,\,S}\r]^p&
\le 1+\sum_{k=-\lfloor\log_{\rho}2\rfloor}^{N^{(\rho)}_{R,\,S}}
\lf[\frac{\mu(\rho^{k+\lceil\log_{\rho}2\rceil+N^{(\rho)}_{B,R}}B)}
{\lz(c_R,\rho^{k+N^{(\rho)}_{B,R}-1}r_B)}\r]^p\\
&\ls 1+\sum_{k=-\lfloor\log_{\rho}2\rfloor}^{N^{(\rho)}_{R,\,S}}
\lf[\frac{\mu(\rho^{k+\lceil\log_{\rho}2\rceil+N^{(\rho)}_{B,R}}B)}
{\lz(c_R,\rho^{k+\lceil\log_{\rho}2\rceil+N^{(\rho)}_{B,R}}r_B)}\r]^p\\
&\sim 1+\sum_{k=N^{(\rho)}_{B,R}-1-\lfloor\log_{\rho}2\rfloor}
^{N^{(\rho)}_{B,R}+N^{(\rho)}_{R,S}-1}
\lf[\frac{\mu(\rho^{k+1+\lceil\log_{\rho}2\rceil}B)}
{\lz(c_B,\rho^{k+1+\lceil\log_{\rho}2\rceil}r_B)}\r]^p\\
&\ls1+\sum_{k=-\lfloor\log_{\rho}2\rfloor}
^{N^{(\rho)}_{B,\,S}+1+\lceil\log_{\rho}2\rceil}
\lf[\frac{\mu(\rho^{k}B)}
{\lz(c_B,\rho^{k}r_B)}\r]^p
\ls\lf[{\wz K}^{(\rho),\,p}_{B,\,S}\r]^p,
\end{align*}
where the implicit positive constants depend only on
$\rho$, $p$ and $\nu$.
This finishes the proof of (v) and hence Lemma \ref{l2.8}.
\end{proof}

Now we show that, for any $\rho_1,\,\rho_2\in(1,\fz)$ and $p\in(0,1]$,
${\wz K}^{(\rho_1),\,p}_{B,\,S}\sim{\wz K}^{(\rho_2),\,p}_{B,\,S}$
for all balls $B\st S$.

\begin{lemma}\label{l2.9}
Let $(\cx,d,\mu)$ be a non-homogeneous metric measure space,
$\rho_1,\,\rho_2\in(1,\fz)$ and $p\in(0,1]$.
Then there exist positive constants
$c_{(\rho_1,\,\rho_2,\,p,\,\nu)}$ and
$C_{(\rho_1,\,\rho_2,\,p,\,\nu)}$, depending on
$\rho_1$, $\rho_2$, $\nu$ and $p$, such that, for all balls $B\st S$,
$$
c_{(\rho_1,\,\rho_2,\,p,\,\nu)}{\wz K}^{(\rho_2),\,p}_{B,\,S}
\le{\wz K}^{(\rho_1),\,p}_{B,\,S}\le C_{(\rho_1,\,\rho_2,\,p,\,\nu)}
{\wz K}^{(\rho_2),\,p}_{B,\,S}.
$$
\end{lemma}

\begin{proof}
For the sake of simplicity, we only prove this lemma for $p=1$.
With some slight modifications, the arguments here are still valid for
$p\in(0,1)$.
For any $\rho_1,\,\rho_2\in(1,\fz)$,
without loss of generality, we may assume that $\rho_1>\rho_2>1$.
For any two fixed balls $B\st S$, let
$N_j:=N^{(\rho_j)}_{B,\,S}$ and
${\wz K}^{(\rho_j)}_{B,\,S}:={\wz K}^{(\rho_j),\,1}_{B,\,S}$
($j\in\{1,2\}$).
It is obvious that $N_1\le N_2$.
Now we consider the following two cases:

\textbf{Case i)} $\rho_1^{N_1}\le\rho_2^{N_2}$.
It is easy to see that
$\rho_2^{N_2-1}\le\rho_1^{N_1}$.
We first prove that
${\wz K}^{(\rho_1)}_{B,\,S}\ls{\wz K}^{(\rho_2)}_{B,\,S}$.

Indeed, for any $n_1\in\{-\lfloor\log_{\rho_1}2\rfloor,\ldots,N_1\}$,
let $n_2$ be the smallest integer such that $\rho_2^{n_2}\ge\rho_1^{n_1}$.
Then we have
\begin{equation}\label{2.4}
n_2\in\{-\lfloor\log_{\rho_2}2\rfloor,\ldots,N_2\}\quad
{\rm and}\quad \rho_2^{n_2-1}
<\rho_1^{n_1}\le\rho_2^{n_2}.
\end{equation}
Consequently, $\rho_2^{n_2-1}B\st\rho_1^{n_1}B\st\rho_2^{n_2}B$.
By some simple calculations, we see that, for any
$n_2\in\{-\lfloor\log_{\rho_2}2\rfloor,\ldots,N_2\}$, there exists at most
one $n_1$ satisfying \eqref{2.4}.
By the above facts, $-\lfloor\log_{\rho_1}2\rfloor
\ge-\lfloor\log_{\rho_2}2\rfloor$ and \eqref{2.1}, we obtain
$$
{\wz K}^{(\rho_1)}_{B,\,S}
\le1+\sum_{n_1=-\lfloor\log_{\rho_1}2\rfloor}^{N_1}
\frac{\mu(\rho_2^{n_2}B)}{\lz(c_{B},\rho_2^{n_2-1}r_{B})}
\ls1+\sum_{n_2=-\lfloor\log_{\rho_2}2\rfloor}^{N_2}
\frac{\mu(\rho_2^{n_2}B)}{\lz(c_{B},\rho_2^{n_2}r_{B})}
\sim{\wz K}^{(\rho_2)}_{B,\,S},
$$
where the implicit positive constants depend only
on $\rho_1$, $\rho_2$ and $\nu$.

On the other hand, for the case $N_2<1$, it is obvious that
${\wz K}^{(\rho_2)}_{B,\,S}\ls1\ls{\wz K}^{(\rho_1)}_{B,\,S}$,
which completes the proof of \textbf{Case i)}.
Thus, without loss of generality, we may assume that $N_2\ge1$.
We notice that
$$
{\wz K}^{(\rho_2)}_{B,\,S}\le 2\lf[1+\sum_{n_2=-\lfloor
\log_{\rho_2}2\rfloor}^{N_2-1}
\frac{\mu(\rho_2^{n_2}B)}{\lz(c_{B},\rho_2^{n_2}r_{B})}\r].
$$
For any fixed $n_2\in\{-\lfloor\log_{\rho_2}2\rfloor,\ldots,N_2-1\}$,
let $n_1$ be the smallest
positive integer such that $\rho_1^{n_1}\ge\rho_2^{n_2}$.
Then we have
\begin{equation}\label{2.5}
n_1\in\{-\lfloor\log_{\rho_1}2\rfloor,\ldots,N_1\}
\quad {\rm and}\quad \rho_1^{n_1-1}<\rho_2^{n_2}\le\rho_1^{n_1}.
\end{equation}
Consequently, $\rho_1^{n_1-1}B\st\rho_2^{n_2}B\st\rho_1^{n_1}B$.
By some simple calculations, we see that, for any
$n_1\in\{-\lfloor\log_{\rho_1}2\rfloor,\ldots,N_1\}$, the number of $n_2$
satisfying \eqref{2.5} does not exceed
$\lceil\frac{\ln\rho_1}{\ln\rho_2}\rceil$.
By the above facts and \eqref{2.1}, we know that
\begin{align*}
{\wz K}^{(\rho_2)}_{B,\,S}&\ls1
+\sum_{n_2=-\lfloor\log_{\rho_2}2\rfloor}^{N_2-1}
\frac{\mu(\rho_2^{n_2}B)}{\lz(c_{B},\rho_2^{n_2}r_{B})}
\sim1+\sum_{n_1=-\lfloor\log_{\rho_1}2\rfloor
}^{N_1}\sum_{n_2:\ \rho_1^{n_1-1}<\rho_2^{n_2}\le\rho_1^{n_1}}
\frac{\mu(\rho_2^{n_2}B)}
{\lz(c_{B},\rho_2^{n_2}r_{B})}\\
&\ls 1+\sum_{n_1=-\lfloor\log_{\rho_1}2\rfloor}^{N_1}
\frac{\mu(\rho_1^{n_1}B)}{\lz(c_{B},\rho_1^{n_1}r_{B})}
\sim{\wz K}^{(\rho_1)}_{B,\,S},
\end{align*}
where the implicit positive constants depend only on
$\rho_1$, $\rho_2$ and $\nu$.
This finishes the proof of \textbf{Case i)}.

\textbf{Case ii)} $\rho_2^{N_2}<\rho_1^{N_1}$. The proof of this case
is similar to that of \textbf{Case i)}, the details being omitted.
This finishes the proof of Lemma \ref{l2.9}.
\end{proof}

\section{Atomic Hardy Spaces $\nhp$}\label{s3}

\hskip\parindent In this section, we introduce
the atomic Hardy space $\nhp$ and establish a useful
property.
Before introducing the notion of $\nhp$,
we first recall some notions related to quasi-Banach spaces;
see, for example, \cite{gly}.

\begin{definition}\label{d3.1}
(i) A \emph{quasi-Banach space} $\cb$ is a
vector space endowed with a \emph{quasi-norm}
$\|\cdot\|_{\cb}$ which is non-negative, non-degenerate
(namely, $\|f\|_{\cb}=0$
if and only if $f=0$), homogeneous, and obeys the quasi-triangle inequality,
namely, there exists a constant $K\in[1,\fz)$ such that,
for all $f,\,g\in\cb$,
$$\|f+g\|_{\cb}\le K(\|f\|_{\cb}+\|g\|_{\cb}).$$

(ii) Let $r\in(0,1]$. A quasi-Banach space $\cb_r$ with the quasi-norm
$\|\cdot\|_{\cb_r}$ is called a \emph{$r$-quasi-Banach space} if
$\|f+g\|_{\cb_r}^r\le\|f\|_{\cb_r}^r+\|g\|_{\cb_r}^r$ for
all $f,\,g\in\cb_r$. Hereafter, $\|\cdot\|_{\cb_r}^r$
is called  the \emph{$r$-quasi-norm} of the
$r$-quasi-Banach space $\cb_r$.
\end{definition}

Then we introduce the notion of $\nhp$ over general
non-homogeneous metric measure spaces.

\begin{definition}\label{d3.2}
Let $\rho\in (1,\fz)$, $0<p\le1\le q\le\fz$, $p\neq q$ and $\gz\in[1,\fz)$.
A function $b$ in $\ltw$ when $p\in(0,1)$ and in $\lon$ when $p=1$
is called  a \emph{$(p,q,\gz,\rho)_\lz$-atomic block} if

(i) there exists a ball $B$ such that $\supp(b)\st B$;

(ii) $\int_\cx b(x)\,d\mu(x)=0$;

(iii) for any $j\in\{1,\,2\}$, there exist
a function $a_{j}$ supported on a ball
$B_{j}\st B$ and a number $\lz_j\in\cc$ such that
$b=\lz_1a_1+\lz_2a_2$
and
$$\|a_j\|_\lq\le [\mu(\rho B_j)]^{1/q-1}[\lz(c_B,r_B)]^{1-1/p}
\lf[\wz K^{(\rho),\,p}_{B_j,\,B}\r]^{-\gz}.$$
Moreover, let $|b|_{\nhp}:=|\lz_1|+|\lz_2|$.

A function $f$ is said to belong to the \emph{space} $\pnhp$
if there exists a sequence of $(p,q,\gz,\rho)_\lz$-atomic blocks,
$\{b_{i}\}_{i=1}^\fz$, such that
$f=\sum_{i=1}^{\fz} b_{i}$ in $\ltw$ when $p\in(0,1)$ and in $\lon$
when $p=1$, and
$$\sum_{i=1}^{\fz}|b_{i}|^p_{\nhp}<\fz.$$
Moreover, define
$$\|f\|_{\nhp}:=\inf
\lf\{\lf[\sum_{i=1}^{\fz}|b_{i}|^p_{\nhp}\r]^{1/p}\r\},$$
where the infimum is taken over all possible decompositions of
$f$ as above.

The \emph{atomic Hardy space} $\nhp$ is then defined as the completion of
$\pnhp$ with respect to the $p$-quasi-norm $\|\cdot\|^p_{\nhp}$.
\end{definition}

\begin{remark}\label{r3.3}
(i) By the theorem of completion of Yosida \cite[p.\,56]{y95}, we see that
$\pnhp$ has a completion
space $\nhp$, namely,
for any $f\in\nhp$,
there exists a Cauchy sequence $\{f_k\}_{k=1}^\fz$
in $\pnhp$ such that
\begin{equation}\label{z3.1}
\lim_{k\to\fz}\|f_k-f\|^p_{\nhp}=0.
\end{equation}
Moreover, if $\{f_k\}_{k=1}^{\fz}$ is a Cauchy sequence in
$\pnhp$, then there exists a
unique $f\in\nhp$ such that \eqref{z3.1} holds true.

(ii) When $p=1$, the space $\wz {\mathbb{H}}_{\rm{atb},\,\rho}
^{1,\,q,\,\gz}(\mu)$ was introduced in \cite{hyy} and proved
to be a Banach space. Thus, $\wz {H}_{\rm{atb},\,\rho}^{1,\,q,\,\gz}(\mu)
=\wz {\mathbb{H}}_{\rm{atb},\,\rho}^{1,\,q,\,\gz}(\mu)$;
see also \cite{bd}.

(iii) Fix $p$, $\rho$ and $\gz$ as in Definition \ref{d3.2}.
For $1\le q_1\le q_2\le\fz$, we easily obtain
$$
\pnhpt\st\pnhpo.
$$

(iv) In Definition \ref{d3.2}, it seems natural to assume
$b\in\lq$ and to require $f=\sum_{i=1}^\fz b_i$ also holds true in $\lq$.
However, if so, then it is unclear whether (iii) of this
remark still holds true or not, which is crucial in applications
(see, for example, Remark \ref{r10.11}(i)).
\end{remark}

Now we show that any element in $\nhp$ has a decomposition in terms of
some $(p,\,q,\,\gz,\,\rho)_{\lz}$-atomic blocks,
$\{b_i\}_{i=1}^\fz$, in $\nhp$.

\begin{proposition}\label{p3.4}
Let $(\cx,d,\mu)$ be a non-homogeneous metric measure space,
$\rho\in (1,\fz)$, $0<p\le1\le q\le\fz$, $p\neq q$ and $\gz\in[1,\fz)$.
Then $f\in\nhp$ if and only if
there exist $(p,\,q,\,\gz,\,\rho)_{\lz}$-atomic blocks
$\{b_i\}_{i=1}^\fz$ such that
\begin{equation}\label{3.1}
f=\sum_{i=1}^\fz b_i\quad {\rm in}\quad  \nhp
\end{equation}
and
$\sum_{i=1}^{\fz}|b_{i}|^p_{\nhp}<\fz$. Moreover,
$$
\|f\|^p_{\nhp}=\inf
\lf\{\sum_{i=1}^{\fz}|b_{i}|^p_{\nhp}\r\},
$$
where the infimum is taken over all possible decompositions of
$f$ as in \eqref{3.1}.
\end{proposition}

\begin{proof}
We first assume that $f\in\nhp$.
Observe that, if \eqref{3.1} holds true, it is easy to see that
\begin{equation}\label{3.2}
\|f\|^p_{\nhp}\le\inf
\lf\{\sum_{i=1}^{\fz}|b_{i}|^p_{\nhp}\r\},
\end{equation}
where the infimum is taken over all possible decompositions of
$f$ as in \eqref{3.1}. It remains to prove \eqref{3.1}
and the reverse inequality of \eqref{3.2}.
For any $f\in\nhp$, we consider the following two cases.

\textbf{Case i)} $f\in\pnhp$. By Definition \ref{d3.2},
there exists a sequence of $(p,q,\gz,\rho)_\lz$-atomic blocks,
$\{b_{i}\}_{i=1}^\fz$, such that
$f=\sum_{i=1}^{\fz} b_{i}$ in $\ltw$ when $p\in(0,1)$
and in $\lon$ when $p=1$ and
$\sum_{i=1}^{\fz}|b_{i}|^p_{\nhp}<\fz.$
Now we claim that \eqref{3.1} holds true.

Indeed, for any $M\in\nn$,
$f-\sum_{i=1}^M b_{i}=\sum_{i=M+1}^\fz b_{i}$ in $\ltw$ when $p\in(0,1)$
and in $\lon$ when $p=1$. Then we know that
$$
\lf\|f-\sum_{i=1}^M b_{i}\r\|^p_{\nhp}
\le\sum_{i=M+1}^\fz |b_{i}|^p_{\nhp}\to 0
\quad {\rm as}\quad  M\to\fz.
$$
Then the claim holds true.
Again, by Definition \ref{d3.2} and \eqref{3.2},
we obtain the desired result for \textbf{Case i)}.

\textbf{Case ii)} $f\in\nhp\bh\pnhp$.
By Remark \ref{r3.3}(i), there exists a Cauchy sequence $\{f_k\}_{k=1}^\fz$
in $\pnhp$ such that
$$
\|f-f_k\|^p_{\nhp}\le 2^{-k-2}\|f\|^p_{\nhp}.
$$
It is easy to see that $f=\sum_{k=1}^\fz(f_k-f_{k-1})$ in $\nhp$, where
we let $f_0:=0$.
Since $f_k-f_{k-1}\in\pnhp$ for all $k\in\nn$,
by Definition \ref{d3.2} and \textbf{Case i)}, we see that, for any
$\ez\in(0,\fz)$ and any $k\in\nn$,
there exists a sequence  of $(p,q,\gz,\rho)_\lz$-atomic blocks,
$\{b_{k,\,i}\}_{i=1}^\fz$,  such that
$$f_k-f_{k-1}=\sum_{i=1}^{\fz} b_{k,\,i}\quad {\rm in\ both}\
\ltw\ {\rm when}\ p\in(0,1),\ {\rm or}\ \lon\ {\rm when}\
p=1,\ {\rm and}\ \nhp$$
and
$$\sum_{i=1}^{\fz}|b_{k,\,i}|^p_{\nhp}<\|f_k-f_{k-1}\|_{\nhp}^p
+\frac{\ez}{2^k}.$$
From this and $f=\sum_{k=1}^\fz(f_k-f_{k-1})$ in $\nhp$,
we further deduce that
$$
f=\sum_{k=1}^\fz(f_k-f_{k-1})=\sum_{k=1}^\fz\sum_{i=1}^\fz b_{k,\,i}\quad
{\rm in}\quad \nhp
$$
and
\begin{align*}
\sum_{k=1}^\fz\sum_{i=1}^\fz |b_{k,\,i}|^p_{\nhp}
&\le\sum_{k=1}^\fz\|f_k-f_{k-1}\|^p_{\nhp}+\sum_{k=1}^{\fz}\frac{\ez}{2^k}\\
&\le\sum_{k=1}^\fz\lf[\|f_k-f\|^p_{\nhp}+\|f_{k-1}-f\|^p_{\nhp}\r]+\ez\\
&\le\sum_{k=1}^\fz 2^{-k}\|f\|^p_{\nhp}+\ez=\|f\|^p_{\nhp}+\ez,
\end{align*}
which, together with the arbitrariness of $\ez$,
completes the proof of \textbf{Case ii)} and hence the necessity
of Proposition \ref{p3.4}.

Conversely, let $f=\sum_{i=1}^\fz b_i$
in $\nhp$ and
$$
\sum_{i=1}^\fz|b_i|^p_{\nhp}<\fz.
$$
Then, for each $k\in\nn$, $f_k=\sum_{i=1}^kb_i\in\pnhp$
and $\lim_{k\to\fz}f_k=f$ in $\nhp$.
Thus, $f\in\nhp$, which completes the proof of the sufficiency
and hence Proposition \ref{p3.4}.
\end{proof}

\section{Boundedness of Calder\'on-Zygmund Operators}\label{s4}

\hskip\parindent In this section,
we introduce the notion of the molecular Hardy space $\mhp$
and prove that the Calder\'on-Zygmund operator $T$ is bounded
from ${\widetilde H_{\rm{mb},\,\rho}^{p,\,q,\,\gamma,\,\delta}(\mu)}$
(or $\nhp$) into $\lp$, and from
${\widetilde H_{\rm{atb},\,\rho(\rho+1)}^{p,\,q,\,\gz+1}(\mu)}$ into
${\widetilde H_{\rm{mb},\,\rho}^{p,\,q,\,\gz,\,\frac12(\delta
-\frac{\nu}{p}+\nu)}(\mu)}$,
where $\dz$ is some positive constant depending on $T$;
see Definition \ref{d4.6} below.

We first introduce the notion of molecular Hardy spaces in a
non-homogeneous metric measure space.

\begin{definition}\label{d4.1}
Let $\rho\in (1,\fz)$, $0<p\le1\le q\le\fz$, $p\neq q$,
$\gz\in [1,\fz)$ and $\ez\in(0,\fz)$.
A function $b$ in $\ltw$ when $p\in(0,1)$ and in $\lon$ when
$p=1$ is called  a
\emph{$(p,q,\gz,\ez,\rho)_\lz$-molecular block} if

(i) $\int_\cx b(x)\,d\mu(x)=0$;

(ii) there exist some ball $B:=B(c_B,r_B)$, with
$c_B\in\cx$ and $r_B\in(0,\fz)$, and some constants
$\wz{M},\,M\in\nn$ such that,
for all $k\in\zz_+$ and $j\in\{1, \ldots, M_k\}$
with $M_k:=\wz{M}$ if $k=0$ and $M_k:=M$ if $k\in\nn$,
there exist functions $m_{k,\,j}$ supported
on some balls $B_{k,\,j}\st U_k(B)$
for all $k\in\zz_+$,
where $U_0(B):=\rho^2 B$ and $U_k(B):=\rho^{k+2}B\bh\rho^{k-2}B$
with $k\in\nn$, and
$\lz_{k,\,j}\in\cc$ such that
$b=\sum_{k=0}^{\fz}\sum_{j=1}^{M_k}\lz_{k,\,j}m_{k,\,j}$
in $\ltw$ when $p\in(0,1)$ and in $\lon$ when $p=1$,
\begin{equation}\label{4.1}
\|m_{k,\,j}\|_{\lq}\le \rho^{-k\ez}\lf[\mu(\rho B_{k,\,j})\r]^{1/q-1}
\lf[\lz\lf(c_{B},\rho^{k+2}r_{B}\r)\r]^{1-1/p}
\lf[\wz K^{(\rho),\,p}_{B_{k,\,j},\,\rho^{k+2}B}\r]^{-\gz}
\end{equation}
and
$$|b|^p_{\mhp}:=\sum_{k=0}^{\fz}\sum_{j=1}^{M_k}|\lz_{k,\,j}|^p<\fz.$$

A function $f$ is said to belong to the \emph{space} $\pmhp$
if there exists a sequence of $(p,q,\gz,\ez,\rho)_\lz$-molecular blocks,
$\{b_{i}\}_{i=1}^\fz$, such that
$f=\sum_{i=1}^{\fz} b_{i}$ in $\ltw$ when $p\in(0,1)$ and in $\lon$ when
$p=1$, and
$$\sum_{i=1}^{\fz}|b_{i}|^p_{\mhp}<\fz.$$
Moreover, define
$$\|f\|_{\mhp}:=\inf
\lf\{\lf[\sum_{i=1}^{\fz}|b_{i}|^p_{\mhp}\r]^{1/p}\r\},$$
where the infimum is taken over all possible decompositions of
$f$ as above.

The \emph{molecular Hardy space} $\mhp$
is then defined as the completion of
$\pmhp$ with respect to the $p$-quasi-norm $\|\cdot\|^p_{\mhp}$.
\end{definition}

\begin{remark}\label{r4.2}
(i) From the theorem of completion of Yosida \cite[p.\,56]{y95},
it follows that $\pmhp$ has a completion space $\mhp$, namely,
for any $f\in\mhp$,
there exists a Cauchy sequence $\{f_k\}_{k=1}^\fz$
in $\pmhp$ such that
\begin{equation}\label{z4.2}
\lim_{k\to\fz}\|f_k-f\|^p_{\mhp}=0.
\end{equation}
Moreover, if $\{f_k\}_{k=1}^{\fz}$ is a Cauchy sequence in
$\pmhp$, then there exists a unique $f\in\mhp$ such that
\eqref{z4.2} holds true.

(ii) It was proved, in \cite[Proposition 2.2(i)]{fyy3}, that
$\wz{\mathbb{H}}_{\rm{mb},\,\rho}^{1,\,q,\,\gz,\,\ez}(\mu)$ is a
Banach space and hence
$$
\wz H_{\rm{mb},\,\rho}^{1,\,q,\,\gz,\,\ez}(\mu)
=\wz{\mathbb{H}}_{\rm{mb},\,\rho}^{1,\,q,\,\gz,\,\ez}(\mu).
$$

(iii) Fix $p$, $\rho$, $\ez$ and $\gz$ as in Definition \ref{d4.1}.
For $1\le q_1\le q_2\le\fz$, we easily have
$$
\pmhpt\st\pmhpo.
$$

(iv) We point out that, via replacing the discrete coefficient
$\wz K^{(\rho),\,1}_{B,\,S}$ in Definitions
\ref{d3.2} and \ref{d4.1} by the continuous
coefficient $K_{B,\,S}$ as in \eqref{2.3x}, the atomic Hardy space
$H_{\rm atb,\,\rho}^{1,\,q,\,\gz}(\mu)$ and the molecular Hardy space
$H_{\rm mb,\,\rho}^{1,\,q,\,\gz,\,\ez}(\mu)$ were introduced,
respectively, in \cite{hyy}
and \cite{fyy3}. It was proved,
in \cite[Proposition 3.3(ii) and Theorem 3.8]{hyy},
that $H_{\rm atb,\,\rho}^{1,\,q,\,\gz}(\mu)$
is independent of the choices of $\rho$, $\gz$ and $q$.
Moreover, in \cite[Remark 2.3]{fyy3}, it was proved that
$H_{\rm atb,\,\rho}^{1,\,q,\,\gz}(\mu)$ and
$H_{\rm mb,\,\rho}^{1,\,p,\,\gz,\,\ez}(\mu)$
coincide with equivalent norms and hence
$H_{\rm atb,\,\rho}^{1,\,q,\,\gz}(\mu)$ is
independent of the choices of $\rho$, $\gz$, $q$ and $\ez$.
However, $H_{\rm atb,\,\rho}^{1,\,q,\,\gz}(\mu)$
and $\wz{H}_{\rm atb,\,\rho}^{1,\,q,\,\gz}(\mu)$
(or $H_{\rm mb,\,\rho}^{1,\,q,\,\gz,\,\ez}(\mu)$ and
$\wz{H}_{\rm mb,\,\rho}^{1,\,q,\,\gz,\,\ez}(\mu)$) may not coincide
(see \cite[Remark 1.9]{fyy3}) and the boundedness
of Calder\'on-Zygmund operators
on $H_{\rm atb,\,\rho}^{1,\,q,\,\gz}(\mu)$
over general non-homogeneous metric measure
spaces is also unclear (see \cite[Remark 2.4]{fyy3}).
\end{remark}

By \cite[Theorem 1.11]{fyy3}, we see that
$\wz H_{\rm{atb},\,\rho}^{1,\,q,\,\gz}(\mu)
=\wz H_{\rm{mb},\,\rho}^{1,\,q,\,\gz,\,\ez}(\mu)$.
For $p\in(0,1)$, we have the following conclusion.

\begin{proposition}\label{p4.3}
Suppose that $(\cx,d,\mu)$ is a non-homogeneous metric measure space.
Let $p\in(0,1)$, and $\rho$, $q$, $\gz$ and $\ez$
be as in Definition \ref{d4.1}.
Then $\pnhp\st\pmhp\st\lp$ and
there exist positive constants $C$ and $\wz C$
such that, for all $f\in\pnhp$,
$$
C\|f\|_{\lp}^p\le\|f\|_{\mhp}^p\le\wz C\|f\|_{\nhp}^p.
$$
\end{proposition}

\begin{proof}
Let $\rho$, $p$, $q$, $\gz$ and $\ez$ be as in Proposition \ref{p4.3}.
By the Fatou lemma, it suffices to prove that,
for any $(p,q,\gz,\rho)_\lz$-atomic block
$b$, $b$ is also a $(p,q,\gz,\ez,\rho)_\lz$-molecular block which
belongs to $\lp$ and
\begin{equation}\label{4.2}
\|b\|_{\lp}^p\ls|b|^p_{\mhp}\ls|b|^p_{\nhp}.
\end{equation}
By Definitions \ref{d3.2} and \ref{d4.1}, it is easy to see that,
for any $(p,q,\gz,\rho)_\lz$-atomic block $b$,
$b$ is also a $(p,q,\gz,\ez,\rho)_\lz$-molecular block and
$|b|^p_{\mhp}\ls|b|^p_{\nhp}$.

On the other hand, for any $(p,q,\gz,\ez,\rho)_\lz$-molecular block $b$
with the same notation as in Definition \ref{d4.1},
by the Fatou lemma, the H\"older inequality,
\eqref{4.1}, $B_{k,\,j}\st \rho^{k+2}B$
and \eqref{2.1}, we see that
\begin{align*}
\|b\|^p_\lp&\le\sum_{k=0}^\fz\sum_{j=1}^{M_k}
|\lz_{k,\,j}|^p\|m_{k,\,j}\|^p_\lp
\ls\sum_{k=0}^\fz\sum_{j=1}^{M_k}|\lz_{k,\,j}|^p\|m_{k,\,j}\|^p_\lq
\lf[\mu(B_{k,\,j})\r]^{1-p/q}\\
&\ls\sum_{k=0}^\fz\sum_{j=1}^{M_k}|\lz_{k,\,j}|^p
\lf[\mu(B_{k,\,j})\r]^{1-p/q}\rho^{-kp\ez}
\lf[\mu(\rho B_{k,\,j})\r]^{p/q-p}
\lf[\lz\lf(c_{B},\rho^{k+2}r_{B}\r)\r]^{p-1}\\
&\ls\sum_{k=0}^\fz\sum_{j=1}^{M_k}|\lz_{k,\,j}|^p\sim|b|^p_{\mhp},
\end{align*}
which completes the proof of Proposition \ref{p4.3}.
\end{proof}

\begin{remark}\label{r4.4}
Let $p\in(0,1)$, and $\rho$, $q$, $\gz$ and $\ez$
be as in Definition \ref{d4.1}.
By Proposition \ref{p4.3}, we easily conclude that
there exists a map $T$ from $\nhp$ to $\mhp$ such that,
for any $f\in\nhp$, there is a unique element $\wz f\in\mhp$
satisfying $Tf=\wz f$ and $\|\wz f\|_{\mhp}\ls\|f\|_{\nhp}$,
where the implicit positive constant is independent of $f$.
In this sense, we say that $\nhp\st\mhp$, which is different
from the classical inclusion relation of spaces, since it is still unclear
whether $T$ is an injection and $\|\wz f\|_{\mhp}\sim\|f\|_{\nhp}$ or not.
\end{remark}

Now we show that any element in $\mhp$ can be decomposed into a
sum of a sequence of $(p,\,q,\,\gz,\,\ez,\,\rho)_{\lz}$-molecular blocks,
$\{b_j\}_{j=1}^\fz$, in $\mhp$. The proof is similar to
that of Proposition \ref{p3.4}, the details being omitted.

\begin{proposition}\label{p4.5}
Suppose that $(\cx,d,\mu)$ is a non-homogeneous metric measure space.
Let $p$, $\rho$, $q$, $\gz$ and $\ez$ be as in Definition \ref{d4.1}.
Then $f\in\mhp$ if and only if there exist
$(p,\,q,\,\gz,\,\ez,\,\rho)_{\lz}$-molecular blocks
$\{b_i\}_{i=1}^\fz$ such that
\begin{equation}\label{4.3}
f=\sum_{i=1}^\fz b_i\quad {\rm in}\quad  \mhp
\end{equation}
and
$$\sum_{i=1}^\fz|b_i|^p_{\mhp}<\fz.$$
Moreover,
$$
\|f\|^p_{\mhp}=\inf
\lf\{\sum_{i=1}^{\fz}|b_{i}|^p_{\mhp}\r\},
$$
where the infimum is taken over all possible decompositions of
$f$ as in \eqref{4.3}.
\end{proposition}

Now we consider the boundedness of Calder\'on-Zygmund operators
on these atomic and molecular Hardy spaces.
To this end, we first recall the following notion
of Calder\'on-Zygmund operators from \cite{hm}.

\begin{definition}\label{d4.6}
A function $K\in L_\loc^1((\cx\times
\cx)\bh\{(x,x):x\in\cx\})$ is called  a \emph{Calder\'on-Zygmund
kernel} if there exists a positive constant $C_{(K)}$,
depending on $K$, such that

(i) for all $x,\,y\in\cx$ with $x\ne y$,
\begin{equation}\label{4.4}
|K(x,y)|\le C_{(K)}\frac{1}{\lz(x,d(x,y))};
\end{equation}

(ii) there exist positive constants
$\dz\in (0,1]$ and $c_{(K)}$, depending on $K$,
such that, for all $x,\,\wz x,\,y\in\cx$
with $d(x,y)\ge c_{(K)}d(x,\wz{x})$,
\begin{equation}\label{4.5}
|K(x,y)-K(\wz{x},y)|+|K(y,x)-K(y,\wz{x})|\le C_{(K)}
\frac{[d(x,\wz{x})]^{\dz}}{[d(x,y)]^{\dz}\lz(x,d(x,y))}.
\end{equation}

A linear operator $T$ is called  a \emph{Calder\'on-Zygmund operator}
with kernel $K$ satisfying \eqref{4.4} and \eqref{4.5} if, for
all $f\in L^{\fz}_b(\mu):=\{f\in\li:\ \supp(f)\ {\rm is\ bounded}\}$,
\begin{equation}\label{4.6}
Tf(x):=\int_{\cx}K(x,y)f(y)\,d\mu(y),\quad x\not\in\supp(f).
\end{equation}
\end{definition}

A new example of operators with kernel satisfying \eqref{4.5}
and \eqref{4.6} is the so-called Bergman-type operator appearing
in \cite{vw}; see also \cite{hm} for an explanation.

We first recall the following useful lemma from \cite{hlyy}.

\begin{lemma}\label{l4.7}
Suppose that $(\cx,d,\mu)$ is a non-homogeneous metric measure space.
Let $T$ be a Calder\'on-Zygmund operator defined by
\eqref{4.6} associated with kernel $K$
satisfying \eqref{4.4} and \eqref{4.5}.
Then the following statements are equivalent:

{\rm (i)} $T$ is bounded on $\ltw$;

{\rm (ii)} $T$ is bounded on $\lq$ for all $q\in(1, \fz)$;

{\rm (iii)} $T$ is bounded from $\lon$ to weak-$\lon$.
\end{lemma}

Now we prove the boundedness of Calder\'on-Zygmund operators from
$\wz H_{\rm{mb},\,\rho}^{p,\,q,\,\gz,\,\dz}(\mu)$ into $\lp$.
Hereafter, let
$\nu:=\log_2C_{(\lz)}$, and $\dz$ be as in Definition \ref{d4.6}.

\begin{theorem}\label{t4.8}
Suppose that $(\cx,d,\mu)$ is a non-homogeneous metric measure space.
Let $\rho\in(1,\fz)$, $\frac{\nu}{\nu+\dz}<p\le1<q<\fz$
and $\gz\in[1,\fz)$. Assume that
the Calder\'on-Zygmund operator $T$ defined by
\eqref{4.6} associated with kernel $K$
satisfying \eqref{4.4} and \eqref{4.5} is bounded on $\ltw$. Then
$T$ is bounded from
$\wz H_{\rm{mb},\,\rho}^{p,\,q,\,\gz,\,\dz}(\mu)$ into $\lp$.
\end{theorem}

\begin{proof}
Let $\rho$, $p$, $q$ and $\gz$ be as in the assumptions of Theorem \ref{t4.8}.
For the sake of simplicity, we take $\rho=2$ and $\gz=1$.
With some slight modifications, the arguments here are still valid for
general cases.
We first reduce the proof to showing
that, for all $(p,q,1,\dz,2)_\lz$-molecular blocks $b$,
\begin{equation}\label{4.7}
\|Tb\|_{\lp}\ls|b|_{\mhpd}.
\end{equation}

Indeed, assume that \eqref{4.7} holds true. For any
$f\in\wz {\mathbb{H}}_{\rm{mb},\,2}^{p,\,q,\,1,\,\dz}(\mu)$,
there exists a sequence $\{b_i\}_{i\in\nn}$ of
$(p,q,1,\dz,2)_\lz$-molecular blocks
such that $f=\sum_{i=1}^\fz b_i$ in $\ltw$
when $p\in(0,1)$ and in $\lon$ when $p=1$ and
$$
\sum_{i=1}^\fz|b_i|^p_{\wz H_{\rm{mb},\,2}^{p,\,q,\,1,\,\dz}(\mu)}
\sim\|f\|^p_{\wz H_{\rm{mb},\,2}^{p,\,q,\,1,\,\dz}(\mu)}.
$$
If $f=\sum_{i=1}^\fz b_i$ in $\ltw$, then,
by the boundedness of $T$ on $\ltw$, we see that, for any $N\in\nn$,
$$
\lf\|\sum_{i=1}^NT(b_i)-Tf\r\|_{\ltw}=
\lf\|T\lf(\sum_{i=1}^Nb_i-f\r)\r\|_{\ltw}
\ls\lf\|\sum_{i=1}^Nb_i-f\r\|_{\ltw}\to 0 \quad {\rm as} \quad {N\to\fz},
$$
which further implies that, for all $\eta\in(0,\fz)$,
\begin{equation}\label{4.7x}
\mu\lf(\lf\{x\in\cx:\ \lf|\sum_{i=1}^N T(b_i)(x)-Tf(x)\r|>\eta\r\}\r)
\to 0 \quad {\rm as} \quad {N\to\fz}.
\end{equation}
If $f=\sum_{i=1}^\fz b_i$ in $\lon$, then, by the boundedness of $T$
from $\lon$ to weak-$\lon$, we still know that
\eqref{4.7x} holds true.
Thus, by the Riesz theorem, we know that there exists a subsequence of partial sums,
$\{\sum_{i=1}^{N_k}T(b_i)\}_k$, such that
$$
Tf=\lim_{k\to\fz}\sum_{i=1}^{N_k}T(b_i) \quad
\mu-{\rm almost\ everywhere\ on}\ \cx,
$$
which, together with the Fatou lemma and \eqref{4.7}, implies that
\begin{align*}
\|Tf\|_{\lp}^p&\le\liminf_{k\to\fz}\int_{\cx}
\sum_{i=1}^{N_k}|T(b_i)(x)|^p\,d\mu(x)
\le\sum_{i=1}^\fz\|T(b_i)\|_{\lp}^p\\
&\ls\sum_{i=1}^\fz|b_i|^p_{\wz H_{\rm{mb},\,2}^{p,\,q,\,1,\,\dz}(\mu)}
\sim\|f\|^p_{\wz H_{\rm{mb},\,2}^{p,\,q,\,1,\,\dz}(\mu)}.
\end{align*}
Moreover, by a standard density argument, we extend $T$ to be a
bounded linear operator
from $\wz H_{\rm{mb},\,2}^{p,\,q,\,1,\,\dz}(\mu)$
into $\lp$, which is the desired result.

Now we prove \eqref{4.7}.
Let $b=\sum_{k=0}^{\fz}\sum_{j=1}^{M_k}\lz_{k,\,j}m_{k,\,j}$
be a $(p,q,1,\dz,2)_\lz$-molecular block, where, for
any $k\in\zz_+$ and
$j\in\{1,\ldots,M_k\}$, $\supp(m_{k,\,j})\st B_{k,\,j}\st U_k(B)$
for some balls $B$ and $B_{k,\,j}$ as in
Definition \ref{d4.1}.
Without loss of generality, we may assume that $\wz M=M$ in
Definition \ref{d4.1}.

By the linearity of $T$, we write
\begin{align*}
\|Tb\|^p_{\lp}&\le\sum_{\ell=5}^\fz
\int_{U_\ell(B)}\lf|T\lf(\sum_{k=0}^{\ell-5}
\sum_{j=1}^M\lz_{k,\,j}m_{k,\,j}\r)(x)\r|^p\,d\mu(x)\\
&\hs+\sum_{\ell=5}^\fz\int_{U_\ell(B)}\lf|T\lf(\sum_{k=\ell-4}^{\ell+4}
\sum_{j=1}^M\lz_{k,\,j}m_{k,\,j}\r)(x)\r|^p\,d\mu(x)\\
&\hs+\sum_{\ell=5}^\fz\int_{U_\ell(B)}\lf|T\lf(\sum_{k=\ell+5}^{\fz}
\sum_{j=1}^M\lz_{k,\,j}m_{k,\,j}\r)(x)\r|^p\,d\mu(x)+\sum_{\ell=0}^4\int_{U_\ell(B)}|Tb(x)|^p\,d\mu(x)\\
&=:{\rm I}+{\rm II}+{\rm III}+{\rm IV}.
\end{align*}
Now we first estimate III. By \eqref{4.4}, \eqref{2.1}, \eqref{2.2},
the H\"older inequality and \eqref{4.1}, we obtain
\begin{align*}
{\rm III}&\le\sum_{\ell=5}^\fz\sum_{k=\ell+5}^\fz\sum_{j=1}^M
|\lz_{k,\,j}|^p\int_{U_\ell(B)}\lf[\int_{B_{k,\,j}}|m_{k,\,j}(y)|
|K(x,y)|\,d\mu(y)\r]^p\,d\mu(x)\\
&\ls\sum_{\ell=5}^\fz\sum_{k=\ell+5}^\fz\sum_{j=1}^M
|\lz_{k,\,j}|^p\int_{U_\ell(B)}\lf[\int_{B_{k,\,j}}
\frac{|m_{k,\,j}(y)|}{\lz(x,d(x,y))}\,d\mu(y)\r]^p\,d\mu(x)\\
&\ls\sum_{\ell=5}^\fz\sum_{k=\ell+5}^\fz\sum_{j=1}^M
|\lz_{k,\,j}|^p\int_{U_\ell(B)}
\frac1{[\lz(c_B,d(x,c_B))]^p}\,d\mu(x)\|m_{k,\,j}\|^p_{\lon}\\
&\ls\sum_{\ell=5}^\fz\sum_{k=\ell+5}^\fz\sum_{j=1}^M
|\lz_{k,\,j}|^p
\frac{\mu(2^{\ell+2}B)}{[\lz(c_B,2^{\ell-2}r_B)]^p}
\lf[\mu(B_{k,\,j})\r]^{p/q'}\|m_{k,\,j}\|^p_{\lq}\\
&\ls\sum_{\ell=5}^\fz\sum_{k=\ell+5}^\fz\sum_{j=1}^M
|\lz_{k,\,j}|^p
\lf[\mu\lf(2^{\ell+2}B\r)\r]^{1-p}\lf[\mu(B_{k,\,j})\r]^{p/q'}\\
&\hs\times2^{-k\dz p}\lf[\mu(2B_{k,\,j})\r]^{-p/q'}
\lf[\lz\lf(c_B,2^{k+2}r_B\r)\r]^{p-1}\\
&\ls\sum_{\ell=5}^\fz\sum_{k=\ell+5}^\fz\sum_{j=1}^M
2^{-k\dz p}|\lz_{k,\,j}|^p
\sim\sum_{j=1}^M\sum_{k=10}^\fz\sum_{\ell=5}^{k-5}
2^{-k\dz p}|\lz_{k,\,j}|^p\\
&\ls\sum_{j=1}^M\sum_{k=10}^\fz k
2^{-k\dz p}|\lz_{k,\,j}|^p\ls\sum_{k=0}^\fz\sum_{j=1}^M|\lz_{k,\,j}|^p
\sim|b|^p_{\mhpd}.
\end{align*}

In order to estimate I, write
\begin{align*}
{\rm I}&\le\sum_{\ell=5}^\fz\int_{U_\ell(B)}
\lf|\int_{\cx}\lf[\sum_{k=0}^{\ell-5}
\sum_{j=1}^M\lz_{k,\,j}m_{k,\,j}(y)\r][K(x,y)-K(x,c_B)]
\,d\mu(y)\r|^p\,d\mu(x)\\
&\hs+\sum_{\ell=5}^\fz\int_{U_\ell(B)}\lf|\int_{\cx}\lf[\sum_{k=0}^{\ell-5}
\sum_{j=1}^M\lz_{k,\,j}m_{k,\,j}(y)\r]K(x,c_B)\,d\mu(y)\r|^p\,d\mu(x)
=:{\rm I_1}+{\rm I_2}.
\end{align*}
From \eqref{4.5}, \eqref{2.2}, \eqref{2.1}, the H\"older inequality,
\eqref{4.1} and the fact that $p\in(\frac{\nu}{\nu+\dz},1]$, it follows that
\begin{align*}
{\rm I_1}&\ls\sum_{\ell=5}^\fz\sum_{k=0}^{\ell-5}\sum_{j=1}^M
|\lz_{k,\,j}|^p
\int_{U_\ell(B)}\lf\{\int_{B_{k,\,j}}
\frac{|m_{k,\,j}(y)|[d(y,c_B)]^\dz}
{[d(x,c_B)]^\dz\lz(c_B,d(x,c_B))}\,d\mu(y)\r\}^p\,d\mu(x)\\
&\ls\sum_{\ell=5}^\fz\sum_{k=0}^{\ell-5}\sum_{j=1}^M
|\lz_{k,\,j}|^p\frac{2^{(k+2)\dz p}r_B^{\dz p}\mu(2^{\ell+2}B)}
{2^{(\ell-2)\dz p}r_B^{\dz p}[\lz(c_B,2^{\ell-2}r_B)]^p}
\|m_{k,\,j}\|^p_{\lon}\\
&\ls\sum_{\ell=5}^\fz\sum_{k=0}^{\ell-5}\sum_{j=1}^M
|\lz_{k,\,j}|^p2^{(k-\ell)\dz p}\lf[\mu\lf(2^{\ell+2}B\r)\r]^{1-p}
\lf[\mu\lf(B_{k,\,j}\r)\r]^{p/q'}\|m_{k,\,j}\|^p_{\lq}\\
&\ls\sum_{\ell=5}^\fz\sum_{k=0}^{\ell-5}\sum_{j=1}^M
|\lz_{k,\,j}|^p2^{(k-\ell)\dz p}\lf[\mu\lf(2^{\ell+2}B\r)\r]^{1-p}
\lf[\mu\lf(B_{k,\,j}\r)\r]^{p/q'}\\
&\hs\times2^{-k\dz p}\lf[\mu(2B_{k,\,j})\r]^{-p/q'}
\lf[\lz\lf(c_B,2^{k+2}r_B\r)\r]^{p-1}\\
&\ls\sum_{\ell=5}^\fz\sum_{k=0}^{\ell-5}\sum_{j=1}^M
|\lz_{k,\,j}|^p2^{-\ell\dz p}\lf[\mu\lf(2^{\ell+2}B\r)\r]^{1-p}
\lf[\lz(c_B,2^{\ell+2}r_B)\r]^{p-1}\lf[C_{(\lz)}\r]^{(\ell-k)(1-p)}\\
&\ls\sum_{\ell=5}^\fz\sum_{k=0}^{\ell-5}\sum_{j=1}^M
|\lz_{k,\,j}|^p2^{[\nu(1-p)-\dz p]\ell}
2^{-k\nu(1-p)}
\ls\sum_{k=0}^\fz\sum_{j=1}^M
|\lz_{k,\,j}|^p\sim|b|^p_{\mhpd}.
\end{align*}
For ${\rm I}_2$, the vanishing moment of $b$, together with
\eqref{4.4}, \eqref{2.1} and \eqref{2.2}, implies that
\begin{align*}
{\rm I_2}&=\sum_{\ell=5}^\fz\int_{U_\ell(B)}
\lf|\int_{\cx}\lf[\sum_{k=\ell-4}^{\fz}
\sum_{j=1}^M\lz_{k,\,j}m_{k,\,j}(y)\r]K(x,c_B)\,d\mu(y)\r|^p\,d\mu(x)\\
&\ls\sum_{\ell=5}^\fz\sum_{k=\ell-4}^{\fz}\sum_{j=1}^M
|\lz_{k,\,j}|^p
\int_{U_\ell(B)}\lf[\int_{B_{k,\,j}}|m_{k,\,j}(y)|\frac{1}
{\lz(c_B,d(x,c_B))}\,d\mu(y)\r]^p\,d\mu(x)\\
&\ls\sum_{\ell=5}^\fz\sum_{k=\ell-4}^{\fz}\sum_{j=1}^M
|\lz_{k,\,j}|^p\frac{\mu(2^{\ell+2}B)}{[\lz(c_B,2^{\ell-2}r_B)]^p}
\|m_{k,\,j}\|^p_{\lon}\\
&\ls\sum_{\ell=5}^\fz\sum_{k=\ell-4}^{\fz}\sum_{j=1}^M
|\lz_{k,\,j}|^p\lf[\mu\lf(2^{\ell+2}B\r)\r]^{1-p}
\lf[\mu\lf(B_{k,\,j}\r)\r]^{p/q'}
\|m_{k,\,j}\|^p_{\lq}\\
&\ls\sum_{\ell=5}^\fz\sum_{k=\ell-4}^{\fz}\sum_{j=1}^M
|\lz_{k,\,j}|^p\lf[\mu\lf(2^{\ell+2}B\r)\r]^{1-p}
\lf[\mu\lf(B_{k,\,j}\r)\r]^{p/q'}\\
&\hs\times2^{-k\dz p}\lf[\mu(2B_{k,\,j})\r]^{-p/q'}
\lf[\lz\lf(c_B,2^{k+2}r_B\r)\r]^{p-1}\\
&\ls\sum_{\ell=5}^\fz\sum_{k=\ell-4}^{\fz}\sum_{j=1}^M2^{-k\dz p}
|\lz_{k,\,j}|^p\sim\sum_{j=1}^M\sum_{k=1}^\fz\sum_{\ell=5}^{k+4}
2^{-k\dz p}|\lz_{k,\,j}|^p\\
&\ls\sum_{j=1}^M\sum_{k=0}^\fz k
2^{-k\dz p}|\lz_{k,\,j}|^p\ls\sum_{k=0}^\fz\sum_{j=1}^M|\lz_{k,\,j}|^p
\sim|b|^p_{\mhpd}.
\end{align*}
Combining ${\rm I_1}$ and ${\rm I_2}$, we conclude that
${\rm I}\ls|b|^p_{\mhpd}$.

Then we turn to estimate II. We further write
\begin{align*}
{\rm II}&\le\sum_{\ell=5}^\fz\sum_{k=\ell-4}^{\ell+4}
\sum_{j=1}^M|\lz_{k,\,j}|^p\int_{U_\ell(B)}|Tm_{k,\,j}(x)|^p\,d\mu(x)\\
&\le\sum_{\ell=5}^\fz\sum_{k=\ell-4}^{\ell+4}
\sum_{j=1}^M|\lz_{k,\,j}|^p\int_{2B_{k,\,j}}|Tm_{k,\,j}(x)|^p\,d\mu(x)\\
&\hs+\sum_{\ell=5}^\fz\sum_{k=\ell-4}^{\ell+4}
\sum_{j=1}^M|\lz_{k,\,j}|^p\int_{U_\ell(B)\bh2B_{k,\,j}}
\cdots=:{\rm II_1}+{\rm II_2}.
\end{align*}
By the H\"older inequality, $\ltw$-boundedness of $T$, Lemma \ref{l4.7},
\eqref{4.1} and \eqref{2.1}, we see that
\begin{align*}
{\rm II_1}&\le\sum_{\ell=5}^\fz\sum_{k=\ell-4}^{\ell+4}
\sum_{j=1}^M|\lz_{k,\,j}|^p\lf[\mu(2B_{k,\,j})\r]^{1-p/q}
\|Tm_{k,\,j}\|^p_{\lq}\\
&\ls\sum_{\ell=5}^\fz\sum_{k=\ell-4}^{\ell+4}
\sum_{j=1}^M|\lz_{k,\,j}|^p\lf[\mu(2B_{k,\,j})\r]^{1-p/q}
\|m_{k,\,j}\|^p_{\lq}\\
&\ls\sum_{\ell=5}^\fz\sum_{k=\ell-4}^{\ell+4}
\sum_{j=1}^M|\lz_{k,\,j}|^p\lf[\mu(2B_{k,\,j})\r]^{1-p/q}
2^{-k\dz p}\lf[\mu(2B_{k,\,j})\r]^{-p/q'}
\lf[\lz\lf(c_B,2^{k+2}r_B\r)\r]^{p-1}\\
&\ls\sum_{\ell=5}^\fz\sum_{k=\ell-4}^{\ell+4}
\sum_{j=1}^M 2^{-k\dz p}|\lz_{k,\,j}|^p
\ls\sum_{k=0}^\fz\sum_{j=1}^M|\lz_{k,\,j}|^p
\sim|b|^p_{\mhpd}.
\end{align*}
For ${\rm II_2}$, from \eqref{4.4},
$d(x,y)\ge d(x,c_{B_{k,\,j}})-d(y,c_{B_{k,\,j}})
\ge\frac12 d(x,c_{B_{k,\,j}})$ for
$x\notin 2B_{k,\,j}$ and $y\in B_{k,\,j}$,
\eqref{2.2}, \eqref{2.1}, the H\"older inequality and \eqref{4.1},
we deduce that
\begin{align*}
{\rm II_2}&\ls\sum_{\ell=5}^\fz\sum_{k=\ell-4}^{\ell+4}
\sum_{j=1}^M|\lz_{k,\,j}|^p\int_{U_\ell(B)\bh2B_{k,\,j}}
\lf[\int_{B_{k,\,j}}\frac{|m_{k,\,j}(y)|}
{\lz(c_{B_{k,\,j}},d(x,c_{B_{k,\,j}}))}
\,d\mu(y)\r]^p\,d\mu(x)\\
&\ls\sum_{\ell=5}^\fz\sum_{k=\ell-4}^{\ell+4}
\sum_{j=1}^M|\lz_{k,\,j}|^p\int_{2^{k+6}B\bh B_{k,\,j}}
\frac1{[\lz(c_{B_{k,\,j}},d(x,c_{B_{k,\,j}}))]^p}\,d\mu(x)
\|m_{k,\,j}\|^p_{\lon}\\
&\ls\sum_{\ell=5}^\fz\sum_{k=\ell-4}^{\ell+4}
\sum_{j=1}^M|\lz_{k,\,j}|^p
\|m_{k,\,j}\|^p_{\lon}
\sum_{i=0}^{N^{(2)}_{B_{k,\,j},\,2^{k+5}B}+1}
\frac{\mu(2^{i+1}B_{k,\,j})}
{[\lz(c_{B_{k,\,j}},2^{i}r_{B_{k,\,j}})]^p}\\
&\ls\sum_{\ell=5}^\fz\sum_{k=\ell-4}^{\ell+4}
\sum_{j=1}^M|\lz_{k,\,j}|^p
\|m_{k,\,j}\|^p_{\lon}
\lf[\mu\lf(2^{N^{(2)}_{B_{k,\,j},\,2^{k+5}B}+2}B_{k,\,j}\r)\r]^{1-p}
\lf[\wz K^{(2),\,p}_{B_{k,\,j},\,2^{k+5}B}\r]^p\\
&\ls\sum_{\ell=5}^\fz\sum_{k=\ell-4}^{\ell+4}
\sum_{j=1}^M|\lz_{k,\,j}|^p\lf[\mu\lf(B_{k,\,j}\r)\r]^{p/q'}
\|m_{k,\,j}\|^p_{\lq}
\lf[\mu\lf(2^{k+9}B\r)\r]^{1-p}
\lf[\wz K^{(2),\,p}_{B_{k,\,j},\,2^{k+5}B}\r]^p\\
&\ls\sum_{\ell=5}^\fz\sum_{k=\ell-4}^{\ell+4}
\sum_{j=1}^M|\lz_{k,\,j}|^p\lf[\mu\lf(B_{k,\,j}\r)\r]^{p/q'}
2^{-k\dz p}\lf[\mu(2B_{k,\,j})\r]^{-p/q'}
\lf[\lz\lf(c_B,2^{k+2}r_B\r)\r]^{p-1}\\
&\hs\times\lf[\wz K^{(2),\,p}_{B_{k,\,j},\,2^{k+2}B}\r]^{-p}
\lf[\mu\lf(2^{k+9}B\r)\r]^{1-p}
\lf[\wz K^{(2),\,p}_{B_{k,\,j},\,2^{k+5}B}\r]^p\\
&\ls\sum_{\ell=5}^\fz\sum_{k=\ell-4}^{\ell+4}
\sum_{j=1}^M 2^{-k\dz p}|\lz_{k,\,j}|^p
\ls\sum_{k=0}^\fz\sum_{j=1}^M|\lz_{k,\,j}|^p
\sim|b|^p_{\mhp},
\end{align*}
which, together with the estimate for ${\rm II_1}$,
implies that ${\rm II}\ls|b|^p_{\mhp}$.

To estimate ${\rm IV}$, observe that
\begin{align*}
{\rm IV}&\le\sum_{\ell=0}^4\int_{U_\ell(B)}\lf|T\lf(\sum_{k=0}^{\ell+4}
\sum_{j=1}^M\lz_{k,\,j}m_{k,\,j}\r)(x)\r|^p\,d\mu(x)\\
&\hs+\sum_{\ell=0}^4\int_{U_\ell(B)}\lf|T\lf(\sum_{k=\ell+5}^{\fz}
\sum_{j=1}^M\lz_{k,\,j}m_{k,\,j}\r)(x)\r|^p\,d\mu(x)
=:{\rm IV}_1+{\rm IV}_2.
\end{align*}
By some arguments similar to those used in the estimates for $\rm II_1$
and $\rm III$, we respectively obtain
$${\rm IV}_1\ls|b|^p_{\mhpd}\quad {\rm and}\quad {\rm IV}_2\ls|b|^p_{\mhpd},$$
which, together with the estimates for $\rm III$, $\rm I$
and $\rm II$, completes the proof of Theorem \ref{t4.8}.
\end{proof}

Now we show the boundedness of
Calder\'on-Zygmund operators from $\nhp$ into $\lp$.

\begin{corollary}\label{c4.9}
Suppose that $(\cx,d,\mu)$ is a non-homogeneous metric measure space.
Let $\rho\in(1,\fz)$, $\frac{\nu}{\nu+\dz}<p\le1< q<\fz$
and $\gz\in[1,\fz)$. Assume that
the Calder\'on-Zygmund operator $T$ defined by
\eqref{4.6} associated with kernel $K$
satisfying \eqref{4.4} and \eqref{4.5} is bounded on $\ltw$. Then
$T$ is bounded from $\nhp$ into $\lp$.
\end{corollary}

\begin{proof}
Let $\rho$, $p$, $q$, $\gz$ and $\dz$
be as in assumptions of Corollary \ref{c4.9}.
For the sake of simplicity, we take $\rho=2$ and $\gz=1$.
By an argument similar to that used in the proof of
Theorem \ref{t4.8}, it suffices to show that, for any
$(p,q,1,2)_\lz$-atomic block $b$,
$$
\|Tb\|_{\lp}\ls|b|_{\wz H_{\rm{atb},\,2}^{p,\,q,\,1}(\mu)},
$$
which is an easy consequence of the facts that
$b$ is also a $(p,q,1,\dz,2)_\lz$-molecular block
and $|b|_{\wz H_{\rm{mb},\,2}^{p,\,q,\,1,\,\dz}(\mu)}
\ls|b|_{\wz H_{\rm{atb},\,2}^{p,\,q,\,1}(\mu)}$
(see \eqref{4.2}), together with \eqref{4.7}
from the proof of Theorem \ref{t4.8}.
This finishes the proof of Corollary \ref{c4.9}.
\end{proof}

\begin{remark}\label{r4.10}
When $p=1$, Theorem \ref{t4.8} or Corollary \ref{c4.9} is
a special case of \cite[Theorem 4.1]{hyy}, since, for any
$q\in(1,\fz]$, $\rho\in(1,\fz)$, $\gz\in[1,\fz)$ and $\ez\in(0,\fz)$,
${\wz H}_{\rm{atb},\,\rho}^{1,\,q,\,\gz}(\mu)\st
H_{\rm{atb},\,\rho}^{1,\,q,\,\gz}(\mu)$, where
$H_{\rm{atb},\,\rho}^{1,\,q,\,\gz}(\mu)$
is introduced in \cite{hyy} (see \cite[Remark 1.9(i)]{fyy3}),
and, by \cite[Theorem 1.11 and Remark 1.9(i)]{fyy3}, we know that
${\wz H}_{\rm{atb},\,\rho}^{1,\,q,\,\gz}(\mu)$
is independent of the choices of $q$, $\rho$ and $\gz$, and
${\wz H}_{\rm{atb},\,\rho}^{1,\,q,\,\gz}(\mu)$ and
${\wz H}_{\rm{mb},\,\rho}^{1,\,q,\,\gz,\,\ez}(\mu)$
coincide with equivalent norms.
\end{remark}

Now we establish the (${\widetilde H_{\rm{atb},\,\rho(\rho+1)}^{p,\,q,\,\gz+1}(\mu)}$,
${\widetilde H_{\rm{mb},\,\rho}^{p,\,q,\,\gz,\,\frac12(\delta
-\frac{\nu}{p}+\nu)}(\mu)}$)-boundedness of Calder\'on-Zygmund operators.
In what follows, for $T$ as in Corollary
\ref{c4.9}, $T$ is said to satisfy $T^*1=0$ if,
for all $h\in L^{\fz}_b(\mu)$ with
$\int_{\cx}h(y)\,d\mu(y)=0$,
$$\int_{\cx}Th(y)\,d\mu(y)=0.$$
Observe that, for such $T$ and $h$, by Corollary \ref{c4.9},
we have $Th\in\lon$.

\begin{theorem}\label{t4.11}
Suppose that $(\cx,d,\mu)$ is a non-homogeneous metric measure space.
Let $\rho\in[2,\fz)$, $\frac{\nu}{\nu+\dz}<p\le1< q<\fz$ and
$\gz\in[1,\fz)$. Assume that
the Calder\'on-Zygmund operator $T$ defined by
\eqref{4.6} associated with kernel $K$
satisfying \eqref{4.4} and \eqref{4.5} is bounded on $\ltw$
and $T^*1=0$. Then $T$ is bounded from
${\widetilde H_{\rm{atb},\,\rho(\rho+1)}^{p,\,q,\,\gz+1}(\mu)}$ into
${\widetilde H_{\rm{mb},\,\rho}^{p,\,q,\,\gz,\,\frac12(\delta
-\frac{\nu}{p}+\nu)}(\mu)}$.
\end{theorem}

\begin{proof}
Observe that, when $p=1$, Theorem \ref{t4.11} is a special case of
\cite[Theorem 1.14]{fyy3},
since it was shown, by \cite[Theorem 1.11 and Remark 1.9(i)]{fyy3}, that,
for any $q\in(1,\fz]$, $\rho\in(1,\fz)$, $\gz\in[1,\fz)$
and $\ez\in(0,\fz)$,
${\wz H}_{\rm{atb},\,\rho}^{1,\,q,\,\gz}(\mu)$ and
${\wz H}_{\rm{mb},\,\rho}^{1,\,q,\,\gz,\,\ez}(\mu)$
coincide with equivalent norms,
and ${\wz H}_{\rm{atb},\,\rho}^{1,\,q,\,\gz}(\mu)$ is independent
of the choices of $q$, $\rho$ and $\gz$.
Thus, to show Theorem \ref{t4.11}, we only need to consider
the case $p\in(\frac{\nu}{\nu+\dz},1)$.
Moreover, for the sake of simplicity, we assume that $\gz=1$ and $\rho=2$.
Via some slight modifications, the arguments here are still valid for
general cases.
We first reduce our proof to showing that, for
any $(p,q,2,6)_\lz$-atomic block, $Tb$ is a
$(p,q,1,\frac12(\dz-\frac{\nu}{p}+\nu),2)_{\lz}$-molecular block and
\begin{equation}\label{4.8}
|Tb|_{{\widetilde H_{\rm{mb},\,2}^{p,\,q,\,1,\,\frac12(\delta
-\frac{\nu}{p}+\nu)}(\mu)}}
\ls|b|_{{\widetilde H_{\rm{atb},\,6}^{p,\,q,\,2}(\mu)}}.
\end{equation}

Indeed, assume that \eqref{4.8} holds true. For any
${\widetilde {\mathbb{H}}_{\rm{atb},\,6}^{p,\,q,\,2}(\mu)}$,
there exists a sequence $\{b_i\}_{i\in\nn}$ of $(p,q,2,6)_\lz$-atomic blocks
such that $f=\sum_{i=1}^\fz b_i$ in $\ltw$ and
$$
\sum_{i=1}^\fz|b_i|^p
_{{\widetilde H_{\rm{atb},\,6}^{p,\,q,\,2}(\mu)}}
\sim\|f\|^p_{{\widetilde H_{\rm{atb},\,6}^{p,\,q,\,2}(\mu)}}.
$$
By the boundedness of $T$ on $\ltw$, we see that
$$
\lf\|\sum_{i=1}^NT(b_i)-Tf\r\|_{\ltw}=
\lf\|T\lf(\sum_{i=1}^Nb_i-f\r)\r\|_{\ltw}
\ls\lf\|\sum_{i=1}^Nb_i-f\r\|_{\ltw}\to 0 \quad {\rm as} \quad {N\to\fz}.
$$
Thus, $Tf=\sum_{i=1}^\fz T(b_i)$ in $\ltw$. Moreover, by \eqref{4.8} and
$T(b_i)$ is a $(p,q,1,\frac12(\dz-\frac{\nu}{p}+\nu),2)_{\lz}$-molecular block
for any $i\in\nn$, we know that
$$
\|Tf\|^p_{{\widetilde H_{\rm{mb},\,2}^{p,\,q,\,1,\,\frac12(\delta
-\frac{\nu}{p}+\nu)}(\mu)}}\le\sum_{i=1}^\fz
|T(b_i)|^p_{{\widetilde H_{\rm{mb},\,2}^{p,\,q,\,1,\,\frac12(\delta
-\frac{\nu}{p}+\nu)}(\mu)}}
\ls\sum_{i=1}^\fz|b_i|^p_{{\widetilde H_{\rm{atb},\,6}^{p,\,q,\,2}(\mu)}}
\sim\|f\|^p_{{\widetilde H_{\rm{atb},\,6}^{p,\,q,\,2}(\mu)}}.
$$
Furthermore, by a standard density argument, we extend $T$ to be a bounded
linear operator from ${\widetilde H_{\rm{atb},\,6}^{p,\,q,\,2}(\mu)}$ into
${\widetilde H_{\rm{mb},\,2}^{p,\,q,\,1,\,\frac12(\delta
-\frac{\nu}{p}+\nu)}(\mu)}$.

Now we show that \eqref{4.8} holds true.
Let $b$ be a $(p,q,2,6)_\lz$-atomic block.
Then $b:=\sum_{j=1}^2\lz_j a_j$,
where, for any $j\in\{1,2\}$, $\supp(a_j)\st B_j\st B$
for some balls $B_j$ and $B$
as in Definition \ref{d3.2}.
Let $B_0:=8B$. We write
$$Tb=(Tb)\chi_{B_0}+\sum_{k=1}^{\fz}(Tb)\chi_{2^{k}B_0\bh 2^{k-1}B_0}
=:{\rm A_1}+{\rm A_2}.$$

We first estimate ${\rm A_1}$. Since $B_j\st B$, we have $3B_j\st8B=B_0$.
Let $N_j:=N^{(2)}_{2B_j,\,B_0}$. Obviously, $N_j\ge-1$.
Without loss of generality, we may assume that $N_j\ge3$.
For the case $N_j\in[-1,3)$, we easily observe that
$2B_j\st B_0\st2^{5}B_j$, which can be reduced to the case $N_j\ge3$.
We further decompose
\begin{align*}
{\rm A_1}&=\sum_{j=1}^2\lz_j (Ta_j)\chi_{2B_j}
+\sum_{j=1}^2\sum_{i=1}^{N_j-2}
\lz_j (Ta_j)\chi_{2^{i+1}B_j\bh2^iB_j}
+\sum_{j=1}^2\lz_j (Ta_j)\chi_{B_0\bh 2^{N_j-1}B_j}\\
&=:{\rm A_{1,\,1}}+{\rm A_{1,\,2}}+{\rm A_{1,\,3}}.
\end{align*}

To estimate ${\rm A_{1,\,1}}$, by Definition \ref{d3.2}(iii),
the boundedness of $T$ on $\ltw$, Lemma \ref{l4.7},
(v), (iv) and (ii) of Lemma \ref{l2.8}, Lemma \ref{l2.9}
and $\wz K^{(2),\,p}_{3B_j,\,B_0}\ge1$,
we see that, for any $j\in\{1,2\}$,
\begin{align*}
\lf\|(Ta_j)\chi_{2B_j}\r\|_{\lq}&\ls\|a_j\|_{\lq}\ls[\mu(6B_j)]^{1/q-1}
[\lz(c_B,r_B)]^{1-1/p}\lf[\wz K^{(6),\,p}_{B_j,\,B}\r]^{-2}\\
&\ls[\mu(6B_j)]^{1/q-1}[\lz(c_B,8r_B)]^{1-1/p}
\lf[\wz K^{(2),\,p}_{3B_j,\,8B}\r]^{-2}\\
&\ls[\mu(6B_j)]^{1/q-1}
\lf[\lz(c_{B_0},r_{4B_0})\r]^{1-1/p}\lf[\wz K^{(2),\,p}_{3B_j,\,4B_0}\r]^{-1},
\end{align*}
here and hereafter, $c_B$ and $c_{B_0}$ denote the centers
of $B$ and $B_0$, and $r_B$ and $r_{4B_0}$ denote
the radii of $B$ and $4B_0$, respectively.
Let $c_1$, independent of $a_j$ and $j$, be the implicit positive constant
of the above inequality, $\sz_{j,\,1}:=c_1\lz_j$ and
$n_{j,\,1}:=c_1^{-1}(Ta_j)\chi_{2B_j}$. Then
${\rm A_{1,\,1}}=\sum_{j=1}^2\sz_{j,\,1}n_{j,\,1}$,
$\supp(n_{j,\,1})\st 3B_j\st B_0$ and
$$
\|n_{j,\,1}\|_{\lq}\le[\mu(6B_j)]^{1/q-1}
\lf[\lz(c_{B_0},r_{4B_0})\r]^{1-1/p}\lf[\wz K^{(2),\,p}_{3B_j,\,4B_0}\r]^{-1}.
$$

For ${\rm A_{1,\,3}}$, we observe that
$r_{B_0}\sim r_{2^{N_j-1}B_j}$,
where $r_{B_0}$ and $r_{2^{N_j-1}B_j}$ denote the radii
of $B_0$ and $2^{N_j-1}B_j$, respectively. For any $j\in\{1,2\}$,
let $x_j$ and $r_j$ be the center and the radius of $B_j$,
respectively.
By \eqref{4.4}, \eqref{2.2}, \eqref{2.1}, the H\"older inequality,
Definition \ref{d3.2}(iii), $\wz K^{(6),\,p}_{B_j,\,B}\ge1$,
$B_0\st 2^{N_j+3}B_j$, Lemma \ref{l2.8}(ii) and Lemma \ref{l2.9},
we obtain
\begin{align*}
\lf\|(Ta_j)\chi_{B_0\bh 2^{N_j-1}B_j}\r\|_{\lq}
&\ls\lf\{\int_{8B\bh2^{N_j-1}B_j}\lf[
\int_{B_j}\frac{|a_j(y)|}
{\lz(x,d(x,y))}\,d\mu(y)\r]^q\,d\mu(x)\r\}^{1/q}\\
&\ls\lf\{\int_{8B\bh2^{N_j-1}B_j}\lf[
\int_{B_j}\frac{|a_j(y)|}
{\lz(x_j,d(x,x_j))}\,d\mu(y)\r]^q\,d\mu(x)\r\}^{1/q}\\
&\ls\frac{\lf[\mu(8B\bh2^{N_j-1}B_j)\r]^{1/q}}{\lz(x_j,2^{N_j-1}r_j)}
[\mu(B_j)]^{1/q'}\|a_j\|_{\lq}\\
&\ls[\mu(32B)]^{1/q-1}[\lz(c_{B},r_{B})]^{1-1/p}
\lf[\wz K^{(6),\,p}_{B_j,\,B}\r]^{-2}\\
&\ls[\mu(4B_0)]^{1/q-1}\lf[\lz(c_{B_0},r_{4B_0})\r]^{1-1/p}
\lf[\wz K^{(2),\,p}_{2B_0,\,4B_0}\r]^{-1}.
\end{align*}
Let $c_2$, independent of $a_j$ and $j$, be the implicit positive constant
of the above inequality,
$\sz_{j,\,3}:=c_2\lz_j$ and $n_{j,\,3}
:=c_2^{-1}(Ta_j)\chi_{B_0\bh 2^{N_j-1}B_j}$. Then
${\rm A_{1,\,3}}=\sum_{j=1}^2\sz_{j,\,3}n_{j,\,3}$,
$\supp(n_{j,\,3})\st 16B=2B_0$ and
$$
\ \lf\|n_{j,\,3}\r\|_{\lq}\le[\mu(4B_0)]^{1/q-1}
\lf[\lz(c_{B_0},r_{4B_0})\r]^{1-1/p}\lf[\wz K^{(2),\,p}_{2B_0,\,4B_0}\r]^{-1}.
$$

We now estimate ${\rm A_{1,\,2}}$. By \eqref{4.4}, \eqref{2.2}, \eqref{2.1},
Definition \ref{d3.2}(iii), the H\"older inequality, (v),
(iv) and (ii) of Lemma \ref{l2.8}, and Lemma \ref{l2.9},
we conclude that
\begin{align*}
&\lf\|(Ta_j)\chi_{2^{i+1}B_j\bh 2^i B_j}\r\|_{\lq}\\
&\hs\ls\lf\{\int_{2^{i+1}B_j\bh2^i B_j}\lf[
\int_{B_j}\frac{|a_j(y)|}{\lz(x,d(x,y))}
\,d\mu(y)\r]^q\,d\mu(x)\r\}^{1/q}\\
&\hs\ls\lf\{\int_{2^{i+1}B_j\bh2^i B_j}\lf[
\int_{B_j}\frac{|a_j(y)|}{\lz(x_j,d(x,x_j))}
\,d\mu(y)\r]^q\,d\mu(x)\r\}^{1/q}\\
&\hs\ls\lf[\lz\lf(x_j,2^i r_j\r)\r]^{1/q-1}\|a_j\|_{\lon}
\lf[\int_{2^{i+1}B_j\bh 2^i B_j}\frac1{\lz\lf(x_j,2^i r_j\r)}
\,d\mu(x)\r]^{1/q}\\
&\hs\ls\lf[\lz\lf(x_j,2^i r_j\r)\r]^{1/q-1}\|a_j\|_{\lq}[\mu(B_j)]^{1/q'}
\lf[\frac{\mu\lf(2^{i+1}B_j\r)}{\lz\lf(x_j,2^i r_j\r)}\r]^{1/q}\\
&\hs\ls\frac{\mu\lf(2^{i+3}B_j\r)}{\lz\lf(x_j,2^i r_j\r)}
\lf[\wz K^{(2),\,p}_{B_j,\,B}\r]^{-1}\lf[\mu\lf(2^{i+3}B_j\r)\r]^{1/q-1}
\lf[\lz(c_{B_0},r_{4B_0})\r]^{1-1/p}
\lf[\wz K^{(2),\,p}_{2^{i+2}B_j,\,4B_0}\r]^{-1}.
\end{align*}
Let $c_3$, independent of $a_j$, $j$ and $i$,
be the implicit positive constant
of the above inequality,
$$\sz^{(i)}_{j,\,2}:=c_3\lz_j\frac{\mu(2^{i+3}B_j)}{\lz(x_j,2^i r_j)}
\lf[\wz K^{(2),\,p}_{B_j,\,B}\r]^{-1}$$
and
$$n^{(i)}_{j,\,2}:=\lf[c_3\frac{\mu\lf(2^{i+3}B_j\r)}
{\lz\lf(x_j,2^i r_j\r)}\r]^{-1}\wz K^{(2),\,p}_{B_j,\,B}(Ta_j)
\chi_{2^{i+1}B_j\bh 2^i B_j}.$$
Then
$${\rm A_{1,\,2}}=\sum_{j=1}^2
\sum_{i=1}^{N_j-2}\sz^{(i)}_{j,\,2}n^{(i)}_{j,\,2},$$
$\supp(n^{(i)}_{j,\,2})\st 2^{i+2}B_j\st 2B_0$ and
$$\lf\|n^{(i)}_{j,\,2}\r\|_{\lq}
\le\lf[\mu\lf(2\lf(2^{i+2}B_j\r)\r)\r]^{1/q-1}
\lf[\lz(c_{B_0},r_{4B_0})\r]^{1-1/p}
\lf[\wz K^{(2),\,p}_{2^{i+2}B_j,\,4B_0}\r]^{-1}.$$

Now we turn to estimate ${\rm A_2}$. For any $k\in\nn$,
by the geometrically doubling condition,
there exists a ball covering $\{B_{k,\,j}\}_{j=1}^{M_0}$,
with uniform radius $2^{k-3}r_{B_0}$,
of $\wz U_k(B_0):=2^{k}B_0\bh 2^{k-1}B_0$
such that the cardinality $M_0\le N_0 8^n$.
Without loss of generality, we may assume that
the centers of the balls in the covering
belong to $\wz U_k(B_0)$.

Let $C_{k,\,1}:=B_{k,\,1}$, $C_{k,\,l}:=B_{k,\,l}
\bh\bp_{m=1}^{l-1}B_{k,\,m}$,
$l\in\{2,\ldots,M_0\}$ and $D_{k,\,l}:=C_{k,\,l}\cap\wz U_k(B_0)$ for all
$l\in\{1,\ldots,M_0\}$.
Then we know that $\{D_{k,\,l}\}_{l=1}^{M_0}$ is pairwise disjoint,
$\wz U_k(B_0)=\bp_{l=1}^{M_0}D_{k,\,l}$ and,
for any $l\in\{1,\ldots,M_0\}$,
$$D_{k,\,l}\st 2B_{k,\,l}\st U_k(B_0):=2^{k+2}B_0\bh 2^{k-2}B_0.$$
Thus,
$${\rm A_2}=\sum_{k=1}^{\fz}Tb\sum_{l=1}^{M_0}\chi_{D_{k,\,l}}
=\sum_{k=1}^{\fz}\sum_{l=1}^{M_0}(Tb)\chi_{D_{k,\,l}}.$$

From $\int_{\cx}b(y)\,d\mu(y)=0$, \eqref{4.5}, \eqref{2.2}, \eqref{2.1},
the H\"older inequality, Definition \ref{d3.2}(iii),
$\wz K^{(2),\,p}_{B_j,\,B}\ge1$, $4B_{k,\,l}\st2^{k+1}B_0$
and Lemma \ref{l2.8}(ii), it follows that,
for any $k\in\nn$, $j\in\{1,2\}$ and $l\in\{1,\ldots,M_0\}$,
\begin{align*}
\|(Tb)\chi_{D_{k,\,l}}\|_{\lq}
&\le\lf\{\int_{D_{k,\,l}}\lf[
\int_B|b(y)||K(x,y)-K(x,c_B)|\,d\mu(y)\r]^q\,d\mu(x)\r\}^{1/q}\\
&\ls\lf[\int_{D_{k,\,l}}\lf\{
\int_B|b(y)|\frac{[d(y,c_B)]^{\dz}}{[d(x,c_B)]^{\dz}\lz(c_B,d(x,c_B))}
\,d\mu(y)\r\}^q\,d\mu(x)\r]^{1/q}\\
&\ls\frac{r_B^{\dz}[\mu(D_{k,\,l})]^{1/q}}
{\lz\lf(c_B,r_{2^{k-1}B_0}\r)\lf(r_{2^{k-1}B_0}\r)^{\dz}}
\int_B|b(y)|\,d\mu(y)\\
&\ls2^{-k\dz}
\lf[\mu\lf(2^{k+1}B_0\r)\r]^{1/q-1}\sum_{j=1}^2
|\lz_j|[\mu(B_j)]^{1/q'}\|a_j\|_{\lq}\\
&\ls 2^{-k\dz}
\lf[\mu\lf(2^{k+1}B_0\r)\r]^{1/q-1}\sum_{j=1}^2
|\lz_j|[\lz(c_B,r_B)]^{1-1/p}\\
&\ls2^{-k\dz}\lf[C_{(\lz)}\r]^{(k+2)(\frac1p-1)}
\sum_{j=1}^2
|\lz_j|[\mu(4B_{k,\,l})]^{1/q-1}
\lf[\lz(c_{B},r_{2^{k+2}B_0})\r]^{1-1/p}\\
&\ls2^{-\frac{k}2(\dz-\frac{\nu}{p}+\nu)}
2^{-\frac{k}2(\dz-\frac{\nu}{p}+\nu)}
\sum_{j=1}^2
|\lz_j|[\mu(4B_{k,\,l})]^{1/q-1}\\
&\hs\times\lf[\lz(c_{B},r_{2^{k+2}B_0})\r]^{1-1/p}
\lf[\wz K^{(2),\,p}_{2B_{k,\,l},\,2^{k+2}B_0}\r]^{-1}.
\end{align*}
Let $c_4$, independent of $b$ and $k$, be the implicit positive constant
of the above inequality,
$$
\lz_{k,\,l}:=c_4 2^{-\frac{k}2(\dz-\frac{\nu}{p}+\nu)}\sum_{j=1}^2|\lz_j|
$$
and $m_{k,\,l}:=\lz_{k,\,l}^{-1}
(Tb)\chi_{D_{k,\,l}}$. Then
$${\rm A_2}=\sum_{k=1}^{\fz}\sum_{l=1}^{M_0}\lz_{k,\,l}m_{k,\,l},$$
$\supp(m_{k,\,l})\st 2B_{k,\,l}\st U_k(B_0)$ and
$$\|m_{k,\,l}\|_{\lq}\le 2^{-\frac{k}2(\dz-\frac{\nu}{p}+\nu)}
\lf[\mu(2(2B_{k,\,l}))\r]^{1/q-1}
\lf[\lz(c_{B},r_{2^{k+2}B_0})\r]^{1-1/p}
\lf[\wz K^{(2),\,p}_{2B_{k,\,l},\,2^{k+2}B_0}\r]^{-1}.$$

Combining the estimates for ${\rm A_1}$ and ${\rm A_2}$,
we see that $Tb$ is a
$(p,q,1,\frac12(\dz-\frac{\nu}{p}+\nu),2)_\lz$-molecular block and
\begin{align*}
|Tb|^p_{{{\widetilde H_{\rm{mb},\,2}^{p,\,q,\,1,\,\frac12(\delta
-\frac{\nu}{p}+\nu)}(\mu)}}}
&=\sum_{j=1}^2|\sz_{j,\,1}|^p+\sum_{j=1}^2\sum_{i=1}^{N_j-1}
\lf|\sz^{(i)}_{j,\,2}\r|^p
+\sum_{j=1}^2|\sz_{j,\,3}|^p+
\sum_{k=1}^{\fz}\sum_{l=1}^{M_0}|\lz_{k,\,l}|^p\\
&\ls\sum_{j=1}^2|\lz_j|^p+\sum_{j=1}^2
\sum_{i=1}^{N_j-2}|\lz_j|^p\lf[\frac{\mu\lf(2^{i+3}B_j\r)}
{\lz(x_j,2^ir_j)}\r]^p\lf[\wz K^{(2),\,p}_{B_j,\,B}\r]^{-p}\\
&\hs+\sum_{k=1}^{\fz}\sum_{l=1}^{M_0}
2^{-\frac{k}2(\dz-\frac{\nu}{p}+\nu)p}\sum_{j=1}^2|\lz_j|^p\\
&\ls\sum_{j=1}^2|\lz_j|^p+
\sum_{k=1}^{\fz}2^{-\frac{k}2(\dz-\frac{\nu}{p}+\nu)p}M_0
\sum_{j=1}^2|\lz_j|^p
\ls\sum_{j=1}^2|\lz_j|^p
\sim|b|^p_{{\widetilde H_{\rm{atb},\,6}^{p,\,q,\,2}(\mu)}},
\end{align*}
which completes the proof of Theorem \ref{t4.11}.
\end{proof}

\begin{remark}\label{r4.12} It is still unclear whether the range of $\rho$ in
Theorem \ref{t4.11} is sharp or not.
\end{remark}

\section{Boundedness of Generalized Fractional Integrals}\label{s5}

\hskip\parindent In this section, we establish the boundedness
of the generalized fractional integral $T_{\bz}$ ($\bz\in(0,1)$)
from $\mhop$ (or $\nhop$) into $\lpt$ with $1/p_2=1/p_1-\bz$,
where $\theta$ is some positive constant depending on $T_\bz$.
To this end, we first recall the notion of generalized
fractional integrals from \cite{fyy2}.

\begin{definition}\label{d5.1}
Let $\bz\in (0,1)$.
A function $K_{\bz}\in L_{\rm{loc}}^{1}(\cx\times
\cx\setminus\{(x,x):\ x\in\cx\})$ is called a
\emph{generalized fractional integral
kernel} if there exists a positive constant $C_{(K_\bz)}$,
depending on $K_\bz$,
such that

(i) for all $x,\,y\in\cx$ with $x\ne y$,
\begin{equation}\label{5.1}
|K_{\bz}(x,y)|\le C_{(K_\bz)}\frac{1}{[\lz(x,d(x,y))]^{1-\bz}};
\end{equation}

(ii) there exist positive constants $\tz\in (0,1]$
and $c_{(K_\bz)}$, depending on $K_\bz$, such that,
for all $x,\,\wz x,\,y\in\cx$ with $d(x,y)\ge c_{(K_\bz)}d(x,\wz{x})$,
\begin{equation}\label{5.2}
|K_{\bz}(x,y)-K_{\bz}(\wz x,y)|+|K_{\bz}(y,x)-K_{\bz}(y,\wz x)|\le
C_{(K_\bz)}\frac{[d(x,\wz x)]^{\tz}}{[d(x,y)]^{\tz}[\lz(x,d(x,y))]^{1-\bz}}.
\end{equation}

A linear operator $T_{\bz}$ is called a
\emph{generalized fractional integral} with
kernel $K_{\bz}$ satisfying \eqref{5.1} and
\eqref{5.2} if, for all $f\in L^{\fz}_b(\mu)$ and $x\not\in \supp(f)$,
\begin{equation}\label{5.3}
T_{\bz}f(x):=\int_{\cx}K_{\bz}(x,y)f(y)\,d\mu(y).
\end{equation}
\end{definition}

\begin{remark}\label{r5.2}
It was shown in \cite[Remark 1.10(iii)]{fyy2} that there exists
a specific example of the generalized fractional integral,
which is a natural variant of the
so-called Bergman-type operator;
see \cite{fyy2} for the details.
\end{remark}

Now we show that the generalized fractional integral
$T_{\bz}$ is bounded from $\mhop$ into $\lpt$ for
$1/p_2=1/p_1-\bz$.
Recall that $\nu:=\log_2C_{(\lz)}$ and
$\tz$ is as in Definition \ref{d5.1}.

\begin{theorem}\label{t5.3}
Suppose that $(\cx,d,\mu)$ is a non-homogeneous metric measure space.
Let $\bz\in(0,1/2)$, $\rho\in(1,\fz)$, $\frac{\nu}{\nu+\tz}<p_1<p_2
\le1<q<1/\bz$, $1/p_2=1/p_1-\bz$ and $\gz\in[1,\fz)$. Assume that
the generalized fractional integral $T_{\bz}$ defined by
\eqref{5.3} associated with kernel $K_{\bz}$
satisfying \eqref{5.1} and \eqref{5.2} is bounded from
$\lq$ into $L^{\wz q}(\mu)$, where $1/\wz q:=1/q-\bz$. Then
$T_{\bz}$ is bounded from
$\wz H_{\rm{mb},\,\rho}^{p_1,\,q,\,\gz,\,\tz}(\mu)$ into $\lpt$.
\end{theorem}

\begin{proof}
Let $\bz$, $\rho$, $p_1$, $p_2$, $q$, $\wz q$ and $\gz$ be as
in assumptions of Theorem \ref{t5.3}.
For the sake of simplicity, we take $\rho=2$ and $\gz=1$.
With some minor modifications, the arguments here are still valid for
general cases.

Since $T_{\bz}$ is bounded from $\lq$ to $L^{\wz q}(\mu)$
for $q\in(1,1/\bz)$ and $1/\wz{q}=1/q-\bz$, by \cite[Theorem 1.13]{fyy2},
we know that $T_{\bz}$ is also bounded from $\lon$ to weak-$L^{1/(1-\bz)}(\mu)$.
By the boundedness of $T_\bz$ from $\ltw$ into $L^{2/(1-2\bz)}(\mu)$
or from $\lon$ into weak-$L^{1/(1-\bz)}(\mu)$
and an argument similar to that used in the proof of Theorem \ref{t4.8},
to show Theorem \ref{t5.3},
it suffices to show that, for all
$(p_1,q,1,\tz,2)_\lz$-molecular blocks $b$,
$$\|T_{\bz}b\|_{\lpt}\ls|b|_{\mhpod}.$$
Let $b=\sum_{k=0}^{\fz}\sum_{j=1}^{M_k}\lz_{k,\,j}m_{k,\,j}$
be a $(p_1,q,1,\tz,2)_\lz$-molecular block, where, for
any $k\in\zz_+$ and
$j\in\{1,,\ldots,M_k\}$, $\supp(m_{k,\,j})\st B_{k,\,j}\st U_k(B)$
for some balls $B$ and $B_{k,\,j}$ as in
Definition \ref{d4.1}.
Without loss of generality, we may assume that $\wz M=M$ in
Definition \ref{d4.1}.

By the linearity of $T_{\bz}$, we write
\begin{align*}
\|T_{\bz}b\|^{p_2}_{\lpt}&\le\sum_{\ell=5}^\fz
\int_{U_\ell(B)}\lf|T_{\bz}\lf(\sum_{k=0}^{\ell-5}
\sum_{j=1}^M\lz_{k,\,j}m_{k,\,j}\r)(x)\r|^{p_2}\,d\mu(x)\\
&\hs+\sum_{\ell=5}^\fz\int_{U_\ell(B)}
\lf|T_{\bz}\lf(\sum_{k=\ell-4}^{\ell+4}
\sum_{j=1}^M\lz_{k,\,j}m_{k,\,j}\r)(x)\r|^{p_2}\,d\mu(x)\\
&\hs+\sum_{\ell=5}^\fz\int_{U_\ell(B)}
\lf|T_{\bz}\lf(\sum_{k=\ell+5}^{\fz}
\sum_{j=1}^M\lz_{k,\,j}m_{k,\,j}\r)(x)\r|^{p_2}\,d\mu(x)
+\sum_{\ell=0}^4\int_{U_\ell(B)}|T_{\bz}b(x)|^{p_2}\,d\mu(x)\\
&=:{\rm I}+{\rm II}+{\rm III}+{\rm IV}.
\end{align*}

Now we first estimate III. By \eqref{5.1}, \eqref{2.1}, \eqref{2.2},
the H\"older inequality, \eqref{4.1} and $1/p_2=1/p_1-\bz$, we obtain
\begin{align*}
{\rm III}&\le\sum_{\ell=5}^\fz\sum_{k=\ell+5}^\fz\sum_{j=1}^M
|\lz_{k,\,j}|^{p_2}\int_{U_\ell(B)}\lf[\int_{B_{k,\,j}}|m_{k,\,j}(y)|
|K_{\bz}(x,y)|\,d\mu(y)\r]^{p_2}\,d\mu(x)\\
&\ls\sum_{\ell=5}^\fz\sum_{k=\ell+5}^\fz\sum_{j=1}^M
|\lz_{k,\,j}|^{p_2}\int_{U_\ell(B)}\lf\{\int_{B_{k,\,j}}
\frac{|m_{k,\,j}(y)|}
{[\lz(x,d(x,y))]^{1-\bz}}\,d\mu(y)\r\}^{p_2}\,d\mu(x)\\
&\ls\sum_{\ell=5}^\fz\sum_{k=\ell+5}^\fz\sum_{j=1}^M
|\lz_{k,\,j}|^{p_2}\int_{U_\ell(B)}
\frac1{[\lz(c_B,d(x,c_B))]^{p_2(1-\bz)}}\,d\mu(x)
\|m_{k,\,j}\|^{p_2}_{\lon}\\
&\ls\sum_{\ell=5}^\fz\sum_{k=\ell+5}^\fz\sum_{j=1}^M
|\lz_{k,\,j}|^{p_2}
\frac{\mu(2^{\ell+2}B)}{[\lz(c_B,2^{\ell-2}r_B)]^{p_2(1-\bz)}}
\lf[\mu\lf(B_{k,\,j}\r)\r]^{p_2/q'}\|m_{k,\,j}\|^{p_2}_{\lq}\\
&\ls\sum_{\ell=5}^\fz\sum_{k=\ell+5}^\fz\sum_{j=1}^M
|\lz_{k,\,j}|^{p_2}\lf[\mu\lf(2^{\ell+2}B\r)\r]^{1-p_2(1-\bz)}
\lf[\mu\lf(B_{k,\,j}\r)\r]^{p_2/q'}\\
&\hs\times2^{-k\tz p_2}\lf[\mu(2B_{k,\,j})\r]^{-p_2/q'}
\lf[\lz\lf(c_B,2^{k+2}r_B\r)\r]^{p_2(1-1/p_1)}\\
&\ls\sum_{\ell=5}^\fz\sum_{k=\ell+5}^\fz\sum_{j=1}^M
2^{-k\tz p_2}|\lz_{k,\,j}|^{p_2}
\sim\sum_{j=1}^M\sum_{k=10}^\fz\sum_{\ell=5}^{k-5}
2^{-k\tz p_2}|\lz_{k,\,j}|^{p_2}\\
&\ls\sum_{j=1}^M\sum_{k=10}^\fz k
2^{-k\tz p_2}|\lz_{k,\,j}|^{p_2}
\ls\lf\{\sum_{k=0}^\fz\sum_{j=1}^M|\lz_{k,\,j}|^{p_1}\r\}^{p_2/p_1}
\sim|b|^{p_2}_{\mhpod}.
\end{align*}

In order to estimate I, write
\begin{align*}
{\rm I}&\le\sum_{\ell=5}^\fz\int_{U_\ell(B)}\lf|\int_{\cx}
\lf[\sum_{k=0}^{\ell-5}
\sum_{j=1}^M\lz_{k,\,j}m_{k,\,j}(y)\r][K_{\bz}(x,y)-K_{\bz}(x,c_B)]
\,d\mu(y)\r|^{p_2}\,d\mu(x)\\
&\hs+\sum_{\ell=5}^\fz\int_{U_\ell(B)}
\lf|\int_{\cx}\lf[\sum_{k=0}^{\ell-5}
\sum_{j=1}^M\lz_{k,\,j}m_{k,\,j}(y)\r]
K_{\bz}(x,c_B)\,d\mu(y)\r|^{p_2}\,d\mu(x)
=:{\rm I_1}+{\rm I_2}.
\end{align*}
From \eqref{5.2}, \eqref{2.2}, \eqref{2.1}, the H\"older inequality,
\eqref{4.1} and $p_1\in(\frac{\nu}{\nu+\tz},1]$ and
$1/p_2=1/p_1-\bz$, it follows that
\begin{align*}
{\rm I_1}&\ls\sum_{\ell=5}^\fz\sum_{k=0}^{\ell-5}\sum_{j=1}^M
|\lz_{k,\,j}|^{p_2}
\int_{U_\ell(B)}\lf\{\int_{B_{k,\,j}}
\frac{|m_{k,\,j}(y)|[d(y,c_B)]^\tz}
{[d(x,c_B)]^\tz[\lz(c_B,d(x,c_B))]^{1-\bz}}\,d\mu(y)\r\}^{p_2}\,d\mu(x)\\
&\ls\sum_{\ell=5}^\fz\sum_{k=0}^{\ell-5}\sum_{j=1}^M
|\lz_{k,\,j}|^{p_2}\frac{2^{(k+2)\tz p_2}r_B^{\tz p_2}\mu(2^{\ell+2}B)}
{2^{(\ell-2)\tz p_2}r_B^{\tz p_2}[\lz(c_B,2^{\ell-2}r_B)]^{p_2(1-\bz)}}
\|m_{k,\,j}\|^{p_2}_{\lon}\\
&\ls\sum_{\ell=5}^\fz\sum_{k=0}^{\ell-5}\sum_{j=1}^M
|\lz_{k,\,j}|^{p_2}2^{(k-\ell)\tz p_2}
\lf[\mu\lf(2^{\ell+2}B\r)\r]^{1-p_2(1-\bz)}
\lf[\mu\lf(B_{k,\,j}\r)\r]^{p_2/q'}\|m_{k,\,j}\|^{p_2}_{\lq}\\
&\ls\sum_{\ell=5}^\fz\sum_{k=0}^{\ell-5}\sum_{j=1}^M
|\lz_{k,\,j}|^{p_2}2^{-\ell\tz p_2}\lf[\mu\lf(2^{\ell+2}B\r)\r]^{1-p_2(1-\bz)}
\lf[\lz\lf(c_B,2^{k+2}r_B\r)\r]^{p_2(1-1/p_1)}\\
&\ls\sum_{\ell=5}^\fz\sum_{k=0}^{\ell-5}\sum_{j=1}^M
|\lz_{k,\,j}|^{p_2}2^{-\ell\tz p_2}\lf[\mu\lf(2^{\ell+2}B\r)\r]^{1-p_2(1-\bz)}
\lf[\lz(c_B,2^{\ell+2}r_B)\r]^{p_2(1-1/p_1)}\lf[C_{(\lz)}\r]^{(\ell-k)(p_2/p_1-p_2)}\\
&\ls\sum_{\ell=5}^\fz\sum_{k=0}^{\ell-5}\sum_{j=1}^M
|\lz_{k,\,j}|^{p_2}2^{[\nu(p_2/p_1-p_2)-\tz p_2]\ell}
2^{-k\nu(p_2/p_1-p_2)}
\ls\sum_{k=0}^\fz\sum_{j=1}^M
|\lz_{k,\,j}|^{p_2}\ls|b|^{p_2}_{\mhpod}.
\end{align*}
For ${\rm I}_2$, the vanishing moment of $b$, together with
\eqref{5.1}, \eqref{2.1}, \eqref{2.2} and $1/p_2=1/p_1-\bz$,
implies that
\begin{align*}
{\rm I_2}&=\sum_{\ell=5}^\fz\int_{U_\ell(B)}
\lf|\int_{\cx}\lf[\sum_{k=\ell-4}^{\fz}
\sum_{j=1}^M\lz_{k,\,j}m_{k,\,j}(y)\r]
K_{\bz}(x,c_B)\,d\mu(y)\r|^{p_2}\,d\mu(x)\\
&\ls\sum_{\ell=5}^\fz\sum_{k=\ell-4}^{\fz}\sum_{j=1}^M
|\lz_{k,\,j}|^{p_2}
\int_{U_\ell(B)}\lf\{\int_{B_{k,\,j}}|m_{k,\,j}(y)|\frac{1}
{[\lz(c_B,d(x,c_B))]^{1-\bz}}\,d\mu(y)\r\}^{p_2}\,d\mu(x)\\
&\ls\sum_{\ell=5}^\fz\sum_{k=\ell-4}^{\fz}\sum_{j=1}^M
|\lz_{k,\,j}|^{p_2}\frac{\mu(2^{\ell+2}B)}
{[\lz(c_B,2^{\ell-2}r_B)]^{p_2(1-\bz)}}
\|m_{k,\,j}\|^{p_2}_{\lon}\\
&\ls\sum_{\ell=5}^\fz\sum_{k=\ell-4}^{\fz}\sum_{j=1}^M
|\lz_{k,\,j}|^{p_2}\lf[\mu\lf(2^{\ell+2}B\r)\r]^{1-p_2(1-\bz)}
\lf[\mu\lf(B_{k,\,j}\r)\r]^{p_2/q'}
\|m_{k,\,j}\|^{p_2}_{\lq}\\
&\ls\sum_{\ell=5}^\fz\sum_{k=\ell-4}^{\fz}\sum_{j=1}^M
|\lz_{k,\,j}|^{p_2}\lf[\mu\lf(2^{\ell+2}B\r)\r]^{1-p_2(1-\bz)}
\lf[\mu\lf(B_{k,\,j}\r)\r]^{p_2/q'}\\
&\hs\times2^{-k\tz p_2}\lf[\mu(2B_{k,\,j})\r]^{-p_2/q'}
\lf[\lz\lf(c_B,2^{k+2}r_B\r)\r]^{p_2(1-1/p_1)}\\
&\ls\sum_{\ell=5}^\fz\sum_{k=\ell-4}^{\fz}\sum_{j=1}^M2^{-k\tz p_2}
|\lz_{k,\,j}|^{p_2}\sim\sum_{j=1}^M\sum_{k=1}^\fz\sum_{\ell=5}^{k+4}
2^{-k\tz p_2}|\lz_{k,\,j}|^{p_2}\\
&\ls\sum_{j=1}^M\sum_{k=0}^\fz k
2^{-k\tz p_2}|\lz_{k,\,j}|^{p_2}
\ls\lf\{\sum_{k=0}^\fz\sum_{j=1}^M|\lz_{k,\,j}|^{p_1}
\r\}^{p_2/p_1}\sim|b|^{p_2}_{\mhpod}.
\end{align*}
Combining ${\rm I_1}$ and ${\rm I_2}$, we conclude that
${\rm I}\ls|b|^{p_2}_{\mhpod}$.

Then we turn to estimate II. We further write
\begin{align*}
{\rm II}&\le\sum_{\ell=5}^\fz\sum_{k=\ell-4}^{\ell+4}
\sum_{j=1}^M|\lz_{k,\,j}|^{p_2}\int_{U_\ell(B)}
|T_{\bz}(m_{k,\,j})(x)|^{p_2}\,d\mu(x)\\
&\le\sum_{\ell=5}^\fz\sum_{k=\ell-4}^{\ell+4}
\sum_{j=1}^M|\lz_{k,\,j}|^{p_2}\int_{2B_{k,\,j}}
|T_{\bz}(m_{k,\,j})(x)|^{p_2}\,d\mu(x)\\
&\hs+\sum_{\ell=5}^\fz\sum_{k=\ell-4}^{\ell+4}
\sum_{j=1}^M|\lz_{k,\,j}|^{p_2}\int_{U_\ell(B)\bh2B_{k,\,j}}
|T_{\bz}(m_{k,\,j})(x)|^{p_2}\,d\mu(x)
=:{\rm II_1}+{\rm II_2}.
\end{align*}
By the H\"older inequality, $(\lq,L^{\wz q}(\mu))$-boundedness
of $T_\bz$, \eqref{4.1}, \eqref{2.1}, $1/p_2=1/p_1-\bz$
and $1/{\wz q}=1/q-\bz$, we see that
\begin{align*}
{\rm II_1}&\le\sum_{\ell=5}^\fz\sum_{k=\ell-4}^{\ell+4}
\sum_{j=1}^M|\lz_{k,\,j}|^{p_2}\lf[\mu(2B_{k,\,j})\r]^{1-p_2/{\wz q}}
\|T_{\bz}(m_{k,\,j})\|^{p_2}_{L^{\wz q}(\mu)}\\
&\ls\sum_{\ell=5}^\fz\sum_{k=\ell-4}^{\ell+4}
\sum_{j=1}^M|\lz_{k,\,j}|^{p_2}\lf[\mu(2B_{k,\,j})\r]^{1-p_2/{\wz q}}
\|m_{k,\,j}\|^{p_2}_{\lq}\\
&\ls\sum_{\ell=5}^\fz\sum_{k=\ell-4}^{\ell+4}
\sum_{j=1}^M|\lz_{k,\,j}|^{p_2}\lf[\mu(2B_{k,\,j})\r]^{1-p_2/{\wz q}}
2^{-k\tz p_2}\lf[\mu(2B_{k,\,j})\r]^{-p_2/q'}\lf[\lz\lf(c_B,2^{k+2}r_B\r)\r]^{p_2(1-1/p_1)}\\
&\ls\sum_{\ell=5}^\fz\sum_{k=\ell-4}^{\ell+4}
\sum_{j=1}^M 2^{-k\tz p_2}|\lz_{k,\,j}|^{p_2}
\ls\sum_{k=0}^\fz\sum_{j=1}^M|\lz_{k,\,j}|^{p_2}
\ls|b|^{p_2}_{\mhpod}.
\end{align*}

For ${\rm II_2}$, from \eqref{5.1},
$d(x,y)\ge d(x,c_{B_{k,\,j}})-d(y,c_{B_{k,\,j}})
\ge\frac12 d(x,c_{B_{k,\,j}})$ for
$x\notin 2B_{k,\,j}$ and $y\in B_{k,\,j}$,
\eqref{2.2}, \eqref{2.1}, $p_2(1-\bz)<p_1$, the H\"older inequality,
\eqref{4.1} and $1/p_2=1/p_1-\bz$, we deduce that
\begin{align*}
{\rm II_2}&\ls\sum_{\ell=5}^\fz\sum_{k=\ell-4}^{\ell+4}
\sum_{j=1}^M|\lz_{k,\,j}|^{p_2}\int_{U_\ell(B)\bh2B_{k,\,j}}
\lf\{\int_{B_{k,\,j}}\frac{|m_{k,\,j}(y)|}
{[\lz(c_{B_{k,\,j}},d(x,c_{B_{k,\,j}}))]^{1-\bz}}
\,d\mu(y)\r\}^{p_2}\,d\mu(x)\\
&\ls\sum_{\ell=5}^\fz\sum_{k=\ell-4}^{\ell+4}
\sum_{j=1}^M|\lz_{k,\,j}|^{p_2}\int_{2^{k+6}B\bh B_{k,\,j}}
\frac1{[\lz(c_{B_{k,\,j}},d(x,c_{B_{k,\,j}}))]^{p_2(1-\bz)}}\,d\mu(x)
\|m_{k,\,j}\|^{p_2}_{\lon}\\
&\ls\sum_{\ell=5}^\fz\sum_{k=\ell-4}^{\ell+4}
\sum_{j=1}^M|\lz_{k,\,j}|^{p_2}\lf\{\int_{2^{k+6}B\bh B_{k,\,j}}
\frac1{[\lz(c_{B_{k,\,j}},d(x,c_{B_{k,\,j}}))]^{p_1}}
\,d\mu(x)\r\}^{\frac{p_2(1-\bz)}{p_1}}\\
&\hs\times\lf[\mu\lf(2^{k+6}B\r)\r]^{\frac{p_1-p_2(1-\bz)}{p_1}}
\|m_{k,\,j}\|^{p_2}_{\lon}\\
&\ls\sum_{\ell=5}^\fz\sum_{k=\ell-4}^{\ell+4}
\sum_{j=1}^M|\lz_{k,\,j}|^{p_2}
\lf\{\sum_{i=0}^{N^{(2)}_{B_{k,\,j},\,2^{k+5}B}+1}
\frac{\mu(2^{i+1}B_{k,\,j})}
{[\lz(c_{B_{k,\,j}},2^{i}r_{B_{k,\,j}})]^{p_1}}
\r\}^{\frac{p_2(1-\bz)}{p_1}}\\
&\hs\times\lf[\mu\lf(2^{k+6}B\r)\r]^{\frac{p_1-p_2(1-\bz)}{p_1}}
\|m_{k,\,j}\|^{p_2}_{\lon}\\
&\ls\sum_{\ell=5}^\fz\sum_{k=\ell-4}^{\ell+4}
\sum_{j=1}^M|\lz_{k,\,j}|^{p_2}
\lf[\mu\lf(2^{N^{(2)}_{B_{k,\,j},\,2^{k+5}B}+2}B_{k,\,j}\r)
\r]^{(1-p_1)\frac{p_2(1-\bz)}{p_1}}\\
&\hs\times\lf\{\sum_{i=0}^{N^{(2)}_{B_{k,\,j},\,2^{k+5}B}+1}
\lf[\frac{\mu(2^{i+1}B_{k,\,j})}
{\lz(c_{B_{k,\,j}},2^{i}r_{B_{k,\,j}})}\r]^{p_1}
\r\}^{\frac{p_2(1-\bz)}{p_1}}\lf[\mu\lf(2^{k+6}B\r)\r]
^{\frac{p_1-p_2(1-\bz)}{p_1}}
\|m_{k,\,j}\|^{p_2}_{\lon}\\
&\ls\sum_{\ell=5}^\fz\sum_{k=\ell-4}^{\ell+4}
\sum_{j=1}^M|\lz_{k,\,j}|^{p_2}
\|m_{k,\,j}\|^{p_2}_{\lon}
\lf[\mu\lf(2^{N^{(2)}_{B_{k,\,j},\,2^{k+5}B}+2}B_{k,\,j}\r)
\r]^{(1-p_1)\frac{p_2(1-\bz)}{p_1}}\\
&\hs\times\lf[\mu\lf(2^{k+6}B\r)\r]^{\frac{p_1-p_2(1-\bz)}{p_1}}
\lf[\wz K^{(2),\,p_1}_{B_{k,\,j},\,2^{k+5}B}\r]^{p_2(1-\bz)}\\
&\ls\sum_{\ell=5}^\fz\sum_{k=\ell-4}^{\ell+4}
\sum_{j=1}^M|\lz_{k,\,j}|^{p_2}\lf[\mu\lf(B_{k,\,j}\r)\r]^{p_2/q'}
\|m_{k,\,j}\|^{p_2}_{\lq}
\lf[\mu\lf(2^{k+9}B\r)\r]^{1-p_2(1-\bz)}
\lf[\wz K^{(2),\,p_1}_{B_{k,\,j},\,2^{k+5}B}\r]^{p_2(1-\bz)}\\
&\ls\sum_{\ell=5}^\fz\sum_{k=\ell-4}^{\ell+4}
\sum_{j=1}^M|\lz_{k,\,j}|^{p_2}\lf[\mu\lf(B_{k,\,j}\r)\r]^{p_2/q'}
2^{-k\tz p_2}\lf[\mu(2B_{k,\,j})\r]^{-p_2/q'}\\
&\hs\times
\lf[\lz\lf(c_B,2^{k+2}r_B\r)\r]^{p_2(1-1/p_1)}
\lf[\wz K^{(2),\,p_1}_{B_{k,\,j},\,2^{k+2}B}\r]^{-p_2}
\lf[\mu\lf(2^{k+9}B\r)\r]^{1-p_2(1-\bz)}
\lf[\wz K^{(2),\,p_1}_{B_{k,\,j},\,2^{k+5}B}\r]^{p_2}\\
&\ls\sum_{\ell=5}^\fz\sum_{k=\ell-4}^{\ell+4}
\sum_{j=1}^M 2^{-k\tz p_2}|\lz_{k,\,j}|^{p_2}
\ls\sum_{k=0}^\fz\sum_{j=1}^M|\lz_{k,\,j}|^{p_2}
\ls|b|^{p_2}_{\mhpod},
\end{align*}
which, together with the estimate for
${\rm II_1}$, implies that
$${\rm II}\ls|b|^{p_2}_{\mhpod}.$$

To estimate ${\rm IV}$, observe that
\begin{align*}
{\rm IV}&\le\sum_{\ell=0}^4\int_{U_\ell(B)}\lf|T_{\bz}\lf(\sum_{k=0}^{\ell+4}
\sum_{j=1}^M\lz_{k,\,j}m_{k,\,j}\r)(x)\r|^{p_2}\,d\mu(x)\\
&\hs+\sum_{\ell=0}^4\int_{U_\ell(B)}\lf|T_{\bz}\lf(\sum_{k=\ell+5}^{\fz}
\sum_{j=1}^M\lz_{k,\,j}m_{k,\,j}\r)(x)\r|^{p_2}\,d\mu(x)
=:{\rm IV}_1+{\rm IV}_2.
\end{align*}
By some arguments similar to those used in the estimates for $\rm II_1$
and $\rm III$, we respectively obtain
$${\rm IV}_1\ls|b|^{p_2}_{\mhpod}\quad {\rm and}\quad {\rm IV}_2\ls|b|^{p_2}_{\mhpod},$$
which, together with the estimates for $\rm III$, $\rm I$
and $\rm II$, completes the proof of Theorem \ref{t5.3}.
\end{proof}

Similar to Corollary \ref{c4.9}, by the proof of Theorem \ref{t5.3},
we also obtain the following boundedness of the
generalized fractional integral $T_\bz$ from $\nhop$
into $\lpt$ for $1/p_2=1/p_1-\bz$, the details being omitted.

\begin{corollary}\label{c5.4}
Suppose that $(\cx,d,\mu)$ is a non-homogeneous metric measure space.
Let $\bz\in(0,1/2)$, $\rho\in(1,\fz)$,
$\frac{\nu}{\nu+\tz}<p_1<p_2
\le1<q<1/\bz$, $1/p_2=1/p_1-\bz$ and $\gz\in[1,\fz)$. Assume that
the generalized fractional integral $T_\bz$ defined by
\eqref{5.3} associated with kernel $K_\bz$
satisfying \eqref{5.1} and \eqref{5.2} is bounded from
$\lq$ into $L^{\wz q}(\mu)$, where $1/{\wz q}:=1/q-\bz$. Then
$T_\bz$ is bounded from $\nhop$ into $\lpt$.
\end{corollary}

\begin{remark}\label{r5.5}
(a) When $p_1=1$, Theorem \ref{t5.3} or Corollary \ref{c5.4} is
a special case of \cite[Theorem 1.13]{fyy2} by the same
reasons as those used in Remark \ref{r4.10}.

(b) For all $\bz\in(0,1)$, $f\in L^{\fz}_b(\mu)$ and $x\in\cx$, the
\emph{fractional integral} $I_\bz f(x)$ is defined by
$$
I_\bz f(x):=\int_\cx \frac{f(y)}{[\lambda(y,d(x,y))]^{1-\bz}}\,d\mu(y).
$$
From \cite[Section 4]{fyy2}, we deduce that, for some
$\ez\in(0,\fz)$, under the weak growth condition as in
Remark \ref{r2.4}(iii) and the following
\emph{$\ez$-weak reverse doubling condition}
on the dominating function $\lz$:
for all $r\in(0, 2\,{\mathop\mathrm{diam}}(\cx))$
and $a\in(1,2\,{\mathop\mathrm{diam}}(\cx)/r)$,
there exists a number $C_{(a)}\in[1,\fz)$,
depending only on $a$ and $\cx$, such that, for all
$x\in\cx$, $\lz(x,ar)\ge C_{(a)}\lz(x,r)$
and $\sum_{k=1}^{\fz}\frac1{[C_{(a^k)}]^{\ez}}<\fz$,
the following statements hold true:

$\rm (b)_1$ the fractional integral $I_\bz$ is a special case of
the generalized fractional integral,
which is bounded from $\lq$ into $L^{\wz q}(\mu)$
for all $q\in(1,1/\bz)$ and $1/{\wz q}=1/q-\bz$;

$\rm (b)_2$ all conclusions of Theorem \ref{t5.3}
and Corollary \ref{c5.4} hold true,
if $T_\bz$ is replaced by $I_\bz$,
where $I_\bz$ has the same assumptions as those of $T_\bz$
in Theorem \ref{t5.3} and Corollary \ref{c5.4}, respectively.
\end{remark}

\section{Campanato Spaces $\ceaeg$}\label{s6}

\hskip\parindent In this section, we introduce the
Campanato space $\ceaeg$ and show that $\ceaeg$ is independent
of the choices of $\rho$, $\eta$, $\gz$ and $q$
under the following assumption of the $\rho$-weakly doubling condition.

\begin{definition}\label{d6.1}
Let $\rho\in(1,\fz)$. The Borel measure $\mu$
is said to satisfy the \emph{$\rho$-weakly doubling condition} if,
for all balls $B\st\cx$,
there exists a positive constant
${\wz C}_1$, depending on $\rho$ but independent of $B$, such that
\begin{equation}\label{6.1}
N^{(\rho)}_{B,\,{\wz B}^{\rho}}\le {\wz C}_1,
\end{equation}
where $N^{(\rho)}_{B,\,{\wz B}^{\rho}}$ is defined as
in Definition \ref{d2.6}.
\end{definition}

\begin{remark}\label{r6.2}
(i) Recall that ${\wz B}^{\rho}$ is totally
determined by $\mu$ and $\rho$. Let $(\cx,d,\mu)$ be a space of
homogeneous type and $\lz(x,r):=\mu(B(x,r))$
for all $x\in\cx$ and $r\in(0,\fz)$.
Then $(\cx,d,\mu)$ satisfies \eqref{6.1},
since $N^{(\rho)}_{B,\,\wz{B}^{\rho}}\sim1$ for all balls $B$ with
equivalent positive constants depending only on $\rho\in(1,\fz)$.
However, by Example
\ref{e6.3} below, there exists a non-doubling measure $\mu$ on a subset
of $\rr$ satisfying \eqref{6.1}; by Example \ref{e6.4} below,
there exists a non-doubling measure not satisfying \eqref{6.1}.
In this sense, a measure satisfying
\eqref{6.1} is said to be $\rho$-weakly doubling.

(ii) From the fact that $\rho^{N^{(\rho)}_{B,\,{\wz B}^{\rho}}}B
=\wz{B}^\rho$ and \eqref{6.1}, it follows that there exists a positive
constant $C_{(\rho,\,{\wz C}_1)}$,
depending on $\rho$ and ${\wz C}_1$, such that,
for any ball $B$,
$$r_{\wz B^{\rho}}\le C_{(\rho,\,{\wz C}_1)}r_B,$$
where $r_B$ and $r_{\wz B^{\rho}}$ denote the
radii of balls $B$ and $\wz B^{\rho}$, respectively.
Obviously, we always have $r_B\le r_{\wz B^{\rho}}$.
\end{remark}

In the remainder of this section,
we \emph{always assume} that the Borel measure $\mu$
satisfies the $\rho$-weakly doubling condition.

The following example shows that there exist some non-trivial
non-doubling measures satisfying \eqref{6.1}.

\begin{example}\label{e6.3}
Let
$$
\cx:=[0,1]\bigcup\lf(\bigcup_{k=1}^\fz\lf[2\sum_{j=0}^{k-1}e^{-j^2},\,
2\sum_{j=0}^{k-1}e^{-j^2}+e^{-k^2}\r]\r).
$$
Denote $[0,1]$ by $D_0$ and $[2\sum_{j=0}^{k-1}e^{-j^2},\,
2\sum_{j=0}^{k-1}e^{-j^2}+e^{-k^2}]$ by $D_k$ for
$k\in\nn$. For any $x\in\cx$ and $r\in(0,\fz)$,
we use $B:=B(x,\,r):=\{y\in\cx:\ |y-x|<r\}$ to
denote a ball of $\cx$. Let $\mu$ be the Lebesgue measure restricted to $\cx$.
Notice that $\mu(B(x,r))\le2r$ for all $x\in\cx$ and $r\in(0,\fz)$.
Thus, $\mu$ is an upper doubling measure with
$\lz(x,r):=2r$ for all $x\in\cx$ and $r\in(0,\fz)$.

Then we claim that $\mu$ is a non-doubling measure.
Indeed, notice that
\begin{equation}\label{6.2}
\sum_{j=k+1}^\fz e^{-j^2}
=e^{-k^2}\sum_{j=1}^\fz e^{-2kj-j^2}\le
\lf\{
  \begin{array}{ll}
    \dfrac{\sqrt{\pi}}2 e^{-k^2} &\,\, \hbox{for all $k\in\zz_+$;} \\[3mm]
    \dfrac{\sqrt{\pi}}{e^2 2} e^{-k^2} &\,\,\hbox{for all $k\in\nn$.}
  \end{array}
\r.
\end{equation}
Let $x_k:=2\sum_{j=0}^{k-1}e^{-j^2}$ and $r_k:=e^{-(k-1)^2}$. Then
$$
\mu(B(x_k,\,r_k))\le\lf(1+\frac{\sqrt{\pi}}{2}\r)e^{-k^2}\quad{\rm and}\quad
\mu(B(x_k,\,2r_k))\ge e^{-(k-1)^2}.
$$
Thus,
$$
\dfrac{\mu(B(x_k,\,2r_k))}{\mu(B(x_k,\,r_k))}
\ge\frac{1}{1+\frac{\sqrt{\pi}}{2}}e^{2k-1}
\rightarrow\fz\quad{\rm as}\quad k\rightarrow\fz,
$$
which implies that $\mu$ is non-doubling.

For any ball $B$, the smallest doubling ball of the form
$2^jB$ with $j\in\zz_+$ is denoted by $\wz B^{(2)}$.
Let $N^{(2)}_{B,\,\wz B^{(2)}}$ be the integer such that
$2^{N^{(2)}_{B,\,\wz B^{(2)}}}B=\wz B^{(2)}$. We claim that
\begin{equation}\label{6.3}
N^{(2)}_{B,\,\wz B^{(2)}}\le 2.
\end{equation}
To prove this, we consider the following two cases for $k$.

{\bf Case (I)} $k\in\{0,\,1,\,2\}$. In this case,
for all $x\in D_k$, we have
$$
\lf\{
  \begin{array}{ll}
    \mu(B(x,\,r))\sim r\sim\mu(B(x,\,2r)) &\quad
    \hbox{for $r\in\lf(0,\,2+\sqrt{\pi}\r]$;} \\[4mm]
    \mu(B(x,\,r))\sim 1\sim\mu(B(x,\,2r)) &\quad
    \hbox{for $r\in\lf(2+\sqrt{\pi},\,\fz\r)$.}
  \end{array}
\r.
$$
From this, it is easy to deduce that $B(x,\,r)$ with $x\in D_k$
and $r\in(0,\fz)$ is a doubling ball
and hence $N^{(2)}_{B,\,\wz B^{(2)}}=0$ in this case.

{\bf Case (II)} $k\in\nn\cap(2,\,\fz)$. In this case,
for all $x\in D_k$, we have
\begin{itemize}

\item[(i)] for $r\in(0,\,e^{-k^2}]$, $\mu(B(x,\,r))\sim r$;

\item[(ii)]for $r\in(e^{-k^2},\,e^{-(k-1)^2}]$,
$\mu(B(x,\,r))\sim e^{-k^2}$;

\item[(iii)] for $r\in(e^{-(k-1)^2},\,2e^{-(k-1)^2}]$, $e^{-k^2}
\le\mu(B(x,\,r))<2e^{-(k-1)^2}$;

\item[(iv)] for $r\in(2\sum_{i=1}^{j}e^{-(k-i)^2},\,2\sum_{i=1}^{j}
e^{-(k-i)^2}+e^{-(k-j-1)^2}]$
with $j\in\{1,\,\ldots,k-2\}$, $$\mu(B(x,\,r))\sim e^{-(k-j)^2};$$

\item[(v)] for $r\in(2\sum_{i=1}^{j}e^{-(k-i)^2}+e^{-(k-j-1)^2},\,
2\sum_{i=1}^{j+1}e^{-(k-i)^2}]$
with $j\in\{1,\,\ldots,k-2\}$,
$$e^{-(k-j)^2}\le\mu(B(x,\,r))<2e^{-(k-j-1)^2};$$

\item[(vi)] for $r\in(2\sum_{i=1}^{k-1}e^{-(k-i)^2},\,\fz)$,
$\mu(B(x,\,r))\sim 1$.

\end{itemize}

Now we show that $\eqref{6.3}$ holds true in cases (i) through (vi).

Indeed, in the case (vi), it is easy to see that $B(x,\,r)$ is a doubling ball
and hence $N^{(2)}_{B,\,\wz B^{(2)}}=0$ in this case.
That is, \eqref{6.3} holds true in this case.
Therefore, we only need to show that \eqref{6.3}
holds true in cases (i) through (v).

In the case (i), since $k\ge 3$, we have $2r\in(0,\,e^{-(k-1)^2}]$.
If $2r\in(0,\,e^{-k^2}]$,
then we have
$$\mu(B(x,\,r))\sim\mu(B(x,\,2r))\sim r;$$
if $2r\in(e^{-k^2},\,e^{-(k-1)^2}]$,
then $r\in(\frac{e^{-k^2}}2,\,e^{-k^2}]$,
which, together with (ii), leads to
$$\mu(B(x,\,r))\sim\mu(B(x,\,2r))\sim e^{-k^2}.$$
It then follows that $B(x,\,r)$ is a doubling ball,
which shows that $N^{(2)}_{B,\,\wz B^{(2)}}=0$ in this case.
That is, \eqref{6.3} holds true.

In the case (ii), if $2r\in(e^{-k^2},\,e^{-(k-1)^2}]$, then we have
$\mu(B(x,\,r))\sim\mu(B(x,\,2r))\sim e^{-k^2}$,
which implies that $\mu(B(x,\,r))$ is a doubling ball;
if $2r\in(e^{-(k-1)^2},\,2e^{-(k-1)^2}]$, then,
for sufficiently large $k$, $B(x,\,r)$ may be a non-doubling ball,
for example, $x=2\sum_{j=0}^{k-1}e^{-j^2}$
and $r=e^{-(k-1)^2}$. In the latter case,
if $B(x,\,r)$ is a non-doubling ball, we consider the ball $B(x,\,2r)$.
If $B(x,\,2r)$ is a doubling ball,
namely, there exists a positive constant $C$,
independent of $x$ and $r$, such that
$\mu(B(x,\,4r))\le C\mu(B(x,\,2r))$,
then there is nothing to prove. Otherwise,
we consider the ball $B(x,\,4r)$.
Notice that  $2r\in(e^{-(k-1)^2},\,2e^{-(k-1)^2}]$.
Then we have $r\in(\frac{e^{-(k-1)^2}}{2},\,e^{-(k-1)^2}]$,
which, together with $k\ge3$, shows that
$$
2e^{-(k-1)^2}<4r\le 4e^{-(k-1)^2}<2e^{-(k-1)^2}+e^{-(k-2)^2}
$$
and
$$
2e^{-(k-1)^2}<8r\le 8e^{-(k-1)^2}<2e^{-(k-1)^2}+e^{-(k-2)^2}.
$$
It then follows, from (iv) with $j=1$, that
$\mu(B(x,\,4r))\sim\mu(B(x,\,8r))\sim e^{-(k-1)^2}$,
which implies that $B(x,\,4r)$ is a doubling ball.
From the above estimate, we conclude that
$N^{(2)}_{B,\,\wz B^{(2)}}\le 2$ in this case.

The argument of the case (iii) is similar to that used in the case of
$2r\in(e^{-(k-1)^2},\,2e^{-(k-1)^2}]$ of (ii).
Moreover, we have $N^{(2)}_{B,\,\wz B^{(2)}}\le 1$ in this case.

Before we deal with the case (iv), we first consider the case (v).
In the case (v),
if $B(x,\,r)$ is a doubling ball, then there is nothing to prove.
Otherwise, we consider the following two cases for $j$. If $j=k-2$, we have
$$
4r>2r>4\sum_{i=1}^{j}e^{-(k-i)^2}+2e^{-(k-j-1)^2}
>2\sum_{i=1}^{j+1}e^{-(k-i)^2}=2\sum_{i=1}^{k-1}e^{-(k-i)^2},
$$
which, together with (vi), shows that
$\mu(B(x,\,4r))\sim\mu(B(x,\,2r))\sim 1$.
Thus, $B(x,\,2r)$ is a doubling ball.
If $j\le k-3$, then, by \eqref{6.2}, we see that
$$
6\sum_{i=1}^{j+1}e^{-(k-i)^2}<e^{-[k-(j+1)-1]^2}.
$$
It then follows that
$$
2\sum_{i=1}^{j+1}e^{-(k-i)^2}<2r<4r\le8\sum_{i=1}^{j+1}e^{-(k-i)^2}
<2\sum_{i=1}^{j+1}e^{-(k-i)^2}+e^{-[k-(j+1)-1]^2}.
$$
This via (iv) shows that $\mu(B(x,\,2r))
\sim\mu(B(x,\,4r))\sim e^{-(k-j-1)^2}$,
which implies that $B(x,\,2r)$ is a doubling ball
and hence $N^{(2)}_{B,\,\wz B^{(2)}}\le 1$ in this case.

In the case (iv), we see that $2r\in(4\sum_{i=1}^{j}e^{-(k-i)^2},\,
4\sum_{i=1}^{j}e^{-(k-i)^2}
+2e^{-(k-j-1)^2}]$.
Notice that, by \eqref{6.2} and $j\le k-2$, we see that
$$
2\sum_{i=1}^{j}e^{-(k-i)^2}=2\sum_{i=k-j}^{k-1}e^{-i^2}
\le\frac{2\sqrt{\pi}}{2e^2}e^{-(k-j-1)^2}<e^{-(k-j-1)^2}.
$$
Thus, we consider the following three cases for $2r$.

{\bf Case (a)} $2r\in(4\sum_{i=1}^{j}e^{-(k-i)^2},
\,2\sum_{i=1}^{j}e^{-(k-i)^2}
+e^{-(k-j-1)^2}]$.
In this case, it is easy to see that
$$\mu(B(x,\,2r))\sim\mu(B(x,\,r))\sim e^{-(k-j)^2},$$
which implies that $B(x,\,r)$ is a doubling ball.

{\bf Case (b)} $2r\in(2\sum_{i=1}^{j}e^{-(k-i)^2}+e^{-(k-j-1)^2},\,
2\sum_{i=1}^{j+1}e^{-(k-i)^2}]$.
In this case, by an argument similar to that used in (v),
we conclude that $N^{(2)}_{B,\,\wz B^{(2)}}\le 2$.

{\bf Case (c)} $2r\in(2\sum_{i=1}^{j+1}e^{-(k-i)^2},\,
4\sum_{i=1}^{j}e^{-(k-i)^2}+2e^{-(k-j-1)^2}]$.
In this case, if $j=k-2$, we have
$$
4r>2r>2\sum_{i=1}^{k-1}e^{-(k-i)^2},
$$
which, together with (vi), implies that $B(x,\,2r)$ is a doubling ball;
if $j\le k-3$, by \eqref{6.2}, we know that
\begin{align*}
2\sum_{i=1}^{j+1}e^{-(k-i)^2}<2r<4r&
\le8\sum_{i=1}^{j}e^{-(k-i)^2}+4e^{-(k-j-1)^2}
<2\sum_{i=1}^{j+1}e^{-(k-i)^2}+6\sum_{i=1}^{j+1}e^{-(k-i)^2}\\
&<2\sum_{i=1}^{j+1}e^{-(k-i)^2}+e^{-[k-(j+1)-1]^2}.
\end{align*}
This via (iv) shows that
$\mu(B(x,\,2r))\sim\mu(B(x,\,4r))\sim e^{-(k-j-1)^2}$,
which implies that $B(x,\,2r)$
is a doubling ball and hence $N^{(2)}_{B,\,\wz B^{(2)}}\le 1$ in the case (iv).

Combining the above estimates, we obtain \eqref{6.3},
which completes the proof of our claim and hence the example.
\end{example}

On the other hand, it turns out that
there exist many non-homogeneous metric measure spaces
which do not satisfy the $\rho$-weakly
doubling condition \eqref{6.1}. We give the
following Gauss measure on $\rr$ as an example.

\begin{example}\label{e6.4}
Let $(\cx,|\cdot|,\mu):=(\rr,|\cdot|,\mu)$, where
$|\cdot|$ denotes the Euclidean distance
and $\mu$ the Gauss measure on $\rr$,
that is, $d\mu(x):=\pi^{-\frac12}e^{-x^2}dx$ for all $x\in\rr$.
As in Example \ref{e6.3}, for any $x\in\rr$ and $r\in(0,\fz)$,
we use $B:=B(x,\,r):=\{y\in\rr:\ |y-x|<r\}$ to denote a ball of $\rr$.
First, we show that $\mu$ is a non-doubling
measure with the dominating function $\lz(x,\,r):=2\pi^{-\frac12}r$.
Indeed, for all $x\in\rr$ and $r\in(0,\fz)$,
$$
\mu(B(x,\,r))=\pi^{-\frac12}\dint^{x+r}_{x-r}e^{-y^2}dy
\le 2\pi^{-\frac12}r=\lz(x,\,r).
$$
On the other hand, let $x_k=2^{2k}$ and $r_{k,\,j}=2^j$
with $k\in\nn$ and $j\in\{1,\ldots,k\}$. Then,
for all $k\in\nn$ and $j\in\{1,\ldots,k\}$, we observe that
$$
\mu(B(x_k,\,r_{k,\,j}))=\pi^{-\frac12}\dint_{2^{2k}-2^j}^{2^{2k}+2^j}
e^{-x^2}dx\le \pi^{-\frac12}e^{-(2^{2k}-2^j)^2}2^{j+1}
$$
and
\begin{align*}
\mu(B(x_k,\,2r_{k,\,j}))&=\pi^{-\frac12}
\dint_{2^{2k}-2^{j+1}}^{2^{2k}+2^{j+1}}e^{-x^2}dx\\
&\ge\pi^{-\frac12}\dint_{2^{2k}-2^{j+1}}^{2^{2k}-3\times2^{j-1}}e^{-x^2}dx
\ge \pi^{-\frac12}e^{-(2^{2k}-3\times2^{j-1})^2}2^{j-1}.
\end{align*}
Thus,
\begin{equation}\label{6.4}
\dfrac{\mu(B(x_k,\,2r_{k,\,j}))}{\mu(B(x_k,\,r_{k,\,j}))}
\ge\frac{1}{4}e^{2^{2k+j}-\frac{5}{4}2^{2j}}
\ge\frac{1}{4}e^{2^{2k+1}-\frac{5}{4}2^{2k}}
\ge\frac{1}{4}e^{2^{2k-1}}\rightarrow\fz\quad {\rm as}
\quad k\rightarrow\fz,
\end{equation}
which implies that $\mu$ is non-doubling.

Now we claim that, for any $\rho\in(1,\fz)$,
there exists a ball $B$ such that
the number $N^{(\rho)}_{B,\,{\wz B}^{\rho}}$ can be arbitrarily large.
For the sake of simplicity, we only prove our claim for $\rho=2$.
With some simple modifications, the arguments here are still valid for all
$\rho\in(1,\,\fz)$. Recall that a ball $B\st \rr$ is said to be
$(2,\bz_2)$-doubling if $\mu(2B)\le \bz_2\mu(B)$ and
$\wz B^{(2)}$ is the smallest $(2,\bz_2)$-doubling
ball of the form $2^jB$ with $j\in\zz_+$.
Let $k_0$ be the smallest positive integer
such that $\frac{1}{4}e^{2^{2k_0-1}}>\bz_2$.
Then, for all $k\in\nn\cap[k_0,\fz)$,
we have $\frac{1}{4}e^{2^{2k-1}}>\bz_2$.
Let $B_k:=B(2^{2k},\,2)$. By \eqref{6.4},
it is easy to see that, for all $j\in\{0,\ldots,k-1\}$,
$2^jB_k$ is not a $(2,\bz_2)$-doubling ball.
It then follows, from the definition of
$N^{(2)}_{B_k,\,{\wz B_k}^{2}}$, that
$$
N^{(2)}_{B_k,\,{\wz B_k}^{2}}>k-1,
$$
which implies our claim and completes the proof of Example \ref{e6.4}.
\end{example}

We now state the definition of the Campanato space $\ceaeg$.

\begin{definition}\label{d6.5}
Let $\alpha\in[0, \fz)$, $\eta\in(1,\fz)$, $\rho\in[\eta, \fz)$
and $q,\,\gz\in[1, \fz)$.
A function $f\in L^1_{\rm loc}(\mu)$ is
said to belong to the {\it Campanato space} $\ceaeg$ if
\begin{align*}
\|f\|_{\ceaeg}
&:=\sup_B\lf\{\frac1{\mu(\eta B)}\frac1{[\lz(c_B,\,r_B)]^{\alpha q}}
\int_B\lf|f(y)-m_{\wz B^{\rho}}(f)\r|^q\,d\mu(y)\r\}^{1/q}\\
&\quad+\sup_{B\st S:\ B,\,S\ (\rho,\,\bz_{\rho})-{\rm doubling}}
\frac{|m_B(f)-m_S(f)|}{[\lz(c_S,\,r_S)]^\alpha
[\kbsa]^{\gz}}<\fz,
\end{align*}
here and hereafter, $m_B(f):=\frac1{\mu(B)}\int_Bf(x)\,d\mu(x)$
for all balls $B$ and $f\in L^1_{\rm loc}(\mu)$.
\end{definition}

\begin{remark}\label{r6.6}
Arguing as in \cite[Lemma 3.2]{ly10}, we see that
$\ce_{\rho,\,\rho,\,\gz}^{0,\,1}(\mu)$
and $\wz{{\rm RBMO}}_{\rho,\,\gz}(\mu)$
coincide with equivalent norms,
where $\wz{{\rm RBMO}}_{\rho,\,\gz}(\mu)$  was
introduced in \cite{fyy3}; see also \cite{h10, hyy}.
\end{remark}

\begin{proposition}\label{p6.7}
Suppose that $(\cx,d,\mu)$ is a non-homogeneous metric
measure space satisfying \eqref{6.1}.
Let $\alpha\in[0, \fz)$, $\eta\in(1,\fz)$, $\rho\in[\eta, \fz)$
and $q,\,\gz\in[1, \fz)$. The following statements hold true:

{\rm (a)} There exists a positive constant $C$
such that, for all balls $B\st S$ and all functions $f\in\ceaeg$,
$$
\lf|m_{\wz B^{\rho}}(f)-m_{\wz S^{\rho}}(f)\r|
\le C\lf[\kbsa\r]^{\gz}[\lz(c_S,r_S)]^{\az}\|f\|_{\ceaeg}.
$$

{\rm (b)} If $f,g\in\ceaeg$ are real-valued functions, then
$\max\{f,\,g\}\in\ceaeg$ and $\min\{f,\,g\}\in\ceaeg$.
Moreover, there exists a positive constant $C$,
independent of $f$ and $g$, such that
$$
\|\max\{f,\,g\}\|_{\ceaeg}+\|\min\{f,\,g\}\|_{\ceaeg}
\le C\lf[\|f\|_{\ceaeg}+\|g\|_{\ceaeg}\r].
$$
\end{proposition}

\begin{proof}
To show (a), we consider the following two cases:

\textbf{Case (i)} $r(\wz B^{\rho})<r(\wz S^{\rho})$. It is obvious that
$\wz B^{\rho}\st2\wz S^{\rho}$. Let $S_0:=\wz{(2\wz S^{\rho})}^{\rho}$.
By (iv), (ii) and (iii) of Lemma \ref{l2.8} with $p=1/(\az+1)$, we have
$$
\lf[{\wz K^{(\rho),\,1/(\az+1)}_{\wz S^{\rho},\,S_0}}\r]^{\gz}
\ls\lf[{\wz K^{(\rho),\,1/(\az+1)}_{\wz S^{\rho},\,2\wz S^{\rho}}}\r]^{\gz}
+\lf[{\wz K^{(\rho),\,1/(\az+1)}_{2\wz S^{\rho},\,S_0}}\r]^{\gz}\ls1.
$$
From this and (v), (iv), (ii) and (iii) of Lemma \ref{l2.8}
with $p=1/(\az+1)$, it follows that
\begin{align*}
\lf[{\wz K^{(\rho),\,1/(\az+1)}_{\wz B^{\rho},\,S_0}}\r]^{\gz}
&\ls\lf[{\wz K^{(\rho),\,1/(\az+1)}_{B,\,S_0}}\r]^{\gz}\\
&\ls\lf[{\wz K^{(\rho),\,1/(\az+1)}_{B,\,S}}\r]^{\gz}
+\lf[{\wz K^{(\rho),\,1/(\az+1)}_{S,\,\wz S^{\rho}}}\r]^{\gz}
+\lf[{\wz K^{(\rho),\,1/(\az+1)}_{\wz S^{\rho},\,S_0}}\r]^{\gz}\\
&\ls\lf[{\wz K^{(\rho),\,1/(\az+1)}_{B,\,S}}\r]^{\gz}.
\end{align*}
Thus, by the above two inequalities, Remark \ref{r6.2}(ii)
and \eqref{2.1}, we have
\begin{align*}
\lf|m_{\wz B^{\rho}}(f)-m_{\wz S^{\rho}}(f)\r|
&\le\lf|m_{\wz B^{\rho}}(f)-m_{S_0}(f)\r|
+\lf|m_{\wz S^{\rho}}(f)-m_{S_0}(f)\r|\\
&\ls\lf\{\lf[{\wz K^{(\rho),\,1/(\az+1)}_{\wz B^{\rho},\,S_0}}\r]^{\gz}
+\lf[{\wz K^{(\rho),\,1/(\az+1)}_{\wz S^{\rho},\,S_0}}\r]^{\gz}\r\}
\lf[\lz(c_{S_0},r_{S_0})\r]^{\az}\|f\|_{\ceaeg}\\
&\ls\lf[{\wz K^{(\rho),\,1/(\az+1)}_{B,\,S}}\r]^{\gz}
\lf[\lz(c_{S_0},r_{S_0})\r]^{\az}\|f\|_{\ceaeg}\\
&\ls\lf[{\wz K^{(\rho),\,1/(\az+1)}_{B,\,S}}\r]^{\gz}
\lf[\lz(c_S,r_S)\r]^{\az}\|f\|_{\ceaeg}.
\end{align*}
This finishes the proof of \textbf{Case (i)}.

\textbf{Case (ii)} $r(\wz S^{\rho})\le r(\wz B^{\rho})$.
Obviously, $\wz S^{\rho}\st 2\wz B^{\rho}$. Let
$B_0:=\wz{(2\wz B^{\rho})}^{\rho}$.
From (iv), (ii) and (iii) of Lemma \ref{l2.8}
with $p=1/(\az+1)$, we deduce that
$$
\lf[{\wz K^{(\rho),\,1/(\az+1)}_{\wz B^{\rho},\,B_0}}\r]^{\gz}
\ls\lf[{\wz K^{(\rho),\,1/(\az+1)}_{\wz B^{\rho},\,2\wz B^{\rho}}}\r]^{\gz}
+\lf[{\wz K^{(\rho),\,1/(\az+1)}_{2\wz B^{\rho},\,B_0}}\r]^{\gz}\ls1.
$$
By this, $\wz S^{\rho}\supset S\supset B$
and (v), (iv) and (iii) of Lemma \ref{l2.8} with
$p=1/(\az+1)$, we know that
\begin{align*}
\lf[{\wz K^{(\rho),\,1/(\az+1)}_{\wz S^{\rho},\,B_0}}\r]^{\gz}
\ls\lf[{\wz K^{(\rho),\,1/(\az+1)}_{B,\,B_0}}\r]^{\gz}
\ls\lf[{\wz K^{(\rho),\,1/(\az+1)}_{B,\,\wz B^{\rho}}}\r]^{\gz}
+\lf[{\wz K^{(\rho),\,1/(\az+1)}_{\wz B^{\rho},\,B_0}}\r]^{\gz}
\ls1.
\end{align*}
Thus, combining the above two inequalities, \eqref{2.2},
\eqref{2.1} and Remark \ref{r6.2}(ii), we have
\begin{align*}
\lf|m_{\wz B^{\rho}}(f)-m_{\wz S^{\rho}}(f)\r|
&\le\lf|m_{\wz B^{\rho}}(f)-m_{B_0}(f)\r|
+\lf|m_{\wz S^{\rho}}(f)-m_{B_0}(f)\r|\\
&\ls\lf\{\lf[{\wz K^{(\rho),\,1/(\az+1)}_{\wz B^{\rho},\,B_0}}\r]^{\gz}
+\lf[{\wz K^{(\rho),\,1/(\az+1)}_{\wz S^{\rho},\,B_0}}\r]^{\gz}\r\}
\lf[\lz(c_{B_0},r_{B_0})\r]^{\az}\|f\|_{\ceaeg}\\
&\ls\lf[\lz(c_{B_0},r_{B_0})\r]^{\az}\|f\|_{\ceaeg}
\sim\lf[\lz(c_S,r_{B_0})\r]^{\az}\|f\|_{\ceaeg}\\
&\ls\lf[{\wz K^{(\rho),\,1/(\az+1)}_{B,\,S}}\r]^{\gz}
\lf[\lz(c_S,r_B)\r]^{\az}\|f\|_{\ceaeg}\\
&\ls\lf[{\wz K^{(\rho),\,1/(\az+1)}_{B,\,S}}\r]^{\gz}
\lf[\lz(c_S,r_S)\r]^{\az}\|f\|_{\ceaeg}.
\end{align*}
This finishes the proof of \textbf{Case (ii)} and hence (a).

To prove (b), since $\max\{f,\,g\}=\frac{f+g+|f-g|}2$
and $\min\{f,\,g\}=\frac{f+g-|f-g|}2$,
it suffices to show that, for any real-valued function $h\in\ceaeg$,
$|h|\in\ceaeg$ and
$$
\||h|\|_{\ceaeg}\ls\|h\|_{\ceaeg}
.$$

To this end, by Definition \ref{d6.5}, the H\"older inequality,
Remark \ref{r6.2}(ii) and \eqref{2.1}, we see that, for any ball $B$,
\begin{align*}
&\lf\{\frac1{\mu(\eta B)}\int_B\lf||h(y)|-m_{\wz B^\rho}(|h|)\r|^q
\,d\mu(y)\r\}^{1/q}\\
&\hs\le\lf\{\frac1{\mu(\eta B)}
\int_B\lf||h(y)|-\lf|m_{\wz B^\rho}(h)\r|\r|^q\,d\mu(y)\r\}^{1/q}
+\lf|\lf|m_{\wz B^\rho}(h)\r|-m_{\wz B^\rho}(|h|)\r|\\
&\hs\le\lf\{\frac1{\mu(\eta B)}
\int_B\lf|h(y)-m_{\wz B^\rho}(h)\r|^q\,d\mu(y)\r\}^{1/q}
+m_{\wz B^\rho}\lf(|h-m_{\wz B^\rho}(h)|\r)\\
&\hs\ls\lf\{\lf[\lz(c_B,r_B)\r]^\az
+\lf[\lz\lf(c_B, r_{\wz B^\rho}\r)\r]^{\az}\r\}
\|h\|_{\ceaeg}\ls\lf[\lz(c_B,r_B)\r]^\az\|h\|_{\ceaeg}.
\end{align*}

On the other hand, by Definition \ref{d6.5}, \eqref{2.1}
and \eqref{2.2}, we find that, for all
$(\rho,\bz_{\rho})$-doubling balls $B\st S$,
\begin{align*}
&|m_B(|h|)-m_S(|h|)|\\
&\hs\le|m_B(|h|)-|m_B(h)||+||m_B(h)|-|m_S(h)||
+||m_S(h)|-m_S(|h|)|\\
&\hs\le m_B(|h-m_B(h)|)+|m_B(h)-m_S(h)|+m_S(|h-m_S(h)|)\\
&\hs\ls\lf\{[\lz(c_B,r_B)]^\az+[\lz(c_S,r_S)]^\az
\lf[{\wz K}^{(\rho),\,1/(\az+1)}_{B,\,S}\r]^\gz
+[\lz(c_S,r_S)]^\az\r\}\|h\|_{\ceaeg}\\
&\hs\ls[\lz(c_S,r_S)]^\az\lf[{\wz K}^{(\rho),\,1/(\az+1)}_{B,\,S}\r]^\gz
\|h\|_{\ceaeg}.
\end{align*}
The above two inequalities imply that $|h|\in\ceaeg$ and
$\||h|\|_{\ceaeg}\ls\|h\|_{\ceaeg}$, which completes
the proof of (b) and hence Proposition \ref{p6.7}.
\end{proof}

We now show that the space $\ceaeg$ is independent
of the choices of $\rho$ and $\eta$.

\begin{proposition}\label{p6.8}
Suppose that $(\cx,d,\mu)$ is a non-homogeneous metric
measure space satisfying \eqref{6.1}.
Let $\az\in[0,\fz)$ and $q,\,\gz\in[1,\fz)$. The following statements
hold true:

\emph{(i)} for any $\eta_1$, $\eta_2$ and $\rho$ satisfying
$1<\eta_1<\eta_2\le\rho<\fz$,
$\ceaog$ and $\ceatg$ coincide with equivalent norms;

\emph{(ii)} for any $\rho_1$, $\rho_2$ and $\eta$ satisfying
$1<\eta\le\rho_1,\,\rho_2<\fz$,
$\ceago$ and $\ceagt$ coincide with equivalent norms.
\end{proposition}

\begin{proof}
We first prove (i). Fix $\az\in[0,\fz)$ and $q,\,\gz\in[1,\fz)$.
Let $\eta_1$, $\eta_2$ and $\rho$ satisfy
$1<\eta_1<\eta_2\le\rho<\fz$. It is obvious that
$\ceaog\st\ceatg$ and, for all $f\in\ceaog$,
$\|f\|_{\ceatg}\le\|f\|_{\ceaog}$.

Conversely, let $f\in\ceatg$. We show that $f\in\ceaog$
and $\|f\|_{\ceaog}\ls\|f\|_{\ceatg}$.
To this end, it suffices to show that, for any ball $B$,
\begin{equation}\label{6.5}
\lf\{\frac1{\mu(\eta_1 B)}\int_B\lf|f(y)
-m_{\wz B^{\rho}}(f)\r|^q\,d\mu(y)\r\}^{1/q}
\ls[\lz(c_B,r_B)]^{\az}\|f\|_{\ceatg}.
\end{equation}
To do so, for any $x\in B$, let $B_x$ be the ball
centered at $x$ with radius $\frac{\eta_1-1}{10\eta_2}r_B$.
Then $r_{\eta_2 B_x}=\frac{\eta_1-1}{10}r_B$ and
$\eta_2 B_x\st \eta_1 B$.

By the geometrically doubling condition and Remark \ref{r2.2}(ii),
we see that there exist $N_1\in\nn$, depending on $\eta_1$,
$\eta_2$ and $(\cx,d,\mu)$, and a finite sequence $\{B_{x_i}\}_{i=1}^{N_1}
=:\{B_i\}_{i=1}^{N_1}$ of balls such that $x_i\in B$ for all
$i\in\{1,\ldots,N_1\}$ and $B\st\cup_{i=1}^{N_1}B_i$.
By this, $\eta_2 B_i\st \eta_1 B$ for all $i\in\{1,\ldots,N_1\}$,
\eqref{2.1}, \eqref{2.2}, Proposition \ref{p6.7}(a) and
Lemma \ref{l2.8}(ii), we conclude that
\begin{align*}
&\frac1{\mu(\eta_1 B)}\int_B\lf|f(y)-m_{\wz B^{\rho}}(f)\r|^q\,d\mu(y)\\
&\hs\le\sum_{i=1}^{N_1}\frac1{\mu(\eta_1 B)}
\int_{B_i}\lf|f(y)-m_{\wz B^{\rho}}(f)\r|^q\,d\mu(y)\\
&\hs\ls\sum_{i=1}^{N_1}\frac{\mu(\eta_2 B_i)}{\mu(\eta_1 B)}
\lf\{\frac1{\mu(\eta_2 B_i)}\int_{B_i}\lf|f(y)
-m_{\wz B_i^{\rho}}(f)\r|^q\,d\mu(y)
+\lf|m_{\wz B_i^{\rho}}(f)-m_{\wz B^{\rho}}(f)\r|^q\r\}\\
&\hs\ls\sum_{i=1}^{N_1}\lf\{\lf[\lz(c_{B_i},r_{B_i})\r]^{\az q}
\|f\|_{\ceatg}^q+\lf|m_{\wz B_i^{\rho}}(f)
-m_{\wz{(\eta_1 B)}^{\rho}}(f)\r|^q
+\lf|m_{\wz{(\eta_1 B)}^{\rho}}(f)
-m_{\wz B^{\rho}}(f)\r|^q\r\}\\
&\hs\ls[\lz(c_B,\eta_1 r_B)]^{\az q}\|f\|_{\ceatg}^q
\lf\{\lf[{\wz K}^{(\rho),\,1/(\az+1)}_{B_i,\,\eta_1 B}\r]^{\gz q}
+\lf[{\wz K}^{(\rho),\,1/(\az+1)}_{B,\,\eta_1 B}\r]^{\gz q}\r\}\\
&\hs\ls[\lz(c_B,r_B)]^{\az q}\|f\|_{\ceatg}^q,
\end{align*}
which completes the proof of \eqref{6.5} and hence (i).

To show (ii), fix $\az\in[0,\fz)$ and $q,\,\gz\in[1,\fz)$.
Let $\rho_1$, $\rho_2$ and $\eta$ satisfy
$1<\eta\le\rho_1,\,\rho_2<\fz$. By the symmetry of
$\rho_1$ and $\rho_2$, it suffices to show that
$\ceagt\st\ceago$ and $\|f\|_{\ceago}\ls\|f\|_{\ceagt}$
for all $f\in\ceagt$. Assume that $f\in\ceagt$.
From the Minkowski inequality, the H\"older inequality,
Proposition \ref{p6.7},
$\rho_1\ge\eta$, Remark \ref{r6.2}(ii), Lemma \ref{l2.9} and
Lemma \ref{l2.8}(iii), we deduce that
\begin{align*}
&\lf\{\frac1{\mu(\eta B)}\int_B\lf|f(y)-m_{\wz B^{\rho_1}}(f)\r|^q
\,d\mu(y)\r\}^{1/q}\\
&\hs\le\lf\{\frac1{\mu(\eta B)}\int_B\lf|f(y)-m_{\wz B^{\rho_2}}(f)\r|^q
\,d\mu(y)\r\}^{1/q}+\lf|m_{\wz B^{\rho_1}}(f)-m_{\wz B^{\rho_2}}(f)\r|\\
&\hs\le[\lz(c_B,r_B)]^{\az}\|f\|_{\ceagt}
+\lf|m_{\wz B^{\rho_1}}(f)-m_{\wz{\wz B^{\rho_1}}^{\rho_2}}(f)\r|
+\lf|m_{\wz{\wz B^{\rho_1}}^{\rho_2}}(f)-m_{\wz B^{\rho_2}}(f)\r|\\
&\hs\ls[\lz(c_B,r_B)]^{\az}\|f\|_{\ceagt}+
\lf\{\frac1{\mu(\eta\wz B^{\rho_1})}\int_{\wz B^{\rho_1}}
\lf|f(y)-m_{\wz{\wz B^{\rho_1}}^{\rho_2}}(f)\r|^q\,d\mu(y)\r\}^{1/q}\\
&\hs\hs+\lf[{\wz K}^{(\rho_2),\,1/(\az+1)}_{B,\,\wz B^{\rho_1}}\r]^{\gz}
\lf[\lz\lf(c_B,r_{\wz B^{\rho_1}}\r)\r]^{\az}\|f\|_{\ceagt}\\
&\hs\ls\lf\{[\lz(c_B,r_B)]^{\az}+\lf[\lz\lf(c_B,r_{\wz B^{\rho_1}}\r)
\r]^{\az}\r\}\|f\|_{\ceagt}\ls[\lz(c_B,r_B)]^{\az}\|f\|_{\ceagt}.
\end{align*}
On the other hand, for all $(\rho_1,\bz_{\rho_1})$-doubling
balls $B\st S$, by the H\"older inequality, Proposition \ref{p6.7},
$\rho_1\ge\eta$, Lemma \ref{l2.9}, \eqref{2.1} and \eqref{2.2},
we have
\begin{align*}
&|m_B(f)-m_S(f)|\\
&\hs\le\lf|m_B(f)-m_{\wz B^{\rho_2}}(f)\r|
+\lf|m_{\wz B^{\rho_2}}(f)-m_{\wz S^{\rho_2}}(f)\r|
+\lf|m_{\wz S^{\rho_2}}(f)-m_S(f)\r|\\
&\hs\le\lf\{\frac1{\mu(B)}\int_B\lf|f(y)-m_{\wz B^{\rho_2}}(f)\r|^q
\,d\mu(y)\r\}^{1/q}
+\lf[{\wz K}^{(\rho_2),\,1/(\az+1)}_{B,\,S}\r]^{\gz}
[\lz(c_S,r_S)]^{\az}\|f\|_{\ceagt}\\
&\hs\hs+\lf\{\frac1{\mu(S)}\int_S\lf|f(y)-m_{\wz S^{\rho_2}}(f)\r|^q
\,d\mu(y)\r\}^{1/q}\\
&\hs\ls[\lz(c_B,r_B)]^{\az}\|f\|_{\ceagt}
+\lf[{\wz K}^{(\rho_1),\,1/(\az+1)}_{B,\,S}\r]^{\gz}
[\lz(c_S,r_S)]^{\az}\|f\|_{\ceagt}\\
&\hs\ls\lf[{\wz K}^{(\rho_1),\,1/(\az+1)}_{B,\,S}\r]^{\gz}
[\lz(c_S,r_S)]^{\az}\|f\|_{\ceagt},
\end{align*}
which completes the proof of (ii) and hence Proposition \ref{p6.8}.
\end{proof}

\begin{remark}\label{r6.9}
(i) By Proposition \ref{p6.8}, we know that the space
$\ceaeg$ is independent of the choices of $\rho$ and $\eta$.
From now on, unless explicitly pointed out, we \emph{always
assume} that $\rho=\eta$ in $\ceaeg$ and \emph{write $\ceaeg$
simply by $\ceag$} and its \emph{norm $\|\cdot\|_{\ceaeg}$
simply by $\|\cdot\|_{\ceag}$}.

(ii) It is still unknown whether
$\ceaeg$ is independent of the choices of $\rho$ and $\eta$ or not
on general non-homogeneous metric measure spaces
without the assumption \eqref{6.1}.
\end{remark}

Before we show that $\ceag$ is independent of the choices of
$\gz$ and $q$, we first give a useful characterization of $\ceag$
which is a variant of \cite[Proposition 2.10]{hyy}.

\begin{proposition}\label{p6.10}
Suppose that $(\cx,d,\mu)$ is a non-homogeneous metric
measure space satisfying \eqref{6.1}.
Let $\rho\in(1,\fz)$, $\az\in[0,\fz)$ and $q,\,\gz\in[1,\fz)$.
The following statements are equivalent:

\emph{(a)} $f\in\ceag$;

\emph{(b)} there exists a sequence $\{f_B\}_B$ of complex numbers associated
with balls $B:=B(c_B,r_B)$, with $c_B\in\cx$ and $r_B\in(0,\fz)$, such that
\begin{align*}
\|f\|^{(q)}_{*,\,\rho}
&:=\sup_B\lf\{\frac1{\mu(\rho B)}\frac1{[\lz(c_B,\,r_B)]^{\az q}}
\int_B|f(y)-f_B|^q\,d\mu(y)\r\}^{1/q}\\
&\quad+\sup_{B\st S}\frac{|f_B-f_S|}{[\lz(c_S,\,r_S)]^{\az}
[\kbsa]^{\gz}}<\fz,
\end{align*}
where $c_S\in\cx$ and $r_S\in(0,\fz)$ denote, respectively,
the center and the radius of the ball $S$, and the first
supremum is taking all balls $B\st\cx$ and the second
one all balls $B\st S\st \cx$.

Moreover, the norms $\|\cdot\|_{\ceag}$ and $\|\cdot\|^{(q)}_{*,\,\rho}$
are equivalent.
\end{proposition}

\begin{proof}
Fix $\rho\in(1,\fz)$, $\az\in[0,\fz)$ and $q,\,\gz\in[1,\fz)$.
Let $f\in\ceag$. We first show that $\|f\|^{(q)}_{*,\,\rho}\ls\|f\|_{\ceag}$.
Indeed, for any ball $B$, let $f_B:=m_{\wz B^{\rho}}(f)$.
Then Proposition \ref{p6.7} implies that, for any two balls
$B\st S$,
$$|f_B-f_S|\ls[\lz(c_S,r_S)]^{\az}[\kbsa]^{\gz}\|f\|_{\ceag}.$$
This, together with the fact that, for any ball $B$,
$$
\lf\{\frac1{\mu(\rho B)}\frac1{[\lz(c_B,\,r_B)]^{\az q}}
\int_B|f(y)-f_B|^q\,d\mu(y)\r\}^{1/q}
\le\|f\|_{\ceag},
$$
implies that $\|f\|^{(q)}_{*,\,\rho}\ls\|f\|_{\ceag}$.

Conversely, assume that $\|f\|^{(q)}_{*,\,\rho}<\fz$. If $B$ is a
$(\rho,\bz_{\rho})$-doubling ball, then, by the H\"older
inequality, we have
\begin{equation}\label{6.6}
|f_B-m_B(f)|\ls\lf\{\frac1{\mu(\rho B)}\int_B|f(y)-f_B|^q\,d\mu(y)\r\}^{1/q}
\ls[\lz(c_B,r_B)]^{\az}\|f\|^{(q)}_{*,\,\rho},
\end{equation}
which, together with the Minkowski inequality, implies that
\begin{align*}
&\lf\{\frac1{\mu(B)}\int_B|f(y)-m_B(f)|^q\,d\mu(y)\r\}^{1/q}\\
&\hs\ls\lf\{\frac1{\mu(\rho B)}\int_B|f(y)-f_B|^q\,d\mu(y)\r\}^{1/q}
+|f_B-m_B(f)|\ls[\lz(c_B,r_B)]^{\az}\|f\|^{(q)}_{*,\,\rho}.
\end{align*}
Thus, by this, the Minkowski inequality,
\eqref{6.6}, Remark \ref{r6.2}(ii) and Lemma \ref{l2.8}(iii), we obtain
\begin{align*}
&\lf\{\frac1{\mu(\rho B)}\int_B\lf|f(y)-m_{\wz B^{\rho}}(f)\r|^q
\,d\mu(y)\r\}^{1/q}\\
&\hs\le\lf\{\frac1{\mu(\rho B)}\int_B|f(y)-f_B|^q\,d\mu(y)\r\}^{1/q}
+\lf|f_B-f_{\wz B^{\rho}}\r|+\lf|f_{\wz B^{\rho}}-m_{\wz B^{\rho}}(f)\r|\\
&\hs\ls\lf\{[\lz(c_B,r_B)]^{\az}+[\lz(c_B,r_{\wz B^{\rho}})]^{\az}
\lf[{\wz K}^{(\rho),\,1/(\az+1)}_{B,\,\wz B^{\rho}}\r]^{\gz}\r\}
\|f\|^{(q)}_{*,\,\rho}\\
&\hs\ls[\lz(c_B,r_B)]^{\az}\|f\|^{(q)}_{*,\,\rho}.
\end{align*}
Moreover, by \eqref{6.6}, \eqref{2.1} and \eqref{2.2},
we conclude that, for all $(\rho,\bz_\rho)$-doubling balls $B\st S$,
\begin{align*}
|m_B(f)-m_S(f)|&\le|m_B(f)-f_B|+|f_B-f_S|+|f_S-m_S(f)|\\
&\ls\lf\{[\lz(c_B,r_B)]^{\az}+[\lz(c_S,r_S)]^{\az}
\lf[\kbsa\r]^{\gz}\r\}\|f\|^{(q)}_{*,\,\rho}\\
&\ls[\lz(c_S,r_S)]^{\az}\|f\|^{(q)}_{*,\,\rho}\lf[\kbsa\r]^{\gz},
\end{align*}
which completes the proof of Proposition \ref{p6.10}.
\end{proof}

To show that $\ceag$ is independent of the choice of $\gz\in[1,\fz)$,
we need the following technical lemma,
which is similar to \cite[Lemma 2.6]{hyy}
(see also \cite[Lemma 9.2]{t01a}).

\begin{lemma}\label{l6.11}
Suppose that $(\cx,d,\mu)$ is a non-homogeneous metric measure space.
Let $m\in\nn\cap(1,\fz)$, $\rho\in(1,\fz)$, $p\in(0,1]$
and $B:=B_1\subset\cdots\subset B_m$ be concentric balls
with center $c_B$ and radii of the form $\rho^N r_B$, where $N\in\zz_+$.
If ${\wz K}^{(\rho),\,p}_{B_i,\,B_{i+1}}
>(3+\lfloor\log_{\rho}2\rfloor)^{1/p}$
for any $i\in\{1,\,\ldots,\,m-1\}$, then
\begin{equation}\label{6.7}
\dsum_{i=1}^{m-1}\lf[{\wz K}^{(\rho),\,p}_{B_i,\,B_{i+1}}\r]^p
<\lf(3+\lfloor\log_{\rho}2\rfloor\r)
\lf[{\wz K}^{(\rho),\,p}_{B_1,\,B_{m}}\r]^p.
\end{equation}
\end{lemma}

\begin{proof}
Fix $m\in\nn$, $\rho\in(1,\fz)$ and $p\in(0,1]$.
Assume that, for any $i\in\{1,\,\ldots,\,m\}$,
$r_{B_i}:=\rho^{N_i}r_B$
for some $N_i\in\zz_+$. For any $i\in\{1,\,\ldots,\,m-1\}$,
by ${\wz K}^{(\rho),\,p}_{B_i,\,B_{i+1}}
>(3+\lfloor\log_{\rho}2\rfloor)^{1/p}$, it is easy to see that
$N_{i+1}-N_i=N^{(\rho)}_{B_i,\,B_{i+1}}\ge1$,
$1<\sum_{k=1}^{N^{(\rho)}_{B_i,\,B_{i+1}}}
\lf[\frac{\mu(\rho^kB_i)}{\lz(c_B,\,\rho^kr_{B_i})}\r]^p$
and $N_m=N^{(\rho)}_{B_1,\,B_{m}}$.
From these facts and \eqref{2.1}, we deduce that,
for any $i\in\{1,\,\ldots,\,m-1\}$,
\begin{align*}
\lf[{\wz K}^{(\rho),\,p}_{B_i,\,B_{i+1}}\r]^p
&\le2+\lfloor\log_{\rho}2\rfloor+\sum_{k=1}^{N^{(\rho)}_{B_i,\,B_{i+1}}}
\lf[\frac{\mu(\rho^kB_i)}{\lz(c_B,\,\rho^kr_{B_i})}\r]^p\\
&<\lf(3+\lfloor\log_{\rho}2\rfloor\r)\sum_{k=1}^{N^{(\rho)}_{B_i,\,B_{i+1}}}
\lf[\frac{\mu(\rho^kB_i)}{\lz(c_B,\,\rho^kr_{B_i})}\r]^p\\
&=\lf(3+\lfloor\log_{\rho}2\rfloor\r)\sum_{k=N_i+1}^{N_{i+1}}
\lf[\frac{\mu(\rho^kB)}{\lz(c_B,\,\rho^kr_{B})}\r]^p.
\end{align*}
Notice that $p\in(0,1]$ and $N_m=N^{(\rho)}_{B_1,\,B_{m}}$.
It then follows that
\begin{align*}
\dsum_{i=1}^{m-1}\lf[{\wz K}^{(\rho),\,p}_{B_i,\,B_{i+1}}\r]^p
<\lf(3+\lfloor\log_{\rho}2\rfloor\r)
\lf[{\wz K}^{(\rho),\,p}_{B_1,\,B_{m}}\r]^p,
\end{align*}
which implies \eqref{6.7} and hence completes the proof of Lemma \ref{l6.11}.
\end{proof}

The following lemma is an analogue of \cite[Lemma 2.7]{hyy},
whose proof needs to use Lemma \ref{l6.11}, the details being omitted.

\begin{lemma}\label{l6.12}
Suppose that $(\cx,d,\mu)$ is a non-homogeneous metric measure space.
Let $\az\in[0,\fz)$, $\rho\in(1,\fz)$ and $p\in(0,1]$.
For a large positive constant $C$, the following statement holds true:
let $x\in\cx$ be a fixed point, and $\{f_B\}_{B\ni x}$
some collection of complex numbers associated with balls $B\ni x$.
If there exists a positive constant
$C_x$, depending on $x$, such that, for all balls $B$ and $S$
with $x\in B\st S$ and $\kbsp\le C$,
$|f_B-f_S|\le C_x\kbsp[\lz(c_S,r_S)]^{\az}$, then,
for all balls $B$ and $S$
with $x\in B\st S$,
$$|f_B-f_S|\le CC_x\kbsp[\lz(c_S,r_S)]^{\az}.$$
\end{lemma}

By Lemma \ref{l6.12}, now we are ready to state the result that
$\ceag$ is independent of the choice of
$\gz$, whose proof is similar to that of \cite[Proposition 2.5]{hyy},
the details being omitted.

\begin{proposition}\label{p6.13}
Suppose that $(\cx,d,\mu)$ is a non-homogeneous metric
measure space satisfying \eqref{6.1}.
Let $\az\in[0,\fz)$, $\rho,\,\gz\in(1,\fz)$ and $q\in[1,\fz)$.
Then $\ceag$ and ${{\mathcal E}^{\alpha,\,q}_{\rho,\,1}(\mu)}$
coincide with equivalent norms.
\end{proposition}

\begin{remark}\label{r6.14}
(i) By Proposition \ref{p6.13}, we know that the space
$\ceag$ is independent of the choice of $\gz$.
From now on, unless explicitly pointed out, we \emph{always
assume} that $\gz=1$ in $\ceag$ and write $\ceag$ simply by $\cear$.

(ii) It is still unknown whether
$\ceaeg$ is independent of the choice of $\gz$ or not
on general non-homogeneous metric measure spaces
without the assumption \eqref{6.1},
even on Euclidean spaces endowed with non-doubling measures.
\end{remark}

In order to show that $\cear$ is independent of $q$,
we establish the following John-Nirenberg type inequality
which is a generalization of \cite[Proposition 6.1]{h10}.
Hereafter, \emph{${\mathcal E}^{\alpha,\,1}_{\rho}(\mu)$ is
simply denoted by $\cera$} and its \emph{equivalent norm
$\|\cdot\|^{(1)}_{*,\,\rho}$ simply by $\|\cdot\|_{*,\,\rho}$}.

\begin{proposition}\label{p6.15}
Suppose that $(\cx,d,\mu)$ is a non-homogeneous metric
measure space satisfying \eqref{6.1}.
Let $\az\in[0,\fz)$ and $\rho\in(1,\fz)$. Then there exists
a positive constant $c$ such that, for any $f\in\cera$,
$t\in(0,\fz)$ and every ball $B_0:=B(x_0,r)$
with $x_0\in\cx$ and $r\in(0,\fz)$,
$$
\mu\lf(\lf\{x\in B_0:\ \frac{|f(x)-f_{B_0}|}{[\lz(x_0,r)]^{\az}}>t\r\}\r)
\le2\mu(\rho B_0)e^{-\frac{ct}{\|f\|_{*,\,\rho}}},
$$
where $f_{B_0}$ is as in Proposition \ref{p6.10}\emph{(ii)} with
$B$ replaced by $B_0$.
\end{proposition}

\begin{proof}
Fix $\az\in[0,\fz)$ and $\rho\in(1,\fz)$. Let $\sz:=5\rho$,
$f\in\cera$ and $L$ be a large positive constant
whose value will be determined later.
We first claim that, for $\mu$-almost every $x\in B_0$ with
$\frac{|f(x)-f_{B_0}|}{[\lz(x_0,r)]^{\az}}>2L$, there exists a
$(\sz,\bz_{\sz})$-doubling ball $\widehat{B}_x^{\sz}$ of the form
$B(x,\sz^{-i}r)$, $i\in\nn$, satisfying
\begin{equation}\label{6.8}
\widehat{B}_x^{\sz}\st\sqrt{\rho}B_0\quad {\rm and}\quad
\frac{|f_{\widehat{B}_x^{\sz}}-f_{B_0}|}{[\lz(x_0,r)]^{\az}}>L.
\end{equation}
Indeed, from $\frac{|f(x)-f_{B_0}|}{[\lz(x_0,r)]^{\az}}>2L$ and
\cite[Corollary 3.6]{h10}, it follows that there exists a
$(\sz,\bz_{\sz})$-doubling ball $\widehat{B}_x^{\sz}$ of the
form $B(x,\sz^{-i}r)$, $i\in\nn$, such that
$\widehat{B}_x^{\sz}\st\sqrt{\rho}B_0$ and
$\frac{|m_{\widehat{B}_x^{\sz}}(f)-f_{B_0}|}{[\lz(x_0,r)]^{\az}}>2L$.
Thus, by this, Propositions \ref{p6.8} and \ref{p6.10}, \eqref{2.1}
and \eqref{2.2}, we conclude that
\begin{align*}
\frac{|f_{\widehat{B}_x^{\sz}}-f_{B_0}|}{[\lz(x_0,r)]^{\az}}
&\ge\frac{|m_{\widehat{B}_x^{\sz}}(f)-f_{B_0}|}{[\lz(x_0,r)]^{\az}}
-\frac{|f_{\widehat{B}_x^{\sz}}
-m_{\widehat{B}_x^{\sz}}(f)|}{[\lz(x_0,r)]^{\az}}\\
&>2L-\frac1{[\lz(x_0,r)]^{\az}}\frac1{\mu(\widehat{B}_x^{\sz})}
\int_{\widehat{B}_x^{\sz}}\lf|f(y)-f_{\widehat{B}_x^{\sz}}\r|\,d\mu(y)\\
&\ge2L-\frac{[\lz(x,\sqrt{\rho}r)]^{\az}}{[\lz(x_0,r)]^{\az}}\bz_{\sz}
\|f\|_{*,\,\sqrt\rho}\\
&\ge2L-C_1\|f\|_{*,\,\rho}\ge L,
\end{align*}
provided that $L\ge C_1\|f\|_{*,\,\rho}$ and $C_1$ is
a positive constant, which implies the claim.

Now we let $\widehat{B}_x^{\sz}$ be the biggest
$(\sz,\bz_{\sz})$-doubling ball of the form
$B(x,\sz^{-i}r)$, $i\in\nn$, satisfying \eqref{6.8}.
By \eqref{6.8}, \eqref{2.1} and \eqref{2.2}, we know that
\begin{align}\label{6.9}
&\noz\frac1{\mu(\widehat{B}_x^{\sz})}\int_{\widehat{B}_x^{\sz}}
\frac{|f(y)-f_{B_0}|}{[\lz(x_0,r)]^{\az}}\,d\mu(y)\\
&\noz\hs\ge\frac{|f_{\widehat{B}_x^{\sz}}-f_{B_0}|}{[\lz(x_0,r)]^{\az}}
-\frac1{\mu(\widehat{B}_x^{\sz})}\int_{\widehat{B}_x^{\sz}}
\frac{|f(y)-f_{\widehat{B}_x^{\sz}}|}{[\lz(x_0,r)]^{\az}}\,d\mu(y)\\
&\hs>L-\frac{[\lz(x,\sqrt{\rho}r)]^{\az}}
{[\lz(x_0,r)]^{\az}}\bz_{\sz}\|f\|_{*,\,\sqrt\rho}\ge L
-C_1\|f\|_{*,\,\rho}\ge L/2,
\end{align}
provided that $L\ge2C_1\|f\|_{*,\,\rho}$.

Then we show that the ball $\wz{(\sz\widehat{B}^{\sz}_x)}^{\sz}
=:\widehat{\widehat{B}}^{\sz}_x$ satisfies
\begin{equation}\label{6.10}
\widehat{\widehat{B}}^{\sz}_x\not\subset\sqrt{\rho}B_0\quad
{\rm or}\quad \frac{|f_{\widehat{\widehat{B}}_x^{\sz}}-f_{B_0}|}
{[\lz(x_0,r)]^{\az}}\le L.
\end{equation}
Indeed, it suffices to prove that, if $\widehat{\widehat{B}}^{\sz}_x\st
\sqrt{\rho}B_0$,
then $\frac{|f_{\widehat{\widehat{B}}_x^{\sz}}-f_{B_0}|}
{[\lz(x_0,r)]^{\az}}\le L$.
From Lemma \ref{l2.9}, $B(x,r)\st 2\sqrt{\rho}B_0$,
(v), (iv), (iii) and (ii) of Lemma \ref{l2.8}, it follows that,
if $\widehat{\widehat{B}}^{\sz}_x\st\sqrt{\rho}B_0\st2\sqrt{\rho}B_0$, then
\begin{align*}
\frac{|f_{\widehat{\widehat{B}}_x^{\sz}}-f_{B_0}|}{[\lz(x_0,r)]^{\az}}
&\le\frac{|f_{\widehat{\widehat{B}}_x^{\sz}}-f_{2\sqrt{\rho}B_0}|}
{[\lz(x_0,r)]^{\az}}+\frac{|f_{2\sqrt{\rho}B_0}-f_{B_0}|}
{[\lz(x_0,r)]^{\az}}\\
&\ls\|f\|_{*,\,\rho}\frac{[\lz(x_0,2\sqrt{\rho}r)]^{\az}}
{[\lz(x_0,r)]^{\az}}\lf[{\wz K}^{(\rho),\,1/(\az+1)}
_{\widehat{\widehat{B}}_x^{\sz},\,2\sqrt{\rho}B_0}
+{\wz K}^{(\rho),\,1/(\az+1)}_{B_0,\,2\sqrt{\rho}B_0}\r]\\
&\ls \|f\|_{*,\,\rho}{\wz K}^{(\sz),\,1/(\az+1)}
_{\widehat{B}_x^{\sz},\,2\sqrt{\rho}B_0}
\ls \|f\|_{*,\,\rho}\lf[{\wz K}^{(\sz),\,1/(\az+1)}
_{\widehat{B}_x^{\sz},\,B(x,\,r)}
+{\wz K}^{(\sz),\,1/(\az+1)}_{B(x,\,r),\,2\sqrt{\rho}B_0}\r]\\
&\le C_2\|f\|_{*,\,\rho}\le L,
\end{align*}
provided that $L\ge C_2\|f\|_{*,\,\rho}$ and
$C_2$ is a positive constant, which shows \eqref{6.10}.

Moreover, if $\widehat{\widehat{B}}^{\sz}_x\not\subset\sqrt{\rho}B_0$,
let $\sz^j\widehat{B}^{\sz}_x$ be the smallest ball of the form
$\sz^k\widehat{B}^{\sz}_x$ ($k\in\nn$) satisfying
$\sz^j\widehat{B}^{\sz}_x\not\subset\sqrt{\rho}B_0$. We easily
obtain
\begin{equation*}
r_{\sz^j\widehat{B}^{\sz}_x}\sim r_{B_0}\quad {\rm and}\quad
\widehat{\widehat{B}}_x^{\sz}=\wz{\lf(\sz^j\widehat{B}^{\sz}_x\r)}^{\sz},
\end{equation*}
where $r_{\sz^j\widehat{B}^{\sz}_x}$ and $r_{B_0}$
denote the radii of balls $\sz^j\widehat{B}^{\sz}_x$ and
$B_0$, respectively.
By this, $\sz^j\widehat{B}^{\sz}_x\st3\sz\sqrt{\rho}B_0$,
$\widehat{B}^{\sz}_x\st\sqrt{\rho}B_0$,
Remark \ref{r6.2}(ii), \eqref{2.1}, \eqref{2.2}, Lemma \ref{l2.9},
(ii) and (iii) of Lemma \ref{l2.8}, we have
\begin{align*}
\frac{|f_{\widehat{\widehat{B}}_x^{\sz}}-f_{B_0}|}{[\lz(x_0,r)]^{\az}}
&\le\frac{|f_{\widehat{\widehat{B}}_x^{\sz}}-f_{\sz^j\widehat{B}^{\sz}_x}|}
{[\lz(x_0,r)]^{\az}}
+\frac{|f_{\sz^j\widehat{B}^{\sz}_x}-f_{3\sz\sqrt{\rho}B_0}|}
{[\lz(x_0,r)]^{\az}}+\frac{|f_{3\sz\sqrt{\rho}B_0}-f_{B_0}|}
{[\lz(x_0,r)]^{\az}}\\
&\ls\frac{[\lz(x,r_{\widehat{\widehat{B}}_x^{\sz}})]^{\az}}{[\lz(x_0,r)]^{\az}}
\|f\|_{*,\,\rho}+\frac{[\lz(x_0,3\sz\sqrt{\rho}r)]^{\az}}
{[\lz(x_0,r)]^{\az}}\|f\|_{*,\,\rho}\lf[{\wz K}^{(\rho),\,1/(\az+1)}
_{\sz^j\widehat{B}_x^{\sz},\,3\sz\sqrt{\rho}B_0}
+{\wz K}^{(\rho),\,1/(\az+1)}_{B_0,\,3\sz\sqrt{\rho}B_0}\r]\\
&\ls\frac{[\lz(x,r_{\sz\widehat{B}_x^{\sz}})]^{\az}}{[\lz(x_0,r)]^{\az}}
\|f\|_{*,\,\rho}+\|f\|_{*,\,\rho}\le C_3\|f\|_{*,\,\rho}\le L,
\end{align*}
provided that $L\ge C_3\|f\|_{*,\,\rho}$ and
$C_3$ is a positive constant.

Thus, in any case, we have
\begin{equation}\label{6.11}
 \frac{|f_{\widehat{\widehat{B}}_x^{\sz}}-f_{B_0}|}
{[\lz(x_0,r)]^{\az}}\le L,
\end{equation}
provided that $L\ge\max\{C_2,C_3\}\|f\|_{*,\,\rho}$.

Furthermore, by \cite[Theorem 1.2]{he} and \cite[Lemma 2.5]{h10},
we see that there exists a sequence $\{\widehat{B}_{x_i}^{\sz}\}_{i\in I}$
of disjoint balls such that $x_i\in B_0$ for any $i\in I$ and
$B_0\st\cup_{x\in B_0}\widehat{B}_x^{\sz}
\st\cup_{i\in I}5\widehat{B}_{x_i}^{\sz}$.
Let $B^{(i)}:=5\widehat{B}_{x_i}^{\sz}$ for any $i\in I$.
Observe that, for any $n\in\nn\cap[2,\fz)$,
if $x\in B_0$ and $\frac{|f(x)-f_{B_0}|}{[\lz(x_0,r)]^{\az}}>nL$,
then there exists $i\in I$ such that $x\in B^{(i)}$ and,
from \eqref{6.11}, Remark \ref{r6.2}(ii), \eqref{2.1},
 Lemma \ref{l2.9},
(v), (iv), (ii) and (iii) of Lemma \ref{l2.8}, it follows that
\begin{align}\label{6.12}
\noz\frac{|f(x)-f_{B^{(i)}}|}{[\lz(x_0,r)]^{\az}}&
\ge\frac{|f(x)-f_{B_0}|}{[\lz(x_0,r)]^{\az}}
-\frac{|f_{B_0}-f_{\widehat{\widehat{B}}_{x_i}^{\sz}}|}{[\lz(x_0,r)]^{\az}}
-\frac{|f_{\widehat{\widehat{B}}_{x_i}^{\sz}}-f_{5\widehat{B}_{x_i}^{\sz}}|}
{[\lz(x_0,r)]^{\az}}\\
&\noz>n L-L-\frac{[\lz(x_i,r_{\widehat{\widehat{B}}_{x_i}^{\sz}})]^{\az}}
{[\lz(x_0,r)]^{\az}}
\|f\|_{*,\,\rho}{\wz K}^{(\rho),\,1/(\az+1)}
_{5\widehat{B}_{x_i}^{\sz},\,\widehat{\widehat{B}}_{x_i}^{\sz}}\\
&\ge (n-1)L-C_4\|f\|_{*,\,\rho}\ge(n-2)L,
\end{align}
provided that $L\ge C_4\|f\|_{*,\,\rho}$ and $C_4$
is a positive constant.

By \eqref{6.9}, the disjointness of $\{\widehat{B}_{x_i}^{\sz}\}_{i\in I}$,
$\widehat{B}_{x_i}^{\sz}\st\sqrt{\rho}B_0$ for all
$i\in I$, Proposition \ref{p6.8}(ii), Proposition \ref{p6.10}(b),
Lemma \ref{l2.8}(ii) and \eqref{2.1},
we further see that
\begin{align}\label{6.13}
\noz\sum_{i\in I}\mu(\rho B^{(i)})&\le\bz_{\sz}
\sum_{i\in I}\mu(\widehat{B}_{x_i}^{\sz})
\le\frac{2\bz_{\sz}}{L}\sum_{i\in I}\int_{\widehat{B}_{x_i}^{\sz}}
\frac{|f(y)-f_{B_0}|}{[\lz(x_0,r)]^{\az}}\,d\mu(y)\\
&\noz\le\frac{2\bz_{\sz}}{L}\lf\{\int_{\sqrt{\rho}B_0}
\frac{|f(y)-f_{\sqrt{\rho}B_0}|}{[\lz(x_0,r)]^{\az}}\,d\mu(y)
+\frac{|f_{\sqrt{\rho}B_0}-f_{B_0}|\mu(\sqrt{\rho}B_0)}
{[\lz(x_0,r)]^{\az}}\r\}\\
&\noz\ls \frac1L \frac{[\lz(x_0,\sqrt{\rho}r)]^{\az}}{[\lz(x_0,r)]^{\az}}
\|f\|_{*,\,\rho}\lf\{\mu(\rho B_0)+\mu(\sqrt{\rho}B_0)
{\wz K}^{(\rho),\,1/(\az+1)}_{B_0,\,\sqrt{\rho}B_0}\r\}\\
&\le\frac{C_5}{L}\|f\|_{*,\,\rho}\mu(\rho B_0)\le\frac12 \mu(\rho B_0),
\end{align}
provided that $L\ge2C_5\|f\|_{*,\,\rho}$ and $C_5$ is a positive
constant.

Moreover, for any $t\in(0,\fz)$, there exists $n\in\zz_+$
such that $2nL\le t<2(n+1)L$. By this and \eqref{6.12},
we know that
\begin{align}\label{6.14}
\noz\lf\{x\in B_0:\ \frac{|f(x)-f_{B_0}|}{[\lz(x_0,r)]^{\az}}>t\r\}
&\st\lf\{x\in B_0:\ \frac{|f(x)-f_{B_0}|}{[\lz(x_0,r)]^{\az}}>2nL\r\}\\
&\st\bigcup_{i\in I}\lf\{x\in B^{(i)}:\
\frac{|f(x)-f_{B^{(i)}}|}{[\lz(x_0,r)]^{\az}}>2(n-1)L\r\}.
\end{align}
Finally, by \eqref{6.13}, \eqref{6.14}, iterating with
the balls $B^{(i)}$ in place of $B_0$ and an argument
similar to that used in the proof of \cite[Proposition 6.1]{h10},
we conclude that
$$
\mu\lf(\lf\{x\in B_0:\ \frac{|f(x)-f_{B_0}|}{[\lz(x_0,r)]^{\az}}>t\r\}\r)
\le2\mu(\rho B_0)e^{-\frac{ct}{\|f\|_{*,\,\rho}}}
$$
with $c:=\frac{\ln 2}{2L}\|f\|_{*,\,\rho}$ and
$L:=2\max\{C_i:\ i\in\{1,\ldots,5\}\}$.
This finishes the proof of Proposition \ref{p6.15}.
\end{proof}

By Proposition \ref{p6.15}, we easily obtain the following conclusion.

\begin{corollary}\label{c6.16}
Suppose that $(\cx,d,\mu)$ is a non-homogeneous metric
measure space satisfying \eqref{6.1}.
Let $\az\in[0,\fz)$, $\rho\in(1,\fz)$ and $q\in[1,\fz)$. Then there exists
a positive constant $C$ such that, for any $f\in\cera$
and every ball $B:=B(c_B,r_B)$ with $c_B\in\cx$ and $r_B\in(0,\fz)$,
$$
\lf\{\frac1{\mu(\rho B)}\int_B|f(x)-f_B|^q\,d\mu(x)\r\}^{1/q}
\le C\|f\|_{*,\,\rho}[\lz(c_B,r_B)]^{\az},
$$
where $f_B$ is as in Proposition \ref{p6.10}\emph{(ii)}.
\end{corollary}

\begin{remark}\label{r6.17}
(i) By Corollary \ref{c6.16} and Proposition \ref{p6.10},
together with the H\"older inequality,
we know that $\cear$ is independent
of the choice of $q$, the details being omitted.
From now on, unless explicitly pointed out, we \emph{always
assume} that $q=1$ in $\cear$ and \emph{write
${\mathcal E}^{\alpha,\,1}_{\rho}(\mu)$ simply by $\cera$}
and its \emph{norm $\|\cdot\|_{{\mathcal E}^{\alpha,\,1}_{\rho}(\mu)}$
simply by $\|\cdot\|_{\cera}$}.

(ii) It is still unknown whether
$\ceaeg$ is independent of the choice of $q$ or not
on general non-homogeneous metric measure spaces
without the assumption \eqref{6.1},
even on Euclidean spaces endowed with non-doubling measures.
\end{remark}

We establish
another characterization of $\cera$ which is needed in the later context.
To this end, we first recall
the so-called median value of
a function on balls in \cite{hyy,hmy13}.
Precisely, let $f$ be a measurable function.
The \emph{median value} of $f$ on any ball $B$,
denoted by $m_f(B)$, is defined
as follows.
If $f$ is real-valued, then, for any
ball $B$ with $\mu(B)\not=0$, let $m_f(B)$ be some {\it real number}
such that {$\inf_{c \in\rr}
\frac1{\mu(B)}\int_B|f(x)-c|\,d\mu(x)$
is attained}. It is known that $m_f(B)$ satisfies
\begin{equation*}
\mu(\{x\in B:\ \, f(x)>m_f(B)\})\le \mu(B)/2
\end{equation*}
and
\begin{equation*}
\mu(\{x\in B:\ \, f(x)<m_f(B)\})\le \mu(B)/2.
\end{equation*}
For all balls
$B$ with $\mu(B)=0$, let $m_f(B):=0$. If $f$ is complex-valued, we take
\begin{equation*}m_f(B):=[m_{\Re f}(B)]+i[m_{\Im f}(B)],
\end{equation*}
where $i^2=-1$ and, for any complex number $z$,
denote by $\Re z$ and $\Im z$
its {\it real part} and the {\it imaginary part}, respectively.

Let $\alpha\in[0, \fz)$, $\rho\in[2,\fz)$ and $q,\,\gz\in[1,\fz)$.
The norm $\|f\|_{\circ,\,\rho}$ of a suitable function $f$ is defined by
\begin{align*}
\|f\|_{\circ,\,\rho}
&:=\sup_{B:\ B\ (\rho,\, \bz_{\rho})-{\rm doubling\ ball}}
\frac1{\mu(B)}\frac1{[\lz(c_B,r_B)]^{\az}}
\int_B|f(y)-m_f(B)|\,d\mu(y)\\
&\quad+\sup_{B\st S:\ B,\,S\ (\rho, \,\bz_{\rho})-{\rm doubling\ balls}}
\frac{|m_f(B)-m_f(S)|}{[\lz(c_S,r_S)]^\alpha[\kbsa]^{\gz}}.
\end{align*}

Then we have the following equivalent
characterization of $\cera$.

\begin{proposition}\label{p6.18}
Suppose that $(\cx,d,\mu)$ is a non-homogeneous metric
measure space satisfying \eqref{6.1}.
Let $\alpha\in[0, \fz)$, $\rho\in[6/5,\fz)$ and $q,\,\gz\in[1,\fz)$.
Then the norms $\|\cdot\|_{\circ,\,\rho}$ and
$\|\cdot\|_{\cera}$ are equivalent.
\end{proposition}

\begin{proof}
Fix $\alpha\in[0, \fz)$ and $\rho\in[6/5,\fz)$.
For the sake of the simplicity, we assume that $\gz=1$.
The arguments here are still valid for the general case
with some minor modifications.
Let $f\in \cera$. Now we show that $\|f\|_{\circ,\,\rho}\ls\|f\|_{\cera}$.
For any $(\rho,\bz_{\rho})$-doubling ball $B$, by
the definition of $m_f(B)$, we conclude that
\begin{align}\label{6.15}
\noz|m_f(B)-m_B(f)|&\le\frac1{\mu(B)}\int_B|f(y)-m_f(B)|\,d\mu(y)\\
&\le\frac1{\mu(B)}\int_B|f(y)-m_B(f)|\,d\mu(y)
\ls\|f\|_{\cera}[\lz(c_B,r_B)]^{\az},
\end{align}
which implies that, for any
$(\rho,\bz_{\rho})$-doubling ball $B$,
$$
\frac1{\mu(B)}
\int_B|f(y)-m_f(B)|\,d\mu(y)\ls\|f\|_{\cera}[\lz(c_B,r_B)]^{\az}.
$$
On the other hand, by \eqref{6.15}, \eqref{2.1} and
\eqref{2.2}, we know that,
for all $(\rho,\bz_{\rho})$-doubling balls $B\subset S$,
\begin{align*}
|m_f(B)-m_f(S)|
&\le|m_f(B)- m_B(f)|+| m_B(f)-m_S(f)|
+|m_S(f)-m_f(S)|\\
&\ls\|f\|_{\cera}[\lz(c_S,r_S)]^{\az}\kbsa
+\|f\|_{\cera}\{[\lz(c_B,r_B)]^{\az}+[\lz(c_S,r_S)]^{\az}\}\\
&\ls\|f\|_{\cera}[\lz(c_S,r_S)]^{\az}\kbsa.
\end{align*}
Combining these two inequalities, we conclude that
$\|f\|_{\circ,\,\rho}\ls \|f\|_{\cera}$.

Conversely, let $\|f\|_{\circ,\,\rho}<\fz$.
We now prove that $\|f\|_{\ceag}\ls\|f\|_{\circ,\,\rho}$.
For any ball $B$, if $B$ is $(\rho,\bz_{\rho})$-doubling,
we see that
\begin{align*}
\frac1{\mu(\rho B)}\int_B|f(y)-m_{B}(f)|\,d\mu(y)
&\noz\le\frac1{\mu(B)}\int_B|f(y)-m_f(B)|\,d\mu(y)
+|m_f(B)-m_B(f)|\\
&\ls\frac1{\mu(B)}\int_B|f(y)-m_f(B)|\,d\mu(y)
\ls\|f\|_{\circ,\,\rho}[\lz(c_B,r_B)]^{\az}.
\end{align*}
Thus, we only need to consider the case that $B$ is
non-$(\rho,\bz_{\rho})$-doubling.

Assume that $B$ is a non-$(\rho,\bz_{\rho})$-doubling ball.
For any $x\in B$, let $B_x$ be the biggest
$(5\rho,\bz_{\rho})$-doubling ball centered at $x$
with radius $(5\rho)^{-k}r_B$ for some $k\in\nn$
(since $\bz_{\rho}>(5\rho)^{n_0}$ and \cite[Lemma 3.3]{h10}).
From $\rho\ge6/5$, it follows easily that
$5B_x\st2B\st(6/5)B\st\wz B^{\rho}$.
Moreover, by \cite[Theorem 1.2]{he} and \cite[Lemma 2.5]{h10},
we see that there exists a countable disjoint subfamily
$\{B_{x_i}\}_i$=:$\{B_i\}_i$ of $\{B_x\}_x$ such that
$x_i\in B$ for all $i$ and $B\st\cup_{x\in B}B_x\st\cup_i5B_i$.
For any $i$, by $5B_i\st B(x_i,r_B)\st (6/5)B\st \rho B\st\wz B^{\rho}$,
and (iv), (iii), (ii) and (v) of Lemma \ref{l2.8}, we see that
\begin{equation*}
{\wz K}^{(\rho),\,1/(\az+1)}_{5B_i,\,\wz B^{\rho}}
\ls{\wz K}^{(\rho),\,1/(\az+1)}_{5B_i,\,B(x_i,r_B)}
+{\wz K}^{(\rho),\,1/(\az+1)}_{B(x_i,r_B),\,\rho B}
+{\wz K}^{(\rho),\,1/(\az+1)}_{\rho B,\,\wz B^{\rho}}
\ls1.
\end{equation*}
From this, together with the fact that
$5B_i$ is a $(\rho,\bz_{\rho})$-doubling ball for any $i$,
\eqref{2.1} and \eqref{2.2},
Remark \ref{r6.2}(ii), $5B_i\st(6/5)B\st\rho B\st\wz B^{\rho}$ for any $i$
and the disjointness of $\{B_i\}_i$, it follows that
\begin{align*}
&\int_B\lf|f(y)-m_{\wz B^{\rho}}(f)\r|\,d\mu(y)\\
&\hs\le\sum_i\int_{5B_i}|f(y)-m_f(5B_i)|\,d\mu(y)
+\sum_i\mu(5B_i)\lf[\lf|m_f(5B_i)-m_f\lf(\wz B^{\rho}\r)\r|
+\lf|m_f\lf(\wz B^{\rho}\r)-m_{\wz B^{\rho}}(f)\r|\r]\\
&\hs\ls\|f\|_{\circ,\,\rho}
\sum_i\mu(5\rho B_i)\lf[\lz\lf(c_{B_i},r_{B_i}\r)\r]^{\az}
+\sum_i\mu(5B_i)\Bigg\{\|f\|_{\circ,\,\rho}
\lf[\lz\lf(c_B,r_{\wz B^{\rho}}\r)\r]^{\az}
\lf[{\wz K}^{(\rho),\,1/(\az+1)}_{5B_i,\,\wz B^{\rho}}\r]\\
&\hs\hs+\frac1{\mu(\wz B^{\rho})}
\int_{\wz B^{\rho}}\lf|f(y)-m_f\lf(\wz B^{\rho}\r)\r|\,d\mu(y)
\Bigg\}\\
&\hs\ls\|f\|_{\circ,\,\rho}\sum_i\mu(B_i)
\lf\{[\lz(c_B,r_B)]^{\az}+\lf[\lz\lf(c_B,r_{\wz B^{\rho}}\r)\r]^{\az}\r\}\\
&\hs\ls\|f\|_{\circ,\,\rho}\sum_i\mu(B_i)[\lz(c_B,r_B)]^{\az}
\ls\|f\|_{\circ,\,\rho}\mu(\rho B)[\lz(c_B,r_B)]^{\az}.
\end{align*}

On the other hand, for all $(\rho,\bz_{\rho})$-doubling balls
$B\st S$, by \eqref{2.1} and \eqref{2.2}, we have
\begin{align*}
&|m_B(f)-m_S(f)|\\
&\quad\le|m_B(f)- m_f(B)|+| m_f(B)-m_f(S)|
+|m_f(S)-m_S(f)|\\
&\quad\ls\lf\{[\lz(c_B,r_B)]^{\az}+[\lz(c_S,r_S)]^{\az}\kbsa
+[\lz(c_S,r_S)]^{\az}\r\}\|f\|_{\circ,\,\rho}\\
&\quad\ls\|f\|_{\circ,\,\rho}[\lz(c_S,r_S)]^{\az}\kbsa.
\end{align*}
These two inequalities show that
$\|f\|_{\cera}\ls\|f\|_{\circ,\,\rho}$,
which completes the proof of Proposition \ref{p6.18}.
\end{proof}

We point out that it is still unclear whether the range of $\rho$ in Proposition
\ref{p6.18} is sharp or not.

\section{Atomic Hardy Spaces $\hhp$ and
Molecular Hardy Spaces $\hmp$}\label{s7}

\hskip\parindent In this section, under the assumption of $\rho$-weakly
doubling conditions,
we introduce the atomic Hardy space $\hhp$ and the molecular Hardy space
$\hmp$, and show that the spaces $\hhp$
and $\hmp$ coincide with equivalent quasi-norms.

\begin{definition}\label{d7.1}
Let $\rho\in (1,\fz)$, $0<p\le1\le q\le\fz$, $p\neq q$ and $\gz\in[1,\fz)$.
A function
$b\in\lon$ is called  a \emph{$(p,q,\gz,\rho)_{\lz,\,1}$-atomic block} if
$b$ satisfies (i), (ii) and (iii) of Definition \ref{d3.2}.
Moreover, let $|b|_{\hhp}:=|\lz_1|+|\lz_2|$.
\end{definition}

\begin{remark}\label{r7.2}
It is easy to see that any $(1,q,\gz,\rho)_{\lz,\,1}$-atomic block
is also a $(1,q,\gz,\rho)_\lz$-atomic block
and vice versa. We point out that the difference between
the $(p,q,\gz,\rho)_\lz$-atomic block and the
$(p,q,\gz,\rho)_{\lz,\,1}$-atomic blocks exists in that
the former is an $\ltw$ function when $p\in(0,1)$,
while the latter is only an $\lon$ function.
\end{remark}

Observe that, for any $(p,q,\gz,\rho)_{\lz,\,1}$-atomic block $b$,
there exist some balls $B_j$ ($j\in\{1,\,2\}$) and $B$,
and some numbers $\lz_j\in\cc$ ($j\in\{1,\,2\}$) such that
$\supp(b)\st B$, $b=\lz_1a_1+\lz_2a_2$ and
$\supp(a_j)\st B_j\st B$, $j\in\{1,\,2\}$.
By $\int_\cx b(x)\,d\mu(x)=0$,
Proposition \ref{p6.7}(a), Definition \ref{d3.2}(iii), \eqref{2.1}
and \eqref{2.2}, we know that
\begin{align}\label{7.1}
\noz&\lf|\int_\cx f(x)b(x)\,d\mu(x)\r|\\
&\noz\hs=\lf|\int_\cx \lf[f(x)-m_{\wz B^{\rho}}(f)\r]b(x)\,d\mu(x)\r|\\
&\noz\hs\le\sum_{j=1}^2|\lz_j|
\int_{B_j}\lf|f(x)-m_{\wz B^{\rho}}(f)\r||a_j(x)|\,d\mu(x)\\
&\noz\hs\le\sum_{j=1}^2|\lz_j|
\lf[\int_{B_j}|a_j(x)|^q\,d\mu(x)\r]^{1/q}
\lf[\int_{B_j}\lf|f(x)-m_{\wz B^{\rho}}(f)\r|^{q'}\,d\mu(x)\r]^{1/q'}\\
&\noz\hs\le\sum_{j=1}^2|\lz_j|
[\mu(\rho B_j)]^{1/q-1}[\lz(c_B,r_B)]^{1-1/p}\lf[\kbjp\r]^{-\gz}\\
&\noz\hs\hs\times
\lf\{\lf[\int_{B_j}\lf|f(x)-m_{\wz B_j^{\rho}}(f)\r|^{q'}\,d\mu(x)\r]^{1/q'}
+[\mu(B_j)]^{1/q'}\lf|m_{\wz B_j^{\rho}}(f)
-m_{\wz B^{\rho}}(f)\r|\r\}\\
&\noz\hs\ls\sum_{j=1}^2|\lz_j|
[\mu(\rho B_j)]^{1/q-1}[\lz(c_B,r_B)]^{1-1/p}
\lf[\kbjp\r]^{-\gz}\\
&\noz\hs\hs\times\lf\{[\mu(\rho B_j)]^{1/q'}
+[\mu(B_j)]^{1/q'}\lf[\kbjp\r]^{\gz}\r\}[\lz(c_B,r_B)]^{1/p-1}\|f\|_{\cer}\\
&\hs\ls\sum_{j=1}^2|\lz_j|\|f\|_{\cer}
\sim|b|_{\hhp}\|f\|_{\cer}.
\end{align}
Thus, a $(p,q,\gz,\rho)_{\lz,\,1}$-atomic block $b$ can be seen as an
element in the \emph{dual space} $(\cer)^\ast$ of ${\cer}$.

\begin{definition}\label{d7.3}
Let $0<p\le1\le q\le\fz$ and $p\neq q$.
The \emph{atomic Hardy space $\hhp$} is defined as the subspace of
$(\cer)^\ast$ when $p<1$ and of $\lon$ when $p=1$,
consisting of those linear functional
admitting an atomic decomposition
\begin{equation}\label{7.2}
f=\sum_{i=1}^{\fz} b_{i}
\end{equation}
in $(\cer)^\ast$ when $p<1$ and in $\lon$ when $p=1$,
where $\{b_i\}_{i=1}^\fz$ are
$(p,q,\gz,\rho)_{\lz,\,1}$-atomic blocks such that
$$\sum_{i=1}^{\fz}|b_{i}|^p_{\hhp}<\fz.$$
Moreover, define
$$\|f\|_{\hhp}:=\inf
\lf\{\lf[\sum_{i=1}^{\fz}|b_{i}|^p_{\hhp}\r]^{1/p}\r\},$$
where the infimum is taken over all possible decompositions of
$f$ as above.
\end{definition}

\begin{remark}\label{r7.4}
(i) It follows from Remark \ref{r7.2} that
$\widehat H_{\rm{atb},\,\rho}^{1,\,q,\,\gz}(\mu)$ is the atomic Hardy space
defined via the discrete coefficients ${\wz K}_{B,\,S}^{(\rho)}$
introduced in \cite{fyy3}, where it was shown that
$\widehat H_{\rm{atb},\,\rho}^{1,\,q,\,\gz}(\mu)$ is independent
of the choices of $q$, $\rho$ and $\gz$. Hereafter,
$\widehat H_{\rm{atb},\,\rho}^{1,\,q,\,\gz}(\mu)$ is simply
denoted by $\widehat H_{\rm{atb}}^1(\mu)$.

(ii) Let $\rho\in(1,\fz)$, $\gz\in[1,\fz)$ and $q\in(1,\fz)$.
By Remarks \ref{r3.3}(ii) and (i) of this remark, we know that
$$
\widehat{H}^{1,\,q,\,\gz}_{\rm atb,\,\rho}(\mu)
=\wz{H}^{1,\,q,\,\gz}_{\rm atb,\,\rho}(\mu)
$$ over general non-homogeneous metric measure spaces.

(iii) Fix $p$, $\rho$ and $\gz$ as in Definition \ref{d7.1}.
For $1\le q_1\le q_2\le\fz$ and $q_1>p$, we notice that
$\widehat H^{p,\,q_2,\,\gz}_{{\rm atb},\,\rho}(\mu)
\st \widehat H^{p,\,q_1,\,\gz}_{{\rm atb},\,\rho}(\mu)$.

(iv) By the results in \cite{fyy3}, we know that the
Calder\'on-Zygmund operator is bounded on
$\widehat H^1_{{\rm atb}}(\mu)$. However, when $p\in(0,1)$, it is
still unclear whether the Calder\'on-Zygmund operator is bounded on
$\widehat H^{p,\,q,\,\gz}_{{\rm atb},\,\rho}(\mu)$ or not.
\end{remark}

We now introduce the notion of the molecular Hardy space
$\hmp$ in the non-homogeneous setting by first presenting
the following notion of $(p,q,\gz,\ez,\rho)_{\lz,\,1}$-molecular blocks.

\begin{definition}\label{d7.5}
Let $\rho\in (1,\fz)$, $0<p\le1\le q\le\fz$, $p\neq q$,
$\gz\in [1,\fz)$ and $\ez\in(0,\fz)$. A function
$b\in\lon$ is called  a
\emph{$(p,q,\gz,\ez,\rho)_{\lz,\,1}$-molecular block} if

(i) $\int_\cx b(x)\,d\mu(x)=0$;

(ii) there exist some ball $B$ and some constants
$\wz{M},\,M\in\nn$ such that,
for all $k\in\zz_+$ and $j\in\{1, \ldots, M_k\}$
with $M_k:=\wz{M}$ if $k=0$ and $M_k:=M$ if $k\in\nn$,
there exist functions $m_{k,\,j}$ supported
on some balls $B_{k,\,j}\st U_k(B)$
for all $k\in\zz_+$,
where $U_0(B):=\rho^2 B$ and $U_k(B):=\rho^{k+2}B\bh\rho^{k-2}B$
with $k\in\nn$, and
$\lz_{k,\,j}\in\cc$ such that
$b=\sum_{k=0}^{\fz}\sum_{j=1}^{M_k}\lz_{k,\,j}m_{k,\,j}$ in $\lon$
when $p=1$ and in both $\lon$ and $(\cer)^\ast$ when $p\in(0,1)$,
\begin{equation}\label{7.3}
\|m_{k,\,j}\|_{\lq}\le \rho^{-k\ez}\lf[\mu(\rho B_{k,\,j})\r]^{1/q-1}
\lf[\lz\lf(c_{B},\rho^{k+2}r_{B}\r)\r]^{1-1/p}
\lf[\wz K^{(\rho),\,p}_{B_{k,\,j},\,\rho^{k+2}B}\r]^{-\gz}
\end{equation}
and
$$|b|^p_{\hmp}:=\sum_{k=0}^{\fz}\sum_{j=1}^{M_k}|\lz_{k,\,j}|^p<\fz.$$
\end{definition}

\begin{remark}\label{r7.6}
Observe that any $(1,q,\gz,\ez,\rho)_{\lz,\,1}$-molecular block
is also a $(1,q,\gz,\ez,\rho)_\lz$-molecular block
and vice versa.
\end{remark}

\begin{definition}\label{d7.7}
Let $0<p\le1\le q\le\fz$, $p\neq q$ and $\ez\in(0,\fz)$.
The \emph{molecular Hardy space $\hmp$} is defined as the subspace of
$(\cer)^\ast$ when $p<1$ and of
$\lon$ when $p=1$, consisting of
those linear functional
admitting a molecular decomposition
\begin{equation}\label{7.4}
f=\sum_{i=1}^{\fz} b_{i}
\end{equation}
in $(\cer)^\ast$ when $p<1$ and in $\lon$ when $p=1$,
where $\{b_i\}_{i=1}^\fz$ are $(p,q,\gz,\ez,\rho)_{\lz,\,1}$-molecular
blocks such that
$$\sum_{i=1}^{\fz}|b_{i}|^p_{\hmp}<\fz.$$
Moreover, define
$$\|f\|_{\hmp}:=\inf
\lf\{\lf[\sum_{i=1}^{\fz}|b_{i}|^p_{\hmp}\r]^{1/p}\r\},$$
where the infimum is taken over all possible decompositions of
$f$ as above.
\end{definition}

\begin{remark}\label{r7.8}
(i) It follows from Remark \ref{r7.6} that
$\hp$ is the molecular Hardy space
defined via the discrete coefficients ${\wz K}_{B,\,S}^{(\rho)}$
introduced in \cite{fyy3}.

(ii) Let $\rho$, $p$, $q$, $\gz$ and $\ez$ be as in Definition \ref{d7.5}.
Then each $(p,q,\gz,\rho)_{\lz,\,1}$-atomic block
is a $(p,q,\gz,\ez,\rho)_{\lz,\,1}$-molecular
block and hence $\hhp\st\hmp$ and,
for all $f\in\hhp$,
$$
\|f\|_{\hmp}\le \|f\|_{\hhp}.
$$
\end{remark}

Moreover, we have the following relation between $\hhp$ and $\hmp$.

\begin{theorem}\label{t7.9}
Suppose that $(\cx,d,\mu)$ is a non-homogeneous
metric measure space satisfying \eqref{6.1}.
Let $\rho\in (1,\fz)$, $0<p\le1\le q\le\fz$, $p\neq q$,
$\gz\in [1,\fz)$ and $\ez\in(0,\fz)$.
Then $\hhp$ and $\hmp$ coincide with equivalent quasi-norms.
\end{theorem}

\begin{proof}
When $p=1$, the conclusion of Theorem \ref{t7.9} was obtained
in \cite[Theorem 1.11]{fyy3} without the assumption \eqref{6.1}.
Thus, we only need to consider the case $p\in(0,1)$.
Fix $\rho\in (1,\fz)$, $\gz\in [1,\fz)$, $\ez\in(0,\fz)$,
$0<p<1\le q\le\fz$ and $p\neq q$.
Let $I_B^{(\rho)}:=N^{(\rho)}_{B,\,\wz{B}^{\rho}}$
for any ball $B$.
By Remark \ref{r7.8}(ii), to show Theorem \ref{t7.9}
in this case, it suffices to prove that $\hmp\st\hhp$
and that, for any $f\in \hmp$, $f\in\hhp$ and $\|f\|_{\hhp}\ls \|f\|_{\hmp}$.
To this end, we first show that
any $(p,q,\gz,\ez,\rho)_{\lz,\,1}$-molecular block $b$
can be decomposed into a sum of some $(p,q,\gz,\rho)_{\lz,\,1}$-atomic blocks
and $(p,\fz,\gz,\rho)_{\lz,\,1}$-atomic blocks and
$\|b\|_{\hhp}\ls|b|_{\hmp}$.

For any $(p,q,\gz,\ez,\rho)_{\lz,\,1}$-molecular block $b$,
by Definition \ref{d7.5}, we know that
\begin{equation}\label{7.5}
b=\sum_{k=0}^{\fz}\sum_{j=1}^{M_k}
\lz_{k,\,j}m_{k,\,j}\quad {\rm in}\ \lon\ {\rm and\ } \lf(\cer\r)^\ast,
\end{equation}
where, for any $k\in\zz_+$ and $j\in\{1, \ldots, M_k\}$,
$\lz_{k,\,j}\in\cc$ and
$\supp(m_{k,\,j})\st B_{k,\,j}\st U_k(B)$ with the same notation as in
Definition \ref{d7.5}.  Moreover, observe that, by \eqref{7.5},
the H\"older inequality and \eqref{7.3}, we have
\begin{align}\label{7.6}
\noz\sum_{k=0}^\fz\sum_{j=1}^{M_k}\|\lz_{k,\,j}m_{k,\,j}\|_\lon
&\le\sum_{k=0}^\fz\sum_{j=1}^{M_k}|\lz_{k,\,j}| \rho^{-k\ez}
\lf[\lz\lf(c_{B},\rho^{k+2}r_{B}\r)\r]^{1-1/p}
\lf[\wz K^{(\rho),\,p}_{B_{k,\,j},\,\rho^{k+2}B}\r]^{-\gz}\\
&\ls\lf(\sum_{k=0}^\fz\sum_{j=1}^{M_k}|\lz_{k,\,j}| ^p\r)^{1/p}
[\lz(c_{B},r_{B})]^{1-1/p}<\fz.
\end{align}

For each $k\in\zz_+$, let $b_k:=\sum_{j=1}^{M_k}
\lz_{k,\,j}m_{k,\,j}, $ $B^{\rho}_{k+2}:=\rho^{k+2}B$ and
$\wz B^{\rho}_{k+2}:=\wz{(\rho^{k+2}B)^{2\rho}}$.
By \eqref{7.5}, we write
\begin{align*}
b&=\sum_{k=0}^{\fz}\lf[b_k-\frac{\chi_{\wz B^{\rho}_{k+2}}}
{\mu(\wz B^{\rho}_{k+2})}\int_{\cx}b_k(y)\,d\mu(y)\r]
+\sum_{k=0}^{\fz}\frac{\chi_{\wz B^{\rho}_{k+2}}}
{\mu(\wz B^{\rho}_{k+2})}\int_{\cx}b_k(y)\,d\mu(y)\\
&=\sum_{k=0}^{\fz}\sum_{j=1}^{M_k}\lz_{k,\,j}
\lf[m_{k,\,j}-\frac{\chi_{\wz B^{\rho}_{k+2}}}
{\mu(\wz B^{\rho}_{k+2})}\int_{B_{k,\,j}}m_{k,\,j}(y)\,d\mu(y)\r]
+\sum_{k=0}^{\fz}\frac{\chi_{\wz B^{\rho}_{k+2}}}
{\mu(\wz B^{\rho}_{k+2})}\int_{\cx}b_k(y)\,d\mu(y)\\
&=:\sum_{k=0}^{\fz}\sum_{j=1}^{M_k} b_{k,\,j}
+\sum_{k=0}^{\fz}\chi_k \wz M_k
=:{\rm I}+{\rm II},
\end{align*}
where, for all $k\in\zz_+$ and $j\in\{1,\ldots, M_k\}$,
$$b_{k,\,j}:=\lz_{k,\,j}\lf[m_{k,\,j}-\frac{\chi_{\wz B^{\rho}_{k+2}}}
{\mu(\wz B^{\rho}_{k+2})}\int_{B_{k,\,j}}m_{k,\,j}(y)\,d\mu(y)\r],$$
$\chi_k:=\frac{\chi_{\wz B^{\rho}_{k+2}}}
{\mu(\wz B^{\rho}_{k+2})}$
and $\wz M_k:=\int_{\cx}b_k(y)\,d\mu(y)$. From \eqref{7.6},
it follows that
$\sum_{k=0}^{\fz}\sum_{j=1}^{M_k} b_{k,\,j}$ and
$\sum_{k=0}^{\fz}\chi_k \wz M_k$ both converge in $\lon$.

To estimate I, we first show that, for any
$k\in\zz_+$ and $j\in\{1, \ldots, M_k\}$,
$b_{k,\,j}$ is a $(p,q,\gz,\rho)_{\lz,\,1}$-atomic block. Noticing that
$\supp(b_{k,\,j})\st2\wz B^{\rho}_{k+2}$ and
$\int_{\cx}b_{k,\,j}(y)\,d\mu(y)=0$, to show this,
it only needs to show that $b_{k,\,j}$ satisfies Definition \ref{d3.2}(iii).
To this end, we further decompose $b_{k,\,j}$ into
\begin{align*}
b_{k,\,j}&=\lz_{k,\,j}\lf[m_{k,\,j}-\frac{\chi_{\supp (m_{k,\,j})}}
{\mu(\wz B^{\rho}_{k+2})}\int_{B_{k,\,j}}m_{k,\,j}(y)\,d\mu(y)\r]\\
&\hs-\lz_{k,\,j}\frac{\chi_{\wz B^{\rho}_{k+2}\bh \supp (m_{k,\,j})}}
{\mu(\wz B^{\rho}_{k+2})}\int_{B_{k,\,j}}m_{k,\,j}(y)\,d\mu(y)
=:{\rm A}^{(1)}_{k,\,j}-{\rm A}^{(2)}_{k,\,j}.
\end{align*}
By the Minkowski inequality, the H\"older inequality,
\eqref{7.3}, \eqref{2.1}, Remark \ref{r6.2}(ii),
(iv) and (iii) of Lemma \ref{l2.8},
we know that
 \begin{align*}
&\lf\|{\rm A}_{k,\,j}^{(1)}\r\|_{\lq}\\
&\hs\le|\lz_{k,\,j}|\lf\{\|m_{k,\,j}\|_\lq
+\frac{[\mu(\supp (m_{k,\,j}))]^{1/q}}
{\mu(\wz B^{\rho}_{k+2})}\lf|\int_{B_{k,\,j}}
m_{k,\,j}(y)\,d\mu(y)\r|\r\}\\
&\hs\ls|\lz_{k,\,j}|\lf\{\|m_{k,\,j}\|_\lq
+\frac{[\mu(\supp (m_{k,\,j}))]^{1/q}[\mu(B_{k,\,j})]^{1/q'}}
{\mu(\wz B^{\rho}_{k+2})}
\|m_{k,\,j}\|_\lq\r\}\\
&\hs\ls|\lz_{k,\,j}|\|m_{k,\,j}\|_\lq\ls|\lz_{k,\,j}|\rho^{-k\ez}
\lf[\mu(\rho B_{k,\,j})\r]^{1/q-1}
\lf[\lz\lf(c_{B},r_{B^{\rho}_{k+2}}\r)\r]^{1-1/p}
\lf[\wz K^{(\rho),\,p}_{B_{k,\,j},\,B^{\rho}_{k+2}}\r]^{-\gz}\\
&\hs\ls|\lz_{k,\,j}|\rho^{-k\ez}
\lf[\mu(\rho B_{k,\,j})\r]^{1/q-1}
\lf[\lz\lf(c_{B},r_{2\wz B^{\rho}_{k+2}}\r)\r]^{1-1/p}
\lf[\wz K^{(\rho),\,p}_{B_{k,\,j},\,2\wz B^{\rho}_{k+2}}\r]^{-\gz}.
\end{align*}
Let $c_5$, independent of $k$ and $j$, be the implicit
positive constant of the above inequality,
$$
\mu^{(1)}_{k,\,j}:=c_5|\lz_{k,\,j}|\rho^{-k\ez}
$$
and
$a^{(1)}_{k,\,j}:=\frac1{\mu^{(1)}_{k,\,j}}{\rm A}^{(1)}_{k,\,j}$.
Then ${\rm A}^{(1)}_{k,\,j}=\mu^{(1)}_{k,\,j}a_{k,\,j}^{(1)}$,
$\supp(a_{k,\,j}^{(1)})\st B_{k,\,j}\st2\wz B^{\rho}_{k+2}$ and
$$\lf\|a_{k,\,j}^{(1)}\r\|_\lq\le\lf[\mu(\rho B_{k,\,j})\r]^{1/q-1}
\lf[\lz\lf(c_{B},r_{2\wz B^{\rho}_{k+2}}\r)\r]^{1-1/p}
\lf[\wz K^{(\rho),\,p}_{B_{k,\,j},\,2\wz B^{\rho}_{k+2}}\r]^{-\gz}.$$

From the H\"older inequality, \eqref{7.3}, the fact that
$\wz K^{(\rho),\,p}_{B_{k,\,j},\,B^{\rho}_{k+2}}\ge1$,
the $(2\rho,\bz_{2\rho})$-doubling property of $\wz B^{\rho}_{k+2}$,
Remark \ref{r6.2}(ii) and Lemma \ref{l2.8}(ii), it follows that
\begin{align*}
\lf\|{\rm A}_{k,\,j}^{(2)}\r\|_{\lq}&
=|\lz_{k,\,j}|
\frac{[\mu(\wz B^{\rho}_{k+2}\bh \supp (m_{k,\,j}))]^{1/q}}
{\mu(\wz B^{\rho}_{k+2})}
\lf|\int_{B_{k,\,j}}m_{k,\,j}(y)\,d\mu(y)\r|\\
&\le|\lz_{k,\,j}|\lf[\mu\lf(\wz B^{\rho}_{k+2}\r)\r]^{1/q-1}
\lf[\mu\lf(B_{k,\,j}\r)\r]^{1/q'}
\|m_{k,\,j}\|_\lq\\
&\ls|\lz_{k,\,j}|\lf[\mu\lf(\wz B^{\rho}_{k+2}\r)\r]^{1/q-1}
\lf[\mu\lf(B_{k,\,j}\r)\r]^{1/q'}\rho^{-k\ez}
\lf[\mu(\rho B_{k,\,j})\r]^{1/q-1}\lf[\lz\lf(c_{B},r_{B^{\rho}_{k+2}}\r)\r]^{1-1/p}\\
&\ls |\lz_{k,\,j}|
\lf[\mu\lf(2\rho\wz B^{\rho}_{k+2}\r)\r]^{1/q-1}\rho^{-k\ez}
\lf[\lz\lf(c_{B},r_{B^{\rho}_{k+2}}\r)\r]^{1-1/p}\\
&\ls|\lz_{k,\,j}|\rho^{-k\ez}
\lf[\mu\lf(2\rho\wz B^{\rho}_{k+2}\r)\r]^{1/q-1}
\lf[\lz\lf(c_{B},r_{2\wz B^{\rho}_{k+2}}\r)\r]^{1-1/p}
\lf[\wz K^{(\rho),\,p}_{2\wz B^{\rho}_{k+2},\,
2\wz B^{\rho}_{k+2}}\r]^{-\gz}.
\end{align*}
Let $c_6$, independent of $k$ and $j$, be the
implicit positive constant of the above inequality,
$$
\mu^{(2)}_{k,\,j}:=c_6|\lz_{k,\,j}|\rho^{-k\ez}
$$
and
$a^{(2)}_{k,\,j}:=\frac1{\mu^{(2)}_{k,\,j}}{\rm A}^{(2)}_{k,\,j}$.
Then ${\rm A}^{(2)}_{k,\,j}=\mu^{(2)}_{k,\,j}a_{k,\,j}^{(2)}$,
$\supp(a_{k,\,j}^{(2)})\st2\wz B^{\rho}_{k+2}$ and
$$\lf\|a_{k,\,j}^{(2)}\r\|_\lq
\le\lf[\mu\lf(2\rho\wz B^{\rho}_{k+2}\r)\r]^{1/q-1}
\lf[\lz\lf(c_{B},r_{2\wz B^{\rho}_{k+2}}\r)\r]^{1-1/p}
\lf[\wz K^{(\rho),\,p}_{2\wz B^{\rho}_{k+2},\,2
\wz B^{\rho}_{k+2}}\r]^{-\gz}.$$
Thus, $b_{k,\,j}=\mu_{k,\,j}^{(1)}a_{k,\,j}^{(1)}
+\mu_{k,\,j}^{(2)}a_{k,\,j}^{(2)}$
is a $(p,q,\gz,\rho)_{\lz,\,1}$-atomic block and
$$
|b_{k,\,j}|_{\hhp}\ls|\lz_{k,\,j}|\rho^{-k\ez}.
$$ Moreover, we have
\begin{equation}\label{7.7}
\|{\rm I}\|^p_{\hhp}
\ls\sum_{k=0}^{\fz}\sum_{j=1}^{M_k}|\lz_{k,\,j}|^p\rho^{-kp\ez}
\ls\sum_{k=0}^{\fz}\sum_{j=1}^{M_k}|\lz_{k,\,j}|^p\sim|b|^p_{\hmp}.
\end{equation}

Now we turn to estimate ${\rm II}$. Observe that, by \eqref{7.6} and
the H\"older inequality, we have
$$
\sum_{k=0}^\fz\lf|\wz{M}_k\r|\le\sum_{k=0}^\fz\|b_k\|_{\lon}<\fz.
$$
For each $k\in\zz_+$, let $N_k:=\sum_{i=k}^{\fz}\wz M_i$.
From the H\"older inequality and \eqref{7.3}, it follows that
\begin{align*}
\sum_{k=0}^\fz\|\chi_k N_k\|_\lon
&\le\sum_{k=0}^\fz\sum_{i=k}^\fz\lf\|\chi_k\wz{M}_i\r\|_\lon
\le\sum_{k=0}^\fz\sum_{i=k}^\fz\|b_i\|_\lon\\
&\le\sum_{k=0}^\fz\sum_{i=k}^\fz\sum_{j=1}^{M_i}
|\lz_{i,\,j}|\|m_{i,\,j}\|_\lq[\mu(B_{i,\,j})]^{1/q'}\\
&\le\sum_{k=0}^\fz\sum_{i=k}^\fz\sum_{j=1}^{M_i}
|\lz_{i,\,j}|\rho^{-i\ez}\lf[\lz\lf(c_B,\rho^{i+2}r_B\r)\r]^{1-1/p}\\
&\le\sum_{k=0}^\fz\rho^{-k\ez}\sum_{i=0}^\fz\sum_{j=1}^{M_i}
|\lz_{i,\,j}|[\lz(c_B,r_B)]^{1-1/p}\\
&\le\lf(\sum_{i=0}^\fz\sum_{j=1}^{M_i}
|\lz_{i,\,j}|^p\r)^{1/p}[\lz(c_B,r_B)]^{1-1/p}<\fz.
\end{align*}
Similarly, $\sum_{k=0}^\fz\|\chi_k N_{k+1}\|_\lon<\fz$.
By the above facts, we have
\begin{align*}
\sum_{k=0}^{\fz}\chi_k \wz M_k&=\sum_{k=0}^{\fz}\chi_k(N_k-N_{k+1})
=\sum_{k=0}^{\fz}(\chi_{k+1}-\chi_k)N_{k+1}+\chi_0 N_0\\
&=\sum_{k=0}^{\fz}(\chi_{k+1}-\chi_k)N_{k+1}\\
&=\sum_{k=0}^{\fz}\sum_{i=k+1}^{\fz}\sum_{j=1}^{M_i}\lz_{i,\,j}
(\chi_{k+1}-\chi_k)\int_{B_{i,\,j}}m_{i,\,j}(y)\,d\mu(y)\\
&=:\sum_{k=0}^{\fz}\sum_{i=k+1}^{\fz}\sum_{j=1}^{M_i} b_{k,\,j,\,i},
\end{align*}
where the summation in the last equality holds in $\lon$.

Now we prove that, for any $k\in\zz_+$, $i\in\{k+1,k+2,\ldots\}$
and $j\in\{1,\ldots,M_k\}$,
$b_{k,\,j,\,i}$ is a $(p,\fz,\gz,\rho)_{\lz,1}$-atomic block.
Observing that $\supp(b_{k,\,j,\,i})\st 2\wz B_{k+3}^{\rho}$ and
$\int_\cx b_{k,\,j,\,i}(y)\,d\mu(y)=0$, we only need to show that
$b_{k,\,j,\,i}$ satisfies Definition \ref{d3.2}(iii).

To this end, we further write
$$b_{k,\,j,\,i}=\lz_{i,\,j}\chi_{k+1}
\int_{B_{i,\,j}}m_{i,\,j}(y)\,d\mu(y)
-\lz_{i,\,j}\chi_k\int_{B_{i,\,j}}m_{i,\,j}(y)\,d\mu(y)
=:{\rm A}_{k,\,j,\,i}^{(1)}-{\rm A}_{k,\,j,\,i}^{(2)}.$$
From the H\"older inequality, \eqref{7.3}, the fact that
$\wz K^{(\rho),\,p}_{B_{i,\,j},\,B_{i+2}^{\rho}}\ge1$,
Remark \ref{r6.2}(ii) and Lemma \ref{l2.8}(ii), we deduce that
\begin{align*}
\lf\|{\rm A}_{k,\,j,\,i}^{(1)}\r\|_\li&\le|\lz_{i,\,j}|
\frac{[\mu(B_{i,\,j})]^{1/q'}}
{\mu(\wz B_{k+3}^{\rho})}\|m_{i,\,j}\|_\lq\\
&\le|\lz_{i,\,j}|
\frac{[\mu(B_{i,\,j})]^{1/q'}}
{\mu(\wz B_{k+3}^{\rho})}\rho^{-i\ez}
[\mu(\rho B_{i,\,j})]^{1/q-1}\lf[\lz\lf(c_{B_i},r_{B_{i+2}^{\rho}}\r)\r]^{1-1/p}
\lf[\wz K_{B_{i,\,j},\,B_{i+2}^{\rho}}^{(\rho),\,p}\r]^{-\gz}\\
&\ls|\lz_{i,\,j}|
\rho^{-i\ez}\lf[\mu\lf(2\rho\wz B_{k+3}^{\rho}\r)\r]^{-1}
\lf[\lz\lf(c_{B},r_{B_{k+3}^{\rho}}\r)\r]^{1-1/p}\\
&\ls|\lz_{i,\,j}|\rho^{-i\ez}
\lf[\mu\lf(2\rho\wz B_{k+3}^{\rho}\r)\r]^{-1}
\lf[\lz\lf(c_{B},r_{2\wz B_{k+3}^{\rho}}\r)\r]^{1-1/p}
\lf[\wz K^{(\rho),\,p}_{2\wz B_{k+3}^{\rho},\,
2\wz B_{k+3}^{\rho}}\r]^{-\gz}.
\end{align*}
Let $c_7$, independent of $k$ and $j$, be the
implicit positive constant of the above inequality,
$$
\mu^{(1)}_{k,\,j,\,i}:=c_7|\lz_{i,\,j}|\rho^{-i\ez}
$$ and
$a^{(1)}_{k,\,j,\,i}:=\frac1{\mu^{(1)}_{k,\,j,\,i}}
{\rm A}_{k,\,j,\,i}^{(1)}.$
Then we see that ${\rm A}_{k,\,j,\,i}^{(1)}
=\mu^{(1)}_{k,\,j,\,i}a^{(1)}_{k,\,j,\,i}$,
$\supp(a_{k,\,j,\,i}^{(1)})\st2\wz B_{k+3}^{\rho}$ and
$$\lf\|a_{k,\,j,\,i}^{(1)}\r\|_\li
\le\lf[\mu\lf(2\rho\wz B_{k+3}^{\rho}\r)\r]^{-1}
\lf[\lz\lf(c_{B},r_{2\wz B_{k+3}^{\rho}}\r)\r]^{1-1/p}
\lf[\wz K_{2\wz B_{k+3}^{\rho},\,
2\wz B_{k+3}^{\rho}}^{(\rho),\,p}\r]^{-\gz}.
$$

By an argument similar to that used in the estimate for
${\rm A}_{k,\,j,\,i}^{(1)}$, we conclude that
$$
\lf\|{\rm A}_{k,\,j,\,i}^{(2)}\r\|_\li\le c_8
|\lz_{i,\,j}|\rho^{-i\ez}\lf[\mu\lf(2\rho\wz B_{k+2}^{\rho}\r)\r]^{-1}
\lf[\lz\lf(c_{B},r_{2\wz B_{k+3}^{\rho}}\r)\r]^{1-1/p}
\lf[\wz K_{2\wz B_{k+2}^{\rho},\,
2\wz B_{k+3}^{\rho}}^{(\rho),\,p}\r]^{-\gz},
$$
where $c_8$ is a positive constant independent of $k$, $j$ and $i$.
Let $\mu^{(2)}_{k,\,j,\,i}:=c_8|\lz_{i,\,j}|\rho^{-i\ez}$ and
$$a^{(2)}_{k,\,j,\,i}:=\frac1{\mu^{(2)}_{k,\,j,\,i}}
{\rm A}_{k,\,j,\,i}^{(2)}.$$
Then ${\rm A}_{k,\,j,\,i}^{(2)}=\mu^{(2)}_{k,\,j,\,i}a^{(2)}_{k,\,j,\,i}$,
$\supp(a_{k,\,j,\,i}^{(2)})
\st2\wz B^{\rho}_{k+2} \st2\wz B_{k+3}^{\rho}$
and
$$\lf\|a_{k,\,j,\,i}^{(2)}\r\|_\li
\le\lf[\mu\lf(2\rho\wz B_{k+2}^{\rho}\r)\r]^{-1}
\lf[\lz\lf(c_{B},r_{2\wz B_{k+3}^{\rho}}\r)\r]^{1-1/p}
\lf[\wz K_{2\wz B_{k+2}^{\rho},\,
2\wz B_{k+3}^{\rho}}^{(\rho),\,p}\r]^{-\gz}.$$
Thus, $b_{k,\,j,\,i}=\mu_{k,\,j,\,i}^{(1)}a_{k,\,j,\,i}^{(1)}
+\mu_{k,\,j,\,i}^{(2)}
a_{k,\,j,\,i}^{(2)}$
is a $(p,\fz,\gz,\rho)_{\lz,\,1}$-atomic block and
$$|b_{k,\,j,\,i}|_{\widehat{H}^{p,\,\infty,\,\gz}_{\rm atb,\,\rho}\,(\mu)}
\ls|\lz_{i,\,j}|\rho^{-i\ez}.$$
 Moreover, we have
\begin{align*}
\|{\rm II}\|^p_{H^{p,\,\infty,\,\gz}_{\rm atb}\,(\mu)}
&\ls\sum_{k=0}^{\fz}\sum_{i=k+1}^{\fz}
\sum_{j=1}^{M_i}|b_{k,\,j,\,i}|^p_{\widehat{H}^{p,\,\infty,\,\gz}
_{\rm atb,\,\rho}\,(\mu)}\\
&\ls\sum_{k=0}^{\fz}\sum_{i=k+1}^{\fz}\rho^{-ip\ez}
\sum_{j=1}^{M_i}|\lz_{i,\,j}|^p
\sim\sum_{i=1}^{\fz}\rho^{-ip\ez}\sum_{k=0}^{i-1}
\sum_{j=1}^{M_i}|\lz_{i,\,j}|^p\\
&\ls\sum_{i=1}^{\fz}\sum_{j=1}^{M_i}|\lz_{i,\,j}|^p\sim|b|^p_{\hmp}.
\end{align*}
From this fact and \eqref{7.7}, we deduce that both
$\sum_{k=0}^{\fz}\sum_{j=1}^{M_k} b_{k,\,j}$
and $\sum_{k=0}^{\fz}\sum_{i=k+1}^{\fz}
\sum_{j=1}^{M_i} b_{k,\,j,\,i}$ converge in
$(\cer)^\ast$.

Now we claim that
$$
b=\sum_{k=0}^{\fz}\sum_{j=1}^{M_k} b_{k,\,j}
+\sum_{k=0}^{\fz}\sum_{i=k+1}^{\fz}\sum_{j=1}^{M_i} b_{k,\,j,\,i}\quad
{\rm in}\quad (\cer)^\ast.$$

Indeed, by $b=\sum_{k=0}^\fz\sum_{j=1}^{M_k}\lz_{k,\,j}m_{k,\,j}$
in $(\cer)^\ast$, we see that, for any $g\in\cer$,
\begin{align*}
\int_\cx b(x)g(x)\,d\mu(x)&=\lim_{K\to\fz}\sum_{k=0}^K\sum_{j=1}^{M_k}
\int_\cx \lz_{k,\,j}m_{k,\,j}(x)g(x)\,d\mu(x)\\
&=\lim_{K\to\fz}\sum_{k=0}^K\sum_{j=1}^{M_k}
\int_\cx b_{k,\,j}(x)g(x)\,d\mu(x)
+\lim_{K\to\fz}\sum_{k=0}^K\wz{M}_k\int_\cx \chi_k(x)g(x)\,d\mu(x).
\end{align*}
Moreover, by the fact that $|N_k|<\fz$ for any $k\in\zz_+$
and $N_0=0$, we further write
\begin{align*}
&\lim_{K\to\fz}\sum_{k=0}^K\wz{M}_k\int_\cx \chi_k(x)g(x)\,d\mu(x)\\
&\quad=\lim_{K\to\fz}\int_\cx\sum_{k=0}^K(N_k-N_{k+1})
\chi_k(x)g(x)\,d\mu(x)\\
&\quad=\lim_{K\to\fz}\int_\cx\lf[\chi_0(x)N_0-\chi_K(x)N_{K+1}
+\sum_{k=1}^KN_k\chi_k(x)-\sum_{k=0}^{K-1}N_{k+1}\chi_k(x)\r]g(x)\,d\mu(x)\\
&\quad=-\lim_{K\to\fz}\int_\cx\chi_K(x)N_{K+1}g(x)\,d\mu(x)
+\lim_{K\to\fz}\int_\cx\sum_{k=0}^{K-1}N_{k+1}
[\chi_{k+1}(x)-\chi_k(x)]g(x)\,d\mu(x)\\
&\quad=:{\rm A}+\lim_{K\to\fz}\sum_{k=0}^{K-1}
\sum_{i=K+1}^{\fz}\wz{M}_i\int_\cx
[\chi_{k+1}(x)-\chi_k(x)]g(x)\,d\mu(x)\\
&\quad={\rm A}+\lim_{K\to\fz}\sum_{k=0}^{K-1}\sum_{i=K+1}^{\fz}
\sum_{j=1}^{M_i}\int_\cx b_{k,\,j,\,i}(x)g(x)\,d\mu(x).
\end{align*}

Thus, to prove the above claim, it suffices to show that ${\rm A}=0$.
To this end, by the H\"older inequality and \eqref{7.3}, we conclude that,
for any $K\in\nn$,
\begin{align}\label{7.8}
\noz\lf|N_{K+1}\r|&\le\sum_{i=K+1}^{\fz}\lf|\wz{M}_i\r|
\le\sum_{i=K+1}^\fz\int_\cx|b_i(y)|\,d\mu(y)\\
&\noz\le\sum_{i=K+1}^\fz\sum_{j=1}^{M_i}|\lz_{i,\,j}|
\int_{B_{i,\,j}}|m_{i,\,j}(y)|\,d\mu(y)\\
&\noz\le\sum_{i=K+1}^\fz\sum_{j=1}^{M_i}|\lz_{i,\,j}|
\lf[\mu(B_{i,\,j})\r]^{1/q'}\|m_{i,\,j}\|_{\lq}\\
&\noz\le\sum_{i=K+1}^\fz\sum_{j=1}^{M_i}|\lz_{i,\,j}|
\lf[\mu(B_{i,\,j})\r]^{1/q'}\rho^{-i\ez}
\lf[\mu(\rho B_{i,\,j})\r]^{1/q-1}
\lf[\lz\lf(c_B,\rho^{i+2}r_B\r)\r]^{1-1/p}\\
&\noz\le\sum_{i=K+1}^\fz\sum_{j=1}^{M_i}|\lz_{i,\,j}|
\rho^{-K\ez}\lf[\lz\lf(c_B,\rho^{K+2}r_B\r)\r]^{1-1/p}\\
&\le\rho^{-K\ez}\lf[\lz\lf(c_B,\rho^{K+2}r_B\r)\r]^{1-1/p}
\lf(\sum_{i=K+1}^\fz\sum_{j=1}^{M_i}|\lz_{i,\,j}|^p\r)^{1/p}.
\end{align}
We write
\begin{align*}
\lf|\int_\cx\chi_{K}(x)N_{K+1}g(x)\,d\mu(x)\r|
&=\lf|N_{K+1}\r|\lf|m_{\wz{B}_K^{\rho}}(g)\r|\\
&\le\lf|N_{K+1}\r|\lf|m_{\wz{B}_K^{\rho}}(g)-m_{\wz{B}_0^{\rho}}(g)\r|
+\lf|N_{K+1}\r|\lf|m_{\wz{B}_0^{\rho}}(g)\r|\\
&=:{\rm I}_K+{\rm II}_K.
\end{align*}
By \eqref{7.8}, $g\in L^1_{\loc}(\mu)$ and
$\sum_{k=0}^{\fz}\sum_{j=1}^{M_k}|\lz_{k,\,j}|^p<\fz$, we know that
$$
{\rm II}_K\le\lf|m_{\wz{B}_0^{\rho}}(g)\r|\rho^{-K\ez}
\lf[\lz(c_B,r_B)\r]^{1-1/p}
\lf(\sum_{i=K+1}^\fz\sum_{j=1}^{M_i}|\lz_{i,\,j}|^p\r)^{1/p}\to0
\quad {\rm as}\quad K\to\fz.
$$
From Proposition \ref{p6.7}(a), $g\in\cer$,
\eqref{7.8} and $\sum_{k=0}^{\fz}\sum_{j=1}^{M_k}|\lz_{k,\,j}|^p<\fz$,
we further deduce that
\begin{align*}
{\rm I}_K&\le\lf|N_{K+1}\r|
\lf[\wz K_{\rho^2B,\,\rho^{K+2}B}^{(\rho),\,p}\r]^{\gz}
\lf[\lz\lf(c_B,\rho^{K+2}r_B\r)\r]^{1/p-1}\|g\|_{\cer}\\
&\ls\lf|N_{K+1}\r|K^{\gz/p}\lf[\lz\lf(c_B,\rho^{K+2}r_B\r)\r]^{1/p-1}
\|g\|_{\cer}\\
&\ls\rho^{-K\ez}K^{\gz/p}\lf(\sum_{i=K+1}^\fz\sum_{j=1}^{M_i}
|\lz_{i,\,j}|^p\r)^{1/p}\|g\|_{\cer}
\to 0 \quad {\rm as}\quad K\to\fz,
\end{align*}
which, together with the estimate of ${\rm I}_K$, completes
the proof of the above claim.

By Remark \ref{r7.4}(ii) and the estimates for I and II,
we see that $b\in\hhp$ and
$$\|b\|_{\hhp}^p\le\|{\rm I}\|_{\hhp}^p+\|{\rm II}\|_{\hhp}^p
\ls\|{\rm I}\|_{\hhp}^p+\|{\rm II}\|_{H^{p,\,\infty,\,\gz}
_{\rm atb}\,(\mu)}^p\ls|b|^p_{\hmp},$$
which, together with some standard arguments, then
completes the proof of Theorem \ref{t7.9}.
\end{proof}

\begin{remark}\label{r7.10}
(i) As was pointed out in the proof of Theorem \ref{t7.9},
if $\rho\in(1,\fz)$, $q\in(1,\fz]$, $\gz\in[1,\fz)$
and $\ez\in(0,\fz)$, then $\widehat H_{\rm{atb},\,\rho}^{1,\,q,\,\gz}(\mu)$
and $\widehat H_{\rm{mb},\,\rho}^{1,\,q,\,\gz,\,\ez}(\mu)$
coincide with equivalent norms, which is just \cite[Theorem 1.11]{fyy3};
namely, in this case, the assumption \eqref{6.1} is superfluous.
However, when $p\in(0,1)$, without \eqref{6.1}, it is still unclear whether
Theorem \ref{t7.9} holds true or not.

(ii) By Theorem \ref{t7.9}, we see that $\hmp$ is independent
of the choice of $\ez$ under the assumption \eqref{6.1}.
\end{remark}

The following result is an easy consequence of Theorem \ref{t7.9}
and Remark \ref{r6.2}(i), the details being omitted.

\begin{corollary}\label{c7.11}
Let $(\cx,d,\mu)$ be a space of homogeneous type with
the dominating function
$$\lz(x,r):=\mu(B(x,r))\quad \mathrm{for\ all}\ x\in\cx\
\mathrm{and}\ r\in(0,\fz),$$
and $\rho$, $p$,
$q$, $\gz$ and $\ez$ be as in Theorem \ref{t7.9}.
Then the conclusions in Theorem \ref{t7.9} and Remark \ref{r7.10}
also hold true in this setting.
\end{corollary}

\section{Duality between $\hhp$ and $\cer$}\label{s8}

\hskip\parindent
In this section, we show that $\cer$ is the dual space of $\hhp$.
To this end,  assuming that $(\cx,d,\mu)$
satisfies the assumption \eqref{6.1},
we show that $\hhp$ is independent
of the choices of $\rho$ and $\gz$.
We point out that \emph{all conclusions
in this section hold true for the case $p=1$
without the assumption \eqref{6.1}}; see \cite{fyy3,hyy} for the details.
Thus, we mainly focus on $p\in(0,1)$ in this section.

\begin{proposition}\label{p8.1}
Suppose that $(\cx,d,\mu)$ is a non-homogeneous metric
measure space satisfying \eqref{6.1}.
Let $\rho\in (1,\fz)$, $0<p<1\le q\le\fz$ and $\gz\in[1,\fz)$.
Then the space $\hhp$ is independent of the choice of $\rho\in(1,\fz)$.
\end{proposition}

\begin{proof}
Let $0<p<1\le q\le\fz$ and $\gz\in[1,\fz)$.
Assume that $\rho\ge\rho_1>\rho_2>1$.
It is easy to see that $\widehat H^{p,\,q,\,\gz}_{{\rm atb},\,\rho_1}(\mu)
\st \widehat H^{p,\,q,\,\gz}_{{\rm atb},\,\rho_2}(\mu)$ and,
for all $f\in \widehat H^{p,\,q,\,\gz}_{{\rm atb},\,\rho_1}(\mu)$,
$$
\|f\|^p_{\widehat H^{p,\,q,\,\gz}_{{\rm atb},\,\rho_2}\,(\mu)}
\le\|f\|^p_{\widehat H^{p,\,q,\,\gz}_{{\rm atb},\,\rho_1}\,(\mu)}.
$$

On the other hand, to show that
$\widehat H^{p,\,q,\,\gz}_{{\rm atb},\,\rho_2}(\mu)
\st \widehat H^{p,\,q,\,\gz}_{{\rm atb},\,\rho_1}(\mu)$, let
$$
b=\sum_{j=1}^2\lz_ja_j\in \widehat H^{p,\,q,\,\gz}_{{\rm atb},\,\rho_2}(\mu)
$$
be a $(p,q,\gz,\rho_2)_{\lz,\,1}$-atomic block, where, for any $j\in\{1,2\}$,
$a_j$ is a function supported on $B_j\st B$ for some
balls $B_j$ and $B$ as in Definition \ref{d7.1}.

Now we claim that, without loss of generality, we may assume that $B$ is
$(\rho^2,\bz_{\rho^2})$-doubling. The reasons are as follows:
if $B$ is non-$(\rho^2,\bz_{\rho^2})$-doubling, by Lemma \ref{l2.9},
(iv) and (ii) of Lemma \ref{l2.8}, \eqref{2.1} and Remark \ref{r6.2}(ii),
we see that
\begin{align*}
\|a_j\|_\lq&\le [\mu(\rho_2 B_j)]^{1/q-1}[\lz(c_B,r_B)]^{1-1/p}
\lf[\wz K^{(\rho_2),\,p}_{B_j,\,B}\r]^{-\gz}\\
&\ls [\mu(\rho_2 B_j)]^{1/q-1}
\lf[\lz\lf(c_B,\rho^{2N^{(\rho^2)}_{B,\,{\wz B}^{\rho^2}}}r_B\r)\r]^{1-1/p}
\lf[\wz K^{(\rho_2),\,p}_{B_j,\,\wz{B}^{\rho^2}}\r]^{-\gz}\\
&\ls [\mu(\rho_2 B_j)]^{1/q-1}\lf[\lz\lf(c_B,r_{\wz{B}^{\rho^2}}\r)\r]^{1-1/p}
\lf[\wz K^{(\rho_2),\,p}_{B_j,\,\wz{B}^{\rho^2}}\r]^{-\gz}.
\end{align*}
Thus, we can replace $B$ by $\wz{B}^{\rho^2}$, which shows the claim.

Then, for each $j\in\{1,2\}$, we have
\begin{equation}\label{8.1}
\|a_j\|_\lq\le [\mu(\rho_2 B_j)]^{1/q-1}[\lz(c_B,r_B)]^{1-1/p}
\lf[\wz K^{(\rho_2),\,p}_{B_j,\,B}\r]^{-\gz}.
\end{equation}
From Remark \ref{r2.2}(ii), it follows that there exists a sequence
$\{B_{k,\,j}\}_{k=1}^N$ of balls such that
$$
B_j\st\bigcup_{k=1}^N B_{k,\,j}=:\bigcup_{k=1}^N B
\lf(c_{B_{k,\,j}},\frac{\rho_2-1}
{10\rho_0(\rho_1+1)}r_{B_j}\r)
$$
and $c_{B_{k,\,j}}\in B_j$ for all $k\in\{1,\ldots,N\}$,
where $\rho_0\in(1,\rho_1)$. Observe that
$\rho_1\rho_0 B_{k,\,j}\st\rho_2 B_j$.
For any $k\in\{1,\ldots,N\}$, define
$a_{k,\,j}:=a_j\frac{\chi_{B_{k,\,j}}}{\sum_{k=1}^N\chi_{B_{k,\,j}}}$
and $\lz_{k,\,j}:=\lz_j$.
Then we have
$$
\supp (a_{k,\,j})\st \rho_0B_{k,\,j}\quad {\rm and}\quad
b=\sum_{j=1}^2\lz_ja_j=\sum_{j=1}^2\sum_{k=1}^N\lz_{k,\,j}a_{k,\,j}.
$$
Moreover,
by \eqref{8.1}, the fact that $\rho_2 B_j\st3\rho B$, \eqref{2.1},
Lemma \ref{l2.9}, (i), (ii), (iv) and (v) of Lemma \ref{l2.8}
and the fact that $\rho_0B_{k,\,j}\st\rho B$,
we know that
\begin{align}\label{8.2}
\noz\|a_{k,\,j}\|_{\lq}&\le\|a_j\|_{\lq}\le[\mu(\rho_2 B_j)]^{1/q-1}
[\lz(c_B,r_B)]^{1-1/p}
\lf[\wz K^{(\rho_2),\,p}_{B_j,\,B}\r]^{-\gz}\\
&\noz\ls[\mu(\rho_1\rho_0 B_{k,\,j})]^{1/q-1}
[\lz(c_B,\rho r_B)]^{1-1/p}
\lf[\wz K^{(\rho_1),\,p}_{B_j,\,3\rho B}\r]^{-\gz}\\
&\ls[\mu(\rho_1\rho_0 B_{k,\,j})]^{1/q-1}
[\lz(c_B,\rho r_B)]^{1-1/p}
\lf[\wz K^{(\rho_1),\,p}_{\rho_0B_{k,\,j},\,\rho B}\r]^{-\gz}.
\end{align}
Let $C_{k,\,j}:=\lz_{k,\,j}(a_{k,\,j}+\gz_{k,\,j}\chi_{B})$, where
$\gz_{k,\,j}:=-\frac1{\mu(B)}\int_{\cx}a_{k,\,j}(x)\,d\mu(x)$.
Now we claim that $C_{k,\,j}$ is a
$(p,q,\gz,\rho_1)_{\lz,\,1}$-atomic block.
Indeed, $\supp(C_{k,\,j})\st\rho B$
and $\int_{\cx}C_{k,\,j}(x)\,d\mu(x)=0$.
Moreover, since $B_{k,\,j}\st\rho B$,
by the H\"older inequality,
\eqref{8.2}, $B$ is $(\rho^2,\bz_{\rho^2})$-doubling, $\rho>\rho_1$,
\eqref{2.1} and Lemma \ref{l2.8}(ii),
we conclude that
\begin{align*}
\|\gz_{k,\,j}\chi_{B}\|_{\lq}&\le[\mu(B)]^{1/q-1}\|a_{k,\,j}\|_{\lq}
\lf[\mu\lf(B_{k,\,j}\r)\r]^{1-1/q}\\
&\ls[\mu(\rho_1\rho B)]^{1/q-1}[\lz(c_B,\rho r_B)]^{1-1/p}
\lf[\wz K^{(\rho_1),\,p}_{B_{k,\,j},\,\rho B}\r]^{-\gz}\\
&\ls[\mu(\rho_1\rho B)]^{1/q-1}[\lz(c_B,\rho r_B)]^{1-1/p}
\lf[\wz K^{(\rho_1),\,p}_{\rho B,\,\rho B}\r]^{-\gz}.
\end{align*}
This, together with \eqref{8.2}, \eqref{2.1},
Lemma \ref{l2.8}(ii), implies that
$|C_{k,\,j}|_{\widehat H^{p,\,q,\,\gz}_{{\rm atb},\,\rho_1}(\mu)}
\ls|\lz_{k,\,j}|$. Thus, the claim holds true.

By the above claim and $\int_{\cx}b(x)\,d\mu(x)=0$, we see that
\begin{equation}\label{8.3}
b=\sum_{j=1}^2\sum_{k=1}^N C_{k,\,j}\in
\widehat H^{p,\,q,\,\gz}_{{\rm atb},\,\rho_1}(\mu)
\end{equation}
and
\begin{equation}\label{8.4}
\|b\|^p_{\widehat H^{p,\,q,\,\gz}_{{\rm atb},\,\rho_1}(\mu)}
\ls\sum_{j=1}^2\sum_{k=1}^N|C_{k,\,j}|^p_{\widehat H^{p,\,q,\,\gz}
_{{\rm atb},\,\rho_1}(\mu)}
\ls\sum_{j=1}^2|\lz_j|^p\sim|b|^p_{\widehat{H}^{p,\,q,\,\gz}
_{{\rm atb},\,\rho_2}(\mu)}.
\end{equation}

For all $f\in \widehat H^{p,\,q,\,\gamma}_{{\rm atb},\,\rho_2}\,(\mu)$,
by Proposition \ref{p6.8}, we know that there exists a sequence $\{b_i\}_i$
of $(p,q,\gz,\rho_2)_{\lz,\,1}$-atomic blocks such that
$f=\sum_{i=1}^\fz b_i$ in $(\mathcal E^{1/p-1}_{\rho_2}(\mu))^\ast
=(\mathcal E^{1/p-1}_{\rho_1}(\mu))^\ast$
and $$\sum_{i=1}^\fz|b_i|^p_{\widehat H^{p,\,q,\,\gamma}
_{{\rm atb},\,\rho_2}\,(\mu)}
\ls\|f\|^p_{\widehat H^{p,\,q,\,\gamma}_{{\rm atb},\,\rho_2}\,(\mu)}.$$
From this fact, \eqref{8.3} and \eqref{8.4},
we further deduce that
$f=\sum_{i=1}^\fz\sum_{j=1}^2\sum_{k=1}^N C^i_{k,\,j}$
in $(\mathcal E^{1/p-1}_{\rho_1}(\mu))^\ast$,
where $\{C^i_{k,\,j}\}_{i,\,j,\,k}$ are all $(p,q,\gz,\rho_1)
_{\lz,\,1}$-atomic
blocks as in
\eqref{8.3} satisfying
$$\sum_{i=1}^\fz\sum_{j=1}^2\sum_{k=1}^N
|C^i_{k,\,j}|^p_{\widehat H^{p,\,q,\,\gamma}_{{\rm atb},\,\rho_1}\,(\mu)}
\ls \sum_{i=1}^\fz
|b_i|^p_{\widehat H^{p,\,q,\,\gamma}_{{\rm atb},\,\rho_2}\,(\mu)}
\ls\|f\|^p_{\widehat H^{p,\,q,\,\gamma}_{{\rm atb},\,\rho_2}\,(\mu)},$$
which implies that $f\in \widehat H^{p,\,q,\,\gamma}
_{{\rm atb},\,\rho_1}\,(\mu)$ and
$$
\|f\|_{\widehat H^{p,\,q,\,\gamma}_{{\rm atb},\,\rho_1}
\,(\mu)}\ls\|f\|_{\widehat H^{p,\,q,\,\gamma}_{{\rm atb},\,\rho_2}\,(\mu)}.
$$
This finishes the proof of Proposition \ref{p8.1}.
\end{proof}

\begin{proposition}\label{p8.2}
Let $\rho\in(1,\fz)$, $0<p<1\le q\le\fz$ and $\gz\in[1,\fz)$.
Then the space $\hhp$ is independent of the choice of $\gz\in[1,\fz)$.
\end{proposition}

\begin{proof}
Assume that $1\le\gz_1<\gz_2$.
Notice that $\lf[{\wz K}^{(\rho),\,p}_{B,\,S}\r]^{-\gz_2}
\le\lf[{\wz K}^{(\rho),\,p}_{B,\,S}\r]^{-\gz_1}$ for all balls $B\subset S$.
From this, we deduce that
$\widehat H^{p,\,q,\,\gamma_2}_{{\rm atb},\,\rho}\,(\mu)
\st \widehat H^{p,\,q,\,\gamma_1}_{{\rm atb},\,\rho}\,(\mu)$ and,
for all $f\in \widehat H^{p,\,q,\,\gamma_2}_{{\rm atb},\,\rho}\,(\mu)$,
$f\in \widehat H^{p,\,q,\,\gamma_1}_{{\rm atb},\,\rho}\,(\mu)$ and
$$
\|f\|_{\widehat H^{p,\,q,\,\gamma_1}_{{\rm atb},\,\rho}\,(\mu)}
\le\|f\|_{\widehat H^{p,\,q,\,\gamma_2}_{{\rm atb},\,\rho}\,(\mu)}.
$$

Now we consider the following converse inclusion that
$\widehat H^{p,\,q,\,\gamma_1}_{{\rm atb},\,\rho}\,(\mu)
\st \widehat H^{p,\,q,\,\gamma_2}_{{\rm atb},\,\rho}\,(\mu).$
Let
$$
b=\dsum_{j=1}^2\lz_j a_j
\in \widehat H^{p,\,q,\,\gamma_1}_{{\rm atb},\,\rho}\,(\mu)
$$
be a $(p,\,q,\,\gz_1,\,\rho)_{\lz,\,1}$-atomic block,
where, for any $j\in\{1,\,2\}$,
$a_j$ is a function supported on $B_j\subset B$ for some balls $B_j$ and
$B$ as in Definition \ref{d7.1}.
We first show that any $(p,q,\gz_1,\rho)_{\lz,\,1}$-atomic block
can be decomposed into a sum of some
$(p,q,\gz_2,\rho)_{\lz,\,1}$-atomic blocks and
\begin{equation}\label{8.5}
\|b\|_{\widehat H^{p,\,q,\,\gamma_2}_{{\rm atb},\,\rho}\,(\mu)}
\ls|b|_{\widehat H^{p,\,q,\,\gamma_1}_{{\rm atb},\,\rho}\,(\mu)}.
\end{equation}

To prove \eqref{8.5}, we consider the following four cases:

{\bf Case (I)} For any $j\in\{1,\,2\}$, $\wz K^{(\rho),\,p}_{B_j,\,B}
\le [(3+\lfloor\log_{\rho}2\rfloor)C_{(\rho)}]^{1/p}$,
where $C_{(\rho)}$ is as in Lemma \ref{l2.8}(i);

{\bf Case (II)} $\wz K^{(\rho),\,p}_{B_1,\,B}>
[(3+\lfloor\log_{\rho}2\rfloor)C_{(\rho)}]^{1/p}$
and $\wz K^{(\rho),\,p}_{B_2,\,B}
\le [(3+\lfloor\log_{\rho}2\rfloor)C_{(\rho)}]^{1/p}$;

{\bf Case (III)} $\wz K^{(\rho),\,p}_{B_1,\,B}
\le [(3+\lfloor\log_{\rho}2\rfloor)C_{(\rho)}]^{1/p}$
and $\wz K^{(\rho),\,p}_{B_2,\,B}
>[(3+\lfloor\log_{\rho}2\rfloor)C_{(\rho)}]^{1/p}$;

{\bf Case (IV)} For any $j\in\{1,\,2\}$,
$\wz K^{(\rho),\,p}_{B_j,\,B}
>[(3+\lfloor\log_{\rho}2\rfloor)C_{(\rho)}]^{1/p}$.

In {\bf Case (I)}, for any $j\in\{1,\,2\}$, we have
\begin{align*}
\lf[\wz K^{(\rho),\,p}_{B_j,\,B}\r]^{-\gz_1}&<1
=\lf[\lf(3+\lfloor\log_{\rho}2\rfloor\r)C_{(\rho)}\r]^{\gz_2/p}
\cdot\lf[\lf(3+\lfloor\log_{\rho}2\rfloor\r)C_{(\rho)}\r]^{-\gz_2/p}\\
&\le\lf[\lf(3+\lfloor\log_{\rho}2\rfloor\r)C_{(\rho)}\r]^{\gz_2/p}
\lf[\wz K^{(\rho),\,p}_{B_j,\,B}\r]^{-\gz_2}.
\end{align*}
For any $j\in\{1,\,2\}$,
let $\wz \lz_j:=[(3+\lfloor\log_{\rho}2\rfloor)C_{(\rho)}]^{\gz_2/p}\lz_j$ and
$\wz a_j:=[(3+\lfloor\log_{\rho}2\rfloor)C_{(\rho)}]^{-\gz_2/p}a_j$.
Then $b=\wz\lz_1\wz a_1+\wz\lz_2\wz a_2$. From this, it is easy
to see that $b$ is a $(p,\,q,\,\gz_2,\,\rho)_{\lz,\,1}$-atomic block,
which implies that $b\in \widehat H^{p,\,q,\,\gamma_2}
_{{\rm atb},\,\rho}\,(\mu)$ and
$$
\|b\|_{\widehat H^{p,\,q,\,\gamma_2}_{{\rm atb},\,\rho}\,(\mu)}
\le \lf[\lf(3+\lfloor\log_{\rho}2\rfloor\r)C_{(\rho)}\r]
^{\gz_2/p}(|\lz_1|+|\lz_2|)
\sim|b|_{\widehat H^{p,\,q,\,\gamma_1}_{{\rm atb},\,\rho}\,(\mu)}$$
in this case.

The proofs of {\bf Case (II)}, {\bf Case (III)} and {\bf Case(IV)} are similar.
For brevity, we only prove {\bf Case (II)}.

In {\bf Case (II)}, we have $\wz K^{(\rho),\,p}_{B_1,\,B}>
[(3+\lfloor\log_{\rho}2\rfloor)C_{(\rho)}]^{1/p}$.
We now choose a sequence $\{B^{(i)}_1\}_{i=0}^m$ of balls
with certain $m\in\nn$ as follows. Let $B_1^{(0)}:=B_1$
and $B_0:=\wz{\rho^{N^{(\rho)}_{B_1,\,B}}B_1}^{\rho^2}$.
To choose $B^{(1)}_1$, let $N_1$ be the smallest
positive integer satisfying
$\wz K^{(\rho),\,p}_{B_1^{(0)},\,\rho^{N_1}B_1^{(0)}}
>[(3+\lfloor\log_{\rho}2\rfloor)C_{(\rho)}]^{1/p}$. If
$r_{\wz{\rho^{N_1}B_1^{(0)}}^{\rho^2}}\ge r_{B_0}$, then we let $B^{(1)}_1:=B_0$
and the selection process terminates.
Otherwise, we let $B_1^{(1)}:=\wz{\rho^{N_1}B_1^{(0)}}^{\rho^2}$.
To choose $B^{(2)}_1$, if, for any $N\in\nn$,
$\wz K^{(\rho),\,p}_{B_1^{(1)},\,\rho^{N}B_1^{(1)}}
\le[(3+\lfloor\log_{\rho}2\rfloor)C_{(\rho)}]^{1/p}$,
let $B^{(2)}_1:=B_0$
and the selection process terminates. Otherwise,
let $N_2$ be the smallest
positive integer satisfying
$\wz K^{(\rho),\,p}_{B_1^{(1)},\,\rho^{N_2}B_1^{(1)}}
>[(3+\lfloor\log_{\rho}2\rfloor)C_{(\rho)}]^{1/p}$. If
$r_{\wz{\rho^{N_2}B_1^{(1)}}^{\rho^2}}\ge r_{B_0}$,
then we let $B^{(2)}_1:=B_0$ and the selection process terminates.
Otherwise, we let $B^{(2)}_1:=\wz{\rho^{N_2}B_1^{(1)}}^{\rho^2}$.
We continue as long as this selection process is possible;
clearly, finally the condition
$r_{\wz{\rho^{N_{}i+1}B^{(i)}_1}^{\rho^2}}< r_{B_0}$ is violated
after finitely many steps. Without loss of generality, we may assume
that the process will stop after $m$ ($m\in\nn\cap(1,\fz)$) steps.
Now we conclude that $\{B^{(i)}_1\}_{i=0}^m$ have the following properties:

(i) $B_1^{(0)}:=B_1$, $B^{(i)}_1:=\wz{\rho^{N_i}B^{(i-1)}_1}^{\rho^2}$
for any $i\in\{1,\,\ldots,\,m-1\}$, and $B^{(m)}_1:=B_0$;

(ii) for any $i\in\{1,\,\ldots,\,m-1\}$, by Lemma \ref{l2.8}(i)
and the definition of $N_i$, we have
$$
\wz K^{(\rho),\,p}_{B^{(i-1)}_1,\,B^{(i)}_1}\ge\lf[C_{(\rho)}\r]^{-1/p}
\wz K^{(\rho),\,p}_{B^{(i-1)}_1,\,\rho^{N_i}B^{(i-1)}_1}
>\lf(3+\lfloor\log_{\rho}2\rfloor\r)^{1/p};
$$

(iii) there exists a positive constant $C$ such that,
for any $i\in\{1,\,\ldots,\,m\}$,
$\wz K^{(\rho),\,p}_{B^{(i-1)}_1,\,B^{(i)}_1}\le C$.
Indeed, if, for any $N\in\nn$,
$\wz K^{(\rho),\,p}_{B^{(m-1)}_1,\,\rho^{N}B_1^{m-1}}
\le[(3+\lfloor\log_{\rho}2\rfloor)C_{(\rho)}]^{1/p}$, then,
from the choice of $B^{(m)}_1$, we have
$\wz K^{(\rho),\,p}_{B^{(m-1)}_1,\,B^{(m)}_1}\ls 1$. Otherwise,
by Lemma \ref{l2.9}, (iv), (ii) and (iii) of Lemma \ref{l2.8}
and the definition of $N_i$, we see that, for any $i\in\{1.\ldots,m\}$,
\begin{align*}
\wz K^{(\rho),\,p}_{B^{(i-1)}_1,\,B^{(i)}_1}&\le2^{1-p}
\lf[\wz K^{(\rho),\,p}_{B^{(i-1)}_1,\,\rho^{N_i}B^{(i-1)}_1}
+c_{(\rho,\,p,\,\nu)}{\wz K}^{(\rho),\,p}_{\rho^{N_i}B^{(i-1)}_1,\,
\wz{\rho^{N_i}B^{(i-1)}_1}^{\rho^2}}\r]\\
&\le2^{1-p}\lf[K^{(\rho),\,p}_{B^{(i-1)}_1,\,\rho^{N_i-1}B^{(i-1)}_1}
+c_{(\rho,\,p,\,\nu)}{\wz K}^{(\rho),\,p}_{\rho^{N_i-1}B^{(i-1)}_1,
\,\rho^{N_i}B^{(i-1)}_1}
+c_{(\rho,\,p,\,\nu)}\r]\le C;
\end{align*}

(iv) by (ii), Lemma \ref{l6.11}, the fact that
$B_1^{m-1}\subset B^{(m)}_1\subset2\rho^{2{\wz C}_1+1}B$,
and (i), (iv) and (ii) of Lemma \ref{l2.8},
where ${\wz C}_1$ is as in \eqref{6.1},
we know that
\begin{align*}
m=(m-2)+2&\le\sum_{i=1}^{m-2}\lf[{\wz K}^{(\rho),\,p}
_{B^{(i)}_1,\,B_1^{i+1}}\r]^p+2
<\lf(3+\lfloor\log_{\rho}2\rfloor\r)\lf[\wz K^{(\rho),\,p}
_{B_1,\,B_1^{m-1}}\r]^p+2\\
&\ls\lf[\wz K^{(\rho),\,p}_{B_1,\,2\rho^{2{\wz C}_1+1}B}\r]^p
\ls\lf[\wz K^{(\rho),\,p}_{B_1,\,B}\r]^p+
\lf[\wz K^{(\rho),\,p}_{B,\,2\rho^{2{\wz C}_1+1}B}\r]^p
\ls\lf[\wz K^{(\rho),\,p}_{B_1,\,B}\r]^p.
\end{align*}
Let $C$ be the implicit positive constant of the above inequality,
$(\wz C_b)^p:=C[\wz K^{(\rho),\,p}_{B_1,\,B}]^p$ and
$\wz c_0:=\wz C_b a_1$. For any $i\in\{1,\,\ldots,\,m\}$, let
$$
\wz c_i:=\dfrac{\chi_{B^{(i)}_1}}{\mu(B^{(i)}_1)}
\dint_{\cx}\wz c_{i-1}(y)\,d\mu(y).
$$
If $i=0$, by Definition \ref{d3.2}(iii), \eqref{2.2},
$r_{B_1^{(1)}}\le\rho^{2{\wz C}_1+1}r_B$, \eqref{2.1} and (iii), we have
\begin{align}\label{8.6}
\|\wz c_0\|_{L^q(\mu)}&\ls[\mu(\rho B_1)]^{1/q-1}[\lz(c_B,r_B)]^{1-1/p}
\lf[\wz K^{(\rho),\,p}_{B_1,\,B}\r]^{-\gz_1+1}\nonumber\\
&\ls\lf[\mu\lf(\rho B_1^{(0)}\r)\r]^{1/q-1}
\lf[\lz\lf(c_{B_1},r_B\r)\r]^{1-1/p}\nonumber\\
&\ls\lf[\mu\lf(\rho B_1^{(0)}\r)\r]^{1/q-1}
\lf[\lz\lf(c_{B_1},r_{\rho B_1^{(1)}}\r)\r]^{1-1/p}
\lf[\wz K^{(\rho),\,p}_{B^{0}_1,\,\rho B_1^{(1)}}\r]^{-\gz_2},
\end{align}
where the implicit positive constant is
independent of $\wz K^{(\rho),\,p}_{B_1,\,B}$.
For $i=1$, by the H\"older inequality, Definition \ref{d3.2}(iii), \eqref{2.2},
the facts that $r_{B_1^{(1)}}\le\rho^{2{\wz C}_1+1}r_B$
and $B_1^{(1)}$ is doubling, \eqref{2.1} and Lemma \ref{l2.8}(ii),
we conclude that
\begin{align}\label{8.7}
\|\wz c_1\|_{L^q(\mu)}&\le\lf[\mu\lf(B_1^{(1)}\r)\r]^{1/q-1}[\mu(B_1)]^{1-1/q}
\lf\|\wz C_b a_1\r\|_{L^q(\mu)}\nonumber\\
&\ls\lf[\mu\lf(B_1^{(1)}\r)\r]^{1/q-1}[\mu(B_1)]^{1-1/q}[\mu(\rho B_1)]^{1/q-1}
[\lz(c_B,r_B)]^{1-1/p}
\lf[\wz K^{(\rho),\,p}_{B_1,\,B}\r]^{-\gz_1+1}\nonumber\\
&\ls\lf[\mu\lf(B_1^{(1)}\r)\r]^{1/q-1}
\lf[\lz\lf(c_{B_1},r_B\r)\r]^{1-1/p}\nonumber\\
&\ls\lf[\mu\lf(\rho^2 B_1^{(1)}\r)\r]^{1/q-1}
\lf[\lz\lf(c_{B_1},r_{\rho B_1^{(1)}}\r)\r]^{1-1/p}
\lf[\wz K^{(\rho),\,p}_{\rho B^{(1)}_1,\,\rho B_1^{(1)}}\r]^{-\gz_2}.
\end{align}
Similar to \eqref{8.6} and \eqref{8.7}, respectively,
for any $i\in\{2,\,\ldots,\,m\}$, we have
\begin{equation}\label{8.8}
\|\wz c_{i-1}\|_{L^q(\mu)}\ls
\lf[\mu\lf(\rho^2 B^{(i-1)}_1\r)\r]^{1/q-1}
\lf[\lz\lf(c_{B_1},r_{\rho B_1^{i}}\r)\r]^{1-1/p}
\lf[\wz K^{(\rho),\,p}_{\rho B^{(i-1)}_1,\,\rho B_1^{i}}\r]^{-\gz_2}
\end{equation}
and
\begin{equation}\label{8.9}
\|\wz c_i\|_{L^q(\mu)}\ls
\lf[\mu\lf(\rho^2 B^{(i)}_1\r)\r]^{1/q-1}
\lf[\lz\lf(c_{B_1},r_{\rho B_1^{i}}\r)\r]^{1-1/p}
\lf[\wz K^{(\rho),\,p}_{\rho B^{(i)}_1,\,\rho B_1^{i}}\r]^{-\gz_2}.
\end{equation}
For any $i\in\{1,\,\ldots,\,m\}$, let $c_i:=\frac{\lz_1}
{\wz C_b}(\wz c_{i-1}-\wz c_i)$.
Then $\supp(c_i)\subset\rho B^{(i)}_1$ and
$$
\dint_\cx c_i(x)\,d\mu(x)=0,
$$
which, together with \eqref{8.8} and \eqref{8.9},
implies that $c_i$ is a $(p,\,q,\,\gz_2,\,\rho)_{\lz,\,1}$-atomic block
associated with the ball $\rho B^{(i)}_1$ and
\begin{equation}\label{8.10}
|c_i|_{\widehat H^{p,\,q,\,\gamma_2}_{{\rm atb},\,\rho}\,(\mu)}
\ls \dfrac{|\lz_1|}{\wz C_b}.
\end{equation}

Now we see that
\begin{equation*}
b=\dsum_{i=1}^m c_i
+\dfrac{\lz_1}{\wz C_b}\wz c_m+\lz_2 a_2.
\end{equation*}
Notice that $\int_\cx b(x)\,d\mu(x)=0$ and $\int_\cx c_i(x)\,d\mu(x)=0$.
It then follows that
$$
\dint_\cx\lf[\dfrac{\lz_1}{\wz C_b}\wz c_m(x)+\lz_2 a_2(x)\r]\,d\mu(x)=0.
$$
On the other hand, we have $\supp(\wz c_m)\subset
\rho B^{(m)}_1\subset 2\rho^{{\wz C}_1+2}B=:B'$,
$r_{B'}=2\rho^{{\wz C}_1+2}r_B\le2\rho^{{\wz C}_1+2}r_{B^{(m)}_1}$
and $\supp(a_2)\subset B_2\subset B'$.
An argument similar to that used in the estimate of \eqref{8.7} shows that
\begin{equation*}
\|\wz c_m\|_{L^q(\mu)}\ls\lf[\mu\lf(\rho B^{(m)}_1\r)\r]^{1/q-1}
\lf[\lz\lf(c_{B_1},r_{B'}\r)\r]^{1-1/p}
\lf[\wz K^{(\rho),\,p}_{B^{m}_1,\,B'}\r]^{-\gz_2}.
\end{equation*}
From Definition \ref{d3.2}(iii), \eqref{2.1},
$\wz K^{(\rho),\,p}_{B_2,\,B}
\le [(3+\lfloor\log_{\rho}2\rfloor)C_{(\rho)}]^{1/p}$,
$B'=2\rho^{{\wz C}_1+2}B$ and (iv) and (ii) of Lemma \ref{l2.8},
it follows that
\begin{align*}
\|a_2\|_\lq&\ls [\mu(\rho B_2)]^{1/q-1}[\lz(c_B,r_B)]^{1-1/p}
\lf[\wz K^{(\rho),\,p}_{B_2,\,B}\r]^{-\gz_2}\\
&\ls[\mu(\rho B_2)]^{1/q-1}[\lz(c_B,r_{B'})]^{1-1/p}
\lf[\wz K^{(\rho),\,p}_{B_2,\,B'}\r]^{-\gz_2}.
\end{align*}
Thus, $c_{m+1}:=\dfrac{\lz_1}{\wz C_b}\wz c_m+\lz_2 a_2$ is a
$(p,\,q,\,\gz_2,\,\rho)_{\lz,\,1}$-atomic block
associated with the ball $B'$ and
$$
|c_{m+1}|_{\widehat H^{p,\,q,\,\gamma_2}_{{\rm atb},\,\rho}
\,(\mu)}\ls |\lz_1|+|\lz_2|.
$$
By this fact, the definition of $\wz C_b$, \eqref{8.10} and (iv),
we obtain
$b=\sum_{i=1}^{m+1} c_i
\in \widehat H^{p,\,q,\,\gamma_2}_{{\rm atb},\,\rho}\,(\mu)$
and
\begin{equation*}
\|b\|^p_{\widehat H^{p,\,q,\,\gamma_2}_{{\rm atb},\,\rho}\,(\mu)}
\ls \sum_{i=1}^{m+1} |c_i|^p_{\widehat H^{p,\,q,\,\gamma_2}
_{{\rm atb},\,\rho}\,(\mu)}
\ls (|\lz_1|+|\lz_2|)^p
\sim |b|^p_{\widehat H^{p,\,q,\,\gamma_1}_{{\rm atb},\,\rho}\,(\mu)},
\end{equation*}
where the implicit positive constant is independent of $m$.
This finishes the proof of \eqref{8.5}.

Let $f\in \widehat H^{p,\,q,\,\gamma_1}_{{\rm atb},\,\rho}\,(\mu)$.
Then, by Proposition \ref{p6.13} and Definition \ref{d7.1},
there exists a sequence $\{b_j\}_j$
of $(p,q,\gz_1,\rho)_{\lz,\,1}$-atomic blocks such that
$f=\sum_{j=1}^\fz b_j$ in
$({{\mathcal E}^{\alpha,\,q}_{\rho,\,\gamma_1}(\mu)})^\ast
=({{\mathcal E}^{\alpha,\,q}_{\rho,\,\gamma_2}(\mu)})^\ast$
and $$\sum_{j=1}^\fz|b_j|^p_{\widehat H^{p,\,q,\,\gamma_1}
_{{\rm atb},\,\rho}\,(\mu)}
\ls\|f\|^p_{\widehat H^{p,\,q,\,\gamma_1}_{{\rm atb},\,\rho}\,(\mu)}.$$
From this fact and \eqref{8.5},
we further deduce that
$f=\sum_{j=1}^\fz\sum_{i=1}^{m_j+1} c_{j,\,i}$
in $({{\mathcal E}^{\alpha,\,q}_{\rho,\,\gamma_2}(\mu)})^\ast$,
where $\{c_{j,\,i}\}_{j,\,i}$ are all
$(p,q,\gz_2,\rho)_{\lz,\,1}$-atomic blocks as in
\eqref{8.5} satisfying
$$\sum_{j=1}^\fz\sum_{i=1}^{m_j+1}|c_{j,\,i}|^p_{\widehat H^{p,\,q,\,\gamma_2}
_{{\rm atb},\,\rho}\,(\mu)}\ls \sum_{j=1}^\fz
|b_j|^p_{\widehat H^{p,\,q,\,\gamma_1}_{{\rm atb},\,\rho}\,(\mu)}
\ls\|f\|^p_{\widehat H^{p,\,q,\,\gamma_1}_{{\rm atb},\,\rho}\,(\mu)},$$
which implies that
$f\in \widehat H^{p,\,q,\,\gamma_2}_{{\rm atb},\,\rho}\,(\mu)$ and
$$
\|f\|_{\widehat H^{p,\,q,\,\gamma_2}_{{\rm atb},\,\rho}\,(\mu)}
\ls\|f\|_{\widehat H^{p,\,q,\,\gamma_1}_{{\rm atb},\,\rho}\,(\mu)}.
$$
This finishes the proof of Proposition \ref{p8.2}.
\end{proof}

Now we are ready to show that $\cer$ is the dual space of $\hhp$.

\begin{theorem}\label{t8.3}
Let $p\in (0,1]$, $\rho\in(1,\fz)$, $\gz\in[1,\fz)$ and
$q\in(1, \fz)$. Then
$$
\cer=(\hhp)^\ast.
$$
\end{theorem}

\begin{proof}
When $p=1$, by \cite[Remark 2.6(iii)]{fyy3},
the conclusion of Theorem \ref{t8.3} holds true without the assumption
\eqref{6.1}.
Thus, it remains to consider the case when $p\in(0, 1)$.
Let $p\in (0,1)$, $\rho\in(1,\fz)$, $\gz\in[1,\fz)$
and $q\in(1, \fz)$.
We first show that $\cer\st(\hhp)^\ast$.
To this end, let $f\in\cer$.
Recall that any $h\in\hhp$ is, by Definition \ref{d7.3},
a continuous linear functional in $(\cer)^\ast$.
Let us write $\langle h, f\rangle$ to denote the value of
the linear functional $h$ at $f\in \cer$. Then the mapping
$\ell_f: h \to \langle h, f\rangle$ is a well defined
linear functional on $\hhp$. If $h =\sum_{i=1}^\fz b_i $ is
an atomic decomposition
of $h$ in terms of $(p,q,\gz,\rho)_{\lz,\,1}$-atomic
blocks $\{b_i\}_i$ such that
$\sum_{i=1}^\fz|b_i|^p_{\hhp}\ls \|h\|^p_{\hhp}$,
by \eqref{7.1}, we then have
\begin{align*}
\lf|\langle h, f\rangle\r|&=\lf|\sum_{i=1}^\fz
\dint_\cx b_i(x)f(x)\,d\mu(x)\r|\\
&\ls\lf[\sum_{i=1}^\fz|b_i|^p_{\hhp}\r]^{1/p}\|f\|_{\cer}
\ls \|h\|_{\hhp}\|f\|_\cer.
\end{align*}
Therefore, we conclude that
$$|\ell_f(h)|\ls \|f\|_\cer\|h\|_{\hhp}.$$
This shows that  $\cer\st(\hhp)^\ast$.

To see the converse, for any $q\in(1,\fz)$,
we first claim that, if $\ell\in(\hhp)^\ast$, then
there exists a function $f\in L^{q'}_{\rm loc}(\mu)$
such that, for all $g\in \cup_B L^q_0(B)$,
$$\ell(g)=\int_\cx f(x)g(x)\,d\mu(x),$$
where, for all balls $B\st \cx$, $L^q_0(B)$ denotes
the \emph{subspace} of $L^q(B)$ consisting
of functions having integral zero. Indeed,
let $\{B_k\}_k$ be an increasing sequence of
balls which exhausts $\cx$. For each $k$, let $\cc(B_k)$ denote
the space of functions those are constants on $B_k$. Suppose
that $q\in (1, \fz)$ and $\ell\in (\hhp)^\ast$.
Then $\ell\in (L^q_0(B_k))^\ast=L^{q'}(B_k)/\cc(B_k)$.
Indeed, if $g\in L^q_0(B_k)$, then
$g\in\hhp$ and
$$\|g\|_{\hhp}\le [\mu(\rho B_k)]^{1-1/q}
\lf[\lz\lf(c_{B_k}, r_{B_k}\r)\r]^{1/p-1}\|g\|_{L^q(\mu)}.$$
We further see that
\begin{align*}
|\ell(g)|&\le \|\ell\|_{(\hhp)^\ast}\|g\|_{\hhp}\\
&\le\|\ell\|_{(\hhp)^\ast}[\mu(\rho B_k)]^{1-1/q}
[\lz(c_{B_k}, r_{B_k})]^{1/p-1}\|g\|_{L^q(\mu)}.
\end{align*}
Hence, by the Riesz representation theorem,
there exists a unique $f_k\in L^{q'}(B_k)/\cc(B_k)$ such that,
for every $g\in L^q_0(B_k)$,
$$\ell(g)=\int_{B_k}f_k(x) g(x)\,d\mu(x).$$
Since $\{B_k\}_k$ is increasing, by a standard argument, we see that
there exists a unique function $f\in L^{q'}_{\rm loc}(\mu)$
such that, for all $g\in \cup_B L^q_0(B)$,
$$\ell(g)=\int_\cx f(x)g(x)\,d\mu(x).$$
This proves the claim.

We now show that, if $f\in L^{q'}_{\rm loc}(\mu)$
such that $\ell\in(\hhp)^\ast$,
then $f\in\cer$ and
$$\|f\|_\cer\ls \|\ell\|_{(\hhp)^\ast}.$$
To this end, by Proposition \ref{p6.18},
it suffices to show that, for any
$(\rho,\bz_{\rho})$-doubling ball $B$,
\begin{equation}\label{8.11}
\frac1{\mu(\rho B)}\frac1{[\lz(c_B,\,r_B)]^{1/p-1}}
\int_B|f(x)-m_f(B)|
\,d\mu(x)\ls \|\ell\|_{(\hhp)^\ast}
\end{equation}
and, for all $(\rho ,\bz_\rho )$-doubling balls $B\st S$,
\begin{equation}\label{8.12}
|m_f(B)-m_f(S)|
\ls  \|\ell\|_{(\hhp)^\ast}\lf[\kbsp\r]^{\gz}[\lz(c_S,\,r_S)]^{1/p-1}.
\end{equation}
We first prove \eqref{8.11}. Let $B$ be a $(\rho,\bz_{\rho})$-doubling
ball. Assume that
\begin{align}\label{8.13}
\noz&\int_{\{x\in B:\ f(x)>m_f(B)\}}|f(x)-m_f(B)|^{q'}\,d\mu(x)\\
&\hs\ge \int_{\{x\in B:\ f(x)<m_f(B)\}}|f(x)-m_f(B)|^{q'}\,d\mu(x).
\end{align}
Consider
$$a(x):=\left\{
  \begin{array}{ll}
    |f(x)-m_f(B)|^{q'-1}, \quad & \hbox{if $x\in
    \{x\in B:\ f(x)> m_f(B)\}$;} \\
    \wz C_B, \quad & \hbox{if $x\in \{x\in B:\ f(x)< m_f(B)\}$;} \\
    0, \quad & \hbox{otherwise,}
  \end{array}
\right.
$$
where $\wz C_B$ denotes the constant such that $\int_\cx a(x)\,d\mu(x)=0$.
By the definition of $m_f(B)$, we have
\begin{equation}\label{8.14}
\mu(\{x\in B:\ f(x)>m_f(B)\})\le \mu(B)/2
\le \mu(\{x\in B:\ f(x)\le m_f(B)\}).
\end{equation}
From this fact, we deduce that $\supp (a)\st\sqrt{\rho}B$,
$a$ is a $(p,q,\gz,\sqrt{\rho})_{\lz}$-atomic block and
\begin{align*}
\|a\|_{{\wz H}^{p,\,q,\,\gz}_{{\rm atb},\,\sqrt{\rho}}(\mu)}&
\le \|a\|_{L^q(\mu)}[\mu(\sqrt{\rho}\times\sqrt{\rho} B)]^{1/q'}
[\lz(c_B, r_B)]^{1/p-1}\\
&\le [\mu(\rho B)]^{1/q'}[\lz(c_B, r_B)]^{1/p-1}
\lf[\int_{\{x\in B:\ f(x)>m_f(B)\}}|f(x)-m_f(B)|^{q'}\,d\mu(x)\r.\\
&\quad\lf.+\int_{\{x\in B:\ f(x)\le m_f(B)\}}
\lf|\wz C_B\r|^{q}\,d\mu(x)\r]^{1/q}.
\end{align*}
By \eqref{8.14}, the definition of $\wz C_B$
and the H\"older inequality, we have
\begin{align*}
&\int_{\{x\in B:\ f(x)\le m_f(B)\}}\lf|\wz C_B\r|^{q}\,d\mu(x)\\
&\hs=\lf|\int_{\{x\in B:\ f(x)\le m_f(B)\}}\wz C_B\,d\mu(x)\r|^q
\lf[\mu(\{x\in B:\ f(x)\le m_f(B)\})\r]^{1-q}\\
&\hs\ls\lf|\int_{\{x\in B:\ f(x)>m_f(B)\}}|f(x)-m_f(B)|^{q'-1}\,d\mu(x)\r|^q
[\mu(B)]^{1-q}\\
&\hs\ls\int_{\{x\in B:\ f(x)>m_f(B)\}}|f(x)-m_f(B)|^{q'}\,d\mu(x).
\end{align*}
From this, $\supp (a)\st\sqrt{\rho}B$ and Proposition \ref{p8.1},
it follows that
\begin{align}\label{8.15}
\noz\|a\|_{\hhp}&\sim\|a\|_{{\widehat H}^{p,\,q,\,\gz}
_{{\rm atb},\,\sqrt{\rho}}(\mu)}\\
&\ls [\mu(\rho B)]^{1/q'}[\lz(c_B, r_B)]^{1/p-1}
\lf[\int_{\{x\in B:\ f(x)>m_f(B)\}}
|f(x)-m_f(B)|^{q'}\,d\mu(x)\r]^{1/q}.
\end{align}
On the other hand, by the definition of $a$ and \eqref{8.13}, we see that
\begin{align*}
\int_B f(x)a(x)\,d\mu(x)&=\int_B[f(x)-m_f(B)]a(x)\,d\mu(x)\\
&\ge \int_{\{x\in B:\ f(x)>m_f(B)\}}|f(x)-m_f(B)|^{q'}\,d\mu(x)\\
&\ge \frac12\int_{B}|f(x)-m_f(B)|^{q'}\,d\mu(x),
\end{align*}
which, together with \eqref{8.15}, implies that
\begin{align*}
&\lf[\int_B|f(x)-m_f(B)|^{q'}\,d\mu(x)\r]^{1/q'}\|a\|_{\hhp}\\
&\quad\ls [\mu(\rho B)]^{1/q'}[\lz(c_B, r_B)]^{1/p-1}
\int_B|f(x)-m_f(B)|^{q'}\,d\mu(x)\\
&\quad\ls [\mu(\rho B)]^{1/q'}[\lz(c_B, r_B)]^{1/p-1}
\int_B f(x)a(x)\,d\mu(x)\\
&\quad\ls [\mu(\rho B)]^{1/q'}[\lz(c_B, r_B)]^{1/p-1}
\|\ell\|_{(\hhp)^\ast}\|a\|_{\hhp}.
\end{align*}
From this and the H\"older inequality,
it then follows that
\begin{align*}
&\frac1{\mu(\rho B)}\frac1{[\lz(c_B,\,r_B)]^{1/p-1}}
\int_B|f(x)-m_f(B)|\,d\mu(x)\\
&\quad\le \frac{[\mu(\rho B)]^{-1/q'}}{[\lz(c_B,\,r_B)]^{1/p-1}}
\lf[\int_B|f(x)-m_f(B)|^{q'}
\,d\mu(x)\r]^{1/q'}\ls \|\ell\|_{(\hhp)^\ast}.
\end{align*}
Thus, \eqref{8.11} holds true.

To show \eqref{8.12}, for all $(\rho ,\bz_\rho )$-doubling balls $B\st S$, let
$$a_1:=\frac{|f-m_f(S)|^{q'}}{f-m_f(S)}\chi_{\{x\in S:\ f(x)\not=m_f(S)\}}$$
and $a_2:=\wz C_S\chi_S$, where $\wz C_S$ denotes the constant such that
$\int_\cx[a_1(x)+a_2(x)]\,d\mu(x)=0$.
Observe that
\begin{equation}\label{8.16}
\lf|\wz C_S\r|\le [\mu(S)]^{-1}[\mu(B)]^{1-1/q}
\lf[\int_B|f(x)-m_f(S)|^{q'}\,d\mu(x)\r]^{1/q}.
\end{equation}
From this, together with the fact that $B$ and $S$
are $(\rho,\bz_{\rho})$-doubling and Proposition \ref{p8.1},
it follows that $\supp (b)\st(2\sqrt{\rho}+1)S$,
$$b:=\lz_1\wz a_1+\lz_2\wz a_2\in
{\widehat H}^{p,\,q,\,\gz}_{{\rm atb},\,\sqrt{\rho}}(\mu)
\subset {\widehat H}^{p,\,q,\,\gz}_{{\rm atb},\,\rho}(\mu)$$
and
\begin{align}\label{8.17}
\noz\|b\|_{\hhp}\sim\|b\|_{{\widehat H}^{p,\,q,\,\gz}
_{{\rm atb},\,\sqrt{\rho}}(\mu)}
&\ls [\mu(\sqrt{\rho}\times\sqrt{\rho}B)]^{1-1/q}
\lf[\lz\lf(c_S,\,r_{(2\sqrt{\rho}+1)S}\r)\r]^{1/p-1}\\
&\hs\times\lf[\widetilde{K}^{(\rho),\,p}
_{\sqrt{\rho}B,\,(2\sqrt{\rho}+1)S}\r]^{\gz}
\lf[\int_B|f(x)-m_f(S)|^{q'}\,d\mu(x)\r]^{1/q},
\end{align}
where
\begin{align*}
\wz a_1:=&a_1[\mu(\rho B)]^{1/q-1}\lf[\int_B|f(x)-m_f(S)|^{q'}\,d\mu(x)\r]^{-1/q}\\
&\times\lf[\widetilde{K}^{(\rho),\,p}
_{\sqrt{\rho}B,\,(2\sqrt{\rho}+1)S}\r]^{-\gz}
\lf[\lz(c_S,\,r_{(2\sqrt{\rho}+1)S})\r]^{1-1/p},
\end{align*}
$$\wz a_2:=[\mu(\rho S)]^{-1}\lf[\lz\lf(c_S, r_{(2\sqrt{\rho}+1)S}\r)
\r]^{1-1/p}\lf[\widetilde{K}^{(\rho),\,p}
_{\sqrt{\rho}S,\,(2\sqrt{\rho}+1)S}\r]^{-\gz}\chi_S,$$
\begin{align*}
\lz_1&:=\lf[\int_B|f(x)-m_f(S)|^{q'}\,d\mu(x)\r]^{1/q}
\lf[\widetilde{K}^{(\rho),\,p}
_{\sqrt{\rho}B,\,(2\sqrt{\rho}+1)S}\r]^{\gz}[\mu(\rho B)]^{1-1/q}\\
&\hs\times\lf[\lz\lf(c_S,\,r_{(2\sqrt{\rho}+1)S}\r)\r]^{1/p-1}
\end{align*}
and
$$\lz_2:=\wz C_S\mu(\rho S)\lf[\lz(x_S, r_{(2\sqrt{\rho}+1)S})\r]^{1/p-1}
\lf[\widetilde{K}^{(\rho),\,p}
_{\sqrt{\rho}S,\,(2\sqrt{\rho}+1)S}\r]^{\gz}.$$
Then $\supp (\wz a_1)\st\sqrt{\rho}B\st(2\sqrt{\rho}+1)S$
and $\supp (\wz a_2)\st\sqrt{\rho} S\st(2\sqrt{\rho}+1)S$.
By the definition of $a_1$, the $(\rho,\bz_{\rho})$-doubling
property of $B$ and $S$, the vanishing moment of $b$,
\eqref{8.16}, \eqref{8.11} and \eqref{8.17}, we see that
\begin{align*}
&\int_B|f(x)-m_f(S)|^{q'}\,d\mu(x)\\
&\quad=\int_B a_1(x)[f(x)-m_f(S)]\,d\mu(x)\\
&\quad\le \lf|\int_\cx f(x)b(x)\,d\mu(x)\r|
+\lf|\wz C_S\r|\int_S|f(x)-m_f(S)|\,d\mu(x)\\
&\quad\ls \|\ell\|_{(\hhp)^\ast}\|b\|_{\hhp}
+\lf|\wz C_S\r|\|\ell\|_{(\hhp)^\ast}\mu(\rho S)[\lz(c_S,r_S)]^{1/p-1}\\
&\quad\ls\|\ell\|_{(\hhp)^\ast}\lf[\int_B|f(x)-m_f(S)|^{q'}
\,d\mu(x)\r]^{1/q}\lf[\kbsp\r]^{\gz}
[\mu(B)]^{1-1/q}[\lz(c_S,\,r_S)]^{1/p-1}.
\end{align*}
This implies that
$$\lf[\int_B|f(x)-m_f(S)|^{q'}\,d\mu(x)\r]^{1/q'}
\ls \|\ell\|_{(\hhp)^\ast}\lf[\kbsp\r]^{\gz}
[\mu(B)]^{1-1/q}[\lz(c_S,\,r_S)]^{1/p-1}.$$
Thus, from this, \eqref{8.11}, the $(\rho, \bz_{\rho})$-doubling
property of $B$ and $S$, and the H\"older inequality, it follows that
\begin{align*}
|m_f(B)-m_f(S)|&=\frac1{\mu(B)}\int_B|m_f(B)-m_f(S)|\,d\mu(x)\\
&\le \frac1{\mu(B)}\int_{B}|f(x)-m_f(B)|\,d\mu(x)
+\frac1{\mu(B)}\int_{B}|f(x)-m_f(S)|\,d\mu(x)\\
&\le \lf[\frac1{\mu(B)}\int_{B}|f(x)-m_f(B)|^{q'}\,d\mu(x)\r]^{1/q'}
+\lf[\frac1{\mu(B)}\int_{B}|f(x)-m_f(S)|^{q'}\,d\mu(x)\r]^{1/q'}\\
&\ls \|\ell\|_{(\hhp)^\ast}[\lz(c_B,\,r_B)]^{1/p-1}
+[\lz(c_S,\,r_S)]^{1/p-1}\lf[\kbsp\r]^{\gz}\|\ell\|_{(\hhp)^\ast}\\
&\ls [\lz(c_S,\,r_S)]^{1/p-1}\lf[\kbsp\r]^{\gz}\|\ell\|_{(\hhp)^\ast},
\end{align*}
which implies \eqref{8.12}, and hence completes the
proof of Theorem \ref{t8.3}.
\end{proof}

\begin{remark}\label{r8.4}
It is still unclear whether Theorem \ref{t8.3}
holds true or not for $q=1$ and $p\in(0,1)$, or
$q=\fz$ and $p\in(0,1]$ on non-homogeneous metric measure
spaces satisfying the $\rho$-weakly doubling condition \eqref{6.1}.
\end{remark}

\section{Relations between ${\mathcal E}^{\az}_\rho(\mu)$
and $\lip_{\az,\,q}(\mu)$ or
between $\widehat H^{p,\,q}_{\rm atb}(\mu)$ and
$H^{p,\,q}_{\rm at}(\mu)$}\label{s9}

\hskip\parindent In this section, we investigate the relations
between $\ceaeg$ and $\lip_{\az,\,q}(\mu)$, and between
$\hhp$ and the atomic Hardy space $H^{p,\,q}_{\rm at}(\mu)$
introduced by Coifman and Weiss  \cite{cw77}
over spaces of homogeneous type.

Let $({\mathcal X},d,\mu)$ be a space of homogeneous type with
$\lambda(x,r):=\mu(B(x,r))$
for all $x\in\cx$ and $r\in(0,\fz)$. Recall that \eqref{6.1} holds true
in spaces of homogeneous type.
Thus, all the results obtained in Sections \ref{s6},
\ref{s7} and \ref{s8} are still
valid in this setting and we denote $\ceaeg$ simply by $\cera$.
We first establish an equivalent characterization of $\cera$.
To this end, we recall the notions of spaces $\lip_{\az,\,q}(\mu)$
and $\lip_\az(\mu)$ in \cite{ms1}.
To be precise, let $\alpha\in[0, \fz)$, $q\in[1,\fz)$, $\rho\in(1,\fz)$
and $\delta$ be a quasi-distance on $\cx$.
A function $\phi$ is said to be in the \emph {Lipschitz space}
$\lip_{\az,\,q}(\mu;\dz)$ if
\begin{align*}
\|\phi\|^{(\delta)}_{\alpha,\,q}
:=\sup_{B_{\dz}}\lf\{\frac1{[\mu(B_{\dz})]^{1+q\alpha}}
\int_{B_{\dz}}\lf|f(y)-m_{B_{\dz}}(f)\r|^q\,d\mu(y)\r\}^{1/q}<\fz,
\end{align*}
where the supremum is taken over all balls $B_{\dz}$ from $(\cx,\dz,\mu)$,
and a function $\psi$ is said to be in the \emph{space}
$\lip_\az(\mu;\dz)$ if
\begin{equation}\label{9.1}
\|\psi\|^{(\delta)}_{\alpha}:=\sup_{x\not= y}
\frac{|\psi(x)-\psi(y)|}{[\delta(x, y)]^\alpha}<\fz.
\end{equation}
Then we let $\lip_{\az,\,q}(\mu):=\lip_{\az,\,q}(\mu;d)$
and $\lip_\az(\mu):=\lip_\az(\mu;d)$, respectively.

\begin{remark}\label{r9.1}
By \cite[Theorem 5]{ms1}, we see that, for any $\alpha\in(0, \fz)$,
there exists a quasi-distance $\delta$ on
$\cx$, defined by setting, for all $x,\,y\in\cx$,
$$
\dz(x,y):=\inf\{\mu(B):\ B\ {\rm is\ a\ ball\ containing}\ x\ {\rm and}\ y\}
$$
such that $(\cx, \delta, \mu)$
is a normal space. Namely, there exist
two positive constants, $c_9$ and $c_{10}$, such that
$c_9 r\le\mu(B_{\dz}(x,r))\le c_{10} r$ for every
$x\in\cx$, $r\in(\mu({x}),\mu(\cx))$
and
$$
B_{\dz}(x,r):=\{x\in\cx:\ \dz(x,y)<r\}
$$
and, for any $q\in[1, \fz)$, a function $\phi$ is in the
Lipschitz space $\lip_{\az,\,q}(\mu)$ of $(\cx, d, \mu)$
if and only if there exists a function $\psi$ in the
space $\lip_\az(\mu;\dz)$ of $(\cx, \delta, \mu)$ such that
$\phi=\psi$ for $\mu$-almost every $x\in\cx$.
Moreover, $\|\phi\|^{(d)}_{\alpha,\,q}\sim \|\psi\|^{(\delta)}_{\alpha}$.
\end{remark}

Hereafter, we \emph{always let $\dz$ be as in Remark \ref{r9.1}}.

Now we discuss the relation between $\cera$ and
$\lip_{\az,\,q}(\mu)$.

\begin{proposition}\label{p9.2}
Suppose that $(\cx,d,\mu)$ is a space of homogeneous type,
$\az\in[0,\fz)$, $\rho\in(1,\fz)$ and $q\in[1,\fz)$.
Then $\cera$ and $\lip_{\az,\,q}(\mu)$ coincide with equivalent norms.
\end{proposition}

\begin{proof}
Fix $\az\in[0,\fz)$ and $q\in[1,\fz)$.
By Proposition \ref{p6.8}(ii), without loss of generality,
we may assume that $\rho=2$. By Definition \ref{d6.5},
we know that $\cera\st\lip_{\az,\,q}(\mu)$ and, for all $f\in\cera$,
$\|f\|_{\az,\,q}^{(d)}\le\|f\|_{\cera}$.

Conversely, by Definition \ref{d6.5},
it suffices to prove that, for all $f\in\lip_{\az,\,q}(\mu)$
and balls $B\st S$,
\begin{equation}\label{9.2}
|m_B(f)-m_S(f)|
\ls\|f\|_{\az,\,q}^{(d)}\kbsp[\mu(S)]^{\az}.
\end{equation}
To this end, we consider the following two cases.

\textbf{Case (I)} $\mu(S)\le4C_{(\mu)}\mu(B)$,
where $C_{(\mu)}$ is as in \eqref{1.1}.
Thus, by this and the H\"older inequality, we have
\begin{align*}
|m_B(f)-m_S(f)|&\le\frac1{\mu(B)}\int_B|f(x)-m_S(f)|\,d\mu(x)
\ls\frac1{\mu(S)}\int_S|f(x)-m_S(f)|\,d\mu(x)\\
&\ls\frac1{[\mu(S)]^{1/q}}\lf[\int_S|f(x)-m_S(f)|^q\,d\mu(x)\r]^{1/q}
\ls[\mu(S)]^{\az}\|f\|_{\az,\,q}^{(d)},
\end{align*}
which implies \eqref{9.2} in \textbf{Case (I)}.

\textbf{Case (II)}
\begin{equation}\label{9.3}
\mu(S)>4C_{(\mu)}\mu(B).
\end{equation}
Now we show \eqref{9.2}. Let $N$ be the smallest integer
such that $2^Nr_B\ge r_S$. Let $B_*:=2^{N+1}B$. Then $S\st B_*\st 6S$,
which implies that
\begin{equation}\label{9.4}
\mu(S)\le\mu(B_*)\le\mu(6S)\le \lf[C_{(\mu)}\r]^3\mu(S).
\end{equation}
Furthermore, let $B^{(0)}:=B$. By \eqref{9.3} and \eqref{9.4}, we see that
$$
\mu\lf(2^{N+1}B\r)=\mu(B_*)\ge\mu(S)>4C_{(\mu)}\mu(B)>2\mu\lf(B^{(0)}\r).
$$
Thus, let $B^{(1)}:=2^{N_1}B^{(0)}$, $N_1\in\nn$, be the smallest ball
in the form of $2^kB^{(0)}$ ($k\in\nn$)
such that $\mu(2^kB^{(0)})>2\mu(B^{(0)})$.
Moreover, by \eqref{9.3} and \eqref{9.4}, we know that
$r_{B^{(1)}}\le r_{B_*}$ and
$$
\mu\lf(B^{(1)}\r)\le C_{(\mu)}\mu\lf(2^{-1}B^{(1)}\r)
\le2C_{(\mu)}\mu\lf(B^{(0)}\r).
$$
We further consider the following two subcases.

\textbf{Subcase i)}. There exists $k\in\nn$ such that
$\mu(2^kB^{(1)})>2\mu(B^{(1)})$.
In this case, we let $2^{N_2}B^{(1)}$, $N_2\in\nn$,
be the smallest ball
in the form of $2^kB^{(1)}$, $k\in\nn$, such that
$\mu(2^kB^{(1)})>2\mu(B^{(1)})$.
Now we divide this subcase into two parts:

(a) $r_{2^{N_2}B^{(1)}}\le r_{B_*}$. Let $B^{(2)}:=2^{N_2}B^{(1)}$. Then
$\mu(B^{(2)})\le C_{(\mu)}\mu(2^{-1}B^{(2)})\le2C_{(\mu)}\mu(B^{(1)})$;

(b) $r_{2^{N_2}B^{(1)}}>r_{B_*}$. Let $B^{(2)}:=B_*$.
Then $\mu(B^{(2)})\le2\mu(B^{(1)})$,
where we terminate the construction in this subcase.

\textbf{Subcase ii)}. For any $k\in\nn$,
$\mu(2^kB^{(1)})\le2\mu(B^{(1)})$. Let
$B^{(2)}:=B_*$. Then $\mu(B^{(2)})\le2\mu(B^{(1)})$ and we terminate the
construction in this subcase.

We continue to choose the balls $\{B^{(i)}\}_i$ in this way.
Clearly, finally the condition
$r_{2^{N_{i+1}}B^{(i)}}\le r_{B_*}$ ($i\in\nn$) is violated
after finitely many steps. Without loss of generality, we may assume
that the process stops after $m$ ($m\in\nn\cap(1,\fz)$) steps.
Then we obtain a sequence of balls, $\{B^{(i)}\}_{i=0}^m$, such that

(i) $B=:B^{(0)}\st\cdots\st B^{(m)}:=B_*$;

(ii) for any $i\in\{1,\ldots,m-1\}$,
$2\mu(B^{(i-1)})<\mu(B^{(i)})\le2C_{(\mu)}\mu(B^{(i-1)})$;

(iii) $\mu(B^{(m)})\le2C_{(\mu)}\mu(B^{(m-1)})$ and
$\mu(S)\le\mu(B^{(m)})\le [C_{(\mu)}]^3\mu(S)$.

Observe that $m\le N+1$.
Thus, by the fact that $S\st B_*\st 6S$, the H\"older inequality,
(i), (ii) and (iii), we have
\begin{align*}
&|m_B(f)-m_S(f)|\\
&\hs\le\sum_{i=1}^m\lf|m_{B^{(i-1)}}(f)-m_{B^{(i)}}(f)\r|
+\lf|m_{B^{(m)}}(f)-m_S(f)\r|\\
&\hs\le\sum_{i=1}^m\frac1{\mu(B^{(i-1)})}\int_{B^{(i-1)}}
\lf|f(x)-m_{B^{(i)}}(f)\r|\,d\mu(x)
+\frac1{\mu(S)}\int_{S}\lf|f(x)-m_{B^{(m)}}(f)\r|\,d\mu(x)\\
&\hs\ls\sum_{i=1}^m\frac1{\mu(B^{(i)})}
\int_{B^{(i)}}\lf|f(x)-m_{B^{(i)}}(f)\r|\,d\mu(x)
+\frac1{\mu(B^{(m)})}\int_{B^{(m)}}\lf|f(x)-m_{B^{(m)}}(f)\r|\,d\mu(x)\\
&\hs\ls\sum_{i=1}^m\lf[\frac1{\mu(B^{(i)})}
\int_{B^{(i)}}\lf|f(x)-m_{B^{(i)}}(f)\r|^q\,d\mu(x)\r]^{1/q}
+\lf[\frac1{\mu(B^{(m)})}\int_{B^{(m)}}
\lf|f(x)-m_{B^{(m)}}(f)\r|^q\,d\mu(x)\r]^{1/q}\\
&\hs\ls\sum_{i=1}^m\lf[\mu\lf(B^{(i)}\r)\r]^{\az}\|f\|_{\az,\,q}^{(d)}
+\lf[\mu\lf(B^{(m)}\r)\r]^{\az}\|f\|_{\az,\,q}^{(d)}\\
&\hs\ls(1+m)\lf[\mu\lf(B^{(m)}\r)\r]^{\az}\|f\|_{\az,\,q}^{(d)}
\ls(1+m)[\mu(6S)]^{\az}\|f\|_{\az,\,q}^{(d)}\\
&\hs\ls(1+N)[\mu(6S)]^{\az}\|f\|_{\az,\,q}^{(d)}
\sim\lf[\kbsp\r]^p[\mu(S)]^{\az}\|f\|_{\az,\,q}^{(d)}.
\end{align*}
This finishes the proof of \eqref{9.2}
in \textbf{Case (II)} and hence Proposition \ref{p9.2}.
\end{proof}

\begin{remark}\label{r9.3}
When $\az=0$, Proposition \ref{p9.2} is just \cite[Proposition 4.7]{h10}
with $\lz(x,r)=\mu(B(x,r))$ for all $x\in\cx$ and $r\in(0,\fz)$.
\end{remark}

Now we recall the notion of the atomic Hardy space $H^{p,\,q}_{\rm at}(\mu)$
from \cite{cw77}.
Suppose that $p\in(0, 1]$ and $q\in[1, \fz]\cap(p, \fz]$. A function $a$ on
$\cx$ is called  a {\it $(p, q)$-atom} if

(i) $\supp(a)\subset B$ for some ball $B\subset\cx$;

(ii) $\|a\|_{L^q(\mu)}\le [\mu(B)]^{1/q-1/p}$;

(iii) $\int_\cx a(x)\, d\mu(x)=0$.

A function $f\in L^1(\mu)$ or a linear functional
$f\in(\lip_{1/p-1}(\mu))^*$
when $p\in(0,1)$ is said to be in the {\it Hardy space
$H_{\rm at}^{1,\,q}(\mu)$}
when $p=1$ or $H_{\rm at}^{p,\,q}(\mu)$ when $p\in(0,1)$ if there exist
$(p,\,q)$-atoms $\{a_j\}_{j=1}^\fz$ and
$\{\lz_j\}_{j=1}^\fz\subset\cc$ such that
$$
f=\sum_{j\in\nn}\lz_j a_j,
$$
which converges in $L^1(\mu)$ when $p=1$ or in $(\lip_{1/p-1}(\mu))^*$
when $p\in(0,1)$, and
$$\sum_{j\in\nn}|\lz_j|^p<\fz.$$
Moreover, the {\it norm} of $f$ in $H_{\rm at}^{p,\,q}(\mu)$
with $p\in(0,1]$ and $q\in[1, \fz]\cap(p, \fz]$ is defined by
$$
\|f\|_{H_{\rm at}^{p,\,q}(\mu)}
:=\inf\lf\{\lf(\sum_{j\in\nn}|\lz_j|^p\r)^{1/p}\r\},
$$
where the infimum is taken over all possible decompositions of
$f$ as above.

Coifman and Weiss \cite{cw77} proved that $H_{\rm at}^{p,\,q}(\mu)$ and
$H^{p,\,\fz}_{\rm at}(\mu)$ coincide with equivalent norms for all $p\in(0, 1]$
and $q\in[1, \fz)\cap(p, \fz)$. Thus, we denote $H^{p,\,q}_{\rm at}(\mu)$
simply by $H^p_{\rm at}(\mu)$.

Let $p\in(0,1]$, $q\in(1,\fz]$,
$\gz\in[1,\fz)$ and $\rho\in(1,\fz)$.
Recall that the space $\hhp$ is independent of the choices of
$\gz\in[1,\,\fz)$ and $\rho\in(1,\,\fz)$;
see Propositions \ref{p8.1} and \ref{p8.2}.
Denote \emph{$\hhp$ simply by $\widehat H^{p,\,q}_{\rm atb}(\mu)$}. Moreover,
without loss of generality, we may let $\gz=1/p$ and $\rho=2$.

Now we show that $\widehat H^{p,\,q}_{\rm atb}(\mu)$
and $H^p_{\rm at}(\mu)$ coincide
with equivalent quasi-norms.

\begin{theorem}\label{t9.4}
Let $(\cx,d,\mu)$ be a space of homogeneous type,
$p\in(0,\,1]$ and $q\in(1,\fz]$.
Then the spaces $\widehat H^{p,\,q}_{\rm atb}(\mu)$
and $H^p_{\rm at}(\mu)$ coincide with equivalent quasi-norms.
\end{theorem}

\begin{proof}
Let $p\in(0,1]$ and $q\in(1,\fz]$. We first show
that $H^p_{\rm at}(\mu)\st \hhpq$.
To this end, let $f\in H^p_{\rm at}(\mu)$. Then there exist sequences
of $(p,q)$-atoms, $\{b_k\}_{k=1}^\fz$, and
numbers, $\{\lz_k\}_{k=1}^\fz\st\cc$, such that
$f=\sum_{k=1}^\fz\lz_k b_k$ in $(\lip_{1/p-1}(\mu;\dz))^{\ast}$
and
\begin{equation}\label{9.5}
\sum_{k=1}^\fz|\lz_k|^p\ls\|f\|^p_{H^p_{\rm at}(\mu)}.
\end{equation}
We then claim that, for each $k$, $\lz_k b_k$ is a
$(p,q,\frac1p,2)_{\lz,\,1}$-atomic block
and
\begin{equation}\label{9.6}
|\lz_kb_k|_{\widehat H^{p,\,q}_{\rm atb}(\mu)}\ls |\lz_k|.
\end{equation}
Indeed, let $\rho\in(1,\fz)$ and $\gz\in[1,\fz)$.
It suffices to prove that, if $b$ is a $(p, q)$-atom,
then $b$ is a $(p,q,1/p,2)_{\lz,\,1}$-atomic block.
Suppose that $\supp(b)\subset B(c_B,\,r_B)=: B$, then
$$
\|b\|_{L^q(\mu)}\le [\mu(B)]^{1/q-1/p}.
$$
Let
$$
a_1=a_2:=\lf[\frac{\mu(\rho B)}{\mu(B)}\r]^{1/q-1}
\lf[\wz K^{(\rho),\,p}_{B,\,B}\r]^{-1/p}b
$$
and
$$
\lz_1=\lz_2:=\frac{1}{2}\lf[\frac{\mu(B)}{\mu(\rho B)}\r]^{1/q-1}
\lf[\wz K^{(\rho),\,p}_{B,\,B}\r]^{1/p}.
$$
It then follows that $\supp(a_1)\subset B_1:=B$, $\supp(a_2)
\subset B_2:=B$, $b=\lz_1 a_1+\lz_2 a_2$ and, for $j\in\{1,2\}$,
$$
\|a_j\|_\lq\le [\mu(\rho B_j)]^{1/q-1}[\mu(B)]^{1-1/p}
\lf[\wz K^{(\rho),\,p}_{B_j,\,B}\r]^{-1/p},
$$
which further implies that $b$ is a $(p,q,1/p,2)_{\lz,\,1}$-atomic
block and, moreover,
\begin{equation*}
|b|_{\widehat H^{p,\,q,\,1/p}_{\rm atb,\,2}(\mu)}=|\lz_1|+|\lz_2|\ls1.
\end{equation*}
This finishes the proof of the above claim.
Moreover, from Remark \ref{r9.1} and Proposition \ref{p9.2},
we deduce that $f=\sum_{k=1}^\fz\lz_k b_k$ in $(\cer)^{\ast}$,
which, together with \eqref{9.5} and \eqref{9.6}, further
implies that  $f\in \widehat{H}^{p,\,q}_{\rm atb}(\mu)$ and
$\|f\|_{\widehat{H}^{p,\,q}_{\rm atb}\,(\mu)}\ls \|f\|_{H^p_{\rm at}(\mu)}$.

Now we consider the converse inclusion that $\hhpq\st H^p_{\rm at}(\mu)$.
Let $b=\sum_{j=1}^2\lz_j a_j$
be a $(p,q,1/p,2)_{\lz,\,1}$-atomic block, where, for any $j\in\{1,\,2\}$,
$a_j$ is a function supported on $B_j\subset B$ for some balls $B_j$ and
$B$ as in Definition \ref{d7.1}, and
\begin{equation}\label{9.7}
\|a_j\|_\lq\le [\mu(2 B_j)]^{1/q-1}[\mu(B)]^{1-1/p}
\lf[\wz K^{(2),\,p}_{B_j,\,B}\r]^{-1/p}.
\end{equation}
We consider the following four cases:

{\bf Case (I)} For any $j\in\{1,\,2\}$,
$\frac{\mu(B)}{\mu(B_j)}\le 4C_{(\mu)}$,
where $C_{(\mu)}$ is as in \eqref{1.1};

{\bf Case (II)} $\frac{\mu(B)}{\mu(B_1)}> 4C_{(\mu)}$
and $\frac{\mu(B)}{\mu(B_2)}\le 4C_{(\mu)}$;

{\bf Case (III)} $\frac{\mu(B)}{\mu(B_1)}\le 4C_{(\mu)}$
and $\frac{\mu(B)}{\mu(B_2)}> 4C_{(\mu)}$;

{\bf Case (IV)} For any $j\in\{1,\,2\}$,
$\frac{\mu(B)}{\mu(B_j)}> 4C_{(\mu)}$.

In {\bf Case (I)}, we see that
$$
\|b\|_{L^q(\mu)}\le|\lz_1|\|a_1\|_{L^q(\mu)}+|\lz_2|\|a_2\|_{L^q(\mu)}
\le\lf[4C_{(\mu)}\r]^{1-1/q}(|\lz_1|+|\lz_2|)[\mu(B)]^{1/q-1/p},
$$
which implies that
$\lf[4C_{(\mu)}\r]^{1/q-1}(|\lz_1|+|\lz_2|)^{-1}b$ is a $(p,\,q)$-atom and
$$
\lf\|\lf[4C_{(\mu)}\r]^{1/q-1}(|\lz_1|+|\lz_2|)^{-1}b\r\|
_{H^p_{\rm at}(\mu)}\le 1.
$$

The proofs of {\bf Case (II)}, {\bf Case (III)}
and {\bf Case(IV)} are similar.
For brevity, we only prove {\bf Case (II)}.

In {\bf Case (II)}, $\frac{\mu(B)}{\mu(B_1)}> 4C_{(\mu)}$ and
$\frac{\mu(B)}{\mu(B_2)}\le 4C_{(\mu)}$.
We now choose a sequence of balls, $\{B^{(i)}_1\}_{i=0}^m$
with certain $m\in\nn$, as follows. Let $B_1^{(0)}:=B_1$
and $B_0:=2^{N^{(2)}_{B_1,\,B}+1}B_1$. Then
$$
B\subset B_0\subset 6B,
$$
which, together with \eqref{1.1}, shows that
\begin{equation}\label{9.8}
\mu(B)\le\mu(B_0)\le\mu(6B)\le \lf[C_{(\mu)}\r]^3\mu(B).
\end{equation}

To choose $B^{(1)}_1$, let $N_1$ be the smallest
positive integer satisfying
$\mu(2^{N_1}B_1^{(0)})>2\mu(B_1^{(0)})$. We let $B_1^{(1)}:=2^{N_1}B_1^{(0)}$.
By \eqref{9.8}, \eqref{1.1} and the choice of $B_1^{(1)}$, we have
$r_{B_1^{(1)}}<r_{B_0}$ and
$$
2\mu\lf(B_1^{(0)}\r)<\mu\lf(B_1^{(1)}\r)
\le C_{(\mu)}\mu\lf(2^{-1}B_1^{(1)}\r)\le 2C_{(\mu)}\mu\lf(B_1^{(0)}\r).
$$

To choose $B^{(2)}_1$, if, for any $N\in\nn$,
$\mu(2^{N}B_1^{(1)})\le 2\mu(B_1^{(1)})$,
let $B^{(2)}_1:=B_0$
and the selection process terminates. Otherwise,
let $N_2$ be the smallest
positive integer satisfying
$\mu(2^{N_2}B_1^{(1)})>2\mu(B_1^{(1)})$. If
$r_{2^{N_2}B_1^{(1)}}\ge r_{B_0}$,
then we let $B^{(2)}_1:=B_0$ and the selection process terminates.
Otherwise, we let $B^{(2)}_1:=2^{N_2}B_1^{(1)}$.
We continue as long as this selection process is possible;
clearly, finally the condition
$r_{2^{N_{i+1}}B^{(i)}_1}< r_{B_0}$ is violated
after finitely many steps. Without loss of generality, we may assume
that the process stops after $m$
($m\in\nn\cap(1,N_{B_1,\,B}^{(2)}+1]$) steps.
Then we obtain a family of balls, $\{B^{(i)}_1\}_{i=1}^m$, such that

(i) $B_1^{(0)}:=B_1\subset B$,
$B^{(i)}_1:=2^{N_i}B^{(i-1)}_1\st B^{(m)}_1:=B_0\st 6B$
for any $i\in\{1,\,\ldots,\,m-1\}$;

(ii) for any $i\in\{1,\,\ldots,\,m-1\}$, by \eqref{1.1}
and the definition of $N_i$, we have
$$
2\mu\lf(B^{(i-1)}_1\r)<\mu\lf(B^{(i)}_1\r)\le 2C_{(\mu)}\mu\lf(B^{(i-1)}_1\r);
$$

(iii) $\mu(B^{(m)}_1)\le 2C_{(\mu)}\mu(B_1^{m-1})$
and $\mu(B)\le\mu(B^{(m)}_1)\le[C_{(\mu)}]^3\mu(B)$;

(iv) from the above selection process and the definition of
$\wz K^{(2),\,p}_{B_1,\,B}$, we conclude that
\begin{equation}\label{9.9}
m\le N^{(2)}_{B_1,\,B}+1\le \lf[1+N^{(2)}_{B_1,\,B}\r]^{1/p}
\le\wz K^{(2),\,p}_{B_1,\,B}=:\lf[\widehat{C}_b\r]^p.
\end{equation}
Let $\wz c_0:=\widehat{C}_b a_1$. For any $i\in\{1,\,\ldots,\,m\}$, let
$$
\wz c_i:=\dfrac{\chi_{B^{(i)}_1}}{\mu(B^{(i)}_1)}
\dint_{\cx}\wz c_{i-1}(y)\,d\mu(y).
$$
For $i\in\{1,\,\ldots,\,m\}$, we claim that
\begin{equation}\label{9.10}
\|\wz c_{i-1}\|_{L^q(\mu)}\ls
\lf[\mu(B_1^{i})\r]^{1/q-1/p}
\end{equation}
and
\begin{equation}\label{9.11}
\|\wz c_i\|_{L^q(\mu)}\ls
\lf[\mu(B_1^{i})\r]^{1/q-1/p},
\end{equation}
where the implicit positive constant is independent of $\widehat{C}_b$
and hence $B_1$, $B_2$ and $B$.
Indeed, we prove \eqref{9.10} and \eqref{9.11} by induction.
By \eqref{9.7}, $B_1^{(1)}\subsetneqq B_0\subset 6B$,
\eqref{1.1} and (ii), we have
\begin{equation*}
\|\wz c_0\|_{L^q(\mu)}\ls\lf[\mu\lf(2 B^{(0)}_1\r)\r]^{1/q-1}
[\mu(B)]^{1-1/p}
\lf[\wz K^{(2),\,p}_{B_1,\,B}\r]^{-1/p+1/p}
\ls\lf[\mu\lf(B_1^{(1)}\r)\r]^{1/q-1/p}.
\end{equation*}
For $i=1$, by the H\"older inequality, \eqref{9.7},
$B_1^{(1)}\subsetneqq B_0\subset 6B$, \eqref{1.1} and (ii),
we conclude that
\begin{align*}
\|\wz c_1\|_{L^q(\mu)}&\le\lf[\mu\lf(B_1^{(1)}\r)\r]^{1/q-1}
\lf[\mu\lf(B^{(0)}_1\r)\r]^{1-1/q}\lf\|\widehat{C}_b a_1\r\|_{L^q(\mu)}\\
&\ls\lf[\mu\lf(B_1^{(1)}\r)\r]^{1/q-1}\lf[\mu\lf(B^{(0)}_1\r)\r]^{1-1/q}
\lf[\mu\lf(2 B^{(0)}_1\r)\r]^{1/q-1}
[\mu(B)]^{1-1/p}
\lf[\wz K^{(2),\,p}_{B_1,\,B}\r]^{-1/p+1/p}\\
&\ls\lf[\mu\lf(B_1^{(1)}\r)\r]^{1/q-1/p}.
\end{align*}
Moreover, by (ii) (if $m=2$, we use (iii)), we have
\begin{equation*}
\|\wz c_1\|_{L^q(\mu)}\ls[\mu(B^{(2)}_1)]^{1/q-1/p}.
\end{equation*}
Now, we assume that \eqref{9.10} and \eqref{9.11}
hold true for $i\in\nn\cap[1,\,m)$.
It then follows, from the H\"older inequality,
\eqref{9.11} and (ii) (if $i+1=m$, we use (iii)), that
\begin{align*}
\|\wz c_{i+1}\|_{L^q(\mu)}&\ls\lf[\mu\lf(B_1^{i+1}\r)\r]^{1/q-1}
\lf[\mu\lf(B^{(i)}_1\r)\r]^{1-1/q}\|\wz c_i\|_{L^q(\mu)}\\
&\ls\lf[\mu\lf(B_1^{i+1}\r)\r]^{1/q-1}\lf[\mu\lf(B^{(i)}_1\r)\r]
^{1-1/q}\lf[\mu\lf(B^{(i)}_1\r)\r]^{1/q-1/p}
\ls\lf[\mu\lf(B_1^{i+1}\r)\r]^{1/q-1/p}
\end{align*}
and, moreover, if $i+1<m$, we further have
$$
\|\wz c_{i+1}\|_{L^q(\mu)}\ls\lf[\mu\lf(B_1^{i+2}\r)\r]^{1/q-1/p}.
$$
By induction, we conclude that \eqref{9.10} and \eqref{9.11} hold true.

For $i\in\{1,\,\ldots,\,m\}$, let
$c_i:=[\wz C]^{-1}(\wz c_{i-1}-\wz c_i)$, where
$\wz C$ is a positive constant to be fixed later.
Then $\supp(c_i)\subset2 B^{(i)}_1$ and
$
\int_\cx c_i(x)\,d\mu(x)=0,
$
which, together with \eqref{9.10}, \eqref{9.11} and \eqref{1.1},
implies that $c_i$ is a multiple of a $(p,\,q)$-atom
associated with the ball $2 B^{(i)}_1$ provided $\wz C$ large enough.
We now write
\begin{equation*}
b=\dfrac{\wz C\lz_1}{\widehat{C}_b}\dsum_{i=1}^m c_i
+\dfrac{\lz_1}{\widehat{C}_b}\wz c_m+\lz_2 a_2.
\end{equation*}
Let $c_{m+1}:=[\wz C(|\lz_1|+|\lz_2|)]^{-1}(\frac{\lz_1}
{\widehat{C}_b}\wz c_m+\lz_2 a_2)$.
Notice that $\int_\cx b(x)\,d\mu(x)=0$ and, for each
$i\in\{1,\,\ldots,m\}$, $\int_\cx c_i(x)\,d\mu(x)=0$.
It then follows that
$
\int_\cx c_{m+1}(x)\,d\mu(x)=0.
$
On the other hand, we have $\supp(\wz c_m)\subset
2 B^{(m)}_1\subset 12B=:B'$, $\mu(B')\sim\mu(B)$
and $\supp(a_2)\subset B_2\subset B'$.
By \eqref{9.11}, (iii) and $\mu(B')\sim\mu(B)$, we have
\begin{equation*}
\|\wz c_m\|_{L^q(\mu)}\ls\lf[\mu\lf(B^{(m)}_1\r)\r]
^{1/q-1/p}\ls\lf[\mu(B')\r]^{1/q-1/p}.
\end{equation*}
From \eqref{9.7}, $\frac{\mu(B)}{\mu(B_2)}\le 4C_{(\mu)}$
and $\mu(B')\sim\mu(B)$, it follows that
\begin{equation*}
\|a_2\|_\lq\ls [\mu(2 B_2)]^{1/q-1}[\mu(B)]^{1-1/p}
\lf[\wz K^{(2),\,p}_{B_2,\,B}\r]^{-1/p}\ls\lf[\mu\lf(B'\r)\r]^{1/q-1/p}.
\end{equation*}
Let $\wz C$ be a positive constant, which is independent of
$\widehat C_b$ and $m$, such that
$$\|c_{i}\|_{L^q(\mu)}\le \wz C\lf[\mu\lf(B_1^{i+2}\r)\r]^{1/q-1/p}$$
for each $i\in\{1,\,\ldots,m\}$, and
$$\lf\|\frac{\lz_1}{\widehat{C}_b}\wz c_m+\lz_2 a_2\r\|_{L^q(\mu)}
\le \wz C(|\lz_1|+|\lz_2|)\lf[\mu\lf(B'\r)\r]^{1/q-1/p}.$$
Then we see that $c_{m+1}$ is a $(p,\,q)$-atom associated with the ball $B'$.
From this and (iv), we conclude that
$$
b=\dfrac{\wz C\lz_1}{\widehat{C}_b}\dsum_{i=1}^m c_i
+\wz C(|\lz_1|+|\lz_2|)c_{m+1}\in H^p_{\rm at}(\mu)
$$
and
\begin{equation}\label{9.12}
\|b\|^p_{H^p_{\rm at}(\mu)}\ls m \dfrac{|\lz_1|^p}
{[{\widehat C}_b]^p}+(|\lz_1|+|\lz_2|)^p\ls(|\lz_1|+|\lz_2|)^p\sim |b|^p_{\hhpq},
\end{equation}
where the implicit positive constant is independent of $\widehat C_b$ and $m$.

Moreover, for any $f\in\hhpq$, by Definition \ref{d7.3}
with $\gz=1/p$ and $\rho=2$, we know that
there exists a sequence $\{b_k\}_{k\in\nn}$ of
$(p,q,1/p,2)_{\lz,\,1}$-atomic blocks such that $f=\sum_{k\in\nn}b_k$
in $(\cer)^\ast$ and
\begin{equation}\label{9.13}
\sum_{k\in\nn}|b_k|^p_{\hhpq}
\ls\|f\|^p_{\hhpq}.
\end{equation}
For each $k$, assume that $b_k=\lz_{k,\,1}a_{k,\,1}
+\lz_{k,\,2}a_{k,\,2}$, where
$\supp(b_k)\st B_k$, $\supp(a_{k,\,j})\st B_{k,\,j}$
for $j\in\{1,\,2\}$. Let $[\widehat{C}_k]^p
:=\wz K^{(2),\,p}_{B_{k,\,1},\,B_k}$.
From Remark \ref{r9.1} and Proposition \ref{p9.2},
\eqref{9.9}, \eqref{9.12} and \eqref{9.13}, we deduce that
$$
f=\sum_{k\in\nn}b_k=\sum_{k\in\nn}\lf[\dsum_{i=1}^{m_k}
\dfrac{\wz C\lz_{k,\,1}}{\widehat{C}_k}c_{k,\,i}
+\wz C(|\lz_{k,\,1}|+|\lz_{k,\,2}|)c_{m_k+1}\r]
$$
in $(\cer)^\ast=(\lip_{1/p-1}(\mu;\dz))^{\ast}$ and
$$\sum_{k\in\nn}\lf[\dsum_{i=1}^{m_k}\lf|\dfrac{\lz_{k,\,1}}{\widehat{C}_k}\r|^p
+(|\lz_{k,\,1}|+|\lz_{k,\,2}|)^p\r]\ls \sum_{k\in\nn}
|b_k|^p_{\hhpq}\ls\|f\|^p_{\hhpq}.$$
This implies that $f\in H^p_{\rm at}(\mu)$ and
$\|f\|_{H^p_{\rm at}(\mu)}\ls \|f\|_{\hhpq}$, which completes the proof
of Theorem \ref{t9.4}.
\end{proof}

\begin{remark}\label{r9.5}
(i) Theorem \ref{t9.4} implies that, if $(\cx,d,\mu)$ is a space of
homogeneous type with
$$\lz(x,r):=\mu(B(x,r))$$
for all $x\in\cx$ and
$r\in(0,\fz)$, then $\widehat H^1_{\rm atb}(\mu)$
and $H^1_{\rm at}(\mu)$ coincide with equivalent norms. We notice that
it is not a paradox of the example given by Tolsa
\cite[Example 5.6]{t01a} since $\lz(x,r)\sim r$
for all $x\in\cx$ and $r\in(0,\sqrt{2}]$,
which is not equivalent to $\mu(B(x,r))\sim r^2$, in that example.

(ii) Let $\rho\in(1,\fz)$, $\gz\in[1,\fz)$ and $q\in(1,\fz]$.
Combining Theorem \ref{t9.4} and Remark \ref{r7.4}(ii), we obtain
$$
\widehat{H}^{1,\,q,\,\gz}_{\rm atb,\,\rho}(\mu)
=H^{1}_{\rm at}(\mu)=\wz{H}^{1,\,q,\,\gz}_{\rm atb,\,\rho}(\mu)
$$
over spaces of homogeneous type.

(ii) From Theorem \ref{t9.4} and \cite[Theorem A]{cw77}, it follows that
$\widehat H^{p,\,q}_{\rm atb}(\mu)$ is independent of the choice of
$q$ in spaces of homogeneous type with $\lz(x,r):=\mu(B(x,r))$
for all $x\in\cx$ and $r\in(0,\fz)$.
\end{remark}

\section{Relation between $\nhp$ and $H^{p}_{\rm at}(\mu)$
over RD-spaces}\label{s10}

\hskip\parindent In this section, we investigate the relations
among $\nhp$, $\mhp$ and $H^{p,\,q}_{\rm at}(\mu)$ over RD-spaces.
As a corollary, the relations among $\nhp$, $\mhp$, $H^{p}(\rr^D)$,
$\widehat{H}^{p,\,q,\,\gz}_{\rm atb,\,\rho}(\mu)$
and $\widehat{H}_{\rm{mb},\,\rho}^{p,\,q,\,\gamma,\,\epsilon}(\mu)$
over Euclidean spaces $(\rr^D,|\cdot|)$ endowed
with the $D$-dimensional Lebesgue measure $dx$ are also presented.

Let $(\cx,d,\mu)$ be a space of homogeneous type with $\lz(x,r):=\mu(B(x,r))$
for all $x\in\cx$ and $r\in(0,\fz)$.
The following notion of RD-spaces was introduced by Han, M\"uller and Yang
\cite{hmy08}. A space of homogeneous type is called an
RD-\emph{space} if there exist constants
$\kz\in(0,\,\nu]$ and $C\in[1,\,\fz)$ such that,
for all $x\in\cx$, $r\in(0,\,\diam(\cx)/2)$ and $\lz\in[1,\,\diam(\cx)/(2r))$,
\begin{equation}\label{10.1}
C^{-1}\lz^\kz\mu(B(x,r))\le\mu(B(x,\lz r))\le C\lz^{\nu}\mu(B(x,r)),
\end{equation}
where $\diam(\cx):=\sup_{x,\,y\in\cx}d(x,y)$ and
$\nu:=\log_2C_{(\lz)}$ is as in Section \ref{s1}.
We point out that the RD-space is also a space of homogeneous type.
In the remainder of this section, we \emph{always assume} that $(\cx, d, \mu)$
is an {\rm RD}-space with $\mu(\cx)=\fz$ and let $V_r(x):=\mu(B(x, r))$ and
$V(x, y):=\mu(B(x, d(x, y)))$ for all $x,\,y\in\cx$ and $r\in(0,\fz)$.

The following {\it space of test functions on $\cx$} was
introduced by Han, M\"uller and Yang \cite{hmy06,hmy08}.
Throughout this section, we fix $x_1\in\cx$.

Let $\bz\in(0,\,1]$ and $\gz\in(0,\fz)$. A function $f$ on $\cx$
is said to belong to the
{\it space of test functions, $\cg(\bz,\gz)$}, if
there exists a non-negative constant $\wz C$ such that
\begin{itemize}

\item[(A6)] for all $x\in\cx$,
$$|f(x)|\le \wz{C}\frac{1}{V_1(x_1)+V(x_1,x)}
\lf[\frac{1}{1+d(x_1,x)}\r]^\gz;$$

\item[(A7)] for all $x,\,y\in\cx$ satisfying $d(x,y)\le(1+d(x_1,x))/2$,
$$|f(x)-f(y)|\le \wz{C}\lf[\frac{d(x,y)}{1+d(x_1,x)}\r]^\bz
\frac{1}{V_1(x_1)+V(x_1,x)}\lf[\frac{1}{1+d(x_1,x)}\r]^\gz.$$

\end{itemize}
Moreover, for $f\in\cg(\bz,\gz)$, its {\it norm} is defined by
$$
\|f\|_{\cg(\bz,\gz)}:=\inf\{\wz{C}:\ {\wz C}\ {\rm satisfies\
(A6)\ and\ (A7)}\}.
$$

The \emph{space} $\mathring{\cg}(\bz,\gz)$ is defined as the set of
all functions $f\in\cg(\bz,\gz)$ satisfying $\int_\cx f(x)\,d\mu(x)=0$.
Moreover, we endow the space $\mathring{\cg}(\bz,\gz)$ with
the same norm as the space $\cg(\bz,\gz)$.
Furthermore, $\mathring{\cg}(\bz,\gz)$ is a Banach space.

For any given $\ez\in(0,\,1]$, let $\mathring{\cg}^\ez_0(\bz,\gz)$
be the completion of the set $\mathring{\cg}(\ez,\,\ez)$
in $\mathring{\cg}(\bz,\gz)$ when $\bz,\gz\in(0,\,\ez]$.
Moreover, if $f\in\mathring{\cg}^\ez_0(\bz,\gz)$, we then define
$\|f\|_{\mathring{\cg}^\ez_0(\bz,\gz)}:=\|f\|_{\mathring{\cg}(\bz,\gz)}$.
We define the \emph{dual space} $(\mathring{\cg}^\ez_0(\bz,\gz))^*$
to be the set of all continuous linear functionals $\cl$ from
$\mathring{\cg}^\ez_0(\bz,\gz)$ to $\cc$, and
endow it with the weak$^*$ topology.

Suppose that $\ez_1\in(0,1]$ and $\ez_2,\,\ez_3\in(0,\fz)$.
Let $\{D_t\}_{t\in(0,\fz)}$ be a family of bounded linear operators
on $L^2(\mu)$ such that,
for all $t\in(0,\fz)$, $D_t(x, y)$, the kernel of $D_t$,
is a measurable function from $\cx\times\cx$ to
$\cc$ satisfying the following estimates:
there exists a positive constant $L_0$ such that, for all
$t\in(0,\fz)$ and all $x,\, \wz x,\, y\in\cx$ with
$d(x, \wz x)\le [t+d(x, y)]/2$,

\begin{itemize}

\item[(A1)] $|D_t(x, y)|\le L_0\dfrac{1}{V_t(x)+V_t(y)+V(x, y)}
\lf[\dfrac{t}{t+d(x, y)}\r]^{\ez_2}$;

\item[(A2)] $|D_t(x, y)-D_t(\wz x, y)|
\le L_0\lf[\dfrac{d(x, \wz x)}{t+d(x, y)}\r]^{\ez_1}
\dfrac{1}{V_t(x)+V_t(y)+V(x, y)}\lf[\dfrac{t}{t+d(x, y)}\r]^{\ez_3}$;

\item[(A3)] Property (A2) also holds true
with the roles of $x$ and $y$ interchanged;

\item[(A4)] $\int_\cx D_t(x, y)\,d\mu(x)=0$;

\item[(A5)] $\int_\cx D_t(x, y)\,d\mu(y)=0$.

\end{itemize}

Now we recall the following \emph{Calder\'on reproducing formula}
which is a continuous variant of
\cite[Theorem 3.10]{hmy08}. Hereafter, we let $a\wedge b:=\min\{a,b\}$
and $a\vee b:=\max\{a,b\}$ for all $a,\,b\in\rr$.

\begin{lemma}\label{l10.1}
Let $\ez_1:=1$, $\ez_2,\,\ez_3\in(0,\fz)$,
$\ez\in(0,\ez_1\wedge\ez_2)$ and $\{D_t\}_{t\in(0,\fz)}$ be as above.
Then there exists a family of linear operators
$\{\wz{D}_t\}_{t\in(0,\fz)}$ such that, for all
$f\in\mathring{\cg}^{\ez}_0(\bz,\gz)$ with $\bz,\,\gz\in(0,\ez)$,
$$
f=\int_0^\fz\wz{D}_tD_t(f)\,\frac{dt}{t}
$$
in $\mathring{\cg}^{\ez}_0(\bz,\gz)$ and in $\lq$ for all $q\in(1,\fz)$.
Moreover, the kernels of the operators $\wz{D}_t$ satisfy the conditions
\emph{(A1), (A2), (A4)} and \emph{(A5)} with $\ez_1$ and $\ez_2$ replaced by
$\wz{\ez}\in(\ez,\ez_1\wedge\ez_2)$.
\end{lemma}

To the best of our knowledge, the following useful property
is well known but there exists no complete proof. We present full details here.

\begin{lemma}\label{l10.2}
Let $\ez_1$ be as in {\rm (A2)}, $\ez\in(0,\ez_1]$, $\bz,\,\gz\in(0,\ez]$
and $q\in(1,\fz)$. Then $\mathring{\cg}^\ez_0(\bz,\gz)$ is dense
in $\lq$.
\end{lemma}

To prove this lemma, we need the following two technical conclusions.
Hereafter, for any $\ez\in(0,\fz)$,
we denote $\|\cdot\|_{\ez}^{(d)}$, which is as in \eqref{9.1}
with $\dz$ and $\az$ replaced respectively by $d$ and $\ez$,
simply by $\|\cdot\|_{\ez}$.

\begin{lemma}\label{l10.3}
Let $\ez\in(0,1]$, $F$ be a nonempty closed
set and $G$ an open set containing $F$.
Then there exists $f\in\lip_\ez(\mu)$ such that
$$
f=1\quad {\rm on}\quad F;\quad \supp(f)\st\overline{G};\quad 0\le f\le 1
\quad {\rm on}\quad \cx,
$$
where, for any set $A\st\cx$,
$\overline{A}$ represents the smallest closed set containing $A$.
\end{lemma}

\begin{proof}
For $x\in\cx$, let
$$f(x):=\lf[\frac{d(x,\,G^{\complement})}
{d(x,\,F)+d(x,\,G^{\complement})}\r]^{\ez},$$
here and hereafter, for any two sets $A,\,B\st\cx$,
$A^{\complement}:=\cx\bh A$ and
$$d(A,B):=\inf\{d(a,b):\ a\in A\ {\rm and}\ b\in B\}.$$
It is easy to show that $f$ has all the
required properties in Lemma \ref{l10.3} with
$\|f\|_{\ez}\le\frac1{[d(F,\,G^{\complement})]^{\ez}}$,
which completes the proof of Lemma \ref{l10.3}.
\end{proof}

Moreover, we have the following conclusion.

\begin{lemma}\label{l10.4}
Let $\ez\in(0,1]$. For any ball $B(x_0,r_0)$ and $\eta\in(0,\fz)$, there
exists
$$h\in\lip_{\ez,\,b}(\mu)
:=\{f\in\lip_\ez(\mu):\ \supp(f)\ {\rm is\ bounded}\}$$
such that
$$
\supp(h)\st (B(x_0,r_0))^{\complement},\quad
\int_\cx h(x)\,d\mu(x)=1\quad
{\rm and}\quad \|h\|_{\lq}<\eta.
$$
\end{lemma}

\begin{proof}
For any $k\in\nn$, let $F_k:=\overline{B(x_0,k+r_0+1)}\bh B(x_0,r_0+1)$ and
$$
G_k:={B(x_0,k+r_0+2)}\bh\overline{B(x_0,r_0)}.
$$
From \cite[Remark 1.2(i)]{hmy08}, it follows that
$G_k\supset F_k\neq\emptyset$
for some sufficiently large $k$.
By Lemma \ref{l10.3}
with $F$ and $G$ replaced by $F_k$ and $G_k$, respectively, we conclude
that there exists $\wz{h}_k\in\lip_\ez(\mu)$ ($\|\wz h_k\|_\ez\le1$) such that
$\wz{h}_k=1$ on $F_k$, $\supp(\wz{h}_k)\st\overline{G_k}$
and $0\le\wz h_k\le1$ on $\cx$. Let $h_k:=\frac{\wz{h}_k}{{\rm J}_k}$,
where ${\rm J}_k:=\int_{\cx}\wz{h}_k(x)\,d\mu(x)$.
Thus, $\int_\cx h_k(x)\,d\mu(x)=1$ and
\begin{equation}\label{10.2}
{\rm J}_k\ge\int_{F_k}\,d\mu(x)\ge\mu(B(x_0,k+r_0+1)\bh B(x_0,r_0+1)).
\end{equation}
We only need to show that $\lim_{k\to\fz}\|h_k\|_{\lq}^q=0$.
By \eqref{10.2} and $0\le \wz h_k\le1$ on $\cx$, we have
$$
\|h_k\|_{\lq}^q\le\frac{\mu(B(x_0,k+r_0+3)\bh B(x_0,r_0))}
{\mu(B(x_0,k+r_0+1)\bh B(x_0,r_0+1))}
\frac1{[\mu(B(x_0,k+r_0+1)\bh B(x_0,r_0+1))]^{q-1}}.
$$
Notice that $\mu(\cx)=\fz$, we obtain
$$\lim_{k\to\fz}\frac1{[\mu(B(x_0,k+r_0+1)\bh B(x_0,r_0+1))]^{q-1}}=0,$$
which reduces the proof to the fact that
$$\limsup_{k\to\fz}\frac{\mu(B(x_0,k+r_0+3)\bh B(x_0,r_0))}
{\mu(B(x_0,k+r_0+1)\bh B(x_0,r_0+1))}\le C,$$
where $C$ is as in \eqref{10.1}.

Indeed, from $\mu(\cx)=\fz$ and \eqref{10.1}, we deduce that
\begin{align*}
&\limsup_{k\to\fz}\frac{\mu(B(x_0,k+r_0+3)\bh B(x_0,r_0))}
{\mu(B(x_0,k+r_0+1)\bh B(x_0,r_0+1))}\\
&\hs\le1+\limsup_{k\to\fz}\frac{\mu(B(x_0,k+r_0+3)\bh B(x_0,k+r_0+1))}
{\mu(B(x_0,k+r_0+1)\bh B(x_0,r_0+1))}\\
&\hs\le1+\limsup_{k\to\fz}\frac{[C(\frac{k+r_0+3}{k+r_0+1})^{\nu}-1]
\mu(B(x_0,k+r_0+1))}{\mu(B(x_0,k+r_0+1)\bh B(x_0,r_0+1))}\le C,
\end{align*}
where $\nu$ and $C$ is as in \eqref{10.1}.
This finishes the proof of Lemma \ref{l10.4}.
\end{proof}

Now we are ready to prove Lemma \ref{l10.2}.

\begin{proof}[Proof of Lemma \ref{l10.2}]
By some arguments as in \cite[p.\,19]{hmy08}, we know that
$$
\mathring{\lip}_{\ez,\,b}(\mu)
:=\lf\{f\in\lip_{\ez,\,b}(\mu):\ \int_\cx f(x)\,d\mu(x)=0\r\}
\st\mathring{\cg}(\ez,\ez)\st\mathring{\cg}^\ez_0(\bz,\gz)\st\ltw.
$$
Thus, to show Lemma \ref{l10.2}, it suffices to
prove that $\mathring{\lip}_{\ez,\,b}(\mu)$ is dense in $\lq$.
By the fact that $\lip_{\ez,\,b}(\mu)$ is dense in $\lq$ (see, for example,
\cite[Corollary 2.11(ii)]{hmy08}), we know that, for any  $\eta\in(0,\fz)$
and $f\in\lq$, there exists $g\in \lip_{\ez,\,b}(\mu)$ such that
$\|g-f\|_{\lq}<\eta/2$. Now we show that there exists
$\wz g\in\mathring{\lip}_{\ez,\,b}(\mu)$
such that $\|\wz{g}-f\|_{\lq}<\eta$. We consider the following two cases:

\textbf{Case (i)} $\int_\cx g(x)\,d\mu(x)=0$. The result holds true immediately.

\textbf{Case (ii)} $\int_\cx g(x)\,d\mu(x)=A\neq0$. Let a ball $B(x_0,r_0)
\supset\supp(g)$. By Lemma \ref{l10.4}, there exists $h\in\lip_{\ez,\,b}(\mu)$
such that $\supp(h)\st(B(x_0,r_0))^{\complement}$, $\int_\cx h(x)\,d\mu(x)=1$
and $\|h\|_{\lq}<\frac{\eta}{2|A|}$, which implies that
$\wz{g}:=g-Ah\in\mathring{\lip}_{\ez,\,b}(\mu)$ and
$$
\|f-\wz{g}\|_{\lq}\le\|f-g\|_{\lq}+|A|\|h\|_{\lq}
<\frac{\eta}{2}+|A|\frac{\eta}{2|A|}=\eta.
$$
This, together with \textbf{Case (i)}, finishes the proof of
Lemma \ref{l10.2}.
\end{proof}

Recall that the {\it Littlewood-Paley S-function $S(f)(x)$} for any
$f\in L^q(\mu)$ with $q\in(1, \fz)$ and $x\in\cx$  is defined by
$$
S(f)(x):=\lf\{\int_{\Gamma(x)}|D_t(f)(y)|^2
\frac{d\mu(y)\,dt}{V_t(x)\,t}\r\}^{1/2},
$$
where
$$\Gamma(x):=\{(y, t)\in\cx\times(0, \fz):\  d(y, x)<t\}$$
generalizes the notion of a cone with vertex at $x$ and aperture 1.

In RD-spaces, Han, M\"uller and Yang \cite{hmy06}
introduced the \emph{Hardy space $H^p(\mu)$} defined by
$$
H^p(\mu):=\lf\{f\in(\mathring{\cg}^\ez_0(\bz,\gz))^*:\  S(f)\in L^p(\mu)\r\}
$$
endowed with the quasi-norm
$$
\|f\|_{H^p(\mu)}:=\|S(f)\|_{L^p(\mu)}.
$$
Moreover, Grafakos, Liu and Yang \cite{gly} proved that
$H_{\rm at}^{p}(\cx)$ and $H^p(\mu)$  coincide with equivalent quasi-norms.

Before dealing with the relation between
$\wz H^{p,\,q,\,\gz}_{\rm atb,\,\rho}(\mu)$ and
$H_{\rm at}^p(\mu)$, we need the following construction of dyadic cubes
on spaces of homogeneous type from \cite{c90}; see also \cite{hmy06}.

\begin{lemma}\label{l10.5}
Let $\cx$ be a space of homogeneous type. Then there exists a collection
$$
\lf\{Q_\bz^k\st\cx:\ k\in\zz,\,\bz\in I_k\r\}
$$ of open subsets, where $I_k$ is some index set, and positive
constants $\dz\in(0,1)$ and $L_1,\,L_2$ such that

\rm{(i)} $\mu(\cx\bh\cup_\bz Q_\bz^k)=0$ for each fixed $k$,
and $Q_\bz^k\cap Q_\gz^k=\emptyset$ if $\bz\neq\gz$;

\rm{(ii)} for any $\bz$, $\gz$, $k$ and $l$ with $l\ge k$,
either $Q_\gz^l\st Q_\bz^k$ or $Q_\gz^l\cap Q_\bz^k=\emptyset$;

\rm{(iii)} for each $(k,\bz)$ and each $l<k$, there exists a unique $\gz$
such that $Q_\bz^k\st Q_\gz^l$;

\rm{(iv)} $\diam(Q_\bz^k)\le L_1\dz^k$;

\rm{(v)} each $Q_\bz^k$ contains some ball $B(z_\bz^k,L_2\dz^k)$, where
$z_\bz^k\in\cx$.
\end{lemma}

We further introduce some notation from \cite{hmy06}. Let
$$
\mr:=\lf\{Q_\bz^k\st\cx:\ k\in\zz,\,\bz\in I_k\r\}.
$$
For any $k\in\zz$, let $\Omega_k:=\{x\in\cx:\ S(f)(x)>2^k\}$ and
$$
\mr_k:=\lf\{Q\in\mr:\ \mu(Q\cap\Omega_k)>\frac12\mu(Q)
\ {\rm and}\ \mu(Q\cap\Omega_{k+1})\le\frac12\mu(Q)\r\}.
$$
Moreover, for any $Q_\bz^k\in\mr$, let
$$
\widehat{Q}_\bz^k:=\lf\{(x,t)\in\cx\times(0,\fz):\ x\in Q_\bz^k\ {\rm and}\
L_1\dz^k<t\le L_1\dz^{k-1}\r\},
$$
$$
\mr_k^{\rm mc}:=\lf\{Q\in\mr_k:\ {\rm if}\ {\wz Q}\supset Q\
{\rm and}\ {\wz Q}\in\mr,\ {\rm then}\ {\wz Q}\notin\mr_k\r\},
$$
and, for any $Q_k^{\rm mc}\in\mr_k^{\rm mc}$,
$$
\wz{Q}_k^{\rm mc}:=
\bigcup_{Q\in\mr_k,\,Q\st Q_k^{\rm mc}}\widehat{Q},
$$
here and hereafter, ``mc'' means \emph{maximal cubes}.

We need the following useful lemma.

\begin{lemma}\label{l10.6}
Let $k,\,j\in\zz$ and $k<j$. Then
$$
\lf(\bigcup_{Q_k^{\rm mc}\in\mr_k^{\rm mc}}\wz{Q}_k^{\rm mc}\r)
\cap\lf(\bigcup_{Q_j^{\rm mc}\in\mr_j^{\rm mc}}\wz{Q}_j^{\rm mc}\r)=\emptyset.
$$
\end{lemma}

\begin{proof}
Let $k,\,j\in\zz$ and $k<j$. To prove this lemma,
it suffices to prove that, for any two dyadic cubes $Q$ and $P$
satisfying $Q\in\mr_k$, $Q\st Q_k^{\rm mc}\in\mr_k^{\rm mc}$,
$P\in\mr_j$ and $P\st Q_j^{\rm mc}\in\mr_j^{\rm mc}$, it holds true that
$\widehat{Q}\cap\widehat{P}=\emptyset$.
We prove it by contradiction. Suppose that
$\widehat{Q}\cap\widehat{P}\neq\emptyset$. Then $Q\cap P\neq\emptyset$.
By Lemma \ref{l10.5}(ii), $Q\st P$ or $P\st Q$. Without loss of generality,
we may assume that $Q\st P$. Let $Q=Q_\bz^m$ and $P=Q_\gz^n$
for some $m,\,n\in\zz$, $\bz\in I_m$ and $\gz\in I_n$.
If $m\neq n$, then $\widehat{Q}\cap\widehat{P}=\emptyset$, which contradicts
to the assumption that $\widehat{Q}\cap\widehat{P}\neq\emptyset$.
Thus, $m=n$, which, together with Lemma \ref{l10.5}(i),
implies that $Q_\bz^m=Q_\gz^n$. Moreover,
from $Q\in\mr_k$, it follows that $\mu(Q\cap\Omega_k)>\frac12\mu(Q)$
and $\mu(Q\cap\Omega_{k+1})\le\frac12\mu(Q)$. On the other hand,
by $P\in\mr_j$, we see that $\mu(P\cap\Omega_j)>\frac12\mu(P)$
and $\mu(P\cap\Omega_{j+1})\le\frac12\mu(P)$. Meanwhile,
$k<j$ implies that $\Omega_j\st \Omega_{k+1}$ and hence
$\mu(P\cap\Omega_{k+1})>\frac12\mu(P)$.
Thus, $P\neq Q$, which contradicts to the fact that
$Q=Q_\bz^m=Q_\gz^n=P$. This finishes the proof
of Lemma \ref{l10.6}.
\end{proof}

Now we introduce some useful decompositions of $\wz{D}_t(x,y)$
in Lemma \ref{l10.1} which are easy consequences of
\cite[Proposition 2.9]{hmy08}, the details being omitted.

\begin{lemma}\label{l10.7}
Let $\ez_1\in(0,1]$, $\ez_2\in(0,\fz)$, $\ez\in(0,\ez_1\wedge\ez_2)$,
$\wz{\ez}\in(\ez,\ez_1\wedge\ez_2)$ and
$\{\wz{D}_t\}_{t\in(0,\fz)}$ be as in Lemma \ref{l10.1}. Then,
for any $N\in(0,\wz{\ez}]$, $t\in(0,\fz)$ and $x,\,y\in\cx$,
$$
\wz{D}_t(x,y)=\sum_{\ell=0}^{\fz}2^{-N\ell}\vz_{2^{\ell}t}(x,y),
$$
where $\vz_{2^{\ell}t}(x,y)$ is an \emph{adjust bump function in
$x$ associated with the ball $B(y,2^{\ell}t)$}, which means that
there exists a positive constant $C$ such that, for all
$t\in(0,\fz)$ and $y\in\cx$,

\emph{(i)} $\supp(\vz_{2^{\ell}t}(\cdot,y))\st B(y,2^{\ell}t)$;

\emph{(ii)} $|\vz_{2^{\ell}t}(x,y)|\le C\frac1{V_{2^{\ell}t}(y)}$
for all $x\in\cx$;

\emph{(iii)} $\|\vz_{2^{\ell}t}(\cdot,y)\|_{\ez}
\le C(2^{\ell}t)^{-\eta}\frac1{V_{2^{\ell}t}(y)}$ for all $0<\eta\le\wz{\ez}$;

\emph{(iv)} $\int_{\cx}\vz_{2^{\ell}t}(x,y)\,d\mu(x)=0$.
\end{lemma}

Then we introduce a useful criteria for the boundedness of
some integral operators. Hereafter, we denote the inner
product of $\ltw$ by $(\cdot,\cdot)$.

\begin{lemma}\label{l10.8}
Let $K_t(\cdot,\cdot)$ for $t\in(0,\fz)$ be a measurable function from
$\cx\times\cx$ to $\cc$ and $\{K_t\}_{t\in(0,\fz)}$
a set of $\ltw$-bounded linear operators defined by
$$
K_t(f)(x):=\int_\cx K_t(x,y)f(y)\,d\mu(y) \quad for\ all\
t\in(0,\fz),\ x\in\cx\ and\ f\in\ltw.
$$
If there exist positive constants $\ez_1$,
$\ez_2$ and $C$ such that, for all $x,\,y\in\cx$ and $s,\,t\in(0,\fz)$,
\begin{equation}\label{10.3}
|K_t(x,y)|\le C\frac{1}{V_t(x)+V_t(y)+V(x,y)}\lf[\frac{t}{t+d(x,y)}\r]^{\ez_2},
\end{equation}
and
\begin{equation}\label{10.4}
|K_tK^{*}_s(x,y)|\le C\lf(\frac{t}{s}\wedge\frac{s}{t}\r)^{\ez_1}
\frac{1}{V_{t\vee s}(x)+V_{t\vee s}(y)+V(x,y)}
\lf[\frac{t\vee s}{t\vee s+d(x,y)}\r]^{\ez_2},
\end{equation}
where $K_t^*$ is the adjoint operator of $K_t$.
Then there exists a positive constant $\wz C$ such that, for all
$f\in\ltw$,
$$
\|G(f)\|_{\ltw}\le {\wz C}\|f\|_{\ltw},
$$
where, for all $x\in\cx$,
$$G(f)(x):=\lf\{\int_0^\fz|K_t(f)(x)|^2\,\frac{dt}{t}\r\}^{1/2}.$$
\end{lemma}

\begin{proof}
For any $f\in\ltw$, by the Fubini theorem,
we write
\begin{align*}
\|G(f)\|_{\ltw}^2&=(G(f),G(f))=\int_\cx
\int_0^\fz|K_t(f)(x)|^2\,\frac{dt}{t}d\mu(x)\\
&=\int_0^\fz\int_\cx|K_t(f)(x)|^2\,d\mu(x)\frac{dt}{t}
=\lim_{N\to\fz}\int_{1/N}^N(K_t(f),K_t(f))\frac{dt}{t}\\
&=\lim_{N\to\fz}\int_{1/N}^N\int_\cx
K_t^*K_t(f)(x)f(x)\,d\mu(x)\frac{dt}{t}\\
&=\lim_{N\to\fz}\int_\cx\int_{1/N}^N K_t^*K_t(f)(x)\,\frac{dt}{t}f(x)d\mu(x)
=\lim_{N\to\fz}\lf(\int_{1/N}^N K_t^*K_t(f)\frac{dt}{t},f\r).
\end{align*}
Moreover, from \eqref{10.3}, \eqref{10.4} and the Schur lemma
(see \cite[p.\,457]{g08}), we deduce that, for all $s,\,t\in(0,\fz)$,
$K^*_t$ and $K^{**}_s=K_s$ are bounded on $\ltw$ and
$$
\lf\|K_t^*K_tK^*_sK_s\r\|_{\ltw\to\ltw}
\ls\lf\|K_tK^*_s\r\|_{\ltw\to\ltw}
\ls\lf(\frac{t}{s}\wedge\frac{s}{t}\r)^{\ez_1}.
$$
By the above two inequalities and \cite[p.\,237,\ Exercise 8.5.8]{g09},
we conclude that
$$\|G(f)\|^2_{\ltw}\le\liminf_{N\to\fz}
\lf\|\int_{1/N}^N K_t^*K_t(f)\frac{dt}{t}\r\|_{\ltw}
\|f\|_{\ltw}\ls\|f\|^2_{\ltw},$$
which completes the proof of Lemma \ref{l10.8}.
\end{proof}

Before showing the main result of this section, we introduce another
technical lemma which gives a sufficient condition to the fact
that $f=g$ in $\ltw$ for all $f,\,g\in(\mathring{\cg}^\ez_0(\bz,\gz))^*$.

\begin{lemma}\label{l10.9}
Let $\ez_1$ be as in \emph{(A2)}
and $\bz,\,\gz\in(0,\ez_1)$. If $f,\,g\in\ltw$
and $f=g$ in $(\mathring{\cg}^\ez_0(\bz,\gz))^*$, then $f=g$ in $\ltw$.
\end{lemma}

\begin{proof}
For any $f,\,g\in\ltw$, let $(f,g):=\int_{\cx}f(x)g(x)\,d\mu(x)$.
Now we claim that $(f,\cdot)$ is a bounded functional on
$\mathring{\cg}^\ez_0(\bz,\gz)\st\ltw$. Indeed, by (A6) and
the H\"older inequality, we conclude that,
for any $f\in\ltw$ and $h\in\mathring{\cg}^\ez_0(\bz,\gz),$
\begin{align*}
|(f,h)|&\le\int_{\cx}|f(x)||h(x)|\,d\mu(x)\\
&\ls\|h\|_{\cg(\bz,\gz)}
\int_{\cx}\frac1{V_1(x_1)+V(x_1,x)}
\lf[\frac1{1+d(x,x_1)}\r]^{\gz}|f(x)|\,d\mu(x)\\
&\ls\|h\|_{\cg(\bz,\gz)}\|f\|_{\ltw}
\int_{\cx}\lf\{\frac1{V_1(x_1)+V(x_1,x)}
\lf[\frac1{1+d(x,x_1)}\r]^{\gz}\r\}^2\,d\mu(x)\\
&\ls\|h\|_{\cg(\bz,\gz)}\|f\|_{\ltw}\\
&\hs\times\lf\{\int_{B(x_1,1)}\lf[\frac1{V_1(x_1)}\r]^2\,d\mu(x)
+\frac1{V_1(x_1)}\int_{\cx\bh B(x_1,1)}\frac1{V(x_1,x)}
\lf[\frac1{d(x,x_1)}\r]^{2\gz}\,d\mu(x)\r\}\\
&\ls\|h\|_{\cg(\bz,\gz)}\|f\|_{\ltw}\frac1{V_1(x_1)},
\end{align*}
which implies the claim.

From this, a density argument, Lemma \ref{l10.2}
and $f=g$ in $(\mathring{\cg}^\ez_0(\bz,\gz))^*$, it follows that
\begin{align*}
\|f-g\|_{\ltw}&=\sup\lf\{|(f-g,h)|:\ \|h\|_{\ltw}\le1\r\}\\
&=\sup\lf\{|(f-g,h)|:\ h\in\mathring{\cg}^\ez_0(\bz,\gz)\ {\rm and}\
\|h\|_{\ltw}\le1\r\}\\
&=\sup\lf\{|(f,h)-(g,h)|:\ h\in\mathring{\cg}^\ez_0(\bz,\gz)\ {\rm and}\
\|h\|_{\ltw}\le1\r\}=0,
\end{align*}
which completes the proof of Lemma \ref{l10.9}.
\end{proof}

Now we are ready to prove the following main result of this section.

\begin{theorem}\label{t10.10}
Let $(\cx,\,d,\,\mu)$ be an \emph{RD}-space
with $\mu(\cx)=\fz$, $\frac{\nu}{\nu+1}<p\le1<q\le2$,
$\rho\in(1,\fz)$, $\gz\in[1,\fz)$ and $\ez\in(0,\fz)$,
where $\nu$ is as \eqref{10.1}.
Then $\wz H^{p,\,q,\,\gz,\,\ez}_{\rm mb,\,\rho}(\mu)$,
$\wz H^{p,\,q,\,\gz}_{\rm atb,\,\rho}(\mu)$
and $H_{\rm at}^p(\mu)$ coincide with equivalent quasi-norms.
\end{theorem}

\begin{proof}
Let $p$, $q$, $\rho$, $\gz$ and $\ez$
be as in assumptions of Theorem \ref{t10.10}.
We first claim that $H_{\rm at}^p(\mu)\cap\ltw$
is dense in $H_{\rm at}^p(\mu)$. Indeed, for any $f\in H^p_{\rm at}(\mu)$,
by Theorem \ref{t9.4}, we have
$f=\sum_{i=1}^\fz\lz_ib_i$ in $(\lip_{1/p-1}(\mu))^*$,
where $\{b_i\}_i$ is  a sequence of $(p,2)$-atoms,
$\supp(b_i)\st B_i$ for some ball $B_i$ and
$\|b_i\|_{\ltw}\le[\mu(B_i)]^{1/2-1/p}$.
Let $f_N:=\sum_{i=1}^N\lz_ib_i$, $N\in\nn$. Then $f_N\in\ltw$ for all $N\in\nn$.
Meanwhile, $f-f_N=\sum_{i=N+1}^\fz\lz_i b_i$ in $(\lip_{1/p-1}(\mu))^*$,
and $\|f-f_N\|^p_{H_{\rm at}^p(\mu)}\le\sum_{i=N+1}^\fz|\lz_i|^p\to 0$
as $N\to\fz$. Thus, $H_{\rm at}^p(\mu)\cap\ltw$
is dense in $H_{\rm at}^p(\mu)$, which completes the proof of this claim.

We easily observe that $\ltw\st (\mathring{\cg}^\ez_0(\bz,\gz))^*$.
By \cite[Remark 5.5(ii)]{gly}, we know that $H^p_{\rm at}(\mu)=H^p(\mu)$
with equivalent quasi-norms. By this, the above claim and
a standard density argument, to show Theorem \ref{t10.10},
it suffices to prove that
$$
\lf(H^p_{\rm at}(\mu)\cap\ltw\r)\st\pnhp\st
\wz{\mathbb{H}}^{p,\,q,\,\gz,\,\ez}_{\rm mb,\,\rho}(\mu)
\st \lf(H^p(\mu)\cap\ltw\r).$$
We show this by two steps.

\textbf{Step 1}. Now we show that $(H^p_{\rm at}(\mu)\cap\ltw)\st
\wz{\mathbb{H}}^{p,\,q,\,\gz}_{\rm atb,\,\rho}(\mu)$.
By Remark \ref{r3.3}(iii) and $q\in(1,2]$, it suffices to
show that $(H^p_{\rm at}(\mu)\cap\ltw)\st
\wz{\mathbb{H}}^{p,\,2,\,\gz}_{\rm atb,\,\rho}(\mu)$.
To this end, by Lemma \ref{l10.1} and an argument similar to
that used in the proof of \cite[p.\,1524,\ (2.30)]{hmy08},
we know that, for any
$f\in (\mathring{\cg}^\ez_0(\bz,\gz))^*$ with
$0<\bz,\,\gz<\ez<1\wedge\ez_2$, $\ez_3\in(0,\fz)$
and $\wz{\ez}\in(\ez,1\wedge\ez_2)$
($\ez_2$ and $\ez_3$ are as in (A2)),
$$
f(x)=\sum_{\ell=0}^{\fz}2^{-N\ell}\sum_{k\in\zz}
\sum_{Q_k^{\rm mc}\in\mr_k^{\rm mc}}
\int_{\wz{Q}_k^{\rm mc}}\vz_{2^{\ell}t}(x,y)D_t(f)(y)\,\frac{d\mu(y)dt}{t}
\quad {\rm in}\quad (\mathring{\cg}^\ez_0(\bz,\gz))^*,
$$
where $\vz_{2^{\ell}t}(x,y)$ is as in Lemma \ref{l10.7}.
Then we show that, for any $f\in H_{\rm at}^p(\mu)\cap\ltw$,
\begin{equation}\label{10.5}
f(x)=\sum_{\ell=0}^{\fz}2^{-N\ell}\sum_{k\in\zz}
\sum_{Q_k^{\rm mc}\in\mr_k^{\rm mc}}
\int_{\wz{Q}_k^{\rm mc}}\vz_{2^{\ell}t}(x,y)D_t(f)(y)\,\frac{d\mu(y)dt}{t}
\quad {\rm in}\quad \ltw.
\end{equation}

By Lemma \ref{l10.9}, it suffices to prove that
$$
\lf\|\sum_{\ell=0}^\fz 2^{-N\ell}\sum_{k\in\zz}
\sum_{Q_k^{\rm mc}\in\mr_k^{\rm mc}}\int_{\wz{Q}_k^{\rm mc}}
\vz_{2^{\ell}t}(\cdot,y)D_t(f)(y)\,\frac{d\mu(y)dt}{t}\r\|_{\ltw}<\fz.
$$
To this end, for any $x,\,y\in\cx$, $t,\,s\in(0,\fz)$
and $f\in\ltw$, let
$$
\vz_{2^{\ell}t}(f)(x):=\int_{\cx}\vz_{2^{\ell}t}(x,z)f(z)\,d\mu(z).
$$
Thus
$$
\vz_{2^{\ell}t}^*(f)(x):=\int_{\cx}\vz_{2^{\ell}t}(z,x)f(z)\,d\mu(z)
$$
and
$$
\vz_{2^{\ell}t}^*\vz_{2^{\ell}s}(x,y)
:=\int_{\cx}\vz_{2^{\ell}t}(z,x)\vz_{2^{\ell}t}(z,y)\,d\mu(z).
$$
By a duality method, Lemma \ref{l10.2}, the Fubini theorem,
the H\"older inequality,
Lemma \ref{l10.6} and \cite[Proposition 2.14]{hmy06}, we obtain
\begin{align*}
&\lf\|\sum_{\ell=0}^\fz 2^{-N\ell}\sum_{k\in\zz}
\sum_{Q_k^{\rm mc}\in\mr_k^{\rm mc}}\int_{\wz{Q}_k^{\rm mc}}
\vz_{2^{\ell}t}(\cdot,y)D_t(f)(y)\,\frac{d\mu(y)dt}{t}\r\|_{\ltw}\\
&\hs=\sup_{\begin{subarray}{c}
\|h\|_{\ltw}\le1\\
h\in\mathring{\cg}^\ez_0(\bz,\gz)\end{subarray}}
\lf|\lf\langle\sum_{\ell=0}^\fz
2^{-N\ell}\sum_{k\in\zz}
\sum_{Q_k^{\rm mc}\in\mr_k^{\rm mc}}\int_{\wz{Q}_k^{\rm mc}}
\vz_{2^{\ell}t}(\cdot,y)D_t(f)(y)\,\frac{d\mu(y)dt}{t},h\r\rangle\r|\\
&\hs=\sup_{\begin{subarray}{c}
\|h\|_{\ltw}\le1\\
h\in\mathring{\cg}^\ez_0(\bz,\gz)\end{subarray}}
\lf|\sum_{\ell=0}^\fz
2^{-N\ell}\sum_{k\in\zz}
\sum_{Q_k^{\rm mc}\in\mr_k^{\rm mc}}\lf\langle\int_{\wz{Q}_k^{\rm mc}}
\vz_{2^{\ell}t}(\cdot,y)D_t(f)(y)\,\frac{d\mu(y)dt}{t},h\r\rangle\r|\\
&\hs=\sup_{\begin{subarray}{c}
\|h\|_{\ltw}\le1\\
h\in\mathring{\cg}^\ez_0(\bz,\gz)\end{subarray}}
\lf|\sum_{\ell=0}^\fz 2^{-N\ell}
\sum_{k\in\zz}\sum_{Q_k^{\rm mc}\in\mr_k^{\rm mc}}
\int_{\wz{Q}_k^{\rm mc}}
\lf[\int_\cx\vz_{2^{\ell}t}(x,y)h(x)\,d\mu(x)\r]
D_t(f)(y)\,\frac{d\mu(y)dt}{t}\r|\\
&\hs\le\sup_{\|h\|_{\ltw}\le1}\sum_{\ell=0}^\fz 2^{-N\ell}
\sum_{k\in\zz}\sum_{Q_k^{\rm mc}\in\mr_k^{\rm mc}}
\lf|\int_{\wz{Q}_k^{\rm mc}}
\lf[\int_\cx\vz_{2^{\ell}t}(x,y)h(x)\,d\mu(x)\r]
D_t(f)(y)\,\frac{d\mu(y)dt}{t}\r|\\
&\hs\le\sup_{\|h\|_{\ltw}\le1}\sum_{\ell=0}^\fz 2^{-N\ell}
\lf\{\int_\cx\int_{0}^\fz\lf|\vz_{2^{\ell}t}^*(h)(y)\r|^2\,
\frac{dt}{t}d\mu(y)\r\}^{1/2}
\lf\{\int_\cx\int_{0}^\fz|D_t(f)(y)|^2
\,\frac{dt}{t}d\mu(y)\r\}^{1/2}\\
&\hs\ls\sup_{\|h\|_{\ltw}\le1}\sum_{\ell=0}^\fz 2^{-N\ell/2}
\lf\{\int_\cx\int_{0}^\fz\lf|2^{-N\ell/2}\vz_{2^{\ell}t}^*(h)(y)\r|^2\,
\frac{dt}{t}d\mu(y)\r\}^{1/2}\|f\|_{\ltw},
\end{align*}
where, in the third equality of the above equation, we used the
fact that, for any $Q_k^{\rm mc}\in\mr_k^{\rm mc}$, $f\in\ltw$
and $h\in\mathring{\cg}^\ez_0(\bz,\gz)$,
$$
{\rm O}:=\int_\cx\int_{\wz{Q}_k^{\rm mc}}
\lf|\vz_{2^{\ell}t}(x,y)\r||h(x)|
|D_t(f)(y)|\,\frac{dt}{t}d\mu(y)\,d\mu(x)<\fz.
$$
Indeed, let $Q_k^{\rm mc}:=Q_{\bz_0}^{k_0}$ for some $k_0\in\zz$
and some $\bz_0\in I_{k_0}$.
By the Fubini-Tonelli theorem, (i) and (ii) of Lemma \ref{l10.7},
(A6), $\wz{Q}_k^{\rm mc}\st Q_{\bz_0}^{k_0}\times(0,L_1\dz^{k_0-1}]$
(see \cite[p.\,1524]{hmy06})
and \cite[Proposition 2.14]{hmy06}, we conclude that
\begin{align*}
{\rm O}&=\int_{\wz{Q}_k^{\rm mc}}
\lf[\int_\cx\lf|\vz_{2^{\ell}t}(x,y)\r||h(x)|\,d\mu(x)\r]
|D_t(f)(y)|\,\frac{dt}{t}d\mu(y)\\
&\ls\int_{\wz{Q}_k^{\rm mc}}
\lf[\int_{B(y,2^{\ell}t)}\frac{|h(x)|}{V_{2^{\ell}t}(y)}\,d\mu(x)\r]
|D_t(f)(y)|\,\frac{dt}{t}d\mu(y)\\
&\ls\|h\|_{\mathring{\cg}^\ez_0(\bz,\gz)}
\frac{1}{V_{1}(x_1)}\int_{\wz{Q}_k^{\rm mc}}
|D_t(f)(y)|\,\frac{dt}{t}d\mu(y)\\
&\ls\|h\|_{\mathring{\cg}^\ez_0(\bz,\gz)}
\frac{1}{V_{1}(x_1)}\lf[\int_{\wz{Q}_k^{\rm mc}}\,\frac{dt}{t}d\mu(y)\r]^{1/2}
\lf[\int_{\wz{Q}_k^{\rm mc}}
|D_t(f)(y)|^2\,\frac{dt}{t}d\mu(y)\r]^{1/2}\\
&\ls\|h\|_{\mathring{\cg}^\ez_0(\bz,\gz)}
\frac{1}{V_{1}(x_1)}\lf[\mu\lf(Q_{\bz_0}^{k_0}\r)L_1\dz^{k_0-1}\r]^{1/2}
\lf[\int_{\cx}\int_{0}^{\fz}
|D_t(f)(y)|^2\,\frac{dt}{t}d\mu(y)\r]^{1/2}\\
&\ls\|h\|_{\mathring{\cg}^\ez_0(\bz,\gz)}
\frac{1}{V_{1}(x_1)}\lf[\mu\lf(Q_{\bz_0}^{k_0}\r)L_1\dz^{k_0-1}\r]^{1/2}
\|f\|_{\ltw}<\fz,
\end{align*}
which implies the desired result.

Let $\Phi_\ell(h)(y):=\{\int_{0}^\fz|2^{-N\ell/2}\vz_{2^{\ell}t}^*(h)(y)|^2\,
\frac{dt}{t}\}^{1/2}$ for all $y\in\cx$.
To prove \eqref{10.5}, we only need to show that
\begin{equation}\label{10.6}
\|\Phi_\ell(h)\|_{\ltw}\ls\|h\|_{\ltw},
\end{equation}
where the implicit positive constant is independent of $\ell$.

By Lemma \ref{l10.8}, we need to show that
$\{\varphi_{2^{\ell}t}\}_{t\in(0,\fz)}$ satisfy \eqref{10.3}
and \eqref{10.4}. From Lemma \ref{l10.7}, we easily deduce that
\eqref{10.3} holds true for $\{\varphi_{2^{\ell}t}\}_{t\in(0,\fz)}$.
Thus, it suffices to show that, for all $x,\,y\in\cx$,
$s,\,t\in(0,\fz)$,
\begin{align*}
&\lf|\lf(2^{-N\ell/2}\varphi_{2^{\ell}t}\r)^*
\lf(2^{-N\ell/2}\varphi_{2^{\ell}s}\r)(x,y)\r|\\
&\quad\ls\lf(\frac{t}{s}\wedge\frac{s}{t}\r)^{\eta}
\frac{1}{V_{t\vee s}(x)+V_{t\vee s}(y)+V(x,y)}
\lf[\frac{t\vee s}{t\vee s+d(x,y)}\r]^{N/2}.
\end{align*}

Due to the symmetry of $t$ and $s$, without loss of generality,
we may assume that $s<t$. Thus, we only need to show that,
for all $x,\,y\in\cx$, $0<s<t<\fz$,
$$
\lf|\varphi_{2^{\ell}t}^*\varphi_{2^{\ell}s}(x,y)\r|
\ls2^{N\ell}\lf(\frac{s}{t}\r)^{\eta}
\frac{1}{V_{t}(x)+V_{t}(y)+V(x,y)}
\lf[\frac{t}{t+d(x,y)}\r]^{N/2}.
$$
By Lemma \ref{l10.7}(iv), we write
\begin{align*}
\lf|\varphi_{2^{\ell}t}^*\varphi_{2^{\ell}s}(x,y)\r|
&\le\int_\cx\lf|\varphi_{2^{\ell}t}(z,x)-\varphi_{2^{\ell}t}(y,x)\r|
\lf|\varphi_{2^{\ell}s}(z,y)\r|\,d\mu(z)\\
&\le\int_{\lf\{z\in\cx:\ d(z,y)\le\frac{2^{\ell}t+d(x,y)}2\r\}}
\lf|\varphi_{2^{\ell}t}(z,x)-\varphi_{2^{\ell}t}(y,x)\r|
\lf|\varphi_{2^{\ell}s}(z,y)\r|\,d\mu(z)\\
&\quad+\int_{\lf\{z\in\cx:\ d(z,y)>\frac{2^{\ell}t+d(x,y)}2\r\}}
\lf|\varphi_{2^{\ell}t}(z,x)\r|\lf|\varphi_{2^{\ell}s}(z,y)\r|\,d\mu(z)\\
&\quad+\lf|\varphi_{2^{\ell}t}(y,x)\r|
\int_{\lf\{z\in\cx:\ d(z,y)>\frac{2^{\ell}t+d(x,y)}2\r\}}
\lf|\varphi_{2^{\ell}s}(z,y)\r|\,d\mu(z)\\
&=:{\rm I_1}+{\rm I_2}+{\rm I_3}.
\end{align*}

We first estimate $\rm I_1$.
Observe that, if $z\in B(x,2^{\ell}t)$
and $\frac{2^{\ell}t+d(x,y)}{2}\ge d(y,z)\ge d(x,y)-d(x,z)$,
then $y\in B(x,2^{\ell+2}t)$; if $d(x,y)<2^{\ell+2}t$, then
$$
\lf(\frac{t}{t+d(x,y)}\r)^{N/2}\ge\lf(\frac{t}{t+d(x,y)}\r)^{N}
\gtrsim 2^{-N\ell}.
$$
From the above two facts, (i), (ii) and (iii) of Lemma \ref{l10.7}
and Remark \ref{r2.4}(ii), it follows that
\begin{align*}
{\rm I_1}&\ls\int_{\lf\{z\in\cx:\ d(z,y)\le\frac{2^{\ell}t+d(x,y)}2\r\}}
\lf\|\varphi_{2^{\ell}t}(\cdot,x)\r\|_{\eta}[d(y,z)]^{\eta}
\chi_{B(x,2^{\ell+2}t)}(y)\frac{\chi_{B(y,2^{\ell}s)}(z)}
{V_{2^{\ell}s}(y)}\,d\mu(z)\\
&\ls\int_{\lf\{z\in\cx:\ d(z,y)\le\frac{2^{\ell}t+d(x,y)}2\r\}}
\frac1{(2^{\ell}t)^{\eta}}\frac{\chi_{B(x,2^{\ell+2}t)}(y)}{V_{2^{\ell}t}(x)}
(2^{\ell}s)^{\eta}\frac{\chi_{B(y,2^{\ell}s)}(z)}{V_{2^{\ell}s}(y)}\,d\mu(z)\\
&\ls\lf(\frac{s}{t}\r)^{\eta}\int_{\lf\{z\in\cx:\ d(z,y)
\le\frac{2^{\ell}t+d(x,y)}2\r\}}
\frac{\chi_{B(x,2^{\ell+2}t)}(y)}{V_{2^{\ell+2}t}(x)}
\frac{\chi_{B(y,2^{\ell}s)}(z)}{V_{2^{\ell}s}(y)}\,d\mu(z)\\
&\ls\lf(\frac{s}{t}\r)^{\eta}\frac{\chi_{B(x,2^{\ell+2}t)}(y)}
{V_{2^{\ell+2}t}(x)+V_{2^{\ell+2}t}(y)+V(x,y)}\\
&\ls2^{N\ell}\lf(\frac{s}{t}\r)^{\eta}\frac1
{V_{t}(x)+V_{t}(y)+V(x,y)}\lf[\frac{t}{t+d(x,y)}\r]^{N/2}.
\end{align*}

Now we turn to estimate ${\rm I_2}$. Observe that, if
$\{z\in\cx:\ d(z,y)>\frac{2^{\ell}t+d(x,y)}2\}\neq\emptyset$
and $z\in B(x,2^{\ell}t)\cap B(y,2^{\ell}s)$, then
$y\in B(x,2^{\ell+1}t)$ and $1\ls(\frac{s}{t})^{\eta}$.
By the above facts, (i) and (ii) of Lemma \ref{l10.7}
and some arguments similar to those used in the estimate for $\rm I_1$,
we further have
\begin{align*}
{\rm I_2}&\ls\int_{\lf\{z\in\cx:\ d(z,y)>\frac{2^{\ell}t+d(x,y)}2\r\}}
\frac{\chi_{B(x,2^{\ell}t)}(z)}{V_{2^{\ell}t}(x)}
\frac{\chi_{B(y,2^{\ell}s)}(z)}{V_{2^{\ell}s}(y)}\,d\mu(z)\\
&\ls\frac{\chi_{B(x,2^{\ell+1}t)}(y)}{V_{2^{\ell+1}t}(x)}
\int_{\lf\{z\in\cx:\ d(z,y)>\frac{2^{\ell}t+d(x,y)}2\r\}}
\frac{\chi_{B(y,2^{\ell}s)}(z)}{V_{2^{\ell}s}(y)}\,d\mu(z)\\
&\ls\lf(\frac{s}{t}\r)^{\eta}\frac1
{V_{2^{\ell+1}t}(x)+V_{2^{\ell+1}t}(y)+V(x,y)}\\
&\ls2^{N\ell}\lf(\frac{s}{t}\r)^{\eta}\frac1
{V_{t}(x)+V_{t}(y)+V(x,y)}\lf[\frac{t}{t+d(x,y)}\r]^{N/2}.
\end{align*}

For $\rm I_3$, by some arguments similar to those used in the
estimate for $\rm I_2$, we see that
\begin{align*}
{\rm I_3}&\ls\frac{\chi_{B(x,2^{\ell}t)}(y)}{V_{2^{\ell}t}(x)}
\int_{\lf\{z\in\cx:\ d(z,y)>\frac{2^{\ell}t+d(x,y)}2\r\}}
\frac{\chi_{B(y,2^{\ell}t)}(z)}{V_{2^{\ell}s}(y)}\,d\mu(z)\\
&\ls2^{N\ell}\lf(\frac{s}{t}\r)^{\eta}\frac1
{V_{t}(x)+V_{t}(y)+V(x,y)}\lf[\frac{t}{t+d(x,y)}\r]^{N/2}.
\end{align*}
Combining the estimates for $\rm I_1$, $\rm I_2$ and $\rm I_3$,
we finish the proof of \eqref{10.6} and hence \eqref{10.5}.

Moreover, from the proof of \cite[Theorem 2.21]{hmy06}, it follows that,
for any $f\in H_{\rm at}^p(\mu)\cap\ltw$, there exists a positive
constant $L_3$ such that
$$
f(x)=\sum_{\ell=0}^{\fz}2^{-N\ell}\sum_{k\in\zz}
\sum_{Q_k^{\rm mc}\in\mr_k^{\rm mc}}
\lz^{\ell}_{Q_k^{\rm mc}}a^{\ell}_{Q_k^{\rm mc}}
\quad {\rm in}\quad \ltw,
$$
where
$$
\lz^{\ell}_{Q_k^{\rm mc}}=L_3\lf[\mu\lf(B_k^{\rm mc}\r)\r]^{1/p-1/2}
\lf[\int_{{\wz Q}_k^{\rm mc}}|D_t(f)(y)|^2\,\frac{d\mu(y)dt}{t}\r]^{1/2},
$$
$Q_k^{\rm mc}:=Q^{k_0}_{\bz_0}$, $B_k^{\rm mc}:=B(z^{k_0}_{\bz_0},
(\frac1{\dz}+1)L_12^{\ell}\dz^{k_0})$ and, for all $x\in\cx$,
$$
a^{\ell}_{Q_k^{\rm mc}}(x):=\frac1{\lz^{\ell}_{Q_k^{\rm mc}}}
\int_{\wz{Q}_k^{\rm mc}}\vz_{2^{\ell}t}(x,y)D_t(f)(y)\,\frac{d\mu(y)dt}{t}
$$
is a $(p,2)$-atom supported on $B_k^{\rm mc}$.
By an argument similar to that used in the proof of \eqref{9.6},
we further conclude that
$a^{\ell}_{Q_k^{\rm mc}}$ is also a $(p,2,\gz,\rho)_\lz$-atomic block
and $|a^{\ell}_{Q_k^{\rm mc}}|_{{\wz H}^{p,\,2,\,\gz}_{\rm atb,\,\rho}(\mu)}
\ls|\lz^{\ell}_{Q_k^{\rm mc}}|$.
Thus, $f\in\wz{\mathbb{H}}^{p,\,2,\,\gz}_{\rm atb,\,\rho}(\mu)$ and
$$
\|f\|_{{\wz H}^{p,\,2,\,\gz}_{\rm atb,\,\rho}(\mu)}\ls\|f\|_{H_{\rm at}^p(\mu)},
$$
which completes the proof of \textbf{Step 1}.

\textbf{Step 2}. In this step, we show that
$\pnhp\st\pmhp\st (H^p(\mu)\cap\ltw)$ for any $q\in(1,\fz)$.
By Proposition \ref{p4.3}, we see that
$\pnhp\st\pmhp$ for any $q\in(1,\fz)$.
Thus, to prove the desired conclusion, it suffices to show that
$\pmhp\st (H^p(\mu)\cap\ltw)$ for any $q\in(1,\fz)$.

We first reduce the proof to showing that,
if $b$ is a $(p,q,\gz,\ez,\rho)_\lz$-molecular block,
then
\begin{equation}\label{10.7}
S(b)\in L^p(\mu)\quad {\rm and}\quad \|S(b)\|_{\lp}
\ls|b|_{\wz{H}^{p,\,q,\,\gz,\,\ez}_{\rm mb,\,\rho}(\mu)}.
\end{equation}

Indeed, assume that \eqref{10.7} holds true. For any
$f\in\wz{\mathbb{H}}^{p,\,q,\,\gz,\,\ez}_{\rm mb,\,\rho}(\mu)$,
by Definition \ref{d4.1}, we know that
there exists a sequence $\{b_i\}_{i\in\nn}$ of
$(p,q,\gz,\ez,\rho)_\lz$-molecular blocks such that $f=\sum_{i=1}^\fz b_i$
in $\ltw$ and
\begin{equation}\label{10.8}
\sum_{i=1}^\fz|b_i|^p_{\wz{H}^{p,\,q,\,\gz,\,\ez}_{\rm mb,\,\rho}(\mu)}
\sim\|f\|_{\wz{H}^{p,\,q,\,\gz,\,\ez}_{\rm mb,\,\rho}(\mu)}^p.
\end{equation}
Notice that $D_t(y,\cdot)\in\ltw$ for any $y\in\cx$ and $t\in(0,\fz)$.
Thus, for any $y\in\cx$, we have
$$
|D_t(f)(y)|=|(D_t(y,\cdot),f)|\le\sum_{i=0}^\fz|(D_t(y,\cdot),b_i)|
=\sum_{i=0}^\fz|D_t(b_i)(y)|.
$$
From this, the Fatou lemma, \eqref{10.7} and \eqref{10.8}, we deduce that
$$
\|S(f)\|^p_{\lp}\le\sum_{i=1}^\fz\|S(b_i)\|^p_{\lp}
\ls\sum_{i=1}^\fz|b_i|^p_{\wz{H}^{p,\,q,\,\gz,\,\ez}_{\rm mb,\,\rho}(\mu)}
\sim\|f\|_{\wz{H}^{p,\,q,\,\gz,\,\ez}_{\rm mb,\,\rho}(\mu)}^p,
$$
which completes the proof of \textbf{Step 2}.

Now we prove \eqref{10.7} following the ideas
of the proof of Theorem \ref{t4.8}.
For the sake of simplicity, we assume that
$\gz=1$ and $\rho=2$. Let $b=\sum_{k=0}^\fz\sum_{j=1}^{M_k}\lz_{k,\,j}
a_{k,\,j}$ be a $(p,q,1,\ez,2)_\lz$-molecular block
as in Definition \ref{d4.1} with $\gz=1$ and $\rho=2$,
where, for any $k\in\zz_+$ and $j\in\{1,\ldots,M_k\}$,
$\supp(a_{k,\,j})\subset B_{k,\,j}\subset U_k(B)$ for some $B_{k,\,j}$
and $U_k(B)$ as in Definition \ref{d4.1}. Without loss of generality,
we may assume that $\wz M=M$ in Definition \ref{d4.1}.
Since $S$ is sublinear, we write
\begin{align*}
\|S(b)\|^p_{\lp}&\le\sum_{\ell=5}^\fz\int_{U_\ell(B)}
\lf|S\lf(\sum_{k=0}^{\ell-5}
\sum_{j=1}^M\lz_{k,\,j}m_{k,\,j}\r)(x)\r|^p\,d\mu(x)\\
&\hs+\sum_{\ell=5}^\fz\int_{U_\ell(B)}\lf|S\lf(\sum_{k=\ell-4}^{\ell+4}
\sum_{j=1}^M\lz_{k,\,j}m_{k,\,j}\r)(x)\r|^p\,d\mu(x)\\
&\hs+\sum_{\ell=5}^\fz\int_{U_\ell(B)}\lf|S\lf(\sum_{k=\ell+5}^{\fz}
\sum_{j=1}^M\lz_{k,\,j}m_{k,\,j}\r)(x)\r|^p\,d\mu(x)+\sum_{\ell=0}^4\int_{U_\ell(B)}|S(b)(x)|^p\,d\mu(x)\\
&=:{\rm I}+{\rm II}+{\rm III}+{\rm IV}.
\end{align*}
Now we first estimate III. For any $x\in U_{\ell}(B)$ and
$\ell\in\nn\cap[5,\fz)$, by the Minkowski inequality,
we see that
\begin{align*}
&S\lf(\sum_{k=\ell+5}^{\fz}
\sum_{j=1}^M\lz_{k,\,j}m_{k,\,j}\r)(x)\\
&\hs\le\sum_{k=\ell+5}^{\fz}\sum_{j=1}^M|\lz_{k,\,j}|
\lf\{\int_{\Gamma(x)}\lf[\int_{\cx}|m_{k,\,j}(z)||D_t(y,z)|\,d\mu(z)\r]^2
\frac{d\mu(y)dt}{V_t(x)t}\r\}^{1/2}\\
&\hs\le\sum_{k=\ell+5}^{\fz}\sum_{j=1}^M|\lz_{k,\,j}|
\int_{B_{k,\,j}}\lf[\int_{\Gamma(x)}|m_{k,\,j}(z)|^2|D_t(y,z)|^2
\frac{d\mu(y)dt}{V_t(x)t}\r]^{1/2}\,d\mu(z)\\
&\hs\le\sum_{k=\ell+5}^{\fz}\sum_{j=1}^M|\lz_{k,\,j}|
\int_{B_{k,\,j}}|m_{k,\,j}(z)|[M_1(x,z)+M_2(x,z)]\,d\mu(z),
\end{align*}
where, for all $x\in\cx$ and $z\in B_{k,\,j}$ with $k\in\nn\cap[\ell+5,\fz)$
and $j\in\{1,\ldots,M\}$,
$$
{\rm M_1}(x,z):=\lf[\int_{\gfz{\Gamma(x)\cap}{\{(y,\,t)
\in\cx\times(0,\fz):\ t\le\frac{d(x,z)}2\}}}
|D_t(y,z)|^2\frac{d\mu(y)dt}{V_t(x)t}\r]^{1/2},
$$
$$
{\rm M_2}(x,z):=\lf[\int_{\gfz{\Gamma(x)\cap}{\{(y,\,t)\in
\cx\times(0,\fz):\ t>\frac{d(x,z)}2\}}}
|D_t(y,z)|^2\frac{d\mu(y)dt}{V_t(x)t}\r]^{1/2}.
$$
For any $x,\,y,\,z\in\cx$ satisfying $d(y, x)<t$ and $t\le d(x, z)/2$,
it is easy to see that
$$d(y, z)\ge d(x, z)-d(y, x)\ge\frac{1}{2}d(x, z).$$
It then follows, from this, (A1) and \eqref{10.1}, that
$$
{\rm M_1}(x,z)\ls\lf[\frac{1}{[V(x, z)]^2}\int^{d(x,\,z)/2}_0
\lf(\int_{B(x,t)}\frac{d\mu(y)}{V_t(x)}\r)
\lf(\dfrac{t}{d(x, z)}\r)^{\ez_2}\frac{dt}{t}\r]^{1/2}\ls\frac{1}{V(x, z)}
$$
and
\begin{align*}
{\rm M_2}(x,z)&\ls\lf[\int_{d(x,\,z)/2}^\fz\lf(\int_{B(x,t)}
\frac{d\mu(y)}{V_t(x)}\r)
\frac{dt}{[V_{2t}(z)]^2t}\r]^{1/2}\\
&\ls\frac{1}{V(x, z)}\lf[\int_{d(x,\,z)/2}^\fz
\frac{[d(x,z)]^{2\kz}}{(2t)^{2\kz} t}\,dt\r]^{1/2}\ls\frac{1}{V(x, z)}.
\end{align*}
Moreover, by $z\in B_{k,\,j}\st 2^{k+2}B\bh 2^{k-2}B$, $k\ge \ell+5$,
$x\in2^{\ell+2}B\bh 2^{\ell-2}B$, we have
$d(x,c_B)<2^{\ell+2}r_B$, $d(z,c_B)\ge2^{k-2}r_B\ge2^{\ell+3}r_B$
and
$$
d(x,z)\ge d(z,c_B)-d(x,c_B)\ge2^{\ell+3}r_B
-2^{\ell+2}r_B=2^{\ell+2}r_B>d(x,c_B),
$$
where $c_B$ and $r_B$ denote the center and the radius of $B$,
respectively. Thus, for all $x\in\cx$,
\begin{equation}\label{10.9}
S\lf(\sum_{k=\ell+5}^{\fz}
\sum_{j=1}^M\lz_{k,\,j}m_{k,\,j}\r)(x)
\ls\sum_{k=\ell+5}^{\fz}\sum_{j=1}^M|\lz_{k,\,j}|\frac1{V(x,c_B)}
\|m_{k,\,j}\|_{\lon}.
\end{equation}

From \eqref{10.9}, the H\"older inequality, \eqref{4.1}
and \eqref{10.1}, we deduce that
\begin{align*}
{\rm III}&\ls\sum_{\ell=5}^\fz\sum_{k=\ell+5}^\fz\sum_{j=1}^M
|\lz_{k,\,j}|^p\int_{U_\ell(B)}\frac1{[V(x,c_B)]^p}
\,d\mu(x)\|m_{k,\,j}\|^p_{\lon}\\
&\ls\sum_{\ell=5}^\fz\sum_{k=\ell+5}^\fz\sum_{j=1}^M
|\lz_{k,\,j}|^p\frac{V_{2^{\ell+2}r_B}(c_B)}{[V_{2^{\ell-2}r_B}(c_B)]^p}
\|m_{k,\,j}\|^p_{\lq}[\mu(B_{k,\,j})]^{p/q'}\\
&\ls\sum_{\ell=5}^\fz\sum_{k=\ell+5}^\fz\sum_{j=1}^M
|\lz_{k,\,j}|^p\lf[V_{2^{\ell+2}r_B}(c_B)\r]^{1-p}
2^{-k\ez p}\lf[\lz\lf(c_B,2^{k+2}r_B\r)\r]^{p-1}\\
&\ls\sum_{\ell=5}^\fz\sum_{k=\ell+5}^\fz\sum_{j=1}^M
2^{-k\ez p}|\lz_{k,\,j}|^p
\sim\sum_{j=1}^M\sum_{k=10}^\fz\sum_{\ell=5}^{k-5}
2^{-k\ez p}|\lz_{k,\,j}|^p\\
&\ls\sum_{j=1}^M\sum_{k=10}^\fz k
2^{-k\ez p}|\lz_{k,\,j}|^p\ls\sum_{k=0}^\fz\sum_{j=1}^M|\lz_{k,\,j}|^p
\sim|b|^p_{\mhpe}.
\end{align*}

In order to estimate I, for all $x\in\cx$, we write
\begin{align*}
&S\lf(\sum_{k=0}^{\ell-5}
\sum_{j=1}^M\lz_{k,\,j}m_{k,\,j}\r)(x)\\
&\hs\le\lf\{\int_{\Gamma(x)}
\lf|\int_{\cx}\sum_{k=0}^{\ell-5}
\sum_{j=1}^M\lz_{k,\,j}m_{k,\,j}(z)[D_t(y,z)-D_t(y,c_B)]
\,d\mu(z)\r|^2\frac{d\mu(y)dt}{V_t(x)t}\r\}^{1/2}\\
&\hs\hs+\lf\{\int_{\Gamma(x)}
\lf|\int_{\cx}\sum_{k=0}^{\ell-5}
\sum_{j=1}^M\lz_{k,\,j}m_{k,\,j}(z)D_t(y,c_B)
\,d\mu(z)\r|^2\frac{d\mu(y)dt}{V_t(x)t}\r\}^{1/2}
=:{\rm M}_3(x)+{\rm M}_4(x).
\end{align*}
To estimate ${\rm M}_3(x)$, by the Minkowski inequality,
we further write, for all $x\in\cx$,
\begin{align*}
{\rm M}_3(x)&\le\sum_{k=0}^{\ell-5}
\sum_{j=1}^M|\lz_{k,\,j}|\lf\{\int_{\Gamma(x)}
\lf[\int_{B_{k,\,j}}|m_{k,\,j}(z)||D_t(y,z)-D_t(y,c_B)|
\,d\mu(z)\r]^2\frac{d\mu(y)dt}{V_t(x)t}\r\}^{1/2}\\
&\le\sum_{k=0}^{\ell-5}
\sum_{j=1}^M|\lz_{k,\,j}|\int_{B_{k,\,j}}|m_{k,\,j}(z)|
\lf\{\int_{\Gamma(x)}|D_t(y,z)-D_t(y,c_B)|^2
\frac{d\mu(y)dt}{V_t(x)t}\r\}^{1/2}\,d\mu(z)\\
&\le\sum_{k=0}^{\ell-5}
\sum_{j=1}^M|\lz_{k,\,j}|\int_{B_{k,\,j}}|m_{k,\,j}(z)|
[{\rm M}_{3,\,1}(x,z)+{\rm M}_{3,\,2}(x,z)]\,d\mu(z),
\end{align*}
where, for all $x\in\cx$ and $z\in B_{k,\,j}$ with $k\in\zz_+\cap[0,\ell-5]$
and $j\in\{1,\ldots,M\}$,
$$
{\rm M}_{3,\,1}(x,z):=\lf[\int_{\gfz{\Gamma(x)\cap}{\{(y,\,t)
\in\cx\times(0,\fz):\ t\le\frac{d(x,c_B)}8\}}}
|D_t(y,z)-D_t(y,c_B)|^2\frac{d\mu(y)dt}{V_t(x)t}\r]^{1/2},
$$
$$
{\rm M}_{3,\,2}(x,z):=\lf[\int_{\gfz{\Gamma(x)\cap}{\{(y,\,t)\in
\cx\times(0,\fz):\ t>\frac{d(x,c_B)}8\}}}
|D_t(y,z)-D_t(y,c_B)|^2\frac{d\mu(y)dt}{V_t(x)t}\r]^{1/2}.
$$
Now we give some observations. For any $z\in B_{k,\,j}
\st 2^{k+2}B\bh 2^{k-2}B$,
$k\in\zz_+\cap[0,\ell-5]$, $j\in\{1,\ldots,M\}$
and $y\in\Gamma(x)$, we have
$d(z,c_B)<2^{k+2}r_B$ and $d(y,c_B)\ge d(x,c_B)-d(y,x)\ge2^{\ell-2}r_B-t$
and hence
$$
d(y,c_B)+t\ge2^{\ell-2}r_B\ge2^{k+3}r_B>2d(z,c_B).
$$
Meanwhile, for any $y\in\Gamma(x)\cap\{(y,\,t)
\in\cx\times(0,\fz):\ t\le\frac{d(x,c_B)}8\}$, we have
$$
d(y,c_B)\ge d(x,c_B)-d(x,y)\ge d(x,c_B)-\frac{d(x,c_B)}8=\frac78 d(x,c_B).
$$
From these observations, (A3) and \eqref{10.1}, it follows that,
for all $x\in\cx$ and $z\in B_{k,\,j}$ with
$k\in\zz_+\cap[0,\ell-5]$ and $j\in\{1,\ldots,M\}$,
\begin{align*}
{\rm M}_{3,\,1}(x,z)
&\ls\frac{2^kr_B}{d(x,c_B)}
\lf\{\int^{\frac{d(x,\,c_B)}8}_0\int_{B(x,t)}
\frac1{[V(c_B,x)]^2}\frac{t^{\ez_3-1}}{[d(x,c_B)]^{\ez_3}}
\frac{d\mu(y)dt}{V_t(x)}\r\}^{1/2}\\
&\sim\frac{2^kr_B}{d(x,c_B)}
\frac1{V(c_B,x)}.
\end{align*}
and
\begin{align*}
{\rm M}_{3,\,2}(x,z)
&\ls\frac{2^{k}r_B}{d(x,c_B)}
\lf\{\int^{\fz}_{\frac{d(x,\,c_B)}8}\int_{B(x,t)}
\frac1{[V_t(c_B)]^2}\frac{d\mu(y)dt}{V_t(x)t}\r\}^{1/2}\\
&\ls\frac{2^kr_B}{d(x,c_B)}
\frac1{V(c_B,x)}\lf\{\int^{\fz}_{\frac{d(x,\,c_B)}8}
\lf[\frac{t}{d(x,c_B)}\r]^{-2\kz}\frac{dt}{t}\r\}^{1/2}\\
&\sim\frac{2^kr_B}{d(x,c_B)}
\frac1{V(c_B,x)}.
\end{align*}
Combining the estimates of ${\rm M}_{3,\,1}(x,z)$ and
${\rm M}_{3,\,2}(x,z)$, we find that, for all $x\in\cx$,
$$
{\rm M}_3(x)\ls\sum_{k=0}^{\ell-5}
\sum_{j=1}^M|\lz_{k,\,j}|\|m_{k,\,j}\|_{\lon}
\frac{2^{k}r_B}{d(x,c_B)}
\frac1{V(c_B,x)}.
$$
By this, the H\"older inequality, \eqref{4.1}, \eqref{10.1}
and $p>\frac{\nu}{\nu+1}$, we conclude that
\begin{align*}
&\sum_{\ell=5}^{\fz}\int_{U_{\ell}(B)}[{\rm M}_3(x)]^p\,d\mu(x)\\
&\hs\ls\sum_{\ell=5}^{\fz}\sum_{k=0}^{\ell-5}
\sum_{j=1}^M|\lz_{k,\,j}|^p\|m_{k,\,j}\|^p_{\lon}
\int_{U_{\ell}(B)}\frac{2^{kp}r_B^{p}}{[d(x,c_B)]^{p}}
\frac1{[V(c_B,x)]^p}\,d\mu(x)\\
&\hs\ls\sum_{\ell=5}^{\fz}\sum_{k=0}^{\ell-5}
\sum_{j=1}^M|\lz_{k,\,j}|^p[\mu(B_{k,\,j})]^{p/q'}\|m_{k,\,j}\|^p_{\lq}
\frac{2^{kp}r_B^{p}}{2^{(\ell-2)p}r_B^{p}}
\frac{V_{2^{\ell+2}r_B}(c_B)}{[V_{2^{\ell-2}r_B}(c_B)]^p}\\
&\hs\ls\sum_{\ell=5}^{\fz}\sum_{k=0}^{\ell-5}
\sum_{j=1}^M|\lz_{k,\,j}|^p 2^{-k\ez p}2^{(k-\ell)p}
\lf[V_{2^{\ell+2}r_B}(c_B)\r]^{1-p}
\lf[\lz\lf(c_B,2^{k+2}r_B\r)\r]^{p-1}\\
&\hs\ls\sum_{\ell=5}^{\fz}\sum_{k=0}^{\ell-5}
\sum_{j=1}^M|\lz_{k,\,j}|^p 2^{-k\ez p}2^{(k-\ell)p}
2^{(\ell-k)(1-p)\nu}\\
&\hs\sim\sum_{\ell=5}^{\fz}\sum_{k=0}^{\ell-5}
\sum_{j=1}^M|\lz_{k,\,j}|^p 2^{-k\ez p}
2^{(\ell-k)[(1-p)\nu-p]}\\
&\hs\ls\sum_{k=0}^{\fz}
\sum_{j=1}^M|\lz_{k,\,j}|^p\sim|b|^p_{\mhpe}.
\end{align*}

By $\int_\cx b(x)\,d\mu(x)=0$ and some arguments similar
to those used in the estimate of \eqref{10.9}, we see that,
for all $x\in\cx$,
\begin{align*}
{\rm M}_4(x)&=\lf\{\int_{\Gamma(x)}
\lf|\int_{\cx}\sum_{k=\ell-4}^{\fz}
\sum_{j=1}^M\lz_{k,\,j}m_{k,\,j}(z)D_t(y,c_B)
\,d\mu(z)\r|^2\frac{d\mu(y)dt}{V_t(x)t}\r\}^{1/2}\\
&\ls\sum_{k=\ell-4}^{\fz}
\sum_{j=1}^M|\lz_{k,\,j}|\|m_{k,\,j}\|_{\lon}
\frac1{V(c_B,x)}.
\end{align*}
Again, by some arguments similar
to those used in the estimate of III, we know that
$$
\sum_{\ell=5}^{\fz}\int_{U_\ell(B)}[{\rm M}_4(x)]^p\,d\mu(x)
\ls|b|^p_{\mhpe}.
$$
Thus,
$$
{\rm I}\ls\sum_{\ell=5}^{\fz}\int_{U_\ell(B)}[{\rm M}_3(x)]^p\,d\mu(x)
+\sum_{\ell=5}^{\fz}\int_{U_\ell(B)}[{\rm M}_4(x)]^p\,d\mu(x)
\ls|b|^p_{\mhpe}.
$$

Then we turn to estimate II. We first write
\begin{align*}
{\rm II}&\le\sum_{\ell=5}^\fz\sum_{k=\ell-4}^{\ell+4}
\sum_{j=1}^M|\lz_{k,\,j}|^p\int_{U_\ell(B)}
\lf[S\lf(m_{k,\,j}\r)(x)\r]^p\,d\mu(x)\\
&\le\sum_{\ell=5}^\fz\sum_{k=\ell-4}^{\ell+4}
\sum_{j=1}^M|\lz_{k,\,j}|^p\int_{2B_{k,\,j}}
\lf[S\lf(m_{k,\,j}\r)(x)\r]^p\,d\mu(x)\\
&\hs+\sum_{\ell=5}^\fz\sum_{k=\ell-4}^{\ell+4}
\sum_{j=1}^M|\lz_{k,\,j}|^p\int_{U_\ell(B)\bh2B_{k,\,j}}
\lf[S\lf(m_{k,\,j}\r)(x)\r]^p\,d\mu(x)
=:{\rm II_1}+{\rm II_2}.
\end{align*}
By the H\"older inequality, the $\lq$-boundedness ($q\in(1,\fz)$)
of $S(f)$ and \eqref{4.1}, we conclude that
\begin{align*}
{\rm II}_1&\ls\sum_{\ell=5}^\fz\sum_{k=\ell-4}^{\ell+4}
\sum_{j=1}^M|\lz_{k,\,j}|^p\lf[\mu\lf(2B_{k,\,j}\r)\r]^{1-\frac{p}{q}}
\lf\|S\lf(m_{k,\,j}\r)\r\|^p_{\lq}\\
&\ls\sum_{\ell=5}^\fz\sum_{k=\ell-4}^{\ell+4}
\sum_{j=1}^M|\lz_{k,\,j}|^p\lf[\mu\lf(2B_{k,\,j}\r)\r]^{1-\frac{p}{q}}
\lf\|m_{k,\,j}\r\|^p_{\lq}\\
&\ls\sum_{\ell=5}^\fz\sum_{k=\ell-4}^{\ell+4}
\sum_{j=1}^M|\lz_{k,\,j}|^p2^{-k\ez p}
\ls\sum_{k=0}^{\fz}\sum_{j=1}^M|\lz_{k,\,j}|^p
\sim|b|^p_{\mhpe}.
\end{align*}

To estimate $\rm II_2$, fix $\ell\in\nn\cap[5,\fz)$,
$k\in\{\ell-4,\ldots,\ell+4\}$, $j\in\{1,\ldots,M\}$
and $x\in U_{\ell}(B)\bh 2B_{k,\,j}$.
Notice that, for any $z\in B_{k,\,j}$ and $x\notin 2B_{k,\,j}$,
$$
d(x,z)\ge d(x,c_{B_{k,\,j}})-d(z,c_{B_{k,\,j}})\ge
\frac12 d(x,c_{B_{k,\,j}}).
$$
By this, the Minkowski inequality and an argument similar
to that used in the estimate of \eqref{10.9}, we further obtain
\begin{align*}
S\lf(m_{k,\,j}\r)(x)
&\le\lf\{\int_{\Gamma(x)}\lf[\int_{B_{k,\,j}}|m_{k,\,j}(z)||D_t(y,z)|
\,d\mu(z)\r]^2\frac{d\mu(y)dt}{V_t(x)t}\r\}^{1/2}\\
&\le\int_{B_{k,\,j}}|m_{k,\,j}(z)|\lf[\int_{\Gamma(x)}|D_t(y,z)|^2
\frac{d\mu(y)dt}{V_t(x)t}\r]^{1/2}\,d\mu(z)\\
&\ls\frac1{V(x,c_{B_{k,\,j}})}\|m_{k,\,j}\|_{\lon}.
\end{align*}
From this, the H\"older inequality, \eqref{4.1} and \eqref{10.1},
we deduce that
\begin{align*}
{\rm II_2}&\ls\sum_{\ell=5}^\fz\sum_{k=\ell-4}^{\ell+4}
\sum_{j=1}^M|\lz_{k,\,j}|^p\int_{U_\ell(B)\bh2B_{k,\,j}}
\frac1{[V(x,c_{B_{k,\,j}})]^p}\,d\mu(x)\|m_{k,\,j}\|^p_{\lon}\\
&\ls\sum_{\ell=5}^\fz\sum_{k=\ell-4}^{\ell+4}
\sum_{j=1}^M|\lz_{k,\,j}|^p\lf[\int_{2^{\ell+7}B\bh2B_{k,\,j}}
\frac1{V(x,c_{B_{k,\,j}})}\,d\mu(x)\r]^p
\lf[\mu\lf(2^{\ell+2}B\r)\r]^{1-p}\\
&\hs\times\lf[\mu\lf(B_{k,\,j}\r)\r]^{p/q'}\|m_{k,\,j}\|^p_{\lq}\\
&\ls\sum_{\ell=5}^\fz\sum_{k=\ell-4}^{\ell+4}
\sum_{j=1}^M|\lz_{k,\,j}|^p
\lf[\wz K^{(\rho),\,p}_{B_{k,\,j},\,2^{\ell+6}B}\r]^p
\lf[\mu\lf(2^{\ell+2}B\r)\r]^{1-p}\\
&\hs\times2^{-k\ez p}\lf[\lz\lf(c_B,2^{k+2}r_B\r)\r]^{p-1}
\lf[\wz K^{(\rho),\,p}_{B_{k,\,j},\,2^{k+2}B}\r]^{-p}\\
&\ls\sum_{\ell=5}^\fz\sum_{k=\ell-4}^{\ell+4}
\sum_{j=1}^M|\lz_{k,\,j}|^p2^{-k\ez p}\sim|b|^p_{\mhpe},
\end{align*}
which, together with the estimate for ${\rm II_1}$,
implies that ${\rm II}\ls|b|^p_{\mhpe}$.

To estimate ${\rm IV}$, observe that
\begin{align*}
{\rm IV}&\le\sum_{\ell=0}^4\int_{U_\ell(B)}\lf|S\lf(\sum_{k=0}^{\ell+4}
\sum_{j=1}^M\lz_{k,\,j}m_{k,\,j}\r)(x)\r|^p\,d\mu(x)\\
&\hs+\sum_{\ell=0}^4\int_{U_\ell(B)}\lf|S\lf(\sum_{k=\ell+5}^{\fz}
\sum_{j=1}^M\lz_{k,\,j}m_{k,\,j}\r)(x)\r|^p\,d\mu(x)
=:{\rm IV}_1+{\rm IV}_2.
\end{align*}
By some arguments similar to those used in the estimates for $\rm II_1$
and $\rm III$, we respectively obtain
$${\rm IV}_1\ls|b|^p_{\mhpe}\quad {\rm and}\quad {\rm IV}_2\ls|b|^p_{\mhpe},$$
which, together with the estimates for $\rm III$, $\rm I$
and $\rm II$, completes the proof of
\textbf{Step 2} and hence Theorem \ref{t10.10}.
\end{proof}

\begin{remark}\label{r10.11}
(i) Let $\rho\in(1,\fz)$, $\gz\in[1,\fz)$ and
$\frac{\nu}{\nu+1}<p\le1<q\le 2$.
Combining Propositions \ref{p8.1} and \ref{p8.2}, and
Theorems \ref{t10.10} and \ref{t9.4},
we finally obtain
$$
\widehat{H}^{p,\,q,\,\gz}_{\rm atb,\,\rho}(\mu)
=H^{p}_{\rm at}(\mu)
=\wz{H}^{p,\,q,\,\gz}_{\rm atb,\,\rho}(\mu)
=\wz{H}^{p,\,q,\,\gz,\,\ez}_{\rm mb,\,\rho}(\mu)
$$
over an RD-space $(\cx,d,\mu)$ with $\mu(\cx)=\fz$.

(ii) It is still unclear whether
$\widehat{H}^{p,\,q,\,\gz}_{\rm atb,\,\rho}(\mu)\
({\rm or}\ H^{p}_{\rm at}(\mu))$ and
$\wz{H}^{p,\,q,\,\gz}_{\rm atb,\,\rho}(\mu)$ (or $\mhp$) coincide
or not for any $q\in(2,\fz]$ over RD-spaces $(\cx,d,\mu)$ with $\mu(\cx)=\fz$.

(iii) Let $(\mathcal{X},d,\mu):=(\mathbb{R}^D,|\cdot|,dx)$
with the $D$-dimensional Lebesgue measure $dx$,
$\rho\in(1,\fz)$, $\gz\in[1,\fz)$,
$\frac{D}{D+1}<p\le1<q<\fz$ and $\ez\in(0,\fz)$.
By Theorem \ref{t9.4}, we see that
$\widehat{H}^{p,\,q,\,\gz}_{\rm atb,\,\rho}(\mu)
=H^{p}(\rr^D)$. Now we deal with the relation between
$\wz{H}^{p,\,q,\,\gz}_{\rm atb,\,\rho}(\mu)$ and
$H^{p}(\rr^D)$.
To this end, consider \cite[Theorem 5.4]{bckyy} with
$\varphi(x,t):=t^p$ ($p\in(0,1]$)
and $L=-\Delta$, we notice that
$H^{p}_{-\Delta}(\rr^D)=H^{p}(\rr^D)$ (see \cite{fs}),
$q(\varphi)=1$, $r(\varphi)=\fz$,
$\ell(\varphi)=p=i(\varphi)$, $q\in(1,\fz)$
and $p_{-\Delta}=1$ therein. By $e^{t\Delta}1=1$, $-\Delta$ satisfying
assumptions \textbf{(H1)} and \textbf{(H2)} in \cite{bckyy},
and \cite[p.\,107, (6.16)]{bckyy}
(see also \cite[Remark 5.1]{jy}), we conclude that, for any
$(p,q,M)_L$-atom $a$ defined in \cite[Definition 5.2]{bckyy},
$$
\int_{\rr^D}a(x)\,dx=0.
$$
Thus, $a$ is a $(p,q)$-atom. From this,
\textbf{Step 2} of the proof of Theorem \ref{t10.10} and
the proof of \cite[Theorem 5.4]{bckyy},
we deduce that
\begin{equation}\label{10.10}
\lf(H^{p}(\rr^D)\cap L^2(\rr^D)\r)
=\wz{\mathbb{H}}^{p,\,q,\,\gz}_{\rm atb,\,\rho}(\mu).
\end{equation}
Thus, $H^{p}(\rr^D)=\nhp$.

Moveover, by \textbf{Step 2} of the proof of Theorem \ref{t10.10}
and \eqref{10.10}, we know that
$$
\pnhp\st\pmhp\st H^{p}(\rr^D)\cap L^2(\rr^D)=\pnhp.
$$
Thus, by this and Theorems \ref{t9.4} and \ref{t7.9},
we have
$$
\mhp=\nhp=H^{p}(\rr^D)
=\widehat{H}^{p,\,q,\,\gz}_{\rm atb,\,\rho}(\mu)
=\widehat{H}^{p,\,q,\,\gz,\,\ez}_{\rm mb,\,\rho}(\mu).
$$
\end{remark}

\Acknowledgements{The Second author is supported by the National
Natural Science Foundation of China (Grant No. 11301534) and Da Bei Nong Education Fund (Grant No.
1101-2413002). The third author is supported by the National Natural
Science Foundation of China (Grant Nos. 11171027 and 11361020),
the Specialized Research Fund for the Doctoral Program of Higher Education
of China (Grant No. 20120003110003) and the Fundamental Research Funds
for Central Universities of China (Grant Nos. 2012LYB26, 2012CXQT09, 2013YB60
and 2014KJJCA10). The fourth author is supported by the National Natural Science Foundation
of China (Grant No. 11101339).
The authors would
like to express their deep thanks to Professor Tuomas Hyt\"onen
for indicating them an important example (see Example \ref{e6.3} above).
Dachun Yang also wishes to express his sincerely thanks to Professor
Yasuo Komori-Furuya for a suggestive conversation on the topic of this article in a conference
held in Tokyo Metropolitan University of Japan in 2012.}



\begin{thebibliography}{99}
\bahao\baselineskip 11.5pt

\bibitem{b13} Bui T A. Boundedness of maximal
operators and maximal commutators on
non-homogeneous spaces. In: CMA Proceedings of
AMSI International Conference
on Harmonic Analysis and Applications (Macquarie University, February 2011).
Macquarie University, Australia, 2013, 45: 22-36%

\bibitem{bckyy} Bui T A, Cao J, Ky L D, Yang D, Yang S.
Musielak-Orlicz-Hardy spaces associated with operators satisfying
reinforced off-diagonal estimates. {Anal Geom Metr Spaces},
2013, 1: 69-129%

\bibitem{bd} Bui T A, Duong X T.
Hardy spaces, regularized BMO spaces and the boundedness of
Calder\'on-Zygmund operators on non-homogeneous spaces.
{J Geom Anal}, 2013, 23: 895--932%

\bibitem{c64} Campanato S. Theoremi di interpolazione per trasformazioni
che applicano $L^p$ in $C^{h,\,\az}$ (Italian).
{Ann Scuola Norm Sup Pisa (3)}, 1964, 18: 345-360%

\bibitem{cy} Cao J, Yang D. Hardy spaces
$H^p_L(\rn)$ associated with operators satisfying
$k$-Davies-Gaffney estimates. {Sci China Math}, 2012, 55: 1403-1440%

\bibitem{cmy} Chen W, Meng Y, Yang D. Calder\'on-Zygmund operators on
Hardy spaces without the doubling condition.
{Proc Amer Math Soc}, 2005, 133: 2671-2680%

\bibitem{c90} Christ M. A $T(b)$ theorem with remarks on analytic capacity
and the Cauchy integral.
{Colloq Math}, 1990, 60/61: 601--628%

\bibitem{co1} Coifman R R. A real variable characterization of $H^p$.
{Studia Math}, 1974, 51: 269--274%

\bibitem{co2} Coifman R R. Characterization of Fourier transforms
of Hardy spaces. {Proc Nat Acad Sci U S A}, 1974, 71: 4133--4134%

\bibitem{cw71} Coifman R R, Weiss G.
Analyse Harmonique
Non-Commutative sur Certains Espaces Homog\`enes. In:
Lecture Notes in Math 242. Berlin-New York: Springer-Verlag, 1971%

\bibitem{cw77} Coifman R R, Weiss G. Extensions of
Hardy spaces and their use in analysis.
{Bull Amer Math Soc}, 1977, 83: 569--645%

\bibitem{fs} Fefferman C, Stein E M.
$H^p$ spaces of several variables. {Acta Math}, 1972, 129: 137--193%

\bibitem{fyy3} Fu X, Yang Da, Yang Do.
The molecular characterization of
the Hardy space $H^1$ on non-homogeneous spaces
and its application. {J Math Anal Appl}, 2014, 410: 1028--1042%

\bibitem{fyy1} Fu X, Yang D, Yuan W.
Boundedness on Orlicz spaces for multilinear
commutators of Calder\'on-Zygmund operators
on non-homogeneous spaces. {Taiwanese J Math}, 2012, 16: 2203--2238%

\bibitem{fyy2} Fu X, Yang D, Yuan W.
Generalized fractional integrals and their commutators over
non-homogeneous spaces. {Taiwanese J Math}, 2014, 18: 509--557%

\bibitem{gly13} Gong R, Li J, Yan L.  A local version of Hardy spaces
associated with operators on metric spaces.
{Sci China Math}, 2013, 56: 315--330%

\bibitem{g08} Grafakos L. Classical Fourier Analysis.
Second edition. In: Graduate Texts in Mathematics 249.
New York: Springer, 2008%

\bibitem{g09} Grafakos L. Modern Fourier Analysis. Second edition.
In: Graduate Texts in Mathematics 250. New York: Springer, 2009%

\bibitem{gly} Grafakos L, Liu L, Yang D.
Maximal function characterization of
Hardy spaces on RD-spaces and their applications.
{Sci China Ser A}, 2008, 51: 2253--2284%

\bibitem{gs13} Guliyev V, Sawano Y. Linear and sublinear operators
on generalized Morrey spaces with non-doubling measures.
{Publ Math Debrecen}, 2013, 83: 303--327%

\bibitem{hmy06} Han Y, M\"uller D, Yang D.
Littlewood-Paley characterizations for Hardy
spaces on spaces of homogeneous type.
{Math Nachr}, 2006, 279: 1505--1537%

\bibitem{hmy08} Han Y, M\"uller D, Yang D.
A theory of Besov and Triebel-Lizorkin spaces on metric measure
spaces modeled on Carnot-Carath\'eodory spaces.
{Abstr Appl Anal}, 2008, Art. ID 893409: 250 pp%

\bibitem{he} Heinonen J. Lectures on Analysis on Metric Spaces.
New York: Springer-Verlag, 2001

\bibitem{hmy} Hu G, Meng Y, Yang D.
New atomic characterization of $H^1$ space with non-doubling
measures and its applications. {Math Proc Cambridge Philos Soc},
2005, 138: 151--171%

\bibitem{hmy13} Hu G, Meng Y, Yang D.
A new characterization of regularized BMO spaces on
non-homogeneous spaces and its applications.
{Ann Acad Sci Fenn Math}, 2013, 38: 3--27%

\bibitem{hmy12} Hu G, Meng Y, Yang D.
Weighted norm inequalities for multilinear
Calder\'on-Zygmund operators on non-homogeneous metric
measure spaces. {Forum Math}, 2014, 26: 1289--1322%

\bibitem{h10} Hyt\"onen T. A framework for non-homogeneous
analysis on metric spaces, and the RBMO space of Tolsa.
{Publ Mat}, 2010, 54: 485--504%

\bibitem{hlyy} Hyt\"{o}nen T, Liu S, Yang Da, Yang Do.
Boundedness of Calder\'on-Zygmund operators on non-homogeneous
metric measure spaces. {Canad J Math}, 2012, 64: 892--923%

\bibitem{hm} Hyt\"{o}nen T, Martikainen H.
Non-homogeneous $Tb$ theorem and random dyadic cubes on
metric measure spaces. {J Geom Anal}, 2012, 22: 1071--1107%

\bibitem{hyy} Hyt\"{o}nen T, Yang Da, Yang Do.
The Hardy space $H^1$ on non-homogeneous metric spaces.
{Math Proc Cambridge Philos Soc}, 2012, 153: 9--31%

\bibitem{jy} Jiang R, Yang D.
Orlicz-Hardy spaces associated with operators satisfying
Davies-Gaffney estimates. {Commun Contemp Math}, 2011, 13: 331--373%

\bibitem{la78} Latter R H. A characterization of $H^p(\rr^n)$
in terms of atoms. {Studia Math}, 1978, 62: 93--101%

\bibitem{ly10} Lin H, Yang D.
Spaces of type BLO on non-homogeneous metric measure.
{Front Math China}, 2011, 6: 271--292%

\bibitem{ly12} Lin H, Yang D. An interpolation
theorem for sublinear operators on
non-homogeneous metric measure spaces.
{Banach J Math Anal}, 2012, 6: 168--179%

\bibitem{ly14} Lin H, Yang D. Equivalent
boundedness of Marcinkiewicz integrals on
non-homogeneous metric measure spaces.
{Sci China Math}, 2014, 57: 123--144%

\bibitem{lmy} Liu S, Meng Y, Yang D.
Boundedness of maximal Calder\'on-Zygmund
operators on non-homogeneous metric measure spaces.
{Proc Roy Soc Edinburgh Sect A}, 2014, 144: 567--589%

\bibitem{lyy2} Liu S, Yang Da, Yang Do.
Boundedness of Calder\'on-Zygmund operators
on non-homogeneous metric measure spaces:
Equivalent characterizations.
{J Math Anal Appl}, 2012, 386: 258--272%

\bibitem{lhd} Liu Y, Huang J, Dong J.
Commutators of Calder\'on-Zygmund
operators related to admissible functions on spaces of homogeneous type
and applications to Schr\"odinger operators.
{Sci China Math}, 2013, 56: 1895--1913%

\bibitem{ms1} Mac\'ias R A, Segovia C.
Lipschitz functions on spaces of homogeneous type.
{Adv in Math}, 1979, 33: 257--270%

\bibitem{ns} Nakai E, Sawano Y. Orlicz-Hardy spaces
and their duals. {Sci China Math}, 2014, 57: 903--962%

\bibitem{ntv} Nazarov F, Treil S, Volberg A. The $Tb$-theorem
on non-homogeneous spaces. {Acta Math}, 2003, 190: 151--239%

\bibitem{ss13} Sawano Y, Shimomura T.
Sobolev's inequality for Riesz potentials of functions in generalized
Morrey spaces with variable exponent attaining the value 1
over non-doubling measure spaces. {J Inequal Appl}, 2013, 12: 19 pp%

\bibitem{st05} Sawano Y, Tanaka H.
Morrey spaces for non-doubling measures.
{Acta Math Sin (Engl Ser)}, 2005, 21: 1535--1544%

\bibitem{st07} Sawano Y, Tanaka H. The John-Nirenberg type inequality
for non-doubling measures. {Studia Math}, 2007, 181: 153--170%

\bibitem{st07-2} Sawano Y, Tanaka H.
Sharp maximal inequalities and commutators
on Morrey spaces with non-doubling measures.
{Taiwanese J Math}, 2007, 11:  1091--1112%

\bibitem{st09} Sawano Y, Tanaka H. Predual spaces of Morrey spaces
with non-doubling measures.
{Tokyo J Math}, 2009, 32: 471--486%

\bibitem{st09-2} Sawano Y, Tanaka H. Besov-Morrey spaces and
Triebel-Lizorkin-Morrey spaces for non-doubling measures.
{Math Nachr}, 2009, 282: 1788--1810%

\bibitem{s70} Stein E M. Singular
Integrals and Differentiability Properties
of Functions. Princeton, NJ: Princeton University Press, 1970

\bibitem{s93} Stein E M. Harmonic Analysis:
Real-variable Methods, Orthogonality, and Oscillatory Integrals.
Princeton, NJ: Princeton University Press, 1993

\bibitem{sw} Stein E M, Weiss G.
On the theory of harmonic functions of several variables.
I. The theory of $H^{p}$-spaces.
{Acta Math}, 1960, 103: 25--62%

\bibitem{tw} Taibleson M H, Weiss G.
The molecular characterization of certain Hardy spaces.
In: Representation theorems for Hardy spaces.
{Ast\'erisque}, 1980, 77: 67--149%

\bibitem{tl} Tan C, Li J. Littlewood-Paley theory on metric measure
spaces with non doubling measures and its applications.
Sci China Math, 2015, doi: 10.1007/s11425-014-4950-8, in press%

\bibitem{t01a} Tolsa X. $\rm{BMO}$, $H^1$, and Calder\'on-Zygmund
operators for non doubling measures. {Math Ann}, 2001, 319: 89--149%

\bibitem{t01b} Tolsa X. Littlewood-Paley theory and the $T(1)$ theorem with
non-doubling measures. {Adv Math}, 2001, 164: 57--116%

\bibitem{t03a} Tolsa X. The space $H^1$ for nondoubling measures
in terms of a grand maximal operator.
{Trans Amer Math Soc}, 2003, 355: 315--348%

\bibitem{t03b} Tolsa X. Painlev\'e's problem and the semiadditivity
of analytic capacity. {Acta Math}, 2003, 190: 105--149%

\bibitem{t04} Tolsa X. The semiadditivity of continuous analytic
capacity and the inner boundary conjecture. {Amer J Math},
2004, 126: 523--567%

\bibitem{t05} Tolsa X. Bilipschitz maps, analytic capacity,
and the Cauchy integral. {Ann of Math (2)}, 2005, 162: 1243--1304%

\bibitem{t14} Tolsa X. Analytic Capacity, the Cauchy Transform,
and Non-homogeneous Calder\'on-Zygmund Theory.
In: Progress in Mathematics 307. Cham: Birkh\"auser/Springer, 2014%

\bibitem{vw} Volberg A, Wick B D. Bergman-type
singular operators and the characterization
of Carleson measures for Besov-Sobolev spaces on the complex ball.
{Amer J Math}, 2012, 134: 949--992%

\bibitem{w73} Walsh T. The dual of $H^p (\rr^{n+1}_+)$ for $p<1$.
{Canad J Math}, 1973, 25: 567--577%

\bibitem{yyf} Yang Da, Yang Do, Fu X.
The Hardy space $H^1$ on non-homogeneous spaces
and its applications---a survey.
{Eurasian Math J}, 2013, 4: 104--139%

\bibitem{yyh}  Yang Da, Yang Do, Hu G.
The Hardy Space $H^1$ with
Non-doubling Measures and Their Applications.
In: Lecture Notes in Math 2084.
Berlin: Springer-Verlag, 2013%

\bibitem{yy12} Yang D, Yang S. Local Hardy spaces of
Musielak-Orlicz type and their applications.
{Sci China Math}, 2012, 55: 1677--1720%

\bibitem{y95} Yosida K. Functional Analysis. Berlin: Springer-Verlag, 1995

\end{thebibliography}
\end{document}